\documentclass{memo-l}
\usepackage{amssymb}
\usepackage[arrow,curve,matrix,cmtip,frame]{xy}
\SelectTips{cm}{}\newdir{ >}{{}*!/-5pt/\dir{>}}
\usepackage{xspace}

\ifx\hypersetup\undefined\else\hypersetup{linktoc=none}\fi

\def\sref$#1${\protect\ref{#1}}

\iffalse
\counterwithout{section}{chapter}
\else

\makeatletter\catcode`\:=11

\def\tockludge{\dokludge{lem:ACyc}{thm:fixsmflat}{thm:necircle}@}

\def\savechapter#1{%
\immediate\write\@auxout{\unexpanded{\global\@namedef{rchap@#1}}{\arabic{chapter}}}}
\def\mlabel#1{\savechapter{#1}\oldlabel{#1}}
\def\mlabelindisplay#1{\savechapter{#1}\oldlabel@in@display{#1}}
\def\mref#1{{\def\tempo{\arabic{chapter}}%
\ifcsname rchap@#1\endcsname
  \def\tempt{\csname rchap@#1\endcsname}%
\else
  \def\tempt{0}%
\fi
\if\tempo\tempt\relax\else\ifcase\tempt\relax\or
\or II.\or III.\or IV.\or V.\or VI.\or VII.\or VIII.\or 
IX.\or X.\or XI.\or XII.\or XIII.\or XIV.\or XV.\or XVI.\or 
XVII.\or XVIII.\or XIX.\else\relax\fi\fi}\oldref{#1}}

\def\dokludge#1{\ifx#1@\def\dokludgea##1{}\fi\dokludgea{#1}}
\def\dokludgea#1{\expandafter\ifx\csname rchap@#1\endcsname\relax\@namedef{rchap@#1}{0}\fi 
\expandafter\let\csname oldrchap@#1\expandafter\endcsname\csname rchap@#1\endcsname\@namedef{rchap@#1}{0}\dokludge}
\def\untockludge{%
\def\dokludgea##1{\@namedef{rchap@##1}{\csname oldrchap@##1\endcsname}\dokludge}%
\tockludge}
\fi
\catcode`\:=12\makeatother

\DeclareMathAlphabet{\mathscr}{U}{rsfs}{m}{n}

\newcommand{\Spectra}{\aS}
\newcommand{\Com}{\aC om}
\newcommand{\Alg}{\aA lg}
\newcommand{\Mod}{\aM od}
\newcommand{\ppc}{$p$pc\xspace}

\newcommand{\dR}{\mathbf{R}}

\newcommand{\sdde}{\rightsquigarrow}

\newcommand{\Mto}{\buildrel{\scriptscriptstyle M}\over\to}
\newcommand{\WMto}{\buildrel{\scriptscriptstyle W}\over\to}
\newcommand{\ssma}{({\sma})}

\newcommand{\KG}{K}
\newcommand{\LQ}{P}
\newcommand{\PLambda}{N}
\newcommand{\PGamma}{G}

\newcommand{\EF}[1][{\Gamma}]{EF}

\newcommand{\perm}{\pi}
\newcommand{\permp}{\pi!}

\let\comp\circ

\newcommand{\Nice}{\aE_\Phi}

\newcommand{\aACyc}{\aA^{\Cyc}}
\newcommand{\aApCyc}{\aA^{\Cyc}_{p}}

\newcommand{\ssdot}{\bullet}

\newcommand{\subdot}{_\ssdot}

\let\iso\cong
\let\sma\wedge
\renewcommand{\to}{\mathchoice{\longrightarrow}{\rightarrow}{\rightarrow}{\rightarrow}}

\newcommand{\sto}{\rightarrow}
\newcommand{\overto}[1]{\xrightarrow{\,#1\,}}
\newcommand{\overfrom}[1]{\xleftarrow{\,#1\,}}

\let\catsymbfont\mathcal
\newcommand{\aA}{{\catsymbfont{A}}}
\newcommand{\aC}{{\catsymbfont{C}}}
\newcommand{\aE}{{\catsymbfont{E}}}
\newcommand{\aF}{{\catsymbfont{F}}}
\newcommand{\aH}{{\catsymbfont{H}}}
\newcommand{\aI}{{\catsymbfont{I}}}
\newcommand{\aJ}{{\catsymbfont{J}}}
\newcommand{\aM}{{\catsymbfont{M}}}

\newcommand{\aQ}{{\catsymbfont{Q}}}
\newcommand{\aS}{{\catsymbfont{S}}}
\newcommand{\aU}{{\catsymbfont{U}}}
\newcommand{\aV}{{\catsymbfont{V}}}

\newcommand{\bC}{{\mathbb{C}}}

\newcommand{\bN}{{\mathbb{N}}}
\newcommand{\bP}{{\mathbb{P}}}
\newcommand{\bR}{{\mathbb{R}}}
\newcommand{\bS}{{\mathbb{S}}}
\newcommand{\bT}{{\mathbb{T}}}
\newcommand{\bZ}{{\mathbb{Z}}}

\def\quickop#1{\expandafter\DeclareMathOperator\csname
#1\endcsname{#1}}
\quickop{id}\quickop{Id}
\quickop{holim}
\quickop{hocolim}
\quickop{colim}
\quickop{Fix}
\quickop{fin}
\quickop{Eq}\quickop{hoEq}\let\hoeq\hoEq
\quickop{op}
\quickop{Hom}
\quickop{Sym}
\quickop{reg}
\quickop{cyc}\let\Cyc\cyc
\newcommand{\pcyc}{p\cyc}
\quickop{All}
\quickop{Aut}
\quickop{lcm}
\quickop{Ind}
\quickop{Ad}

\newtheorem{thm}{Theorem}
\numberwithin{thm}{section}

\newtheorem{main}{Theorem}
\newtheorem*{addm}{Addendum~\themain}
\newtheorem*{mcor}{Corollary}

\newtheorem{cor}[thm]{Corollary}
\newtheorem{lem}[thm]{Lemma}
\newtheorem{prop}[thm]{Proposition}

\theoremstyle{definition}
\newtheorem{defn}[thm]{Definition}
\newtheorem{cons}[thm]{Construction}
\newtheorem{notn}[thm]{Notation}
\newtheorem{ter}[thm]{Terminology}
\newtheorem*{ter*}{Terminology}

\theoremstyle{remark}
\newtheorem{rem}[thm]{Remark}
\newtheorem{example}[thm]{Example}

\renewcommand*{\themain}{\Alph{main}}

\makeatletter\let\c@equation\c@thm\makeatother

\ifx\texorpdfstring\undefined\def\texorpdfstring#1#2{#1}\fi

\newcommand{\term}[1]{\textit{#1}}

\begin{document}

\title[Cyclotomic spectra 10 years later]%
{The homotopy theory of cyclotomic spectra\\10 years later}

\author{Andrew J. Blumberg}
\address{Department of Mathematics, Columbia University, 
New York, NY \ 10027}
\email{blumberg@math.columbia.edu}
\thanks{The first author was supported in part by NSF grants
DMS-2052970, DMS-2104420}
\author{Michael A. Mandell}
\address{Department of Mathematics, Indiana University,
Bloomington, IN \ 47405}
\email{mmandell@iu.edu}
\thanks{The second author was supported in part by NSF grants
DMS-2052846, DMS-2104348}


\begin{abstract}
This paper studies the foundations of the geometric fixed point
functor in multiplicative equivariant stable homotopy theory.  We
introduce a new class of equivariant orthogonal spectra called
generalized orbit desuspension spectra and analyze their homotopical
behavior with respect to the geometric fixed point functor and
especially the interaction with smash products and symmetric powers.
This analysis leads to several new foundational results, including the
construction of the derived functor of geometric fixed points on
equivariant commutative ring orthogonal spectra and its comparison to
the derived functor on the underlying equivariant orthogonal
spectra. In addition this theory provides foundations for the category 
of commutative ring pre-cyclotomic spectra and a formula for the
derived space of maps in this category (a formula needed in the
authors' paper with Yuan on relative $TC$). Finally, we prove a new
multiplicative tom Dieck splitting for equivariant commutative ring
spectra obtained as the pushforward of non-equivariant commutative
ring spectra.
\end{abstract}

\maketitle
\makeatletter
\let\oldlabel\label
\let\oldlabel@in@display\label@in@display
\let\label\mlabel
\let\label@in@display\mlabelindisplay
\let\oldref\ref
\let\ref\mref
\makeatother

\tockludge
\addtocontents{toc}{\protect\setcounter{tocdepth}{1}}
\tableofcontents
\untockludge

\chapter{Introduction}

The last 50 years of work on equivariant stable homotopy theory has
led to spectacular accomplishments both internally
(e.g.,~\cite{Carlsson-SegalConjecture,GreenleesMay-MUCompletion}) and
externally (e.g.,~\cite{HopkinsKuhnRavenel,
FreedHopkinsTeleman-Annals, FreedHopkinsTeleman-JAMS, BHM, HMAnnals,
BMS, HHR}.
However, despite monumental work codifying the theory by
Lewis-May-Steinberger~\cite{LMS} (later extended by
Mandell-May~\cite{MM}), the foundations of the theory,
particularly the treatment of multiplicative phenomena, remain
unsettled.  Notably, the spectacular recent advances of
Hill-Hopkins-Ravenel~\cite{HHR} depended on a new foundational
construction, the multiplicative norm, and this has led to an array of
new computations and insights (e.g., see~\cite{Slicey, BlumbergHill-NormsTransfers}).

A particularly glaring gap in the current foundations is the
interaction between geometric fixed points and the homotopy theory of
equivariant $E_{\infty}$ ring spectra. In particular, the following
basic questions do not have good answers.
\begin{enumerate}
\item Is there a derived functor of geometric fixed points on
equivariant commutative ring orthogonal spectra?

\item If so, does it agree with the derived functor on the underlying
equivariant orthogonal spectra?
\end{enumerate}
This issue was cleverly sidestepped in Hill-Hopkins-Ravenel, and arose
again in our work~\cite{BM-cycl} on the homotopy theory of cyclotomic
spectra.   Specifically, lack of
control on the derived functor of geometric fixed points
in the multiplicative setting has prevented the detailed study of
commutative ring objects in cyclotomic spectra.
For work along these lines, the existence of derived functors is
rarely enough and what is needed is some kind of additional
homotopical control, for example, a model structure where the
geometric fixed points have good properties.

The purpose of this paper is to return to these foundations and fill
in the gaps.  In particular for a compact Lie group $G$ and a closed
subgroup $H<G$, we construct a closed model structure $\aM$ (that
depends on $H$) on the category of $G$-equivariant commutative ring
orthogonal spectra with the following properties, positively answering
the questions above. 
\begin{enumerate}
\item The point-set geometric fixed point functor $\Phi^{H}$ preserves
weak equivalences between $\aM$-cofibrant objects.
\item If $R$ is $\aM$-cofibrant and $X\to R$ is a cofibrant
approximation in the standard model structure on orthogonal spectra,
then the induced map $\Phi^{H}X\to \Phi^{H}R$ is a weak equivalence.
\end{enumerate}
Here (ii) implies (i) and is established in
Theorem~\ref{main:commod} below.  The $\aM$-cofibrant objects
also have the following further properties.
\begin{enumerate}
\addtocounter{enumi}{2}
\item If $R$ is $\aM$-cofibrant, then $\Phi^{H}R$ is flat for the
smash product of $WH$-equivariant orthogonal spectra and in particular
flat for the smash product of non-equivariant orthogonal spectra. 
(Proposition~\ref{prop:mainflat})

\item If $R$ is $\aM$-cofibrant and $X$ is any $G$-equivariant 
orthogonal spectrum, then the lax monoidal structure map
\[
\Phi^{H}R\sma \Phi^{H}X\to \Phi^{H}(R\sma X)
\]
is an isomorphism. (Proposition~\ref{prop:mainsym})
\end{enumerate}
We also mention here the following foundational property of the smash
product, which is not directly related to the $\aM$ model structure
and should have been observed earlier; it is essential to our work
with Yuan in~\cite{BMY-MUP} but does not seem to be well-known.
\begin{enumerate}
\addtocounter{enumi}{4}
\item If $M,N$ are $R$-modules with
$M$ cofibrant (in the usual model structure), then the lax symmetric
monoidal structure map 
\[
\Phi^{H}M\sma_{\Phi^{H}R}\Phi^{H}N\to \Phi^{H}(M\sma_{R}N)
\]
is an isomorphism. (Proposition~\mref{prop:mainsmashover})
\end{enumerate}

We expect these new foundations to have a wide variety of
applications; we already deploy them here in a number of interesting
directions. A particularly striking new structural result about
multiplicative equivariant stable homotopy theory establishes a
previously unnoticed multiplicative version of the tom Dieck
splitting.  The tom Dieck splitting, a foundational result in
equivariant stable homotopy theory, is usually stated for suspension
spectra, but it has a version for genuine $G$-equivariant spectra
promoted from non-equivariant spectra, whether or not suspension
spectra.  In the multiplicative context, we consider the left derived
functor $\epsilon^{*}$ that makes a non-equivariant
commutative ring orthogonal spectrum into a
$G$-equivariant commutative ring orthogonal spectrum.  We prove the
following theorem in Chapter~\ref{chap:tds}.

\begin{main}[Multiplicative tom Dieck Splitting]\label{main:tds}
Let $G$ be a finite group and let $\epsilon^{*}$ denote the left derived
functor from the homotopy category of 
non-equivariant commutative ring orthogonal spectra to the homotopy
category of genuine $G$-equivariant commutative ring orthogonal spectra.  There
exists a natural weak equivalence
\[
\Phi^{G}(\epsilon^{*}R) \simeq R\otimes (\coprod BWH)
\]
where $\otimes$ denotes the left derived tensor (topologically indexed
colimit) in the category of commutative ring orthogonal spectra and
the disjoint union is over conjugacy classes of subgroups $H<G$.
\end{main}

Our original motivation for this paper was to provide the technical
foundations for our work~\cite{BMY-MUP} with Yuan on relative
cyclotomic structures.  That project requires us to construct the
homotopy theory of commutative ring objects in cyclotomic spectra and
to compute the derived mapping spaces.  The existence of a model
structure presenting the correct $(\infty,1)$-category of commutative
ring objects in cyclotomic spectra (stated as
Theorem~\ref{main:commod} below) is an immediate consequence of our 
foundational work, and using this we prove the following
characterization of the derived mapping spaces which provides the
commutative ring analogue of the main theorem on 
mapping spectra in~\cite[5.12]{BM-cycl}.

\begin{main}\label{main:map}
If 
$A$ and $B$ are commutative ring $p$-pre-cyclotomic spectra where the
underlying $\bT$-equivariant commutative ring orthogonal spectra of
$A$ and $B$ are cofibrant and fibrant, respectively, in the model
structure of Theorem~\ref{main:Tmod}.(ii) below, then the derived space of maps
of commutative ring $p$-pre-cyclotomic spectra from $A$ to $B$ is
represented by the homotopy equalizer 
\[
\xymatrix@-1pc{%
\Com^{\bT}(A,B)\ar@<-.5ex>[r]\ar@<.5ex>[r]&\Com^{\bT}(\Phi A,B)
}
\]
where $\Com^{\bT}$ denotes the point-set mapping space in
$\bT$-equivariant commutative ring orthogonal spectra.  Here one map
is induced by the cyclotomic structure map $\Phi A\to A$ 
and the other map is induced by applying $\Phi$ and post-composing
with the cyclotomic structure map $\Phi B\to B$.
\end{main}

The foundation for all the theorems above is the identification and
study of a new class of $G$-equivariant orthogonal spectra that behave
well with respect to 
smash product and the point-set geometric fixed point functors.  We
call these spectra ``generalized orbit desuspension spectra''.  A
generalized orbit desuspension spectrum is a certain kind of
equivariant orthogonal spectrum $J\eta$ constructed from an
equivariant vector bundle $\eta$.  (See Terminology~\ref{term:bundle}
and Construction~\ref{cons:orbitds} for specifics on the hypotheses on
$\eta$ and the construction $J\eta$.)  The orbit desuspension spectra
$G/H_{+}\sma F_{V}S^{0}$ and the HHR complete cells
$G\sma_{K}F_{W}S^{0}$ are examples that come from the trivial
$G$-equivariant vector bundles $G/H\times V\to G/H$ and the extended
$G$-equivariant vector bundles $G\times_{K}W\to G/K$, respectively.
The multiplicative homotopy theory of $G$-equivariant orthogonal
spectra works because orbit desuspension spectra are flat for the
smash product of $G$-equivariant orthogonal spectra (smashing with one
preserves weak equivalences between arbitrary $G$-equivariant
orthogonal spectra).  The point-set geometric fixed point functors are
useful because they are homotopically correct on the orbit
desuspension spectra and they commute up to isomorphism with the smash
product of an orbit desuspension spectrum and an arbitrary
$G$-equivariant orthogonal spectrum (the latter being the main
technical result of~\cite[App.~A]{BM-cycl} that makes the homotopy
theory of pre-cyclotomic spectra work).  Under mild hypotheses,
generalized orbit desuspension spectra share these properties. Details
of the theory take up the bulk of the paper, starting in
Chapter~\ref{chap:bundleone}.

We now move on to a more detailed exposition of the homotopical
results of the paper, specifically the construction of the model
structures, with precise statements of some the most important
results.  In this exposition and most of the rest of the paper, we use
the following terminology: 
\begin{ter*}\label{ter}\indent 
\begin{itemize}
\item ``spectrum'' means orthogonal spectrum; for a compact Lie group
$G$, ``$G$-spectrum'' means $G$-equivariant orthogonal spectrum.
Often these refer to point set categories, but when homotopy theory is
implied, we understand the complete universe theory unless explicitly
indicated otherwise.
\item We abbreviate pre-cyclotomic spectrum to ``pc spectrum'' and
$p$-pre-cyclotomic spectrum to ``\ppc spectrum''.  These concepts are
defined in~\cite{BM-cycl}, and we give only minimal review here.
\item For spectra and $G$-spectra, we prepend ``ring'' or
``commutative ring'' to indicate the categories of associative monoids
or commutative monoids for the smash product (e.g., we say
``ring $G$-spectra'' for associative ring $G$-equivariant orthogonal spectra).  Similar conventions apply to module categories.
\item For \ppc spectra and pc spectra, we review ring, commutative
ring, and module categories in Section~\ref{sec:precyc}.
\end{itemize}
\end{ter*}

We start the detailed exposition with our results on commutative ring
\ppc spectra.  In~\cite{BM-cycl}, we introduced a homotopy
theory for cyclotomic spectra and \ppc spectra in order to justify a
construction of
derived mapping spectra and to prove corepresentability of
$TC$. (Corepresentability of $TC(-;p)$, now often treated as a
definition rather than a theorem, was first conjectured
in~\cite[7.9]{Kaledin-ICM2010} and proved in~\cite[1.4]{BM-cycl}.) 
The definition of the category of
\ppc spectra depends on a version of the point-set
multiplicative geometric $C_{p}$ fixed point functor $\Phi^{C_{p}}$.
This functor goes from $\bT$-spectra to
$\bT/C_{p}$-spectra, but using the $p$th root isomorphism $\rho \colon
\bT\iso \bT/C_{p}$, we get a version that is an endofunctor functor on
$\bT$-spectra, which we denote as~$\Phi$:   
\[
\Phi := \rho^{*}\Phi^{C_{p}}\colon \Spectra^{\bT}\to \Spectra^{\bT}. 
\]

A key fact that made the mapping spectra analyzable is that the
endofunctor $\Phi$ takes cofibrant $\bT$-spectra to
cofibrant $\bT$-spectra (in the standard stable model
structures and appropriate family variants of
\cite[IV\S6]{MM}). The same does not hold in the setting of
commutative ring $\bT$-spectra: for a
commutative ring $\bT$-spectrum $R$, $\Phi R$ has the canonical structure of
a commutative ring $\bT$-spectrum, but even when $R$ is
cofibrant as a commutative ring $\bT$-spectrum, $\Phi R$
may not be.  (More generally, $\Phi$ does not preserve cofibrant
objects in the model structures introduced in~\cite{Stolz-Thesis},
\cite[3.3.12]{BrunDundas}.)  We prove the following theorem in
Section~\ref{sec:circle}:

\begin{main}\label{main:Tmod}
Let $\aF$ be either the family
of $\aF_{\fin}$ of all finite subgroups of $\bT$, or the family
$\aF_{p}$ of finite subgroups of $\bT$ of order a power of $p$.
There exists a symmetric monoidal topological model structure
$\aM_{\aF}$ on $\bT$-spectra which has weak 
equivalences the genuine $\aF$-equival\-ences, is enriched over the
standard model structure on non-equivariant spectra (in the sense
of~\cite[3.3]{LewisMandell2}), and which satisfies the  
following properties:
\begin{enumerate}
\item\label{intromainq} The endofunctor $\Phi$ preserves cofibrations and acyclic cofibrations.
\item\label{intromaincommodel} The category of commutative ring
$\bT$-spectra admits a topological 
model structure with fibrations and weak equivalences created by the
forgetful functor to $\bT$-spectra.
\item The endofunctor $\Phi$ preserves cofibrant objects in the model
structure in~(\ref{intromaincommodel}).
\item\label{intromaincof} If $R$ is cofibrant in the model structure
of~(\ref{intromaincommodel}), then it is cofibrant in the
undercategory of $\bS$ in the model structure $\aM_{\aF}$, i.e., the unit
$\bS\to R$ is a cofibration in $\aM_{\aF}$.
\end{enumerate}
Moreover, the forgetful functor from $\bT$-spectra to non-equivariant
spectra takes cofibrations and fibrations in this model structure to
cofibrations and fibrations in the positive convenient $\Sigma$-model
structure of Definition~\ref{defn:Sigmodel}.
\end{main}

The positive convenient $\Sigma$-model structure is the orthogonal
spectrum analogue of Shipley's ``$S$-model structure'' on symmetric
spectra from~\cite{Shipley-Convenient} and a variant of  Stolz's
``$\bS$-model structure'' from~\cite[\S1.3.3]{Stolz-Thesis}
or~\cite[3.3.12]{BrunDundas} for the trivial group.  See
Section~\ref{sec:neACyc} for details. 

By the K.\ Brown lemma, (\ref{intromainq}) implies that $\Phi$
preserves weak equivalences between cofibrant objects; combining
with~(\ref{intromaincof}), we see that for $R$ a cofibrant $\bT$-equivariant commutative
ring spectrum and $X\to R$ a cofibrant approximation in the model
structure $\aM_{\aF}$, the map $\Phi X\to \Phi R$ is a weak equivalence.

The model structures on commutative ring $\bT$-spectra in
the theorem lead to corresponding model structures on modules and
algebras.

\begin{mcor}
Let $R$ be a commutative ring $\bT$-spectra.  For the
model structure $\aM_{\aF}$ constructed in the proof of
Theorem~\ref{main:Tmod}, the categories of $R$-modules and
$R$-algebras admit model structures with the fibrations and weak
equivalences created in $\aM_{\aF}$.  Moreover, if $R$ is cofibrant in the
model structure of~(\ref{intromaincommodel}), then cofibrant
$R$-modules are cofibrant in the model structure $\aM_{\aF}$ and cofibrant
$R$-algebras are cofibrant in the model
structure $\aM_{\aF}$ on the undercategory of~$\bS$. 
\end{mcor}

Recall that a \ppc spectrum consists of a $\bT$-spectrum $X$ together
with a \term{cyclotomic structure map}
\[
\Phi X\to X
\]
and a map of \ppc spectra is a map of spectra which
makes the diagram of cyclotomic structure maps commute.  As shown
in~\cite[5.3]{BM-cycl}, the category of \ppc
spectra have a model structure with fibrations and weak equivalences
created by the forgetful functor to the standard model structure on
$\bT$-spectra with the $\aF_{p}$-equivalences.  Using the model
structure of Theorem~\ref{main:Tmod}, we get a new model structure on
\ppc spectra and a compatible model structure on the 
category of commutative ring \ppc spectra.

\begin{main}\label{main:cycmod}
There exists a topological model structure on the category of \ppc
spectra with weak equivalences the genuine $\aF_{p}$-equivalences and
fibrations created in the model category of
$\bT$-spectra of Theorem~\ref{main:Tmod}. Moreover, it is enriched
over the standard model structure on non-equivariant spectra, and it
satisfies the following additional properties:
\begin{enumerate}
\item The forgetful functor to $\bT$-spectra sends cofibrations to
cofibrations in the model structure of Theorem~\ref{main:Tmod}.
\item The smash product $\sma$ extends to a symmetric monoidal
product on the subcategory of cofibrant \ppc
spectra in this model structure.
\item The category of commutative ring \ppc
spectra admits a topological model structure where fibrations and weak
equivalences are created by the forgetful functor to
\ppc spectra, and the forgetful functor to
\ppc spectra under $\bS$ preserves cofibrant objects.
\end{enumerate}
\end{main}

Because of the issues with $\Phi$, categories of modules and algebras
in \ppc spectra are hard to study in general, but when
the ring \ppc spectrum $R$ has underlying \ppc
spectrum cofibrant in the undercategory of $\bS$, the usual arguments
can be adapted to work.  As a
particular case, we have the following addendum to the previous
theorem for module and algebra model structures. 

\begin{addm}Let $R$ be a cofibrant commutative ring \ppc
spectrum.
\begin{enumerate}
\item The category of $R$-module \ppc spectra admits a
model structure where fibrations and weak equivalences are created by
the forgetful functor to \ppc spectra, and the forgetful
functor preserves cofibrant objects. The model structure is enriched
over the standard model structure on non-equivariant spectra.
Moreover,  $\sma_{R}$ extends to a symmetric monoidal product on the
subcategory of cofibrant $R$-module \ppc spectra.
\item The category of $R$-algebra \ppc spectra admits a
model structure where fibrations and weak equivalences are created by
the forgetful functor to \ppc spectra, and the forgetful
functor to the undercategory of $\bS$ preserves cofibrant objects.
\item The category of commutative $R$-algebra \ppc spectra admits a
model structure where fibrations and weak equivalences are created by
the forgetful functor to \ppc spectra, and the forgetful
functor to the undercategory of $\bS$ preserves cofibrant objects.
\end{enumerate}
\end{addm}

The particulars of Theorems~\ref{main:Tmod} and~\ref{main:cycmod}
allow us to extend the main result on mapping spaces of~\cite{BM-cycl}
to the multiplicative context. We
stated the case for commutative $R$-algebras as Theorem~\ref{main:map}
above; see Section~\ref{sec:precyc} for the other cases and for other
variants of cyclotomic spectra.

The fact that the geometric fixed points $\Phi^{C_{p}}$ admits a
refinement into an endofunctor $\Phi$ is a special feature of the
circle group $\bT$.  As we mentioned above, we have theorems for the
homotopical properties of geometric $H$ fixed points of commutative
ring $G$-spectra more generally, but they are not as clean to state.
The easiest statement that looks useful is the following. It is proved
at the end of Section~\ref{sec:commod}; see
Theorem~\ref{thm:Finexample} for another theorem along these lines.

\begin{main}\label{main:commod}
Let $G$ be a compact Lie group and $\aH$ a finite set of closed
subgroups of $G$.  There exists a model structure on the category of
commutative ring $G$-spectra with the genuine weak
equivalences, having the following property: when $R$ is cofibrant in this model
structure, $X\to R$ a cofibrant approximation in the (genuine) standard or positive stable
model structure on $G$-spectra, and $H\in \aH$, the induced
map on the point-set multiplicative geometric fixed point functor
$\Phi^{H}X\to \Phi^{H}R$ is a weak equivalence.  In particular, for
$H\in \aH$, $\Phi^{H}$ preserves weak equivalences between cofibrant
objects in this model structure.
\end{main}

Theorem~\ref{main:commod} establishes the existence of a left derived
functor for geometric fixed points, viewed as a functor from
commutative ring $G$-spectra to commutative ring 
$WH$-spectra and a natural isomorphism from this derived functor
composed with the forgetful functor to the genuine $WH$-stable
category and the derived geometric fixed point functor on the
underlying $G$-spectrum.

\subsection*{Technical conventions}

In addition to the terminology for spectra and pre-cyclotomic spectra
set forth on page~\pageref{ter} above, we also have the following
technical conventions. 

We use the phrase ``orthogonal $G$-representation'' to mean a finite dimensional
real inner product space with action of $G$ by isometries, rather than
the isomorphism class of such an object (as the phrase is often used
in algebra).  For an orthogonal $G$-representation $V$, we write $[V]$
for its isomorphism class.

For orthogonal $G$-representations $V,W$, we write $V<W$ to denote
that $V$ is a $G$-stable vector subspace of $W$, not necessarily a
proper subspace.  We write $[V]<[W]$ to indicate that $V$ is
isomorphic to a $G$-stable vector subspace of $W$ (not necessarily a
proper subspace).

We write $H<G$ to indicate that $H$ is a closed subgroup of $G$, not
necessarily a proper subgroup.

In Chapters~\ref{chap:bundleone} and~\ref{chap:bundletwo}, $\Gamma$ denotes a compact Lie
group.  Wherever it appears in this paper, the phrase
``$\Gamma$-space'' means space with continuous $\Gamma$ action, and not
$\Gamma$-space in the sense of Segal. 

\subsection*{Acknowledgments}
This project was motivated by the authors' work with Allen Yuan on
relative cyclotomic spectra; we would not have discovered the results
here without their necessity in that work.  We would like to thank our
collaborators on~\cite{ABGHLM}, Vigleik Angeltveit, Teena Gerhardt, Mike Hill, and Tyler
Lawson for many discussions of related topics.  The first author also like to
particularly thank the Mikes Hill and Hopkins for years of illuminating discussions about geometric
fixed points and equivariant commutative ring spectra.

\chapter[Model Categories and Geometric Fixed Points]%
{Model Categories and Geometric Fixed Points on Equivariant
Commutative Ring Spectra (Theorem~\sref$main:commod$)}
\label{sec:model}

This chapter proves Theorem~\ref{main:commod} from the introduction,
which constructs the left derived functor for geometric $H$ fixed
points on commutative ring $G$-spectra and shows that it agrees with
the corresponding functor on the underlying $G$-spectra.  This
requires the construction of some new model structures, and the first
section is devoted to a review of some general facts for constructing
cofibrantly generated model structures on $G$-spectra that are Quillen
equivalent to the standard model structure.

\section{\texorpdfstring%
{$\aV$-constrained $\aF$-model categories and their variants}%
{V-constrained F-model categories and their variants}}

The purpose of this section is to specialize some standard results on
creating new model structures from old ones in the context of
the standard model structure on equivariant orthogonal spectra. We
claim no novelty on these results and in fact we regard the techniques
as well-known enough to omit all proofs.   

Let $G$ be a compact Lie group.  The \term{standard stable model
structure} on $G$-spectra has its weak
equivalences defined to be the maps that induce isomorphisms on
homotopy groups $\pi_{*}^{H}$ for all closed subgroups $H<G$; its
cofibrations are the retracts of relative cell complexes where the
cells are of the form
\[
G/H_{+}\sma F_{V}S^{0}\sma (D^{n},\partial D^{n})_{+}
\]
where $V$ is an orthogonal $G$-representation, and $F_{V}$ denotes the
functor from based $G$-spaces to $G$-spectra
that is left adjoint to the $V$th space functor.  Hill-Hopkins-Ravenel
introduced the \term{complete model structure} that is better behaved with
respect to induction along a closed subgroup.  It has the same weak
equivalences but more cofibrations; a cofibration is a retract of a
relative cell complex where the cells are of the form
\[
G_{+}\sma_{H}\sma F_{V}S^{0}\sma (D^{n},\partial D^{n})_{+}
\]
where now $V$ is allowed to be an orthogonal $H$-representation.

There are variants of the standard and complete stable model
structures in the context of a family $\aF$ of subgroups of $G$.  A
family $\aF$ is a non-empty set of closed subgroups of $G$ with the
property that when $H\in \aF$ so is every conjugate and every closed
subgroup of $H$.  Given a family $\aF$, the $\aF$-equivalences are the
maps that induce isomorphisms on homotopy groups $\pi_{*}^{H}$ for $H$
in $\aF$.  The standard and complete $\aF$-model structures have
cofibrations the retracts of relative cell complexes with the cells displayed
above (for the standard and complete stable model structures,
respectively) but with the restriction that the subgroup $H$ be in~$\aF$.

From the very beginning, dealing with commutative ring orthogonal spectra
necessitated the creation of additional model structures, the
so-called \term{positive} model structures, which put a constraint on
the orthogonal $G$-representations $V$ used in the definition of cells
for defining cofibrations.  We have positive variants of all the model
structures mentioned above where we require $V$ to contain a positive
dimensional trivial representation.  In later sections, we will see
that asking for even stricter constraints on $V$ will lead to
homotopical results on the geometric fixed points of commutative ring
orthogonal spectra.  The following definition gives the framework for
the types of constraints we consider.  

\begin{defn}\label{defn:frb}
A \term{family representation constraint} $(\aF,\aV)$ consists of a family
$\aF$ of subgroups of $G$ and a set of isomorphism classes of
orthogonal $H$-representations 
$\aV(H)$ for each $H\in \aF$ such that:
\begin{enumerate}
\item If $[V],[W]\in \aV(H)$, then $[V]\oplus [W]\in \aV(H)$.
\item For any orthogonal $H$-representation $W$, there exists $[V]\in \aV(H)$
with $[W]<[V]$.
\item If $H'=gHg^{-1}$, then $[V]\mapsto [gV]$ gives a bijection from
$\aV(H)$ to $\aV(H')$.
\item If $K<H$, then the isomorphism classes of underlying
$K$-representations of the elements of $\aV(H)$ lie in $\aV(K)$.
\end{enumerate}
We say that $\aV$ is \term{standard compatible} if it has the property that for
every orthogonal $G$-representation $W$ and every $H\in \aF$, there
exists an orthogonal $G$-representation $V$ with $W<V$ and the
underlying isomorphism class of orthogonal $H$-representations in
$\aV(H)$.
\end{defn}

We have the following basic examples.

\begin{example}\label{ex:trivposreg}
For any family, we have a constraint $\aV_{0}$ defined by taking
$\aV_{0}(H)$ to be all isomorphism classes of orthogonal
$H$-representations.  We have a constraint $\aV_{+}$ defined by
taking $\aV_{+}(H)$ to be all isomorphism classes $[V]$ such that
$V^{H}\neq 0$ (i.e., $[V]$ contains a copy of the trivial
representation).  For a family consisting of finite groups, we have a
constraint $\aV_{\reg}$ defined by taking $\aV_{\reg}(H)$ to consist
of sums of regular representations.  The constraints $\aV_{0}$ and
$\aV_{+}$ are standard compatible and $\aV_{\reg}$ is standard
compatible in the case when $G$ is finite or $G=\bT$. 
\end{example}

We also have the following example that plays an important role in Theorem~\ref{main:Tmod}.

\begin{example}\label{ex:perm}
Recall that a \term{permutation representation} of $H$ is an
orthogonal $H$-representation isomorphic to one of the form
$\bR\langle X\rangle$, where $X$ is a finite $H$-set; equivalently, it
is an orthogonal $H$-representation that has an orthonormal basis on
which the $H$ action restricts. We have a constraint $\aV_{\perm}$
defined by taking $\aV(H)$ to consist of the isomorphism classes of
permutation representations.  In the work on pc spectra, we need a
family constraint that combines $\aV_{\reg}$ and $\aV_{\perm}$.  Let
$\aV_{\permp}$ be the family constraint where $[V]\in
\aV_{\permp}(H)$ when $V\iso \bR\langle X\rangle$ for an $H$-set $X$
that has at least one free $H$-orbit. We call these representations
the \term{semiregular permutation representations}. The constraints $\aV_{\perm}$
and $\aV_{\permp}$ are standard compatible when $G$ is finite or
$G=\bT$. 
\end{example}

A family representation constraint specifies a constraint on the allowed
cell types as follows.

\begin{notn}
Given a family representation constraint $(\aF,\aV)$, let
$\aA^{\natural}(\aF,\aV)$ denote the set of $G$-spectra $G/H_{+}\sma F_{V}S^{0}$ for $H\in \aF$ and $V$ an orthogonal $G$ action on $\bR^{n}$ (for some $n\geq 0$) such that its underlying
orthogonal $H$-representation satisfies $[V]\in \aV(H)$; let
$\aA^{\comp}(\aF,\aV)$ denote the set of $G$-spectra $G_{+}\sma_{H}F_{V}S^{0}$ for $H\in \aF$ and $V$ an orthogonal $H$ action on $\bR^{n}$ (for some $n\geq 0$) that satisfies
$[V]\in \aV(H)$.
\end{notn}

The set $\aA^{\natural}(\aF,\aV)$ is only useful for us when $\aV$ is
standard compatible.  With $\aA=\aA^{\natural}(\aF,\aV)$ and
$\aA=\aA^{\comp}(\aF,\aV)$ in the following definition, we get the
standard and complete cells (respectively) with subgroup and
representation constraints specified by $\aF$ and $\aV$.

\begin{notn}
For any subset $\aA$ of $G$-spectra, let $I_{\aA}$ be the set of maps of the form
\[
A\sma (\partial D^{n}_{+}\to D^{n}_{+})
\]
for $A\in \aA$ and $n\geq 0$ (with $\partial D^{0}$ the empty set).
\end{notn}

As a technical hypothesis, we will always assume that the orthogonal spectra in $\aA$ are spacewise non-degenerately based. This hypothesis
implies that the maps in $I_{\aA}$ and the resulting cofibrations
defined below are spacewise unbased Hurewicz cofibrations.  While this
plays no role in the (omitted) proofs of the model category structures
below, it is indispensable in other contexts where Lillig's theorem is
used, for example, checking that geometric realization preserves
levelwise weak equivalences. 

\begin{defn}\label{defn:cell}
Let $\aA$ be a set of spacewise non-degenerately based
$G$-spectra.  A \term{relative $\aA$-cell
complex} is a map $X\to Y$ such that $Y\iso \colim X_{n}$ where
$X_{0}=X$, and inductively $X_{n+1}$ is formed as a pushout from
$X_{n}$ over a wedge of maps in $I_{\aA}$ (with the isomorphism to the
colimit inducing the given map $X\to Y$).  A \term{regular
$I_{\aA}$-cofibration} is a map $X\to Y$ given by transfinite
composition of pushouts along maps in $I_{\aA}$: there exists an
ordinal $\lambda$, $G$-spectra $X_{\alpha}$ for $\alpha
\leq \lambda$, and compatible maps $X_{\alpha}\to X_{\beta}$ for
$\alpha <\beta\leq \lambda$ such that $X_{0}=X$, $X_{\lambda}\iso Y$
(compatibly with the given map $X\to Y$), for every $\alpha<\lambda$,
$X_{\alpha+1}$ is formed by pushout along a map in $I_{\alpha}$, and
for every limit ordinal $\beta \leq \lambda$,
$X_{\beta}=\colim_{\alpha<\beta}X_{\alpha}$.
\end{defn}

It is easy to see by transfinite induction that for a regular
$I_{\aA}$-cofibration $X\to X_{\lambda}$, all the maps $X_{\alpha}\to
X_{\beta}$ in the composition system are Hurewicz cofibrations of
equivariant orthogonal spectra.  That observation also implies that
for any $G$-spectrum $Z$, when $\lambda$ is a
cardinal and greater than the sum of the cardinalities of the spaces
$Z(\bR^{n})$, then any map $Z\to Y$ factors through some $X_{\alpha}$
with $\alpha <\lambda$.  This allows use of the small object argument
in this context.

A relative $\aA$-cell complex is in particular a regular
$I_{\aA}$-cofibration (for example by well ordering the set of maps in
each pushout).  Under certain ``smallness'' hypotheses, the reverse is
also true.

\begin{defn}\label{defn:tpssmall}
A $G$-spectrum $A$ is topologically point-set small if the set
of maps out of $A$ commutes with transfinite compositions when the
constituent maps are spacewise closed inclusions.
\end{defn}

The orbit desuspension spectrum $F_{V}G/H_{+}$ is point-set
topologically small since maps out of it represent the functor
$X\mapsto X(V)^{H}$.  By adjunction it follows that $A$ is point-set
topologically small if and only if the function spectrum functor
$F(A,-)$ commutes with transfinite compositions when the
constituent maps are spacewise closed inclusions.  Adjunction again
implies that the smash product of topologically point-set small
$G$-spectra is topologically point-set small.
The relevance of this concept is the following proposition which
follows from a transfinite induction argument (and the observation
above that the maps in the composition system for a regular
$I_{\aA}$-cofibration are Hurewicz cofibrations and so in particular
spacewise closed inclusions).

\begin{prop}\label{prop:tpssmall}
Let $\aA$ be as in Definition~\ref{defn:cell}. If every object $A$ in
$\aA$ is topologically point-set small then every regular $I_{\aA}$-cofibration
is a relative $\aA$-cell complex.
\end{prop}

Standard techniques then give the following theorems, simultaneously
generalizing the standard, positive, and complete stable and
$\aF$-model structures to more general family representation
constraints.

\begin{thm}[$\aV$-constrained $\aF$-model structure]
\label{thm:FVmodel}
For any compact Lie group $G$ and any family representation constraint
$(\aF,\aV)$ with $\aV$ standard compatible, the category of 
$G$-spectra admits a symmetric monoidal
topological model
structure where:
\begin{enumerate}
\item The cofibrations are the retracts of
relative $\aA^{\natural}(\aF,\aV)$-cell complexes.
\item The fibrations are the maps $X\to Y$ such that for every $H\in
\aF$ and every orthogonal $G$-representation $V$ whose underlying
$H$-representation satisfies $[V]\in [\aV(H)]$:
\begin{itemize}
\item The map of non-equivariant spaces $X(V)^{H}\to
Y(V)^{H}$ is a Serre fibration, and
\item For every orthogonal $G$-representation $W$ with $[V\oplus W]\in
\aV(H)$, the square of non-equivariant spaces
\[
\xymatrix@-1pc{%
X(V)^{H}\ar[r]\ar[d]&(\Omega^{W}X(V\oplus W))^{H}\ar[d]\\
Y(V)^{H}\ar[r]&(\Omega^{W}Y(V\oplus W))^{H}
}
\]
is homotopy cartesian.
\end{itemize}
\item The weak equivalences are the $\aF$-equivalences.
\end{enumerate}
\end{thm}

\begin{thm}[$\aV$-constrained complete $\aF$-model structure]
\label{thm:FVcmodel}
For any compact Lie group $G$ and any family representation constraint
$(\aF,\aV)$, the category of 
$G$-spectra admits a symmetric monoidal
topological model structure where:
\begin{enumerate}
\item The cofibrations are the retracts of
$\aA^{\comp}(\aF,\aV)$-cell complexes.
\item The fibrations are the maps $X\to Y$ such that for every $H\in
\aF$ and every orthogonal $H$-representation $V$ 
that satisfies $[V]\in \aV(H)$:
\begin{itemize}
\item The map of non-equivariant spaces $X(V)^{H}\to
Y(V)^{H}$ is a Serre fibration, and
\item For every orthogonal $H$-representation $W$ with $[V\oplus W]\in
\aV(H)$, the square of non-equivariant spaces
\[
\xymatrix@-1pc{%
X(V)^{H}\ar[r]\ar[d]&(\Omega^{W}X(V\oplus W))^{H}\ar[d]\\
Y(V)^{H}\ar[r]&(\Omega^{W}Y(V\oplus W))^{H}
}
\]
is homotopy cartesian.
\end{itemize}
\item The weak equivalences are the $\aF$-equivalences.
\end{enumerate}
\end{thm}

See~\cite[2.1]{SSAlgMod} for a definition of monoidal model structure
and e.g.,~\cite[5.12]{MMSS} for a definition of topological model structure.

In the case when $\aV=\aV_{0}$, we get the standard stable model
structure and complete model structure with the family of all
subgroups or the corresponding $\aF$-model structures for an
arbitrary family.  When $\aV=\aV_{+}$, we get the corresponding
positive model structures.

For the theorems in the introduction, we need to consider more general
classes of cells $\aA$.  (For the following theorem, we note that
$\aA^{\natural}(\aF,\aV)\subset \aA^{\comp}(\aF,\aV)$.  A version
holds for constraints $\aV$ that are not standard compatible where we
replace $\aA^{\natural}(\aF,\aV)$ with $\aA^{\comp}(\aF,\aV)$ and
replace ``$\aV$-constrained $\aF$-model structure'' with ``$\aV$-constrained complete $\aF$-model structure''.) 

\begin{thm}[Variant $\aV$-constrained $\aF$-model structures]
\label{thm:varAmodel}
Let $G$ be a compact Lie group, and $(\aF,\aV)$ a family representation
constraint with $\aV$ standard compatible.  Let $\aA$ be a
set of spacewise non-degenerately based $G$-spectra with $\aA^{\natural}(\aF,\aV)\subset 
\aA$. Then the category of $G$-spectra admits a topological model structure where:
\begin{enumerate}
\item The cofibrations are the $I_{\aA}$-regular cofibrations.
\item The fibrations are the maps $X\to Y$ such that for every $A\in \aA$:
\begin{itemize}
\item The induced map of mapping spaces $\Spectra^{G}(A,X)\to
\Spectra^{G}(A,Y)$ is a Serre fibration, and 
\item For every cofibrant object $\tilde A$ in the 
$\aV$-constrained $\aF$-model structure and $\aF$-equivalence
$\tilde A\to A$, the square of
mapping spaces
\[
\xymatrix@-1pc{%
\Spectra^{G}(A,X)\ar[r]\ar[d]&\Spectra^{G}(\tilde A,X)\ar[d]\\
\Spectra^{G}(A,Y)\ar[r]&\Spectra^{G}(\tilde A,Y)
}
\]
is homotopy cartesian.
\end{itemize}
\item The weak equivalences are the $\aF$-equivalences.
\end{enumerate}
Moreover:
\begin{itemize}
\item If every object in $\aA$ is topologically point-set small, then the
cofibrations are the retracts of the relative $\aA$-cell complexes.
\item If all objects $A$ of $\aA$ are flat in the sense that
$A\sma(-)$ preserves weak equivalences, then all cofibrant objects are
flat and this model structure satisfies the unit axiom and monoid
axiom of Schwede-Shipley~\cite[3.2,3.3]{SSAlgMod}.
\item If for all $A,B\in \aA$, $A\sma B$ is cofibrant in this model
structure, then this model structure satisfies the pushout-product
axiom of Schwede-Shipley~\cite[3.4]{SSAlgMod}.
\end{itemize}
\end{thm}

We also have the following alternative identification of the
fibrations.  (It mainly changes the quantifier on the $\aF$-equivalences
$\tilde A\to A$ in the last bullet point.)

\begin{prop}\label{prop:varAfib}
For $\aA$, $\aF$, $\aV$ in the previous theorem, the fibrations in the
model structure there are precisely the maps $X\to Y$ that satisfy all
of the following conditions.
\begin{itemize}
\item $X\to Y$ is a fibration in the $\aV$-constrained $\aF$-model structure.
\item For every $A\in \aA$, the map $\Spectra^{G}(A,X)\to
\Spectra^{G}(A,Y)$ is a Serre fibration.
\item For every $A\in \aA$, there exists a cofibrant object $\tilde A$
in the $\aV$-constrained $\aF$-model structure and an
$\aF$-equivalence $\tilde A\to A$ such that the square 
\[
\xymatrix@-1pc{%
\Spectra^{G}(A,X)\ar[r]\ar[d]&\Spectra^{G}(\tilde A,X)\ar[d]\\
\Spectra^{G}(A,Y)\ar[r]&\Spectra^{G}(\tilde A,Y)
}
\]
is homotopy cartesian.
\end{itemize}
\end{prop}

When the model structure satisfies
the unit axiom, monoid axiom, and pushout-product axiom, we get corresponding model
structures on categories of modules and algebras by~\cite{SSAlgMod}.
For commutative rings and commutative algebras we need additional
hypotheses.

\begin{thm}\label{thm:multAvar}
Let $G$ be a compact Lie group, and $(\aF,\aV)$ a family representation
constraint with $\aV$ standard compatible.  Let $\aA$ be a set of spacewise non-degenerately based
$G$-spectra with $\aA^{\natural}(\aF,\aV)\subset \aA$.
Assume that the smash product of any two elements of $\aA$ is
cofibrant in the $\aA$-variant $\aV$-constrained $\aF$-model
structure and assume that
for every $A\in \aA$, and every $n\geq 1$:
\begin{enumerate}
\item the $n$th symmetric power $A^{(n)}/\Sigma_{n}$ is flat for the
smash product of $G$-spectra, and
\item for every cofibrant approximation $\tilde A\to A$ in the
$\aV$-constrained $\aF$-model structure, the induced map on
$n$th symmetric powers 
\[
\tilde A^{(n)}/\Sigma_{n} \to A^{(n)}/\Sigma_{n}
\]
is an $\aF$-equivalence.
\end{enumerate}
Then the forgetful functor from commutative ring
$G$-spectra to $G$-spectra with the
$\aA$-variant $\aV$-constrained $\aF$-model structure creates
fibrations and weak equivalences in a model structure.
Moreover, if $\aA$ is closed up to retract under $n$th symmetric powers for
all $n\geq 1$, then commutative ring
$G$-spectra that are cofibrant in this model structure are
cofibrant in the undercategory of $\bS$ in
the $\aA$-variant $\aV$-constrained $\aF$-model structure on
$G$-spectra. 
\end{thm}

We note that in the previous theorem, $G/H_{+}\sma \bS$ does not satisfy
condition~(ii) and so cannot be an element of $\aA$. In particular,
$\aV(H)$ must contain only positive dimensional isomorphism classes
for every $H\in \aF$.  

The hypotheses in Theorem~\ref{thm:multAvar} include the
extra hypotheses of Theorem~\ref{thm:varAmodel} that imply the unit,
monoid, and pushout product axioms, and in particular give model
structures on categories of $R$-modules and $R$-algebras.
Since a commutative $R$-algebra is just a map out of $R$ in the
commutative ring category, the theorem also constructs model
structures on categories of commutative $R$-algebras.

Several theorems in this section assert symmetric monoidal model
structures (in some cases under extra hypotheses).  When the
$\aA$-variant $\aV$-constrained $\aF$-model structure is
symmetric monoidal, then the model is
enriched over itself in the sense of~\cite[3.3]{LewisMandell2}, and
\textit{a fortiori} enriched over the $\aV$-constrained $\aF$-model
category of $G$-spectra. We caution that it is
not always the case that the $\aA$-variant $\aV$-constrained
$\aF$-model category of $G$-spectra is enriched over
a model structure on the category of non-equivariant spectra: this
depends on whether the objects in $\aA$ smashed with the 
non-equivariant desuspension spectra $F_{\bR^{n}}S^{0}$ are cofibrant
in the $\aA$-variant $\aV$-constrained $\aF$-model
structure. (This does not hold for example when
$G$ is finite or $G=\bT$ and $A=\aA^{\natural}(\aF,\aV_{\reg})$ for
any family $\aF$ containing a non-trivial subgroup.) Using
the formulation of enrichment in terms of the Pushout-Tensor Axiom 
of~\cite[3.4]{LewisMandell2}, this
is clearly a necessary condition and given the characterization of the
cofibrations, it is easy to see that it also sufficient.

\begin{prop}\label{prop:enriched}
With notation as in Theorem~\ref{thm:varAmodel}, assume that for every
$n\geq 1$ and every $A\in \aA$, the $G$-spectrum $A\sma
F_{\bR^{n}}S^{0}$ is cofibrant in the $\aA$-variant 
$\aV$-constrained $\aF$-model structure; then this model structure is enriched
over the standard model structure on non-equivariant spectra in the
sense of~\cite[3.3]{LewisMandell2}.
\end{prop}

\section[Theorem~\sref$main:commod$]%
{The multiplicative geometric fixed point functor on variant
model structures and Theorem~\sref$main:commod$}
\label{sec:geofix}
\label{sec:commod}

The purpose of this section is to provide enough examples of
$\aA$-variant $\aV$-constrained model structures to be able to construct
the left derived functor of the multiplicative geometric fixed point
functor $\Phi^{H}$ on the category of commutative ring
$G$-spectra.  As a consequence, we prove
Theorem~\ref{main:commod} from the introduction; at the end of the
section, we review the individual definitions and results to give a
self-contained description of a model structure with the
properties asserted in Theorem~\ref{main:commod}.

Some of the proofs in this section involve the properties of the generalized orbit
desuspension spectra studied in
Chapter~\ref{chap:bundleone}--\ref{chap:bundletwo}.
The background needed on the theory of generalized orbit desuspension
spectra is Proposition~\ref{prop:Sym}, which relates the symmetric
power construction to the theory and Theorem~\ref{thm:bundlegeofix},
which under a mild faithfulness hypothesis
(Definition~\ref{defn:Sigmafaithful}) relates the geometric fixed point
construction to the theory.  The faithfulness hypothesis of
Definition~\ref{defn:isf} is also needed in places.  With this
background, the arguments below rely on specific references to results
in Chapter~\ref{chap:bundleone}.

We begin with the following basic result.

\begin{thm}
Let $G$ be a compact Lie group, $\aF,\aV$ a family representation
constraint with $\aV$ standard compatible, and $\aA$ a set of
spacewise non-degenerately based $G$-spectra with
$\aA^{\natural}(\aF,\aV)\subset \aA$. Fix a closed subgroup $H<G$ and
a family $\aF_{WH}$ of closed subgroups of $WH$.  Assume that for
every $A\in \aA$ and every cofibrant approximation $\tilde A\to A$ in
the $\aV$-constrained $\aF$-model structure, the induced map
$\Phi^{H}\tilde A\to \Phi^{H}A$ is an 
$\aF_{WH}$-equivalence.  Then $\Phi^{H}$ sends
$\aF$-equivalences between cofibrant objects in the
$\aA$-variant $\aV$-constrained $\aF$-model structure to
$\aF_{WH}$-equivalences of $WH$-spectra.
Moreover, for any cofibrant object $X$ in the $\aA$-variant 
$\aV$-constrained $\aF$-model structure and any cofibrant approximation $\tilde
X\to X$ in the $\aV$-constrained $\aF$-model structure, the induced map
$\Phi^{H}\tilde X\to \Phi^{H}X$ is an
$\aF_{WH}$-equivalence. 
\end{thm}

\begin{proof}
Let $\aA^{\Phi H}$ be the set of $WH$-spectra 
\[
\aA^{\Phi H}=\{\Phi^{H}A\mid A\in \aA\}.
\]
By transfinite induction, $\Phi^{H}$ sends an
$I_{\aA}$-regular cofibration to the corresponding
$\aI_{\aA^{\Phi H}}$-regular cofibration.  Standard arguments then prove
the statements.
\end{proof}

\begin{cor}\label{cor:modelPhi}
Let $G$ be a compact Lie group, and $\aF,\aV$ a family representation
constraint with $\aV$ standard compatible.  Let $\aA$ be a set of
spacewise non-degenerately based 
$G$-spectra satisfying the hypotheses of
Theorem~\ref{thm:multAvar} that is closed under $n$th symmetric
powers for all $n\geq 1$.  Fix a closed subgroup $H<G$ and a family
$\aF_{WH}$ of closed subgroups of $WH$.  Assume that for every $A\in
\aA$ and every cofibrant approximation $\tilde A\to A$ in the
$\aV$-constrained $\aF$-model structure, the induced map
$\Phi^{H}\tilde A\to \Phi^{H}A$ is an 
$\aF_{WH}$-equivalence.  Then $\Phi^{H}$ sends
$\aF$-equivalences between cofibrant objects in the
$\aA$-variant $\aV$-constrained $\aF$-model structure on
commutative ring $G$-spectra to
$\aF_{WH}$-equivalences of commutative ring $WH$-spectra.  Moreover, for any commutative
ring $G$-spectrum $R$ that is cofibrant in the $\aA$-variant
$\aV$-constrained $\aF$-model structure and any 
cofibrant approximation $\tilde X\to R$ in the
$\aV$-constrained $\aF$-model structure on $G$-spectra, the induced map 
$\Phi^{H}\tilde X\to \Phi^{H}R$ is an $\aF_{WH}$-equivalence.
\end{cor}

\begin{proof}
By Theorem~\ref{thm:multAvar}, if $R$ is cofibrant in the $\aA$-variant
$\aV$-constrained $\aF$-model structure of $G$-equivariant commutative ring
orthogonal spectra, the unit map $\bS\to R$ is a
cofibration of the underlying $G$-spectra in
the $\aA$-variant $\aV$-constrained $\aF$-model structure.  The
statements then follow from the previous theorem.
\end{proof}

To give the previous theorems substance, we need examples where the
hypotheses hold.  Here is one such example.

\begin{thm}\label{thm:Hexample}
Let $G$ be a compact Lie group.  Let $\aF$ be any family of closed
subgroups of $G$, $H$ any closed subgroup of $G$, and $\aF_{WH}$ any
family of closed subgroups of $WH$.  Let $\aV$ be a family
representation constraint for $\aF$ that is standard compatible with
the property that for every $K\in \aF$, every $[V]\in \aV(K)$ satisfies
the following conditions:
\begin{enumerate}
\item[(1)] $V$ contains a copy of the trivial representation, and 
\item[(2)] If $K<H$ and contains the identity component of $H$ then 
$V$ contains a copy of the regular representation of $\pi_{0}K$. 
\end{enumerate}
Let $\aA$ be the
closure of $\aA^{\natural}(\aF,\aV)$ or $\aA^{\comp}(\aF,\aV)$
under $n$th symmetric power ($n\geq 1$) and (finite, non-empty) smash products.
Then $\aF$, $V$, $H$, $\aF_{WH}$, and $\aA$ satisfy the hypotheses of
Corollary~\ref{cor:modelPhi}.
\end{thm}

\begin{proof}
We need to check the following
properties hold for every $A\in \aA$: 
\begin{enumerate}
\item For every $n\geq 1$, $A^{(n)}/\Sigma_{n}$ is flat for the smash
product of $G$-spectra.
\item For every $n\geq 1$ and every cofibrant approximation $\tilde
A\to A$ in the $\aV$-constrained $\aF$-model structure, $\tilde
A^{(n)}/\Sigma_{n}\to A^{(n)}/\Sigma_{n}$ is an $\aF$-equivalence.
\item For every cofibrant approximation $\tilde
A\to A$ in the $\aV$-constrained $\aF$-model structure, $\Phi^{H}\tilde
A\to \Phi^{H}A$ is an $\aF_{WH}$-equivalence.
\end{enumerate}
We use the theory of
Chapter~\ref{chap:bundleone}, and particularly
the theorems in Section~\ref{sec:geobundle}, with $\KG$ the
trivial group and $G=\Gamma=\Gamma/\KG$.  
We note that by induction, every $A\in \aA$ can be expressed as a
generalized orbit desuspension spectrum $J\eta$ for some $(G,Q)$
vector bundle $\eta$ of the form
\[
\eta_{1}^{n_{1}}\times \dotsb \times \eta_{r}^{n_{r}}
\]
where $Q<\Sigma_{n_{1}}\times \dotsb \times \Sigma_{n_{r}}$ and each
$\eta_{i}$ is of the form  
\[
G\times_{K}V\to G/K 
\]
for $K\in \aF$ and $[V]\in \aV(H)$.  To apply the results of
Section~\ref{sec:geobundle}, we need to
check that $\eta$ is $Q$-faithful and for some results that
$\eta(H)$ (in the notation of
Definition~\ref{defn:philambdaone}) is $Q$-faithful.  For this, 
consider the bundle
\[
\theta = \eta_{1}\amalg \dotsb \amalg \eta_{r},
\]
and the bundle $\Sym^{n}_{Q}\theta$ (in the notation of
Proposition~\ref{prop:Sym}) where $n=n_{1}+\dotsb+n_{r}$ and we regard
$Q$ as a subgroup of $\Sigma_{n}$ via block sum.  Then $\eta$ is
obtained from $\Sym^{n}_{Q}\theta$ by the restriction (pullback) along
the inclusion of a $(G\times Q^{\op})$-stable subset of the base;
likewise $\eta(H)$ is obtained from $\Sym^{n}_{Q}\theta(H)$ by
the restriction along the inclusion of a $(G\times Q^{\op})$-stable
subset of the base.  Thus (by Proposition~\ref{prop:Sfrestr}.(ii)), to
see that $\eta$ and $\eta(H)$ are $Q$-faithful, it is enough to
show that $\Sym^{n}_{Q}\theta$ and $\Sym^{n}_{Q}\theta(H)$ are
$Q$-faithful.  Proposition~\ref{prop:trivsf} (and
Proposition~\ref{prop:Sfrestr}.(i)) shows that $\Sym^{n}_{Q}\theta$ is
$Q$-faithful.  Proposition~\ref{prop:ifcritex}.(b) applies to $\theta$
(with $\Lambda=H$) because of the requirement on regular
representations in $\aV$ and so by 
Theorem~\ref{thm:ifcrit}, $\theta$ is inheritably faithful for
$\Phi^{H}$.  This 
implies that $\Sym^{n}\theta(H)$ is $\Sigma_{n}$-faithful,
which implies that $\Sym^{n}_{Q}\theta(H)$ is $Q$-faithful
(Proposition~\ref{prop:isfrestr}).

By Theorem~\ref{thm:bundleflat}, we
then have that $A$ is flat for all $A\in \aA$, and since $\aA$ is
closed under symmetric power, property~(i) holds.  Property~(ii)
follows by Theorem~\ref{thm:dersym}.   Property~(iii) holds by
Theorem~\ref{thm:derphi}.
\end{proof}

Such a representation constraint always exists: for example, the Peter-Weyl
theorem implies that for any orthogonal $H$-representation $V$ there
exists an orthogonal $G$-representation $W$ whose underlying
$H$-representation contains a copy of $V$.  Applying this with $V$ the
regular representation of $\pi_{0}H$ and taking $\aV(K)$ (for all
$K\in \aF$) to be the set of isomorphism classes of orthogonal
$K$-representations that contain a copy of the underlying
$K$-representation of $W$ gives a representation constraint
satisfying the hypothesis of the theorem.

The previous theorem is general enough that it establishes the
existence of derived functor of $\Phi^{H}$ on 
commutative ring $G$-spectra and a comparison with the derived 
functor of $\Phi^{H}$ on $G$-spectra.

\begin{cor}
For any compact Lie group $G$ and any closed subgroup $H$, the
geometric $H$ fixed point functor from commutative
ring $G$-spectra to commutative ring
$WH$-spectra has a left derived functor, which lifts the derived
geometric $H$ fixed point functor from $G$-spectra to commutative ring
$WH$-spectra.
\end{cor}

The previous result only treats the geometric fixed points of a single
subgroup (or a finite set of subgroups) at a time.  We also have the
following variant, which handles all finite subgroups at the same
time.  The proof is identical, replacing part~(b) of
Proposition~\ref{prop:ifcritex} with part~(a) (applied with $Q$
the trivial group).

\begin{thm}\label{thm:Finexample}
Let $G$ be a compact Lie group.  Let $\aF$ be any family of finite subgroups, and let
$\aV$ be a family representation constraint for $\aF$ 
that is standard compatible and that has the property that for every
subgroup $K\in \aF$, every $[V]\in \aV(K)$ contains the
regular representation of $K$.  Let $\aA$ be the closure of
$\aA^{\comp}(\aF,\aV)$ under $n$th symmetric power ($n\geq 1$)
and smash product.  Then $\aF$, $V$, and $\aA$ satisfy the hypotheses
of Corollary~\ref{cor:modelPhi} for any finite subgroup $H$ of $G$ and
any family $\aF_{WH}$ of closed subgroups of $WH$.
\end{thm}

In both Theorems~\ref{thm:Hexample} and~\ref{thm:Finexample}, the
cofibrant objects are built from cells of the form $A\sma
(D^{n},\partial D^{n})_{+}$ where $A$ is a generalized orbit
desuspension spectrum $J\eta$ for some $(G,Q)$ vector bundle (for some
finite group $Q$).  Using the additional properties of
generalized orbit desuspension spectra from
Chapter~\ref{chap:bundleone}, we see that the cofibrant objects in
these model structures have additional good properties, which we
catalog below.  In the statements below, $H$ denotes the particular subgroup
in Theorem~\ref{thm:Hexample} or a subgroup in $\aF$ in
Theorem~\ref{thm:Finexample}. In each case all that is needed is that
$\eta$ and $\eta(H)$ are $Q$-faithful and have base spaces isomorphic
to $(G\times Q^{\op})$-cell complexes.  Theorem~\ref{thm:bundleflat} implies:

\begin{prop}\label{prop:mainflat}
Suppose $X$ is cofibrant or cofibrant under $\bS$ in the model
structure of Theorem~\ref{thm:Hexample} or~\ref{thm:Finexample}. Then
$X$ is flat for the smash product of $G$-spectra and $\Phi^{H}X$ is
flat for the smash product of $WH$-spectra.
\end{prop}

Theorem~\ref{thm:sm} implies:

\begin{prop}\label{prop:mainsym}
Suppose $X$ is cofibrant or cofibrant under $\bS$ in the model
structure of Theorem~\ref{thm:Hexample} or~\ref{thm:Finexample}. Then
for any $G$-spectrum $Y$, the canonical map
\[
\Phi^{H}X\sma \Phi^{H}Y\to \Phi^{H}(X\sma Y)
\]
is an isomorphism.
\end{prop}

The previous proposition and a cell by cell argument also gives the
following result.

\begin{prop}\label{prop:mainsmashover}
Let $R$ be an associative ring $G$-spectrum. Let $M$ and $N$ be left and
right $R$-modules (respectively) and assume one of $M$, $N$ is
cofibrant either in the standard model structure, the model
structure of Theorem~\ref{thm:Hexample}, or the model structure of
Theorem~\ref{thm:Finexample} for its respective category of
$R$-modules. Then the canonical map 
\[
\Phi^{H}M\sma_{\Phi^{H}R} \Phi^{H}N\to \Phi^{H}(M\sma_{R} N)
\]
is an isomorphism.
\end{prop}

\subsection*{Model structures for Theorem~\protect\ref{main:commod}}
Theorem~\ref{thm:Hexample} combined with Theorem~\ref{thm:multAvar}
gives a model structure that works for an arbitrary finite set $\aH$
of closed subgroups.  The following is a step-by-step guide through
the definitions and results above to see this.  (We now assume that
$\aH$ is non-empty, as otherwise there is nothing to show.)

Let $W$ be an orthogonal $G$-representation with the property that for
every $H\in \aH$, the restriction of $W$ as an orthogonal
$H$-representation contains a copy of the regular representation of 
$\pi_{0}H$.  Let $\aA^{\natural}_{W}$ be the set of spectra of the
form
\[
F_{V}(G/K)_{+}
\]
for $K$ a closed subgroup of $G$ and $V$ an orthogonal
$G$-representation containing $W$. Let $\aA^{\ssma}_{W}$ denote the
closure of $\aA^{\natural}_{W}$ under (finite, non-empty) smash
products and symmetric powers: an 
element of $\aA^{\ssma}_{W}$ is in particular a spectrum of the form
\[
((F_{V_{1}}(G/K_{1})_{+})^{(q_{1})}
\sma \dotsb \sma 
(F_{V_{r}}(G/K_{r})_{+})^{(q_{r})})/Q
\]
where $r\geq 1$, $q_{i}\geq 1$, each $V_{i}$ contains $W$, and
$Q$ is a subgroup of the product of symmetric groups
$\Sigma_{q_{1}}\times \dotsb \times \Sigma_{q_{r}}$.  The objects in
$\aA^{\natural}_{W}$ are 
topologically point-set small (Definition~\ref{defn:tpssmall}) because
$F_{V}(G/K)_{+}$ is and
the class of topologically point-set small objects is closed under
smash products and quotients.  

We get two model structures on
$G$-spectra with 
$\aA=\aA^{\natural}_{W}$ and $\aA=\aA^{\ssma}_{W}$, resp., where:
\begin{itemize}
\item The cofibrations are the maps $X\to Y$ that are retracts of
relative $\aA$-cell complexes (see Definition~\ref{defn:cell}).
\item The weak equivalences are the standard weak equivalences
\end{itemize}
In the case of $\aA^{\natural}_{W}$, a map $X\to Y$ is a fibration if
and only if for every orthogonal $G$-representation $V$ containing
$W$, the map $X(V)\to Y(V)$ is a $G$-equivariant Serre fibration and
for every orthogonal $G$-representation $U$ containing $V$ (which
contains $W$) the diagram  
\[
\xymatrix{%
X(V)\ar[r]\ar[d]&\Omega^{U-V}X(U)\ar[d]\\
Y(V)\ar[r]&\Omega^{U-V}Y(U)
}
\]
is homotopy cartesian.  In the case of $\aA^{\ssma}_{W}$, a map $X\to Y$
is a fibration if and only if:
\begin{itemize}
\item $X\to Y$ is a fibration in the model structure for $\aA^{\natural}_{W}$,
\item For every $A\in \aA^{\ssma}_{W}$, the map $\Spectra^{G}(A,X)\to
\Spectra^{G}(A,Y)$ is a Serre fibration, and
\item For every $A\in \aA^{\ssma}_{W}$, there exists a cofibrant object
$\tilde A$ in the model structure for $\aA^{\natural}_{W}$ and a weak
equivalence $\tilde A\to A$ such that the diagram
\[
\xymatrix{%
\Spectra^{G}(A,X)\ar[r]\ar[d]
&\Spectra^{G}(\tilde A,X)\ar[d]\\
\Spectra^{G}(A,Y)\ar[r]
&\Spectra^{G}(\tilde A,Y)
}
\]
is homotopy cartesian.
\end{itemize}
See for example Theorem~\ref{thm:FVmodel} for $\aA^{\natural}_{W}$
and Theorem~\ref{thm:varAmodel} with Proposition~\ref{prop:varAfib}
for $\aA^{\ssma}_{W}$.

The identity gives a zigzag of Quillen equivalences of model
structures relating these two model structures and the standard (or
positive) stable model structure on $G$-spectra:
the identity is a Quillen left adjoint from the model
structure for $\aA^{\natural}_{W}$ to the standard (or positive) model
structure and to the model structure for $\aA^{\ssma}_{W}$.

Theorem~\ref{thm:Hexample} implies that every $A\in \aA^{\ssma}_{W}$ satisfies
the following properties:
\begin{itemize}
\item For any $n\geq 1$, $A^{(n)}/\Sigma_{n}$ is flat for the smash product of
$G$-spectra
\item For any $n\geq 1$ and any cofibrant approximation $\tilde A\to A$ in the model
category for $\aA^{\natural}_{W}$, the induced map $\tilde
A^{(n)}/\Sigma_{n}\to A^{(n)}/\Sigma_{n}$ is a weak equivalence.
\end{itemize}
(While the statement of Theorem~\ref{thm:Hexample} is peripatetic in
this regard, this assertion is also made in the first few lines of its
proof.) Since $\aA^{\natural}_{W}\subset \aA^{\ssma}_{W}$,
Theorem~\ref{thm:multAvar} applies to show that the forgetful functor
from commutative ring $G$-spectra to
$G$-spectra creates a model structure using
either the model structure for $\aA^{\natural}_{W}$ or the model
structure for $\aA^{\ssma}_{W}$; moreover, using the latter model
structure, cofibrant commutative ring $G$-spectra are cofibrant in the undercategory of $\bS$ in
$G$-spectra.  

Now assume that $R$ is a cofibrant commutative ring
$G$-spectrum in the model structure for $\aA^{\ssma}_{W}$ (which
holds in particular if $R$ a cofibrant commutative
ring $G$-spectrum in the model structure for
$\aA^{\natural}_{W}$).  We need to show that for any cofibrant
approximation $X\to R$ in the standard model structure on
$G$-spectra, the induced map $\Phi^{H}X\to
\Phi^{H}R$ is a weak equivalence for every $H\in \aH$.  Since
$\Phi^{H}$ preserves weak equivalences between cofibrant objects in
the standard model structure, it suffices to show that there exists a
cofibrant approximation with this property.  Let $X\to R$ be a
cofibrant approximation of $R$ in the undercategory of $\bS$ in the
model structure for $\aA^{\natural}_{W}$ on $G$-spectra; then $X\to R$ is a cofibrant approximation in the standard
model structure on the (absolute) category of $G$-spectra.  The inclusions of $\bS$ in $X$ and $R$ are
Hurewicz cofibrations and $\Phi^{H}$ preserves Hurewicz cofibrations
and point-set quotients by Hurewicz cofibrations, so we get a map of
long exact sequences of homotopy groups
\[
\xymatrix@-1pc{%
\dotsb\ar[r]&\pi_{q}^{K}(\Phi^{H}\bS)\ar[r]\ar[d]_{=}
&\pi_{q}^{K}(\Phi^{H}X)\ar[r]\ar[d]
&\pi_{q}^{K}(\Phi^{H}(X/\bS))\ar[r]\ar[d]
&\pi_{q-1}^{K}(\Phi^{H}\bS)\ar[d]_{=}\ar[r]
&\dotsb\\
\dotsb\ar[r]&\pi_{q}^{K}(\Phi^{H}\bS)\ar[r]
&\pi_{q}^{K}(\Phi^{H}R)\ar[r]
&\pi_{q}^{K}(\Phi^{H}(R/\bS))\ar[r]
&\pi_{q-1}^{K}(\Phi^{H}\bS)\ar[r]
&\dotsb
}
\]
for every $K<WH$ (for
example, by \cite[III.3.5.(iv),(vi)]{MM}).  The map $X/\bS\to Y/\bS$
is a weak equivalence of cofibrant objects in the model structure for
$\aA^{\ssma}_{W}$ on $G$-spectra and so induces
a weak equivalence on $\Phi^{H}$ by Corollary~\ref{cor:modelPhi},
which holds by Theorem~\ref{thm:Hexample}.  It follows that
$\Phi^{H}X\to \Phi^{H}R$ is a weak equivalence.

\chapter{The Circle Group and Cyclotomic Spectra (Theorems~B, C, D)}

This chapter treats the case when the compact Lie group is the circle
group $\bT$ of unit complex numbers.  In this case, we have significantly better results
on geometric fixed point functors than we have for general compact Lie
groups.  We study this in Section~\ref{sec:circle}, where we prove
Theorem~\ref{main:Tmod}.  We then apply the results on the circle
group to study \ppc and pc spectra in Section~\ref{sec:precyc}, where
we prove Theorems~\ref{main:map} and~\ref{main:cycmod}.  The
construction of the model structure used in these sections is stated
in terms of a black box set of $\bT$-spectra $\aACyc$ (and its subset
$\aApCyc$); Section~\ref{sec:ACyc} opens this box, filling in the
details and proving the equivariant properties of the $\bT$-spectra in
$\aACyc$ needed to prove the geometric fixed point results.
Section~\ref{sec:convenient} fills in the properties of $\aACyc$
needed to compare the model structure on $\bT$-spectra to the positive
convenient $\Sigma$-model structure on non-equivariant spectra.

\section{The circle group and Theorem~\sref$main:Tmod$}
\label{sec:circle}

The purpose of this section is to begin the proof of
Theorem~\ref{main:Tmod} from the introduction.  Here we have broken it
into two parts: Theorem~\ref{thm:circlephi} below generalizes the
model structure statement of Theorem~\ref{main:Tmod} and its
properties that involve just the equivariant categories, while
Theorem~\ref{thm:necircle} restates the strong relationship asserted
between the model structures on the equivariant and non-equivariant
categories.  As in the statement of Theorem~\ref{main:Tmod}, we let
$\aF_{\fin}$ denote the family of finite subgroups of $\bT$ and
$\aF_{p}$ denote the family of finite $p$-subgroups of $\bT$.  In
contrast, we use slightly different notation for the geometric fixed
point endofunctor denoted $\Phi$ in the introduction in order to
consider the construction for more general subgroups of $\bT$.

\begin{notn}\label{notn:Phin}
Define endofunctors $\Phi_{n}$, $n\geq 1$ by 
\[
\Phi_{n}:=\rho^{*}_{n}\Phi^{C_{n}}.
\]
Here we first take the geometric fixed points $\Phi^{C_{n}}$ to
$\bT/C_{n}$-spectra and then use the $n$th root
isomorphism $\rho_{n}\colon \bT\iso \bT/C_{n}$ to convert back to
$\bT$-spectra.  
\end{notn}

To establish the model structure, we use the general result of
Theorem~\ref{thm:multAvar}, applied with the family $\aF_{\fin}$ or
$\aF_{p}$ and the representation constraint $\aV_{\permp}$ from
Example~\ref{ex:trivposreg}.  Namely, for $C\in \aF_{\fin}$ (or
$\aF_{p}$), $\aV_{\permp}(C)$ consists of the semiregular permutation
representations, that is, those orthogonal $C$-representations which
admit an orthonormal basis to which the $C$ action restricts, and
which has at least one free orbit.  The set of model cells that we
take for $\aA$ in Theorem~\ref{thm:multAvar} is denoted $\aACyc$ (in
the case of $\aF_{\fin}$) and $\aApCyc$ (in the case of $\aF_{p}$).
The $\bT$-spectra in these sets are complicated to specify, but have
very nice properties, as stated in Lemma~\ref{lem:ACyc}.  We postpone
the details of the construction of 
these $\bT$-spectra to Section~\ref{sec:ACyc}, as 
the list of properties suffices for what we need here.

\begin{thm}\label{thm:circlephi}
The sets $\aACyc$ and $\aApCyc$ of $\bT$-spectra in
Definition~\ref{defn:ACyc} below satisfy the hypotheses of
Theorem~\ref{thm:multAvar} for $\aF=\aF_{\fin}$, $\aV=\aV_{\permp}$ and
for $\aF=\aF_{p}$, $\aV=\aV_{\permp}$, respectively. In particular, 
the forgetful functor from commutative ring
$\bT$-spectra to $\bT$-spectra creates
model structures for both the 
$\aACyc$-variant $\aV_{\permp}$-constrained $\aF_{\fin}$-model
structure and the 
$\aApCyc$-variant $\aV_{\permp}$-constrained $\aF_{p}$-model structure.  
These model structures are enriched over the standard model structure
in non-equivariant spectra.
Moreover in these
variant model structures:
\begin{enumerate}
\item For $\aACyc$, the endofunctors $\Phi_{n}$ preserve cofibrations
and acyclic cofibrations of both $\bT$-spectra
and commutative ring $\bT$-spectra. 
For $\aApCyc$, the endofunctors $\Phi_{p^{n}}$ preserve cofibrations
and acyclic cofibrations of both $\bT$-spectra
and commutative ring $\bT$-spectra.
\item If $R$ is a cofibrant commutative ring
$\bT$-spectrum, then $\bS\to R$ is a cofibration of
$\bT$-spectra.
\end{enumerate}
\end{thm}

The remainder of Theorem~\ref{main:Tmod} is the relationship to
non-equivariant spectra.  We state the following result here but prove
it in Section~\ref{sec:neACyc}. 

\begin{thm}\label{thm:necircle}
The forgetful functor from $\bT$-spectra to non-equivariant spectra
takes cofibrant objects in the $\aACyc$-variant 
$\aV_{\permp}$-constrained $\aF_{\fin}$-model structure to cofibrant objects in the
positive convenient $\Sigma$-model structure of
Definition~\ref{defn:Sigmodel}.  The free functor from non-equivariant
spectra to $\bT$-spectra takes cofibrant objects in the positive
convenient $\Sigma$-model structure on spectra of Definition~\ref{defn:Sigmodel}
to cofibrant objects in the $\aApCyc$-variant 
$\aV_{\permp}$-constrained $\aF_{p}$-model structure on $\bT$-spectra.
\end{thm}

The hypotheses of Theorem~\ref{thm:multAvar} imply in particular the
last two bullet point hypotheses of Theorem~\ref{thm:varAmodel}, which
give the unit, monoid, and pushout-product axioms
of~\cite[3.2--4]{SSAlgMod}, which then give model structures on module
and algebra categories.  When $R$ is a commutative
ring $\bT$-spectrum whose underlying $\bT$-spectrum is cofibrant in
the undercategory of $\bS$, then cofibrant
$R$-algebras are cofibrant as $\bT$-spectra in
the undercategory of $\bS$ and cofibrant $R$-modules are cofibrant as
$\bT$-spectra in these model structures.

The definition of $\aACyc$ and $\aApCyc$ is in terms of specific
generalized orbit desuspension spectra described in detail in
Section~\ref{sec:ACyc}.  What we need for the proof of the theorem is
given by the following lemma, proved in Section~\ref{sec:ACyc}.

\begin{lem}\label{lem:ACyc}
Let $A\in \aACyc$, $k\geq 1$.
\begin{enumerate}
\item $A$ is $i^{*}J^{\Gamma/\KG}\theta_{A}$, for some
$\Gamma/\KG$-compatible $Q$-faithful positive dimensional $(\Gamma,Q)$ vector
bundle $\theta_{A}$, where $\KG\lhd\Gamma$ is a finite
normal subgroup, $i$ is an isomorphism $i\colon \bT\iso
\Gamma/\KG$, and the vector bundle $\theta_{A}$ has base space a
finite product of orbits of $\Gamma$. Moreover, for every $\KG<\Lambda
\lhd \Gamma$, the $(\Gamma,Q)$ vector
bundle $\theta_{A}(\Lambda |\KG)$ is $Q$-faithful.
\item If $B\in \aACyc$, then $A\sma B$ is isomorphic to an element of
$\aACyc$; if $A,B\in \aApCyc$, the $A\sma B$ is isomorphic to an
element of $\aApCyc$.
\item $\Phi_{m}A$ is isomorphic to an element of $\aACyc$; if $A\in
\aApCyc$ and $m$ is a power of $p$, then $\Phi_{m}A$ is isomorphic to
an element of $\aACyc_{p}$.
\item $A^{(k)}/\Sigma_{k}$ is isomorphic to an element of $\aACyc$; if
$A\in \aApCyc$, then $A^{(k)}/\Sigma_{k}$ is isomorphic to an element
of $\aApCyc$.
\item There exists a cofibrant object $\tilde A$ in the complete
positive model structure on $\bT$-spectra and a
weak equivalence $\tilde A\to A$ such that the induced map $\tilde
A^{(k)}/\Sigma_{k}\to A^{(k)}/\Sigma_{k}$ is a weak equivalence.
\item For any $n\geq 1$, $A\sma F_{\bR^{n}}S^{0}$ is isomorphic to an
object of $\aACyc$, which is in $\aApCyc$ if $A$ is.
\end{enumerate}
\end{lem}

As a
consequence of the lemma, objects of $\aACyc$ behave well with respect
to smash product and the functors $\Phi_{n}$.  For other statements,
it is useful to provide a label for the spectra with the properties
listed below.

\begin{defn}\label{defn:Nice}
Let $\Nice$ be the class of $\bT$-spectra $A$
that satisfy all of the following properties:
\begin{enumerate}
\item For every $m\geq 1$, $\Phi_{m}A$ is flat for the smash product
of $\bT$-spectra.
\item For every $m\geq 1$ and every cofibrant approximation $\tilde
A\to A$ in the $\aV_{\permp}$-constrained $\aF_{\fin}$-model
structure, the map $\Phi_{m}\tilde A\to \Phi_{m}A$ is a
$\aF_{\fin}$-equivalence. 
\item For every $m,n\geq 1$, the canonical map $\Phi_{mn}A\to
\Phi_{n}\Phi_{m}A$ of~\cite[2.6]{BM-cycl} is an isomorphism.
\item For every $m,n\geq 1$ and every $\bT$-spectrum $X$, the canonical map
\[
\Phi_{n}(\Phi_{m}A)\sma \Phi_{n}X\to
\Phi_{n}(\Phi_{m}A\sma X)
\]
is an isomorphism.
\end{enumerate}
\end{defn}

It is clear that $\Nice$ is closed under $\Phi_{m}$, retracts, smash
products, pushouts along Hurewicz cofibrations, and transfinite
compositions of Hurewicz cofibrations.  We note that $\bS\in \Nice$ and
more generally, the cofibrant objects in the standard stable and
complete model structures
on $\bT$-spectra are in $\Nice$. 

\begin{prop}\label{prop:ACycE}
$\aACyc\subset \Nice$.
\end{prop}

\begin{proof}
This is simplified by remembering that if $A\in \aACyc$, then
$\Phi_{n}A$ is isomorphic to an object of $\aACyc$ by
Lemma~\ref{lem:ACyc}.(iii). Applying Lemma~\ref{lem:ACyc}.(i), we see
that objects of $\aACyc$ have property~(i) by
Theorem~\ref{thm:bundleflat}, have property~(ii) by
Theorem~\ref{thm:derphi}, have property~(iii) by
Theorem~\ref{thm:bundleiterphi}, and have property~(iv) by
Theorem~\ref{thm:sm}.
\end{proof}

Before beginning the proof of Theorem~\ref{thm:circlephi},
we state some auxiliary results that simplify some aspects of the
proof and are also
needed in other contexts.
The first allows us to use relative cell complexes in place of regular
cofibrations.  It is an easy consequence of the fact that in the
equivariant vector bundles used to construct objects of $\aACyc$, the
base spaces are compact.

\begin{prop}\label{prop:aACycsmall}
Objects in $\aACyc$ are topologically point-set small.
\end{prop}

The other result we need regards the interaction of the balanced smash
product with the endofunctors $\Phi_{n}$.  In the
statement, the undercategory case is designed to handle cofibrations
of commutative ring orthogonal spectra: if $R\to A$ is a cofibration
of commutative ring orthogonal spectra in one of the model categories
of Theorem~\ref{thm:circlephi}, then $R\to A$ is a cofibrant object in
the undercategory of $R$ in the category of (symmetric) $R$-modules.

\begin{prop}\label{prop:gfsma}
Let $R$ be a associative ring $\bT$-spectrum.  Let $M$ be a right $R$-module and
$N$ a left $R$-module.  If either $M$ or $N$ is cofibrant in the
respective $R$-module model structure (from
Theorem~\ref{thm:circlephi}) or in the undercategory of $R$ in the
respective $R$-module model structure, then for any $m\geq 1$, the canonical map
\[
\Phi_{m}M\sma_{\Phi_{m}R}\Phi_{m}N\to \Phi_{m}(M\sma_{R}N)
\]
is an isomorphism.
\end{prop}

\begin{proof}
Without loss of generality, we can assume that the cofibrant object is
a cell complex.  The $R$-module cells are then in particular of the form 
\[ R\sma A\sma (D^{n},\partial D^{n})_{+} \]
for $A\in \aACyc$, with the left (for $M$) or right (for $N$)
$R$-module structure on $R$.  Working cell by cell, the result follows from
property~(iv) defining~$\Nice$ for the objects in $\aACyc$.
\end{proof}

We are now ready for the proof of Theorem~\ref{thm:circlephi}.

\begin{proof}[Proof of Theorem~\ref{thm:circlephi}]
First, to check the hypotheses of
Theorem~\ref{thm:multAvar}, we observe the following:
\begin{itemize}
\item Smash product of objects of $\aACyc$ or $\aApCyc$ are isomorphic
to objects of $\aACyc$ or $\aApCyc$ (Lemma~\ref{lem:ACyc}.(ii)) and
are therefore cofibrant in their variant model structures.
\item The symmetric power of any object of $\aACyc$ is flat for the
smash product by Lemma~\ref{lem:ACyc}.(iv) and
Proposition~\ref{prop:ACycE}
(q.v.~Definition~\ref{defn:Nice}.(i)). 
\item Any $\aF_{\fin}$- or $\aF_{p}$-equivalence
$\tilde A\to A$ from a positive complete cofibrant object to an object
of $\aACyc$ induces an $\aF_{\fin}$- or $\aF_{p}$-equivalence on
symmetric powers by Lemma~\ref{lem:ACyc}.(v).   
\end{itemize}
The last bullet point holds because the symmetric power functor
preserves $\aF_{\fin}$- and $\aF_{p}$-equivalences between cofibrant
objects in the positive complete model structure (e.g.,
by applying~\cite[B.111,B.116]{HHR} and restricting the group from $\bT$
to $C_{m}$ or $C_{p^{k}}$).  Objects that are cofibrant in the
$\aV_{\permp}$-constrained $\aF_{\fin}$-model structure and in the 
$\aV_{\permp}$-constrained $\aF_{p}$-model structure are in particular
cofibrant in 
the positive complete model structure, so the last bullet point in
particular implies hypothesis~(ii) of Theorem~\ref{thm:multAvar}.
We therefore get $\aACyc$ and $\aApCyc$ variant model structures on
$\bT$-equivariant commutative ring spectra. 

To check the enrichment statement, we apply Lemma~\ref{lem:ACyc}.(vi) and
Proposition~\ref{prop:enriched}.

To check statement~(ii), that cofibrant commutative ring $\bT$-spectra have underlying
$\bT$-spectra cofibrant in the undercategory of $\bS$, we need to
see that $\aACyc$ and $\aApCyc$ are
closed up to isomorphism under symmetric powers; this is
Lemma~\ref{lem:ACyc}.(iv).

To check statement~(i) in the case of $\bT$-spectra, we use the fact
that a cofibration is a retract of a relative cell complex, with cells
of the form 
\[
A\sma (D^{n},\partial D^{n})_{+}
\]
for $A$ in $\aACyc$ or $\aApCyc$.
Because $\aACyc$ is closed up to isomorphism under $\Phi_{n}$,
$\Phi_{n}$ preserves cell attachments of this form (with $A$ replaced
by $\Phi_{n}A$).  Likewise $\aApCyc$ is closed up to isomorphism under
$\Phi_{p^{n}}$.  The endofunctors $\Phi_{n}$ preserve Hurewicz cofibrations,
and pushouts along Hurewicz cofibrations give Mayer-Vietoris sequences
of homotopy groups in $\bT$-spectra, so
$\Phi_{n}$ and $\Phi_{p^{n}}$ also preserve acyclic cofibrations in
the $\aACyc$-variant $\aV_{\permp}$-constrained $\aF_{\fin}$-model
structure and $\aApCyc$-variant $\aV_{\permp}$-constrained
$\aF_{p}$-model structure on $\bT$-spectra, respectively.

To check statement~(i) in the case of commutative ring $\bT$-spectra,
the main new tool is 
Theorem~\ref{thm:calcgeoSym}.  The usual retract and colimit
argument reduces the statement to showing that applying $\Phi_{n}$ to
a cell attachment in this model structure
\[
R\to R\sma_{\bP(A\sma \partial D^{k}_{+})}(\bP(A\sma D^{k}_{+})) 
\]
gives a cofibration in this model structure. (Here $\bP$ denotes the
free functor from $\bT$-spectra to
commutative ring $\bT$-spectra.)  By
Proposition~\ref{prop:gfsma}, we have
\[
\Phi_{n}(R\sma_{\bP(A\sma \partial D^{k}_{+})}(\bP(A\sma D^{k}_{+})))
\iso (\Phi_{n}R)\sma_{\Phi_{n}(\bP(A\sma \partial D^{k}_{+}))}
\Phi_{n}(\bP(A\sma D^{k}_{+})), 
\]
so it suffices to check that 
\[
\Phi_{n}(\bP(A\sma \partial D^{k}_{+}))\to
\Phi_{n}(\bP(A\sma D^{k}_{+}))
\]
is a cofibration when $A\in \aACyc$.  By Theorem~\ref{thm:calcgeoSym},
this map is the retract of a finite smash product of maps of the
form
\[
\bP(X\sma (\partial D^{k})^{q}_{+})\to 
\bP(X\sma (D^{k})^{q}_{+})
\]
where $X$ is a retract of $A^{(q)}/\Sigma_{q}$. Since
$A^{(q)}/\Sigma_{q}$ is a isomorphic to an object of $\aACyc$, $X$ is a
retract of an object of $\aACyc$, and the map is a cofibration in the
$\aACyc$-variant model structure for commutative
ring $\bT$-spectra.  Any acyclic cofibration is a transfinite
composition of maps, each of which is formed as a pushout along an
acyclic cofibration of cofibrant objects.  Because as just shown
$\Phi_{n}$ preserves cofibrations and because it preserves weak
equivalences between cofibrant objects, it follows that $\Phi_{n}$
preserves acyclic cofibrations in the $\aACyc$-variant model structure
for commutative ring $\bT$-spectra.  The
argument for the case of $\aApCyc$ is entirely similar.
\end{proof}

\subsection*{Model structures for Theorem~\protect\ref{main:Tmod}}
Theorem~\ref{thm:circlephi} combined with Theorem~\ref{thm:multAvar}
gives a model structure that satisfies the conclusions of
Theorem~\ref{main:Tmod}.  The following is a step-by-step guide
through the definitions and results above to see this.
Theorem~\ref{main:Tmod} gives two sets of statements, one for
$\aF_{\fin}$ and one for $\aF_{p}$; let $\aF$ denote $\aF_{\fin}$ or
$\aF_{p}$ in the respective cases, and let $\aA$ denote the set
$\aACyc$ or $\aApCyc$ of Definition~\ref{defn:ACyc} in the respective
cases.  Write $\aA^{\comp}$ for the set of $\bT$-spectra of the form
\[
\bT_{+}\sma_{C}F_{V}S^{0}
\]
where $C\in \aF$ and $V$ is an orthogonal $C$-representation
whose isomorphism class is in $\aV_{\permp}(C)$: $V$ is non-zero and
has an orthonormal basis to which the $C$ action restricts. Finally,
let $\aA^{\natural}$ be the set of $\bT$-spectra of the form
$\bT/C_{+}\sma F_{V}S^{0}$ where $C\in \aF$ and $V$ is an orthogonal
$\bT$-representation whose underlying isomorphism class of orthogonal
$C$-representation is in $\aV_{\permp}(C)$.  (We will more concisely
write this condition as ``$V$ is in $\aV_{\permp}(C)$'' in what follows.)

We then get model structures on
$\bT$-spectra for $\aA$, $\aA^{\comp}$, and $\aA^{\natural}$, where:
\begin{itemize}
\item The cofibrations are the maps $X\to Y$ that are retracts of
relative $\aA$-, $\aA^{\comp}$, and $\aA^{\natural}$-cell complexes,
respectively (see Definition~\ref{defn:cell}). 
\item The weak equivalences are the $\aF$-equivalences
\end{itemize}
In the case of $\aA^{\natural}$ (or $\aA^{\comp}$), a map $X\to Y$ is a
fibration if and 
only if for every $C\in \aF$ and every orthogonal $\bT$-representation
(or $C$-representation, respectively)
$V$ in $\aV_{\permp}(C)$, the map $X(V)\to 
Y(V)$ is a $C$-equivariant Serre fibration (a Serre fibration on all
fixed point spaces) and for every orthogonal $\bT$-representation (or
$C$-representation, respectively) $W$ in $\aV_{\permp}(C)$
containing such a $V$, the diagram of $C$-spaces
\[
\xymatrix{%
X(V)\ar[r]\ar[d]&\Omega^{W-V}X(W)\ar[d]\\
Y(V)\ar[r]&\Omega^{W-V}Y(W)
}
\]
is homotopy cartesian.  In the case of $\aA$, a map $X\to Y$
is a fibration if and only if:
\begin{itemize}
\item $X\to Y$ is a fibration in the model structure for $\aA^{\natural}$,
\item For every $A\in \aA$, the map $\Spectra^{G}(A,X)\to
\Spectra^{G}(A,Y)$ is a Serre fibration, and
\item For every $A\in \aA$, there exists a cofibrant object
$\tilde A$ in the model structure for $\aA^{\natural}$ and a weak
equivalence $\tilde A\to A$ such that the diagram
\[
\xymatrix{%
\Spectra^{G}(A,X)\ar[r]\ar[d]
&\Spectra^{G}(\tilde A,X)\ar[d]\\
\Spectra^{G}(A,Y)\ar[r]
&\Spectra^{G}(\tilde A,Y)
}
\]
is homotopy cartesian.  
\end{itemize}
(The fibrations in the model structure in the case of $\aA$ can
alternatively be characterized in terms of $\aA^{\comp}$, just
replacing $\aA^{\natural}$ by~$\aA^{\comp}$ in the statements.)

These model structures follow for example from
Theorems~\ref{thm:FVmodel} and~\ref{thm:FVcmodel} in the case of
$\aA^{\natural}$ and $\aA^{\comp}$ (respectively)
and Theorem~\ref{thm:varAmodel} with Proposition~\ref{prop:varAfib}
and Proposition~\ref{prop:aACycsmall} in the case of $\aA$.

The identity gives a zigzag of Quillen equivalences of model
structures relating these two model structures and the standard (or
positive) complete stable model structure on $\bT$-spectra: the identity is a Quillen left adjoint from the
model structure for $\aA^{\natural}$ to the standard (or positive)
model structure and to the model structure for $\aA$ and the
identity is also a Quillen left adjoint from the model structure for
$\aA^{\comp}$ to the complete (or positive complete) model structure
and to the model structure for $\aA$.  Writing the functors in the
direction of Quillen left adjoints, we have the following zigzags of
Quillen equivalences of model structures on the category of
$\bT$-spectra with underlying functor the identity.
\begin{equation}\label{eq:quillenzigzag}
\begin{gathered}
\xymatrix@R-2pc@C-1pc{%
\text{standard}&\text{positive}\ar[l]
&\text{model}(\aA^{\natural})\ar[l]\ar[r]
&\text{model}(\aA)\\
\text{complete}&\text{pos. complete}\ar[l]
&\text{model}(\aA^{\comp})\ar[l]\ar[r]
&\text{model}(\aA)
}
\end{gathered}
\end{equation}

Theorem~\ref{thm:circlephi} now proves Theorem~\ref{main:Tmod}:
Theorem~\ref{thm:circlephi} asserts that $\Phi_{n}$ preserves
cofibrations and acyclic cofibrations in the model structure for $\aA$ on $\bT$-spectra, which is statement~(i) of Theorem~\ref{main:Tmod}.

Theorem~\ref{thm:circlephi} asserts that
the hypotheses of Theorem~\ref{thm:multAvar} hold which implies that
the forgetful functor creates a topological model structure on
commutative ring $\bT$-spectra.  This is
statement~(ii) of Theorem~\ref{main:Tmod}.

Theorem~\ref{thm:circlephi} asserts that $\Phi_{n}$ preserves
cofibrations in the model structure for $\aA$ on the category of
commutative ring $\bT$-spectra, and this in
particular implies statement~(iii) of Theorem~\ref{main:Tmod}.

Theorem~\ref{thm:circlephi} asserts that in the model structures for
$\aA$, the forgetful functor takes the cofibrant
commutative ring $\bT$-spectra to cofibrant
objects in the undercategory of $\bS$ in $\bT$-spectra.  This is
statement~(iv) of Theorem~\ref{main:Tmod}.

The final part of Theorem~\ref{main:Tmod}, relating the model
structures on $\bT$-spectra to the positive convenient $\Sigma$-model
on non-equivariant spectra is proved as follows.  Since cofibrations
both equivariantly and non-equivariantly are formed as sequential
colimits of pushouts along certain maps between certain cofibrant
objects, Theorem~\ref{thm:necircle} (and the fact that $\aApCyc\subset
\aACyc$) implies that both the forgetful functor and free functor
preserve cofibrations.  Since the free functor preserves all weak
equivalences, by adjunction, the forgetful functor also preserves
fibrations.

\section{{P}re-cyclotomic spectra and Theorems~\sref$main:cycmod$
and~\sref$main:map$}
\label{sec:precyc}

The purpose of this section is to prove Theorem~\ref{main:map} and
Theorem~\ref{main:cycmod} (and its addendum) using the results of the
previous section.

The model structure on \ppc spectra alluded to in the
statement of Theorem~\ref{main:cycmod} is the one arising out of
Theorem~\ref{thm:circlephi} above.  Here is the precise statement both
in this case and in the case of pc spectra.  The proof
given in~\cite{BM-cycl} for the model structure considered there
extends to prove this case.

\begin{thm}\label{thm:cycmodel}
The category of \ppc spectra (respectively,
pc spectra) admits a topological model
structure with fibrations and weak equivalences created by the
forgetful functor to $\bT$-equivariant spectra with the
$\aApCyc$-variant $\aV_{\permp}$-constrained $\aF_{p}$-model
structure (resp., the $\aACyc$-variant 
$\aV_{\permp}$-constrained $\aF_{\fin}$-model structure).  This model structure is
enriched over the standard model structure on non-equivariant
spectra. 
\end{thm}

The generating cofibrations in the model structure are of the form
$(\bC A)\sma (D^{n},\partial D^{n})_{+}$ where
\[
\bC A=\bigvee_{k\geq 0} \Phi_{p^{k}} A
\]
for $A\in \aApCyc$, or in the pc case
\[
\bC A=\bigvee_{n\geq 1}\Phi_{n}A
\]
for $A\in \aACyc$.  We see in both cases that these generating
cofibrations are cofibrations in the corresponding model structure of
Theorem~\ref{thm:circlephi}, which implies the following proposition. 

\begin{prop}\label{prop:cyccof}
The forgetful functor from \ppc spectra (resp.,
pc spectra) to $\bT$-spectra
preserves cofibrations and acyclic cofibrations for the relevant model
structures in Theorem~\ref{thm:cycmodel}.
\end{prop}

Proposition~\ref{prop:ACycE} implies that cofibrant objects in the
model structure above are in the class $\Nice$ of~\ref{defn:Nice}.  By
definition, if $X\in \Nice$ and $Y$ is any $\bT$-spectrum, then the canonical map 
\[
\Phi_{n}X\sma \Phi_{n}Y\to \Phi_{n}(X\sma Y)
\]
is an isomorphism.  For $X$ and $Y$ \ppc spectra whose
underlying $\bT$-spectra are in $\Nice$, $X\sma
Y$ gets a \ppc structure map
\[
\Phi_{p}(X\sma Y)\iso \Phi_{p}X\sma \Phi_{p}Y\to X\sma Y
\]
from the smash product of the \ppc structure maps on $X$
and $Y$.  This constructs a symmetric monoidal product on the
full subcategory of \ppc spectra whose underlying
objects are in $\Nice$.  In particular, this establishes part~(i) of
Theorem~\ref{main:cycmod}.  The analogue for pc spectra holds by the same
observation. 

For part~(ii) of Theorem~\ref{main:cycmod}, we recall the definition
of associative ring and commutative ring \ppc spectra.
An associative ring \ppc (respectively, pc)
spectrum consists of a \ppc (resp., pc)
spectrum $A$ together with a \term{unit} map $\bS\to A$ that is a map
of \ppc (resp., pc) spectra and a
\term{multiplication} map of $\bT$-spectra
$A\sma A\to A$ such that the unit and multiplication give $A$ the
structure of an associative ring $\bT$-spectrum
and the diagram
\[
\xymatrix{%
\Phi_{p} A\sma \Phi_{p}A\ar[r]\ar[d]&\Phi_{p}(A\sma A)\ar[r]&\Phi_{p}A\ar[d]\\ 
A\sma A\ar[rr]&&A
}
\]
constructed from the multiplication and structure maps commutes (for
all primes $p$ in the pc case).  An associative ring
\ppc (resp., pc) spectrum is a commutative
ring when the multiplication is commutative.

Proposition~\ref{prop:gfsma} is the key to the construction of the
model structure for commutative ring \ppc spectra.  The
generating cofibrations in this model structure are obtained by
applying the free commutative ring functor
\[
\bP X = \bigvee_{k\geq 0} X^{(k)}/\Sigma_{k}
\]
to the generating cofibrations for the model structure on
\ppc spectra of Theorem~\ref{thm:cycmodel}.  In the
resulting map of commutative ring \ppc spectra, both the
source and target have underlying $\bT$-spectra
formed as the wedge of $\bS$ with a wedge of objects in
$\aApCyc$ (by Lemma~\ref{lem:ACyc}), and in particular are in
$\Nice$.  If we write such a cell as $A\to B$, then for any
commutative ring \ppc spectrum $R$, and map $A\to R$, the cell attachment
$R\sma_{A}B$ satisfies
\[
\Phi_{p}(R\sma_{A}B)\iso \Phi_{p}R\sma_{\Phi_{p}A} \Phi_{p}B
\]
and we can give it the \ppc structure map induced by the
\ppc structure maps on $R$, $A$, and $B$ (which by
hypothesis are maps of commutative ring $\bT$-spectra). Thus, we can form the required cell attachments in
commutative ring \ppc spectra to reproduce the usual
argument constructing a model structure on the commutative ring
objects.  This gives the model structure for part~(ii) of
Theorem~\ref{main:cycmod}. The case for commutative ring
pc spectra is entirely similar.

\begin{thm}\label{thm:crcycmodel}
The category of commutative ring \ppc spectra
(resp., pc spectra)
admits a model structure with fibrations and weak equivalences created
in the model structure of \ppc spectra (resp.,
pc spectra) of Theorem~\ref{thm:cycmodel}. 
\end{thm}

The cells that we used to construct cofibrant commutative ring
\ppc spectra are of the form 
\[
\bP\bC (A\sma \partial D^{n}_{+})\to 
\bP\bC (A\sma D^{n}_{+})
\]
for $A\in \aApCyc$.  On underlying commutative ring
$\bT$-spectra, this is a cofibration of cofibrant objects in the
$\aApCyc$-variant $\aV_{\permp}$-constrained $\aF_{p}$-model
structure (by Proposition~\ref{prop:cyccof}). We see that the
forgetful functor from commutative ring \ppc spectra to 
commutative ring $\bT$-spectra preserves
cofibrations.  Analogous observations hold in the case of pc
spectra. 

\begin{prop}\label{prop:crcyccof}
The forgetful functor from commutative ring \ppc spectra (resp.,
pc spectra) to commutative ring $\bT$-spectra
preserves cofibrations and acyclic cofibrations for the model
structure in Theorem~\ref{thm:crcycmodel} on the source and the
$\aApCyc$-variant $\aV_{\permp}$-constrained  $\aF_{p}$-model
structure (resp.,
$\aACyc$-variant $\aV_{\permp}$-constrained $\aF_{\fin}$-model
structure) on the target.
\end{prop}

To finish the proof of Theorem~\ref{main:cycmod}, we need to see that
for a cofibrant commutative ring \ppc spectrum, the
underlying \ppc spectrum is cofibrant in the
undercategory of $\bS$.  This easily reduces to showing that for
$A_{1},\dotsc,A_{r}\in \aApCyc$,
\[
\bC A_{1}\sma\dotsb\sma\bC A_{r}
\]
is a wedge of \ppc spectra of the form $\bC A$ with
$A\in \aApCyc$.  This is then an easy combinatorial argument using the fact
that $\aApCyc$ is closed (up to isomorphism) under $\Phi_{p^{k}}$.

For Addendum~\ref{main:cycmod}, recall that for $R$ an associative 
ring \ppc (resp., pc) spectrum an
$R$-module \ppc (resp., pc) spectrum consists of an
$R$-module $M$ (in $\bT$-spectra) together with
a \ppc (resp., pc) structure on $M$ such
that the diagram  
\[
\xymatrix{%
\Phi_{p}R\sma \Phi_{p}M\ar[r]\ar[d]
&\Phi_{p}(R\sma M)\ar[r]
&\Phi_{p}M\ar[d]\\
R\sma M\ar[rr]&&M
}
\]
constructed from the structure maps and the action map commutes (for
all primes $p$ in the pc case).  When $R$ is commutative,
we also have notions of associative $R$-algebras and commutative
$R$-algebras in \ppc spectra (resp., pc
spectra): an associative $R$-algebra consists of a map of associative
ring \ppc spectra (resp., pc spectra) $\eta
\colon R\to A$ with the property that $\eta$ is also a map of associative
ring \ppc spectra (resp., pc spectra) $R\to
A^{\op}$. A commutative ring \ppc spectrum (resp.,
pc spectrum) is just a map of commutative ring
\ppc spectra (resp., pc spectra) $R\to A$.
The statements in the Addendum and the corresponding results for
pc spectra are clear from the work above.
We also have the following generalization to the case when $R$ is just
in $\Nice$.

\begin{thm}\label{thm:cycmodnice}
Let $R$ be a commutative ring \ppc (resp.,
pc) spectrum whose underlying $\bT$-spectrum is in $\Nice$ (e.g., if it cofibrant in the model structure
of Theorem~\ref{thm:crcycmodel}).  
\begin{enumerate}

\item The category of $R$-module \ppc (resp., pc) spectra admits a
topological model structure with fibrations and weak equivalences
created by the forgetful functor to the model structure on \ppc
(resp., pc) spectra of Theorem~\ref{thm:cycmodel}.  This model
category is enriched over the standard model category of
non-equivariant spectra. 
\item The category of $R$-algebra \ppc (resp., pc) spectra admits a
topological model structure with fibrations and weak equivalences
created by the forgetful functor to the model structure on \ppc
(resp., pc) spectra of Theorem~\ref{thm:cycmodel}.
\item The category of commutative $R$-algebra \ppc (resp., pc) spectra
admits a topological model structure with fibrations and weak
equivalences created by the forgetful functor to the model structure
on \ppc (resp., pc) spectra of Theorem~\ref{thm:cycmodel}.
\end{enumerate} 
Moreover, in each of these model structures, the underlying
$\bT$-spectra of the cofibrant objects in this model category lie in
$\Nice$. 
\end{thm}

When $R$ is a cofibrant commutative ring \ppc spectrum,
or more generally any commutative ring \ppc spectrum
whose underlying $\bT$-spectrum is in $\Nice$,
the smash product $\sma_{R}$ can be partially defined just like the
smash product $\sma$ on \ppc spectra.  When $M$ is a
cofibrant $R$-module in \ppc spectra and $X$ is any
$R$-module in \ppc spectra, we get a canonical
\ppc structure on the smash product $M\sma_{R}X$ of
their underlying $\bT$-spectra with the structure map
\[
\Phi_{p}(M\sma_{R}X)\iso \Phi_{p}M\sma_{\Phi_{p}R}\Phi_{p}X
\to M\sma_{R}X
\]
induced by the structure maps on $M$, $R$, and $X$, using the
isomorphism of Proposition~\ref{prop:gfsma}.  This in particular
constructs a symmetric monoidal smash product on the full subcategory
of cofibrant $R$-modules in \ppc spectra.  

Before turning to Theorem~\ref{main:map}, we make a brief remark about
monadicity.  While equivariant commutative ring spectra may be defined
in terms of the monad $\bP$ on equivariant orthogonal spectra, there
is no corresponding definition for commutative ring \ppc
spectra, because the smash product does not extend to the whole
category of \ppc spectra.  However, by the work above,
if we restrict to the subcategory of \ppc spectra under
$\bS$ that are cofibrant, we do have a smash product, and on this
category we can define the monad $\tilde P$ by
\[
\tilde P X = \coprod_{n=1}^{\infty}X^{(n)}/\Sigma_{n} 
\]
(where the coproduct is taken in the undercategory of $\bS$, i.e., it
is wedge with all copies of $\bS$ identified).  Analogous observations
hold for pc spectra.  Formally, we get the
following observation.

\begin{prop}
The category of algebras over the monad $\tilde P$ in the category of
cofibrant objects in the undercategory of $\bS$ in \ppc
spectra (respectively, pc spectra) with the model
structure of Theorem~\ref{thm:cycmodel} is the full subcategory of the
category of commutative ring \ppc spectra (respectively,
commutative ring pc spectra) whose underlying
\ppc spectra (respectively, pc spectra)
under $\bS$ are cofibrant.
\end{prop}

For Theorem~\ref{main:map}, the main ingredient beyond
Proposition~\ref{prop:crcyccof}  we need is that the
model structure above is topological, which implies
that maps into a fibrant object convert cofibrations into fibrations.
The proof of the corresponding result~\cite[5.13]{BM-cycl} for mapping
spectra of \ppc spectra in the standard model structure
works in this context to prove the mapping space result for
commutative ring \ppc spectra in this model structure.
The following is a restatement of Theorem~\ref{main:map} naming the
specific model structure.

\begin{thm}\label{thm:mainmap}
If $A$ and $B$ are commutative ring \ppc spectra where
the underlying commutative ring $\bT$-spectra
of $A$ and $B$ are cofibrant and fibrant, respectively, in the
$\aApCyc$-variant $\aV_{\permp}$-constrained $\aF_{p}$-model
structure, then the derived space of maps of commutative ring
\ppc spectra from $A$ to $B$, $\dR\Com^{\pcyc}(A,B)$, is
given by the homotopy equalizer
\[
\dR\Com^{\pcyc}(A,B) \simeq
\xymatrix@-1pc{%
\hoeq\big(\Com^{\bT}(A,B)\ar@<-.5ex>[r]\ar@<.5ex>[r]&\Com^{\bT}(\Phi A,B)\big)
}
\]
where $\Com^{\bT}$ denotes the point-set mapping space in
commutative ring $\bT$-spectra.  Here one map
is induced by the cyclotomic structure map $\Phi A\to A$ 
and the other map is induced by applying $\Phi$ and post-composing
with the cyclotomic structure map $\Phi B\to B$.
\end{thm}

The proof works just like the proof of~\cite[5.13]{BM-cycl}: A check
of definitions shows that the point-set space of maps of 
\ppc spectra from $A$ to $B$, $\Com^{\pcyc}(A,B)$ is
the equalizer
\[
\xymatrix@-1pc{%
\Eq\big(\Com^{\bT}(A,B)\ar@<-.5ex>[r]\ar@<.5ex>[r]&\Com^{\bT}(\Phi A,B)\big)
}
\]
and the point-set mapping space represents the derived mapping space
when $A$ is cofibrant and $B$ is fibrant in the model structure on
commutative ring \ppc spectra of
Theorem~\ref{thm:crcycmodel}.  The proof of~\cite[5.12]{BM-cycl}
generalizes to show that the map from the equalizer to the homotopy
equalizer is a weak equivalence under these hypotheses.  If $A$ and
$B$ are just cofibrant and fibrant, respectively, as commutative ring $\bT$-spectra, their cofibrant and fibrant
approximation $A_{c}\to A$ and $B\to B_{f}$, respectively, in
commutative ring \ppc spectra calculate the derived
mapping space, and the map from the homotopy equalizer for $A,B$ to
the homotopy equalizer for $A_{c},B_{f}$ is a weak equivalence.

In the case when $R$ is cofibrant as a commutative ring
\ppc spectrum (or has underlying $\bT$-spectrum in $\Nice$), we also
get the corresponding result in 
the categories of $R$-modules and $R$-algebras. For convenience we
write $\Alg^{\bT}_{R}$ for the category of associative $R$-algebras in
$\bT$-spectra and $\Alg^{\pcyc}_{R}$ for the category of associative
$R$-algebras in \ppc spectra.  We also write $\Alg^{\bT}_{R}(-,-)$ and
$\Alg^{\pcyc}(-,-)$ for the corresponding (point-set) mapping spaces and
$\dR\Alg^{\bT}_{R}(-,-)$ and $\dR\Alg^{\pcyc}_{R}(-,-)$ for the derived mapping
spaces.  Checking definitions, we have that for $R$-algebras $A$ and
$B$, $\Alg^{\pcyc}_{R}(A,B)$ is the (point-set) equalizer
\[
\xymatrix@-1pc{%
\Eq\big(\Alg^{\bT}_{R}(A,B)\ar@<.5ex>[r]\ar@<-.5ex>[r]
&\Alg^{\bT}_{\Phi_{p}R}(\Phi_{p}A,B)\big)
}
\]
where on the right $B$ becomes a $\Phi_{p}R$-algebra via the structure
map $\Phi_{p}R\to R$.  In the equalizer one map is induced by the
reduction of scalars map $\Alg^{\bT}_{R}\to \Alg^{\bT}_{\Phi_{p}R}$ and the
structure map $\Phi_{p}A\to A$ and the other by applying the
continuous functor $\Phi_{p}\colon \Alg^{\bT}_{R}\to \Alg^{\bT}_{\Phi_{p}R}$ and
using the structure map $\Phi_{p}B\to B$.  From here, the following
result follows from the same argument outline as~\cite[5.13]{BM-cycl}.

\begin{thm}
Let $R$ be a commutative ring \ppc spectrum, whose underlying
$\bT$-spectrum is in $\Nice$ (which holds for example if 
$R$ is cofibrant).  Let $A$ and $B$ be associative $R$-algebras in
\ppc spectra.  Assume that on underlying $\bT$-equivariant associative $R$-algebras, $A$ is cofibrant as a $\bT$-equivariant associative
$R$-algebra and $B$ is fibrant as a $\bT$-equivariant associative
$R$-algebra in the $\aApCyc$-variant
$\aV_{\permp}$-constrained  $\aF_{p}$-model structure.  Then the derived mapping space
$\dR\Alg_{R}^{\pcyc}(A,B)$ is given by the homotopy equalizer
\[
\dR\Alg_{R}^{\pcyc}(A,B)\simeq
\xymatrix@-1pc{%
\hoeq\big(\Alg^{\bT}_{R}(A,B)\ar@<.5ex>[r]\ar@<-.5ex>[r]
&\Alg^{\bT}_{\Phi_{p}R}(\Phi_{p}A,B)\big).
}
\]
\end{thm}

For $R$ any associative ring \ppc spectrum and $X$, $Y$
any $R$-modules in \ppc spectra (without hypotheses), we have the mapping
spectrum $F^{\pcyc}_{R}(X,Y)$, which is defined to be the orthogonal spectrum
obtained as the equalizer
\[
\xymatrix@-1pc{%
\Eq\big(F^{\bT}_{R}(X,Y) \ar@<-.5ex>[r]\ar@<.5ex>[r]&F^{\bT}_{\Phi_{p}R}(\Phi_{p}X,Y)\big)
}
\]
where $F^{\bT}_{R}$ denotes the (non-equivariant) spectrum of
$R$-module maps of the underlying $\bT$-equivariant $R$-modules.
The maps are analogous to those above: one map is induced by the
reduction of scalars from $R$-modules to $\Phi_{p}R$-modules followed
by the structure map $\Phi_{p}X\to X$ for $X$ and the other is induced
by the spectrally enriched functor $\Phi_{p}$ from $R$-modules to
$\Phi_{p}R$-modules and the structure map $\Phi_{p}Y\to Y$ for $Y$.
The space of $R$-module \ppc spectrum maps from $X$ to
$Y$ is the zeroth space of the orthogonal spectrum $F^{\pcyc}_{R}$
above (because $F^{\bT}_{R}$ has this property for the category of
$R$-modules and zeroth space commutes with equalizers). This
construction satisfies the spectral version of Quillen's Axiom~SM7.
(This is precisely the assertion of Theorem~\ref{thm:cycmodnice}.(i)
of enrichment over non-equivariant spectra.)

This implies that the derived functor $\dR F_{R}^{\pcyc}(X,Y)$ exists and
is calculated by cofibrant approximation of $X$ and fibrant
approximation of $Y$.  Again the argument of~\cite[5.13]{BM-cycl} generalizes to prove the
following result.

\begin{thm}
Let $R$ be an associative ring \ppc spectrum, whose underlying
$\bT$-spectrum is in $\Nice$ (which holds for example if 
$R$ is cofibrant or cofibrant commutative).
Let $X$ and $Y$ be $R$-modules in
\ppc spectra.    Assume that on underlying $\bT$-equivariant
$R$-modules, $A$ is cofibrant as a $\bT$-equivariant 
$R$-module and $B$ is fibrant as a $\bT$-equivariant 
$R$-module in either the standard $\aF_{p}$-model structure or the
model structure from Theorem~\ref{thm:cycmodnice}.  Then the derived mapping space
$\dR\Mod_{R}^{\pcyc}(A,B)$ is given by the homotopy equalizer
\[
\dR\Mod_{R}^{\pcyc}(A,B)\simeq
\xymatrix@-1pc{%
\big(\Mod^{\bT}_{R}(A,B)\ar@<.5ex>[r]\ar@<-.5ex>[r]
&\Mod^{\bT}_{\Phi_{p}R}(\Phi_{p}A,B)\big)
}
\]
and the derived mapping spectrum $\dR F_{R}^{\pcyc}(X,Y)$ is
given by the homotopy equalizer of
\[
\dR F_{R}^{\pcyc}(X,Y)\simeq
\xymatrix@-1pc{%
\hoeq\big(F^{\bT}_{R}(A,B)\ar@<.5ex>[r]\ar@<-.5ex>[r]
&F^{\bT}_{\Phi_{p}R}(\Phi_{p}A,B)\big).
}
\]
\end{thm}

The pc case is more complicated than the
\ppc case in precisely the same way as in the
unstructured case of~\cite[\S5]{BM-cycl}.   
For commutative ring pc spectra, the proof
of~\cite[5.18]{BM-cycl} works in the current context to prove the
following theorem, the pc analogue of Theorem~\ref{main:map}.  See~\cite[5.15]{BM-cycl} for notation.

\begin{thm}
If $A$ and $B$ are commutative ring pc spectra where
the underlying commutative ring $\bT$-spectra
of $A$ and $B$ are cofibrant and fibrant, respectively, in the
$\aACyc$-variant $\aV_{\permp}$-constrained $\aF_{fin}$-model
structure, then the derived space of maps of commutative ring
pc spectra from $A$ to $B$, $\dR\Com^{\cyc}(A,B)$ is
given by the homotopy limit
\[
\dR\Com^{\cyc}(A,B)\simeq \holim_{\Theta}\Com^{\bT}(\Phi_{n}A,B)
\]
with maps as in~\cite[(5.14)]{BM-cycl}.
\end{thm}

The argument works the same for $R$-algebras and $R$-modules to give the following results.

\begin{thm}
Let $R$ be a commutative ring pc spectrum, whose underlying
$\bT$-spectrum is in $\Nice$ (which holds for example if 
$R$ is cofibrant). 
Let $A$ and $B$ be associative $R$-algebras in
pc spectra.   Assume that on underlying $\bT$-equivariant associative $R$-algebras, $A$ is cofibrant as a $\bT$-equivariant associative
$R$-algebra and $B$ is fibrant as a $\bT$-equivariant associative
$R$-algebra in the $\aACyc$-variant 
$\aV_{\permp}$-constrained $\aF_{\fin}$-model structure.  Then the derived mapping space
$\dR\Alg_{R}^{\cyc}(A,B)$ is given by the homotopy limit
\[
\dR\Alg_{R}^{\cyc}(A,B)\simeq
\holim_{\Theta}\Alg^{\bT}_{\Phi_{n}R}(\Phi_{n}A,B)
\]
with maps as in~\cite[(5.14)]{BM-cycl}.
\end{thm}

\begin{thm}
Let $R$ be an associative ring pc spectrum, whose
underlying $\bT$-spectrum is in $\Nice$
(which holds for example if $R$ is cofibrant or cofibrant
commutative). Let $X$
and $Y$ be $R$-modules in pc spectra.   Assume that on underlying $\bT$-equivariant
$R$-modules, $A$ is cofibrant and $B$ is fibrant in either the standard
$\aF_{\fin}$-model structure or the model structure of
Theorem~\ref{thm:cycmodnice}.  Then the derived mapping space 
$\dR\Mod_{R}^{\cyc}(A,B)$ is given by the homotopy limit
\[
\dR\Mod_{R}^{\cyc}(A,B)\simeq
\holim_{\Theta}\Mod^{\bT}_{\Phi_{n}R}(\Phi_{n}X,Y)
\]
and the derived mapping spectrum $\dR F_{R}^{\cyc}(X,Y)$ is
given by the homotopy limit
\[
\dR F_{R}^{\cyc}(X,Y)\simeq
\holim_{\Theta}F^{\bT}_{\Phi_{n}R}(\Phi_{n}X,Y)
\]
with maps as in~\cite[(5.14)]{BM-cycl}.
\end{thm}

\section{Construction of \texorpdfstring{$\aACyc$}{ACyc} and proof of Lemma~\sref$lem:ACyc$}
\label{sec:ACyc}

We construct $\aACyc$ and $\aApCyc$ by specifying explicit generalized
orbit desuspension spectra that we can show have the symmetric power
property required for the hypotheses of Theorem~\ref{thm:multAvar}
while at the same time form sets that are closed (up to
isomorphism) under $\Phi_{n}$ (or $\Phi_{p^{n}}$ in the case of
$\aApCyc$), smash products, and symmetric powers.  This requires
leaving the context of the Lie group $\bT$ and working with finite
extensions.  The most convenient way to set this up is as follows.
For $n\geq 1$, let
\[
\bT(n):=\bR/n\bZ, \qquad C_{m}(n)=\langle \tfrac{n}m\rangle<\bT(n),
\qquad \Lambda(n)=\langle1\rangle = C_{n}(n)<\bT(n).
\]
($C_{m}(n)<\bT(n)$ is the subgroup with $m$ elements.)
We regard $\bT(1)$ as canonically isomorphic to $\bT$ (via the map
$t\mapsto e^{2\pi i t}$) and use this isomorphism to convert
$\bT(1)$-spectra to $\bT$-spectra without comment or notation.  We also use the
canonical isomorphism $\bT(1)\iso\bT(n)/\Lambda(n)$ without additional
notation.  Combining these conventions, we write $\bT$-compatible and
$J^{\bT}$ for $\bT(n)/\Lambda(n)$-compatible and
$J^{\bT(n)/\Lambda(n)}$ in the context of
Terminology~\ref{term:compatible} and Notation~\ref{notn:compatible}.

The following gives the first basic building block for the definition of
the objects in $\aACyc$ and $\aApCyc$.

\begin{notn}\label{notn:hhrcells}
Let $n\geq 1$, $m\geq 1$, and let $V$ be an orthogonal $C_{m}(n)$-representation. Let $\eta[n,m,V]$ be the $(\bT(n),1)$
vector bundle
\[
\bT(n)\times_{C_{m}(n)}V\to \bT(n)/C_{m}(n).
\]
\end{notn}

The definition of the objects in $\aACyc$ and $\aApCyc$ also needs the
ideas from Definition~\ref{defn:philambdatwo}.  Specifically, we need the
following variant of the constructions there.

\begin{defn}\label{defn:acyccells}
For the purposes of this section, \term{data} consists of
\begin{itemize}
\item Integers $n\geq 1$, $r\geq 1$, $q_{0}\geq 0$,
$q_{1},\dotsc,q_{r}\geq 1$, $m_{1},\dotsc,m_{r}\geq 1$,
$\ell_{1},\dotsc,\ell_{r}\geq 1$ with $\ell_{i}$ dividing both $n$ and
$m_{i}$.  
\item For $i=1,\dotsc,r$, a
$C_{\ell_{i}}(n)$-trivial orthogonal $C_{m_{i}}(n)$-representation $V_{i}$. 
\item A subgroup $Q<\Sigma_{q_{0}}\times \dotsb \times \Sigma_{q_{r}}$.  
\end{itemize}
We abbreviate this data as
\[
\aQ=(n,r,(q_{0},\dotsc,q_{r}),(m_{1},\dotsc,m_{r}),(\ell_{1},\dotsc,\ell_{r}),(V_{1},\dotsc,V_{r}),Q),
\]
and write $\Sigma=\Sigma_{q_{0}}\times \dotsb \times \Sigma_{q_{r}}$.
Given data $\aQ$ write $n^{\aQ}$, $r^{\aQ}$, $q_{i}^{\aQ}$,
$m_{i}^{\aQ}$, $\ell_{i}^{\aQ}$, $V_{i}^{\aQ}$, and $Q^{\aQ}$ for its
constituent elements.  Let $\eta_{\aQ}$ be the $(\bT(n),Q)$ vector bundle
obtained from
\[
\Sym^{q_{0}}(\bR\to *)\times \Sym^{q_{1}}\eta[n,m_{1},V_{1}]\times \dotsb \times 
 \Sym^{q_{r}}\eta[n,m_{r},V_{r}]
\]
by reducing the group for quotients from $\Sigma$ to $Q$.
Let $\Hom(\aQ)$ denote the set of homomorphisms $\sigma \colon
\Lambda(n)\to Q$ such that for each $i=1,\dotsc,r$, the composite
\[
\Lambda(n)\to Q\to \Sigma_{q_{1}}\times \dotsb \times
\Sigma_{q_{r}}\to \Sigma_{q_{i}}
\]
has kernel containing $C_{\ell_{i}}(n)$; we give $\Hom(\aQ)$ the right
$Q$ action by conjugation (and the trivial $\bT(n)$ action).  Writing
$(-)^{\Lambda(n),\sigma}$ for $\sigma$-twisted $\Lambda(n)$ fixed
points (as in Definition~\ref{defn:philambdaone}), let
$\theta_{\aQ}$ be the $(\bT(n),Q)$ vector bundle
\[
\coprod_{\sigma \in \Hom(\aQ)} \eta_{\aQ}^{\Lambda(n),\sigma}
\]
with $Q$ action
coming from the $Q$ action on both $\Hom(\aQ)$ and
$\eta_{\aQ}$. Finally, when $X$ is a subset of the base of
$\theta_{\aQ}$ that is a $(\bT(n)\times Q^{\op})$-equivariant subcomplex
for a $(\bT(n)\times Q^{\op})$-equivariant cell complex structure, then
we let $\theta_{\aQ,X}$ be the $(\bT(n),Q)$ vector bundle obtained by
base change to $X$.
\end{defn}

In other words, $\theta_{\aQ}$ is the subset of 
\[
\eta_{\aQ}(\Lambda(n))=
\coprod_{\sigma \colon \Lambda(n)\to Q} \eta_{\aQ}^{\Lambda(n),\sigma}
\]
consisting of the summands indexed by $\sigma \in \Hom(\aQ)$.
In particular $\theta_\aQ$ is $\bT$-compatible as is $\theta_{\aQ,X}$
for any $X$.

\begin{defn}\label{defn:ACyc}
Let $D(\aACyc)$ be the set of ordered pairs $(\aQ,X)$ where $\aQ$ is
data as above, $X$ is an equivariant subcomplex as above, $Q$
satisfies the property 
\begin{itemize}
\item For any non-identity $s \in Q$, there exists $j>0$ such
that $s$ is not in the kernel of the map $Q\to \Sigma \to
\Sigma_{q_{j}}$, 
\end{itemize}
and for each $i$, the
$C_{m_{i}^{\aQ}}(n^{\aQ})/C_{\ell_{i}^{\aQ}}(n^{\aQ})$-representation
$V_{i}^{\aQ}$ satisfies the following properties:
\begin{itemize}
\item $V_{i}^{\aQ}$ is a semiregular permutation representation (a
permutation representation containing at
least one copy of the regular representation),
and
\item $V_{i}^{\aQ}$ has underlying non-equivariant inner product space given by
$\bR^{k}$ for some $k$.
\end{itemize}
Let $D(\aApCyc)$ be the subset of $D(\aACyc)$ consisting of those
$(\aQ,X)$ where $n^{\aQ}$ and $m_{i}^{\aQ}$ are powers of $p$.  Let
$\aACyc$ be the set of $\bT$-spectra
$J^{\bT}(\theta_{\aQ,X})$ for $(\aQ,X)\in D(\aACyc)$ and let
$\aApCyc\subset \aACyc$ be the set of $\bT$-spectra
$J^{\bT}(\theta_{\aQ,X})$ for $(\aQ,X)\in D(\aApCyc)$.
\end{defn}

We note that $\aACyc \supset \aA^{\comp}(\aF_{\fin},\aV_{\permp})$ and
$\aApCyc \supset \aA^{\comp}(\aF_{p},\aV_{\permp})$ with the elements of
$\aA^{\comp}(\aF_{\fin},\aV_{\permp})$ and
$\aA^{\comp}(\aF_{p},\aV_{\permp})$ given by the $J^{\bT}(\theta_{\aQ})$
where $n^{\aQ}=1$, $r^{\aQ}=1$, $q^{\aQ}_{0}=0$, and $q^{\aQ}_{1}=1$.  In general, for
$A\in \aACyc$, the $(\aQ,X)\in D(\aACyc)$ for which
$A=J^{\bT}(\theta_{\aQ,X})$ may not be unique, but choosing one, we
write $\aQ_{A}=\aQ$, $X_{A}=X$, $\eta_{A}=\eta_{\aQ_{A}}$,
$\theta_{A}=\theta_{\aQ_{A}}$, $\theta_{A,X}=\theta_{\aQ,X}$, and in
the notation for the data in 
$\aQ_{A}$, abbreviate the superscript $\aQ_{A}$ to just $A$.  (In the
case when $A\in \aApCyc$, we choose $(\aQ,X)\in D(\aApCyc)$.) 

To explain the idea behind the definition, we introduce the following
construction and the following observation about it.

\begin{cons}\label{cons:kQ}
Let $\aQ$ be data as in Definition~\ref{defn:acyccells} and let $k\geq 1$.
define $k\aQ$ by 
\[
n^{k\aQ}:= kn^{\aQ}, r^{k\aQ}:=r^{\aQ}, 
m^{k\aQ}_{i}:=km^{\aQ}_{i},
\ell^{k\aQ}_{i}:=k\ell^{\aQ}_{i},
q^{k\aQ}_{i}:=q^{\aQ}_{i}, 
Q^{k\aQ}:=Q^{\aQ}, 
\]
and take $V_{i}^{k\aQ}$ to be $V_{i}^{\aQ}$ converted to a
$C_{k\ell_{i}^{\aQ}}(kn^{\aQ})$-acyclic orthogonal 
$C_{km_{i}^{\aQ}}(kn^{\aQ})$-representation using the
canonical isomorphism
\[
C_{km_{i}^{\aQ}}(kn^{\aQ})/C_{k\ell_{i}^{\aQ}}(kn^{\aQ})\iso
C_{m_{i}^{\aQ}}(n^{\aQ})/C_{\ell_{i}^{\aQ}}(n^{\aQ}).
\]
Let $\Hom_{k}(k\aQ)\subset \Hom(k\aQ)$ consist of those homomorphisms $\sigma \colon
\Lambda(n^{k\aQ})\to Q^{k\aQ}$ where the composite 
\[
\Lambda(n^{k\aQ})\to Q^{k\aQ}
\to \Sigma_{q_{0}^{k\aQ}} \times \dotsb \times \Sigma_{q^{k\aQ}_{r^{k\aQ}}}
\to \Sigma_{q_{0}^{k\aQ}}
\]
has kernel containing $C_{k}(n^{k\aQ})$.  Let $B_{k}(k\aQ)$ denote the
subset of the base of $\theta_{k\aQ}$ consisting of 
those components indexed by homomorphisms in $\Hom_{k}(k\aQ)$.
\end{cons}

\begin{prop}\label{prop:kQ}
Let $\aQ,X$ be as in Definition~\ref{defn:acyccells}, let $k\geq
1$, and write $Q=Q^{\aQ}=Q^{k\aQ}$. Then:
\begin{enumerate}
\item The action of $C_{k}(n^{k\aQ})$ on $B_{k}(k\aQ)$ is trivial and
under the canonical isomorphism 
\[
\bT(n^{k\aQ})/C_{k}(n^{k\aQ})\times (Q^{k\aQ})^{\op}
\iso \bT(n^{\aQ})\times (Q^{\aQ})^{\op},
\]
$B_{k}(k\aQ)$ is equivariantly isomorphic to the base of $\theta_{\aQ}$.
\item There is an isomorphism of $\bT$-spectra
$J^{\bT}\theta_{k\aQ,kX}\iso J^{\bT}\theta_{\aQ,X}$, where $kX$ is
subcomplex of $B_{k}(k\aQ)$ corresponding to $X$ under the isomorphism of
(ii).  
\item When $\aQ$ is
in $D(\aACyc)$ then so is $k\aQ$; when $k$ is a power of $p$, if $\aQ$
is in $D(\aApCyc)$, then so is $k\aQ$.
\end{enumerate}
\end{prop}

\begin{proof}
Since
\[
C_{k}(n^{k\aQ}) <
C_{\ell_{i}^{k\aQ}}(n^{k\aQ}) <
C_{m_{i}^{k\aQ}}(n^{k\aQ}),\qquad 
C_{k}(n^{k\aQ}) < \Lambda(n^{k\aQ})
\]
(and $C_{k}(n^{k\aQ})$ is abelian),
we have that $C_{k}(n^{k\aQ})$ acts trivially on the vector
bundle $\theta_{k\aQ}$ before restricting the base to $B_{k}(k\aQ)$.
Since $C_{k}(n^{k\aQ}) <
C_{\ell_{i}^{k\aQ}}(n^{k\aQ})$ for all $i$, any $\sigma \in
\Hom_{k}(k\aQ)$ has $C_{k}(n^{k\aQ})$ in
the kernel. The canonical isomorphism 
\[
\Lambda(n^{\aQ})\iso \Lambda(kn^{\aQ})/C_{k}(kn^{\aQ})
\]
therefore induces a bijection between $\Hom(\aQ)$ and
$\Hom_{k}(k\aQ)$.  This proves statement~(i).
For the proof of statement~(ii), we use the notation of
Construction~\ref{cons:orbitds} for the specifics of the $\bT$-spectra
$J$.  Let $W$ be an orthogonal $\bT$-representation. The observations
above about the relations of the subgroups then give  isomorphisms of 
$(\bT(n^{\aQ})\times Q^{\op})$-equivariant fiber bundles 
\[
\aI(\theta_{k\aQ,B_{k}(k\aQ)},W)\iso
\aI(\theta_{\aQ},W)
\]
and $(\bT(n^{\aQ})\times Q^{\op})$-equivariant vector bundles
\[
\Im^{\perp}\aI(\theta_{k\aQ,B_{k}(k\aQ)},W)\iso
\Im^{\perp}\aI(\theta_{\aQ},W).
\]
Since $J$ is constructed as the spectrum formed by the Thom spaces of
these vector bundles, this proves~(ii).  For~(iii), we note that if
$V_{i}^{\aQ}$ is a semiregular permutation representation of
$C_{m_{i}^{\aQ}}(n^{\aQ})/C_{\ell_{i}^{\aQ}}(n^{\aQ})$, then
$V_{i}^{k\aQ}$ is a semiregular permutation representation of
$C_{m_{i}^{k\aQ}}(n^{k\aQ})/C_{\ell_{i}^{k\aQ}}(n^{k\aQ})$.
\end{proof}

\begin{notn}
For $\aQ,X$ as in Definition~\ref{defn:acyccells} and $k\geq
1$, let $k(\aQ,X)=(k\aQ,kX)$ where $kX$ is as in~(iii) above.
\end{notn}

The idea behind the definition of $\aACyc$ is as follows. Looking at
data $\aQ$ (and dropping the $\aQ$ superscripts from the notation),
when $q_{0}=0$ and $Q=Q_{1}\times \dotsb \times Q_{r}$ for $Q_{i}<\Sigma_{q_{i}}$,
the previous proposition combined with Proposition~\ref{prop:Sym} and
Theorem~\ref{thm:bundlegeofix} implies that under faithfulness
hypotheses (that we will show to hold when $\aQ\in D(\aACyc)$)
\[
J^{\bT}\theta_{\aQ}\iso
\Phi_{n/\ell_{1}}((J\eta[\tfrac{n}{\ell_{1}},\tfrac{m_{1}}{\ell_{1}},V'_{1}])^{(q_{1})}/Q_{1})
\sma\dotsb\sma
\Phi_{n/\ell_{r}}((J\eta[\tfrac{n}{\ell_{r}},\tfrac{m_{r}}{\ell_{r}},V'_{r}])^{(q_{r})}/Q_{r})
\]
where $V'$ denotes $V$ converted to the appropriate type of
representation by the canonical isomorphism of quotient groups (the
reverse of the conversion in Construction~\ref{cons:kQ}).  The purpose
of the $\ell_{i}$ is thus to put together a smash product of geometric
fixed points for differing subgroups.
At the cost of passing to a retract (which we handle by using a
subspace $X$ of the base), the functors $\Phi_{n}$
pull outside a symmetric power, and allowing $Q$ to be a
more general subgroup of $\Sigma$ than a product treats the case of a
smash product inside of a symmetric power.  (The restriction on the
group $Q$ when $q_{0}>0$ is purely technical and arises from the
restriction imposed by the requirement that the $Q$ action be
faithful; see the proof of part~(i) of Lemma~\ref{lem:ACyc} below.)

It remains to prove Lemma~\ref{lem:ACyc}.  We do this one part at a time.

\begin{proof}[Proof of Lemma~\ref{lem:ACyc}.(i)] 
Most of the statement follows by construction: we just need to check that
$\theta_{\aQ_{A},X_{A}}$ is $Q^{A}$-faithful and also that for every $k\geq
1$, $\theta_{\aQ_{A},X_{A}}(C_{kn}(n)|\Lambda(n))$ is
$Q^{A}$-faithful.  For this, it is enough to show that $\theta_{A}$
and $\theta_{A}(C_{kn}(n)|\Lambda(n))$ are $Q^{A}$-faithful. 
Let $\aQ=\aQ_{A}$. Since only one object is involved, we drop the
superscripts $\aQ$ or $A$ on the data of $\aQ$ and the subscripts
$\aQ$ or $A$ on the bundles. Fix $k$, and let $\aH_{k}$ be the set of
homomorphisms $\sigma \colon C_{kn}(n)\to Q$ such that the composite
to $\Sigma_{q_{i}}$ contains $C_{\ell_{i}}(n)$ for all
$i=1,\dotsc,r$, and let 
\[
\theta\langle k\rangle=\coprod_{\sigma \in \aH_{k}}\eta^{C_{kn}(n),\sigma}.
\]
Then when $k=1$, $\aH_{1}=\Hom(\aQ)$ and $\theta\langle
1\rangle=\theta$, and so it 
suffices to show that $\theta\langle k\rangle$ is $Q$-faithful: the case $k=1$
shows that $\theta$ is $Q$-faithful, and then the proof of
Theorem~\ref{thm:bundleiterphi} in Section~\ref{sec:qfpl} shows that 
$\theta\langle k\rangle=\theta(C_{kn}(n)|\Lambda(n))$. 

We use the notation $E(\eta)$ and $B(\eta)$ for the base and total
spaces of $\eta$ and $E$ and $B$ for the base and total spaces of
$\theta\langle k\rangle$: 
$\eta \colon E(\eta)\to B(\eta)$ and $\theta\langle k\rangle \colon E\to B$.   Let
\[
s\in Q<\Sigma_{q_{1}}\times \dotsb \times \Sigma_{q_{r}}
\]
not the identity, and let $b\in B$ with $bs=b$; we need to show there
exists an element $e\in E_{b}$ with $es\neq e$.  Let $\sigma\in \aH_{k}$
be the homomorphism indexing the summand containing $b$.  Let $b'$ be
the corresponding element of $B(\eta)^{C_{kn}(n),\sigma}$.  Since
$s$ is not the identity, there exists an $i>0$ so that the projection
$s_{i}$ of $s$ to $\Sigma_{q_{i}}$ is not the identity; let $\sigma_{i}\colon
C_{kn}(n)\to \Sigma_{q_{i}}$ be the composite of $\sigma$ with the
projection.  By definition of $\aH_{k}$, $\sigma_{i}$ has kernel
containing $C_{\ell_{i}}(n)$. 

While we have defined $\eta[n,m_{i},V_{i}]$
as a $(\bT(n),1)$ vector bundle, because $V_{i}$ is
$C_{\ell_{i}}(n)$-trivial, the group of equivariance restricts to
$\bT(n)/C_{\ell_{i}}(n)\iso \bT(n/\ell_{i})$. Viewing
$\eta[n,m_{i},V_{i}]$ as a $(\bT(n/\ell_{i}),1)$ vector bundle, it is
canonically isomorphic to 
\[
\bT(n/\ell_{i})\times_{C_{m_{i}}(n)/C_{\ell_{i}}(n)}V_{i}\to 
\bT(n/\ell_{i})/C_{m_{i}/\ell_{i}}(n/\ell_{i}),
\]
and the $C_{m_{i}/\ell_{i}}(n/\ell_{i})\iso
C_{m_{i}}(n)/C_{\ell_{i}}(n)$-representation $V_{i}$
contains a copy of the regular representation of
$C_{m/\ell_{i}}(n/\ell_{i})$.  Then
\[
\xi := (\Sym^{q_{i}}\eta[n,m_{i},V_{i}])(C_{kn/\ell_{i}}(n/\ell_{i}))
=\hspace{-2em}
\coprod_{\tau \colon C_{kn/\ell_{i}}(n/\ell_{i})\to \Sigma_{q_{i}}}
\hspace{-2em}
 (\Sym^{q_{i}}\eta[n,m_{i},V_{i}])^{C_{kn/\ell_{i}}(n/\ell_{i}),\tau}
\]
is $\Sigma_{q_{i}}$-faithful by Proposition~\ref{prop:ifcritex}.(a) and
Theorem~\ref{thm:ifcrit}. Write
\[
\bar\sigma_{i}\colon C_{kn/\ell_{i}}(n/\ell_{i})\to \Sigma_{q_{i}}
\]
for the
composite of the canonical isomorphism $C_{kn/\ell_{i}}(n/\ell_{i})\iso
C_{kn}(n)/C_{\ell_{i}}(n)$ and the homomorphism
$C_{kn}(n)/C_{\ell_{i}}(n)\to \Sigma_{q_{i}}$ induced by $\sigma_{i}$.
Let $b_{i}$ be the element in the base of
$\xi$ in the $\bar\sigma_{i}$ summand corresponding to the
projection of $b'$.  Because $\sigma$ is fixed by the action of $s$ on
$\aH_{k}$, $\bar\sigma$ is fixed by the action of $s_{i}$ on
$\Hom(C_{kn/\ell_{i}}(n/\ell_{i}),\Sigma_{q_{i}})$, and because $bs=b$,
$b_{i}s_{i}=b_{i}$. We can therefore find $e_{i}$ in the fiber of
$\xi$ at $b_{i}$ such that $e_{i}s_{i}\neq e_{i}$. 

Let $e'$ be the element of $E(\eta)_{b'}$ that is zero in the
coordinates for $\eta[n,m_{j},V_{j}]^{q_{j}}$ with $j\neq i$ and is given by
$e_{i}$ in the coordinate for $\eta[n,m_{i},V_{i}]^{q_{i}}$.  Because
$e_{i}$ is a $\bar\sigma_{i}$-twisted $C_{kn/\ell_{i}}(n/\ell_{i})$ fixed point
of the total space of $\Sym^{q_{i}}\eta[n,m_{i},V_{i}]$ viewed as a
$(\bT(n/\ell_{i}),\Sigma_{q_{i}})$ vector bundle, it is a
$\sigma_{i}$-twisted $C_{kn}(n)$ fixed point of the total space of 
$\Sym^{q_{i}}\eta[n,m_{i},V_{i}]$ viewed as a
$(\bT(n),\Sigma_{q_{i}})$ vector bundle.  It follows that $e'$ is a
$\sigma$-twisted $C_{kn}(n)$ fixed point of $E(\eta)_{b'}$.  We then
have a corresponding point $e$ in $E_{b}$, and because $e_{i}s_{i}\neq
e_{i}$, $es\neq e$.
\end{proof}

\begin{proof}[Proof of Lemma~\ref{lem:ACyc}.(ii)]
Let $A,B\in \aACyc$.  
Define $\aQ$ by taking $n=\lcm(n^{A},n^{B})$ and concatenating the
data for $(n/n^{A})\aQ_{A}$ and $(n/n^{B})\aQ_{B}$:
\begin{align*}
r&=r^{A}+r^{B}&Q&=Q^{A}\times Q^{B}\\
m_{i}&=\begin{cases}
(n/n^{A})m^{A}_{i}&i\leq r^{A}\\
(n/n^{B})m^{B}_{i-r_{A}}&i>r^{A}
\end{cases}
&q_{i}&=\begin{cases}
q^{A}_{0}+q^{B}_{0}&i=0\\
q^{A}_{i}&0<i\leq r^{A}\\
q^{B}_{i-r_{A}}&i>r^{A}
\end{cases}\\
\ell_{i}&=\begin{cases}
(n/n^{A})\ell^{A}_{i}&i\leq r^{A}\\
(n/n^{B})\ell^{B}_{i-r_{A}}\ &i>r^{A}
\end{cases}
&V_{i}&=\begin{cases}
V^{(n/n^{A})\aQ_{A}}_{i}&i\leq r^{A}\\
V^{(n/n^{B})\aQ_{B}}_{i-r_{A}}&i>r^{A}.
\end{cases}
\end{align*}
(We use block sum to identify $\Sigma_{q_{0}^{A}}\times
\Sigma_{q_{0}^{B}}$ as a subgroup of $\Sigma_{q_{0}}$, and define $Q$
as the subgroup of $\Sigma_{q_{0}}\times \dotsb \times \Sigma_{q_{r}}$
corresponding to $Q^{A}\times Q^{B}$.)
Since $\aQ_{A},\aQ_{B}\in D(\aACyc)$, so are $(n/n^{A})\aQ_{A}$,
$(n/n^{B})\aQ_{B}$ and it is easy to see that $\aQ\in D(\aACyc)$.
We then have $n^{(n/n^{A})\aQ_{A}}=n=n^{(n/n^{B})\aQ_{B}}$, and an
isomorphism of $(\bT(n),Q)$ vector bundles
\[
\theta_{\aQ}\iso \theta_{(n/n^{A})\aQ_{A}}\times \theta_{(n/n^{B})\aQ_{B}}.
\]
Let $X$ be the subspace of the base of $\theta_{\aQ}$ corresponding to
$X_{A}\times X_{B}$ under this isomorphism (and the isomorphism of
Proposition~\ref{prop:kQ}.(ii)).  Combining this isomorphism with the isomorphisms from
Proposition~\ref{prop:bundleprod} and~\ref{prop:kQ}.(iii), we get an isomorphism
\[
J^{\bT}(\theta_{\aQ,X})\iso 
J^{\bT}(\theta_{(n/n^{A})(\aQ_{A},X_{A})})\sma J^{\bT}(\theta_{(n/n^{B})(\aQ_{B},X_{B})})
\iso A\sma B.
\]
When $A$ and $B$ are in $\aApCyc$, $n$ and $m_{i}$ are powers of $p$
and so $\aQ\in D(\aApCyc)$. 
\end{proof}

For the proof of Lemma~\ref{lem:ACyc}.(iii), write $\rho_{k}(kn)$ for
the $k$th root isomorphism $\bT(kn)\iso \bT(n)$ and $\rho_{k}(kn)^{*}$
(universally) for the resulting isomorphisms from categories of
$\bT(n)$-equivariant objects to the corresponding categories of
$\bT(kn)$-equivariant objects.  These change of group functors
$\rho_{k}(kn)^{*}$ and the change of group functors $\rho_{k}^{*}$ (for
$\bT\to \bT/C_{k}$) typically commute in the expected way with
reasonable functors between $\bT(n)$-spectra and
$\bT(n')$-spectra.  For example, we have a natural
isomorphism
\[
\Phi_{k}=\rho_{k}^{*}\circ \Phi^{C_{k}}\iso \Phi^{\Lambda(k)}\circ \rho_{k}(k)^{*}.
\]
We have $(\rho_{k}(kn))^{-1}(\Lambda(n))=C_{n}(kn)$, so for 
$\aQ,X$ as in Definition~\ref{defn:acyccells}, we have that $\rho_{k}(kn)^{*}(\theta_{\aQ,X})$ is a
$\bT(k)\iso\bT(kn)/C_{n}(kn)$-compatible $(\bT(kn),Q^{\aQ})$ vector
bundle and we get an isomorphism 
\[
\rho_{k}(k)^{*}(J^{\bT}\theta_{\aQ,X})\iso 
J^{\bT(k)}(\rho_{k}(kn)^{*}(\theta_{\aQ,X})),
\]
Combining the previous two isomorphisms,
we get an isomorphism
\[
\Phi_{k}(J^{\bT}\theta_{\aQ,X})\iso
\Phi^{\Lambda(k)}(J^{\bT(k)}(\rho_{k}(kn)^{*}(\theta_{\aQ,X}))). 
\]

\begin{proof}[Proof of Lemma~\ref{lem:ACyc}.(iii)]
Given $(\aQ,X)\in D(\aACyc)$, it suffices to construct a pair
$(\aQ(k),X(k))$ in $D(\aACyc)$ and an isomorphism 
\[
\Phi^{\Lambda(k)}(J^{\bT(k)}(\rho_{k}(kn)^{*}(\theta_{\aQ,X})))
\iso J^{\bT}\theta_{\aQ(k),X(k)}
\]
where $n:=n^{\aQ}$.
Since $\theta_{\aQ,X}$ is $Q:=Q^{\aQ}$-faithful, so is
$\rho_{k}(kn)^{*}(\theta_{\aQ,X})$, and we will take $Q^{\aQ(k)}=Q$.
Then by Theorem~\ref{thm:bundlegeofix}, it suffices to construct
an isomorphism of $(\bT(kn),Q)$ vector bundles
\[
(\rho_{k}(kn)^{*}(\theta_{\aQ,X}))(\Lambda(kn)|C_{n}(kn))
\iso \theta_{\aQ(k),X(k)}.
\]

We begin with the definition of $\aQ(k)$.  Let
\begin{align*}
n^{\aQ(k)}&:= kn, &r^{\aQ(k)}&:=r:=r^{\aQ}, \\
m^{\aQ(k)}_{i}&:=m_{i}:=m^{\aQ}_{i},
&\ell^{\aQ(k)}_{i}&:=\ell_{i}:=\ell^{\aQ}_{i},\\
q^{\aQ(k)}_{i}&:=q_{i}:=q^{\aQ}_{i}, 
&Q^{\aQ(k)}&:=Q:=Q^{\aQ}, \\
V_{i}^{\aQ(k)}&:=\rho_{k}(kn)^{*}(V_{i}^{\aQ}).
\end{align*}
We note that since $V_{i}^{\aQ}$ is a semiregular permutation representation of
$C_{m_{i}}(n)/C_{\ell_{i}}(n)$,
$V_{i}^{\aQ(k)}$ is a semiregular permutation representation of
$C_{m_{i}}(kn)/C_{\ell_{i}}(kn)$.  We see that since
$(\aQ,X)\in D(\aACyc)$, so is $(\aQ(k),X(k))$ for any choice of
$X(k)$, and when $k$ is a power of $p$, if 
$(\aQ,X)\in D(\aApCyc)$, then so is $(\aQ(k),X(k))$.

Next we analyze $\rho_{k}(kn)^{*}(\theta_{\aQ})$.  First we note that 
\[
\rho_{k}(kn)^{*}(\eta[n,m,V])=\eta[kn,m,\rho_{k}(kn)^{*}(V)]
\]
and so we get 
\[
\rho_{k}(kn)^{*}(\eta_{\aQ})=\eta_{\aQ(k)}.
\]
Let $\rho_{k}(kn)^{*}\Hom(\aQ)$ denote the set of homomorphisms
$C_{n}(kn)\to Q$ such that for each $i=1,\dotsc,r$ the composite 
\[
C_{n}(kn)\to Q\to \Sigma_{q_{0}}\times \dotsb \times \Sigma_{q_{r}}\to 
\Sigma_{q_{i}}
\]
has kernel containing $C_{\ell_{i}}(kn)$.  Then we have
\[
\rho_{k}(kn)^{*}(\theta_{\aQ})\iso
\coprod_{\tau \in \rho_{k}(kn)^{*}\Hom(\aQ)} \eta_{\aQ(k)}^{C_{n}(kn),\tau}.
\]
Write $B=\coprod B_{\tau}$ and $E=\coprod E_{\tau}$ for the base and
total space of $\rho_{k}(kn)^{*}(\theta_{\aQ})$ where we write
$B_{\tau}$ and $E_{\tau}$ for the base and total space of
$\eta_{\aQ(k)}^{C_{n}(kn),\tau}$. 

By Definition~\ref{defn:philambdatwo},
$(\rho_{k}(kn)^{*}(\theta_{\aQ}))(\Lambda(kn)|C_{n}(kn))$ is the
restriction of 
\[
(\rho_{k}(kn)^{*}(\theta_{\aQ}))(\Lambda(kn))=
\coprod_{\sigma\colon \Lambda(kn)\to Q} 
(\rho_{k}(kn)^{*}\theta(\aQ))^{\Lambda(kn),\sigma}
\]
along the inclusion of the subset 
\[
\coprod_{\sigma \colon \Lambda(kn)\to Q} B^{\Lambda(kn),\sigma}\{C_{n}(kn)\}
\subset 
\coprod_{\sigma \colon \Lambda(kn)\to Q} B^{\Lambda(kn),\sigma}
\]
of the base of $(\rho_{k}(kn)^{*}(\theta_{\aQ}))(\Lambda(n))$.
By definition, an element $b\in B^{\Lambda(kn),\sigma}$ in the $\sigma$
summand is in the subset $B^{\Lambda(kn),\sigma}\{C_{n}(kn)\}$ when
the $\sigma|_{C_{n}(kn)}$-twisted $C_{n}(kn)$ action on the fiber $E_{b}$ is
trivial.  Since $\rho_{k}(kn)^{*}(\theta_{\aQ})$ is $Q$-faithful, this
happens exactly when $b$ is in the $\tau =\sigma|_{C_{n}(kn)}$ summand
of $B$.  This subset is either empty (when $\sigma|_{C_{n}(kn)}\notin
\rho_{k}(kn)^{*}\Hom(\aQ)$) or consists of the summand
$B_{\sigma|_{C_{n}(kn)}}\subset B$ (when $\sigma|_{C_{n}(kn)}\in
\rho_{k}(kn)^{*}\Hom(\aQ)$).  The condition that a homomorphism
$\sigma \colon \Lambda(kn)\to Q$ satisfy $\sigma|_{C_{n}(kn)}\in
\rho_{k}(kn)^{*}\Hom(\aQ)$ is a hypothesis on kernels that coincides
with the condition $\sigma \in \Hom(\aQ(k))$. This gives an
isomorphism 
\[
(\rho_{k}(kn)^{*}(\theta_{\aQ}))(\Lambda(kn)|C_{n}(kn))
\iso 
\coprod_{\sigma\in \Hom(\aQ(k))}
\eta_{\aQ(k)}^{\Lambda(kn),\sigma}=\theta_{\aQ(k)}.
\]
Taking $X(k)$ to be the subspace of the base of $\theta_{\aQ(k)}$
corresponding under this isomorphism to the subspace of the base of 
$(\rho_{k}(kn)^{*}(\theta_{\aQ}))(\Lambda(kn)|C_{n}(kn))$
given by the base of
$(\rho_{k}(kn)^{*}(\theta_{\aQ,X}))(\Lambda(kn)|C_{n}(kn))$, 
the isomorphism above restricts to give the isomorphism 
\[
(\rho_{k}(kn)^{*}(\theta_{\aQ,X}))(\Lambda(kn)|C_{n}(kn))
\iso 
\theta_{\aQ(k),X(k)}.
\qedhere
\]
\end{proof}

\begin{proof}[Proof of Lemma~\ref{lem:ACyc}.(iv)]
Given $\aQ,X$ as in Definition~\ref{defn:acyccells} and $k\geq 1$, we
construct  $S^{k}(\aQ,X):=(S^{k}\aQ,S^{k}X)$ as follows.
Let $S^{k}\aQ$ be the data with the
same values for $n$, $r$, $m_{i}$, $\ell_{i}$, and $V_{i}$ as $\aQ$, but with
$q_{i}^{S^{k}\aQ}=kq_{i}$ and 
\[
Q^{S^{k}\aQ}<\Sigma_{k}\wr \Sigma_{q_{0}}\times \dotsb \times
\Sigma_{k}\wr \Sigma_{q_{r}}
<\Sigma_{kq_{0}}\times \dotsb \times \Sigma_{kq_{r}}
\]
the subset where the projection to $(\Sigma_{k})^{r}$ lands in the
diagonal and for each $j=1,\dotsc,k$ the $j$th factor projection 
\[
\Sigma_{k}\wr \Sigma_{q_{0}}\times \dotsb \times
\Sigma_{k}\wr \Sigma_{q_{r}}=
\Sigma_{k}\ltimes \Sigma_{q_{0}}^{k}\times \dotsb \times
\Sigma_{k}\ltimes \Sigma_{q_{r}}^{k}
\to \Sigma_{q_{0}}\times \dotsb \times \Sigma_{q_{r}}
\]
lands in $Q$.  In
other words, the image of $Q^{S^{k}\aQ}$ in
$\Sigma_{kq_{1}+\dotsb+kq_{r}}$ is conjugate to the image of
$\Sigma_{k}\wr Q$ via the permutation associated to the canonical
bijection
\[
(\{1,\dotsc,q_{0}\}\amalg \dotsb \amalg \{1,\dotsc,q_{r}\})^{\amalg k}
\iso 
\{1,\dotsc,q_{0} \}^{\amalg k}\amalg \dotsb \amalg \{1,\dotsc,q_{r} \}^{\amalg k}.
\]
Writing $\psi$ for the isomorphism $Q^{S^{k}\aQ}\iso \Sigma_{k}\wr Q$,
we get an isomorphism of $(\bT(n),Q^{S^{k}\aQ})$ vector bundles
\[
\psi^{*}(\Sym^{k}\eta_{\aQ})\iso \eta_{S^{k}\aQ}.
\]
Let $S^{k}\Hom(\aQ)$ be the subset of $\Hom(S^{k}\aQ)$ of those
homomorphisms $\sigma \colon \Lambda(n)\to Q^{S^{k}\aQ}$ whose
composite with $\psi$
\[
\Lambda(n)\to Q^{S^{k}\aQ}\overto{\iso}\Sigma_{k}\wr Q
\]
lands in $Q^{k}$.  Then we get an isomorphism of
$(\bT(n),\aQ^{S^{k}\aQ})$ vector bundles
\[
\psi^{*}(\Sym^{k}\theta_{\aQ})\iso 
\coprod_{\sigma \in S^{k}\Hom(\aQ)}\eta_{S^{k}\aQ}^{\Lambda(n),\sigma}.
\]
This is the restriction of $\theta_{S^{k}\aQ}$ to the summands indexed
by $\sigma \in S^{k}\Hom(\aQ)\subset \Hom(S^{k}\aQ)$.  Let $S^{k}X$ be
the subspace of the base of $\theta_{S^{k}\aQ}$ corresponding to the
subspace $\psi^{*}X^{k}$ in the base of
$\psi^{*}\Sym^{k}\theta_{\aQ}$.  The isomorphism above then induces an
isomorphism
\[
J^{\bT}_{(\bT(n),\Sigma_{k}\wr Q)}(\Sym^{k}\theta_{\aQ,X})\iso 
J^{\bT}_{(\bT(n),Q^{S^{k}\aQ})}(\psi^{*}(\Sym^{k}\theta_{\aQ,X}))\to
J^{\bT}\theta_{S^{k}(\aQ,X)}.
\]
We note that when $(\aQ,X)$ is in $D(\aACyc)$ or $D(\aApCyc)$, then so
is $S^{k}(\aQ,X)$. 
\end{proof}

\begin{proof}[Proof of Lemma~\ref{lem:ACyc}.(v)]
This follows from part~(i) and Theorem~\ref{thm:dersym}.
\end{proof}

\begin{proof}[Proof of Lemma~\ref{lem:ACyc}.(vi)]
Given $\aQ,X$ as in Definition~\ref{defn:acyccells}, define $\aQ'$ to
have the same values for $n$, $r$, $m_{i}$, $\ell_{i}$, $V_{i}$, and
$q_{i}$ for $i\geq 1$ as $\aQ$, but with
$q_{0}^{\aQ'}=q_{0}^{\aQ}+1$.  Regarding $\Sigma_{q}$ as a subgroup of
$\Sigma_{q+1}$ in the usual way, we can view $Q^{\aQ}$ as a subgroup
of
\[
\Sigma_{q^{\aQ}_{0}+1}\times 
\Sigma_{q^{\aQ}_{1}}\times \dotsb \times \Sigma_{q^{\aQ}_{r^{\aQ}}}
=\Sigma_{q^{\aQ'}_{0}}\times \Sigma_{q^{\aQ'}_{r^{\aQ'}}}
\]
and we take $Q^{\aQ'}=Q^{\aQ}$.  Then $\theta_{\aQ'}$ and
$\theta_{\aQ}$ have the same base space, so $\theta_{\aQ',X}$ makes
sense, and it is easy to see that $J^{\bT}\theta_{\aQ',X}$ is
isomorphic to $J^{\bT}\theta_{\aQ,X}\sma F_{\bR}S^{0}$.
\end{proof}

\section{\texorpdfstring%
{The positive convenient $\Sigma$-model structure}%
{The positive convenient Sigma-model structure} and Theorem~\sref$thm:necircle$}
\label{sec:neACyc}
\label{sec:convenient}

This section defines the positive convenient $\Sigma$-model structure
on non-equivariant spectra and proves Theorem~\ref{thm:necircle} which
compares cofibrations in this structure with cofibrations in the model
structures on $\bT$-spectra from Theorem~\ref{thm:circlephi}.  

We begin with the definition of the model structure.  The model
category results from Section~\ref{sec:model} apply in particular to
non-equivariant spectra by taking the compact Lie group $G$ to be the
trivial group $1$.  Using this theory, we specify the positive
convenient $\Sigma$-model structure by specifying a set of spectra
containing the set $\aA(\aF,\aV_{+})$ (where $\aF$ is the unique
family of subgroups of the trivial group and $\aV_{+}$ is the
representation constraint of Example~\ref{ex:trivposreg}).  In the
definition, we use the $\Sigma_{q}$ action on $F_{\bR^{q}}S^{0}$
induced by the action of $O(q)$ on $\bR^{q}$.

\begin{defn}\label{defn:Sigmodel}
Let $\aA_{\Sigma}$ denote the set of spectra $(F_{\bR^{q}}S^{0})/Q$
where $Q<\Sigma_{q}$.  The \term{positive complete $\Sigma$ model
structure} on spectra is the $\aA_{\Sigma}$-variant $\aV_{+}$-constrained
model structure of Theorem~\ref{thm:varAmodel}.
\end{defn}

One direction of Theorem~\ref{thm:necircle} is easy.  Applying the
free functor to a spectrum $(F_{\bR^{n}}S^{0})/Q$ in $\aA_{\Sigma}$,
we get the $\bT$-spectrum $\bT_{+}\sma (F_{\bR^{n}}S^{0})/Q$.  This is
isomorphic to the $\bT$-spectrum $J\theta_{\aQ,\Delta}$ in $\aApCyc$
where $\Delta$ is the diagonal of $\bT^{q}=\bT(1)^{q}$ and $\aQ$ is
the data where $n=1$, $r=1$, $q_{0}=0$, $q_{1}=q$, $m_{1}=1$, $\ell_{1}=1$,
$Q=Q$, and $V=\bR^{1}$. In other words, it is the restriction to the
diagonal $\bT$ in the base of the $Q^{\op}$-equivariant vector bundle
\[
\Sym^{q}(\bT \times \bR\to \bT)
\]
with the group for quotients reduced to $Q<\Sigma_{q}$.  Since the
free functor sends spectra in $\aA_{\Sigma}$ to spectra in $\aApCyc$
(up to isomorphism), it follows that it sends cofibrations in the
positive convenient $\Sigma$-model structure on spectra to
cofibrations in either of the model structures on $\bT$-spectra in
Theorem~\ref{thm:circlephi}.

The proof of Theorem~\ref{thm:necircle} is then completed by showing
that the forgetful functor sends objects in $\aACyc$ to cofibrant
objects in the positive convenient $\Sigma$-model structure.  
The main observation we need for that is the following lemma.

\begin{lem}\label{lem:necircle}
Let $\aQ\in D(\aACyc)$, let $Q=Q^{\aQ}$, and let $f$ be a $Q^{\op}$-equivariant map
from an orbit $Q^{\op}/\LQ^{\op}$ to the base of $\theta_{\aQ}$.  Then
the spectrum $J_{(1,Q)}f^{*}\theta_{\aQ}$ constructed from the
$(1,Q^{\op})$ vector bundle $f^{*}\theta_{\aQ}$ is isomorphic to
an object in $\aA_{\Sigma}$. 
\end{lem}

\begin{proof}[Proof of Theorem~\ref{thm:necircle} from Lemma~\ref{lem:necircle}]
As observed above, it suffices to show that the forgetful functor
sends objects in $\aACyc$ to cofibrant objects in the positive
convenient $\Sigma$-model structure.  Let $(\aQ,X)\in D(\aACyc)$. The
underlying non-equivariant spectrum of the corresponding element
$J^{\bT}\theta_{\aQ,X}$ of $\aACyc$ is the generalized orbit
desuspension spectrum $J_{(1,Q)}\theta_{\aQ,X}$ (for the trivial
group).  Ignoring the $\bT(n)$ action, $X$ is a $Q^{\op}$-cell complex
and for any $Q^{\op}$-cell $\alpha \colon Q^{\op}/\LQ^{\op}\times
D^{s}\to X$, base change gives a $(1,Q)$ vector bundle
$\alpha^{*}\theta_{\aQ,X}$ and a spectrum
$J_{(1,Q)}(\alpha^{*}\theta_{\aQ,X})$.  We write
$J_{(1,Q)}(\partial\alpha^{*}\theta_{\aQ,X})$ for the spectrum obtained
the same way on the boundary $\partial \alpha \colon Q/\LQ\times
\partial D^{s}\to X$.  The spectrum $J_{(1,Q)}\theta_{\aQ,X}$ is built
as a sequential colimit of pushouts over cofibrations of these
restrictions.  It is therefore enough to show that for every
$Q^{\op}$-cell $\alpha$, the pair
\[
(J_{(1,Q)}(\alpha^{*}\theta_{\aQ,X}),
J_{(1,Q)}(\partial\alpha^{*}\theta_{\aQ,X}))
\]
is isomorphic to a cell $A\sma (D^{s},\partial D^{s})$ for $A$ in
$\aA_{\Sigma}$.  Using the contraction to the center of the disk, we
see that 
\[
(\alpha^{*}\theta_{\aQ,X},\partial \alpha^{*}\theta_{\aQ,X})\iso \eta
\times (D^{s},\partial D^{s})
\]
for the $(1,Q)$ vector bundle $\eta$ over $Q^{\op}/\LQ^{\op}$ obtained
by base change over the restriction of $\alpha$ to the center point.
The result now follows from Lemma~\ref{lem:necircle}.
\end{proof}

The rest of the section proves Lemma~\ref{lem:necircle}.  We fix
\[
\aQ=(n,r,(q_{0},\dotsc,q_{r}),(m_{1},\dotsc,m_{r}),
(\ell_{1},\dotsc,\ell_{r}),(V_{1},\dotsc,V_{r}))\in D(\aACyc)
\]
and a point $x=(x_{1,1},\dotsc,x_{r,q_{r}})$ in the base of $\theta_{\aQ}$.
Let  $\sigma \colon \Lambda(n)\to Q$ in $\Hom(\aQ)$ denote the
homomorphism indexing the component of $x$.  We further fix a subgroup $\LQ<Q$
which is contained in the isotropy subgroup of $x$ for the
right $Q$ action on $\theta_{\aQ}$, which then specifies the
$Q^{\op}$-equivariant map $f$ from $Q^{\op}/\LQ^{\op}$ to the base of
$\theta_{\aQ}$ (sending the base point of the orbit to the point
$x$).  Let $\theta =f^{*}\theta_{\aQ}$; we need to show that
$J_{(1,Q)}\theta$ is isomorphic to an object of $\aA_{\Sigma}$.  

The argument proceeds by studying the fiber of $\theta$ over the
basepoint and the inherent right $\LQ$ action on it.  Writing $E_{x}$ for
this fiber and $V_{i,j}$ for the fiber of
$\eta[n,m_{i},V_{i}]$ over the point $x_{i,j}\in \bT(n)/C_{m_{i}}(n)$,
the construction of $\theta_{\aQ}$ identifies $E_{x}$ as  
\[
(\bR^{q_{0}}\times V_{{1,1}}\times \dotsb \times V_{{r,q_{r}}})^{\Lambda(n),\sigma}.
\] 
Writing $\sigma_{i}\colon \Lambda(n)\to \Sigma_{q_{i}}$ for the
composite map 
\[
\sigma_{i}\colon \Lambda(n)\to Q\to \Sigma_{q_{0}}\times \dotsb \times
\Sigma_{q_{r}}\to \Sigma_{q_{i}},
\]
we can further identify $E_{x}$ as
\[
E_{x}=
(\bR^{q_{0}})^{\Lambda(n),\sigma_{0}}\times 
(V_{{1,1}}\times \dotsb \times V_{1,q_{1}})^{\Lambda(n),\sigma_{1}}
\times \dotsb \times 
(V_{{r,1}}\times \dotsb \times V_{{r,q_{r}}})^{\Lambda(n),\sigma_{r}}.
\]

The zeroth factor $(\bR^{q_{0}})^{\Lambda(n),\sigma_{0}}$ is the fixed
point set of the action of the subgroup
$\sigma_{0}(\Lambda(n))<\Sigma_{q_{0}}$ given by 
the image of $\sigma_{0}$.  Since
$\bR^{q_{0}}=\bR\langle\{1,\dotsc,q_{0}\}\rangle$ is a permutation
representation of $\Sigma_{q_{0}}$, the fixed point set is also a permutation
representation
\[
(\bR^{q_{0}})^{\Lambda(n),\sigma_{0}}\iso 
\bR\langle\{1,\dotsc,q_{0}\}/\sigma_{0}(\Lambda(n))\rangle.
\]
A choice of enumeration of $\{1,\dotsc,q_{0}\}/\sigma_{0}(\Lambda(n))$
as $\{1,\dotsc,k_{0}\}$ defines a homomorphism $\Sigma_{q_{0}}\to
\Sigma_{k_{0}}$ and an isomorphism of $\Sigma_{q_{0}}$-representations  
\[
(\bR^{q_{0}})^{\Lambda(n),\sigma_{0}}\iso \bR^{k_{0}}
\]
with the action of $\Sigma_{q_{0}}$ on $\bR^{k_{0}}$ via the homomorphism,
and in particular, by permutation of coordinates.

To complete the proof of Lemma~\ref{lem:necircle}, it suffices to show
that for each $i$ there exists an orthonormal basis of
\[
(V_{{i,1}}\times \dotsb \times V_{{i,q_{i}}})^{\Lambda(n),\sigma_{i}}
\]
such that the $\LQ$ action on
the vector space restricts to an action on the basis: the choice of
these bases specify a homomorphism 
\[
\LQ\to \Sigma_{k_{1}}\times \dotsb \times \Sigma_{k_{r}}\to \Sigma_{k}
\]
(for $k=\sum k_{r}$) and an isomorphism of $(1,Q)$ vector bundles
\[
\theta \iso \bR^{k}\times_{\LQ}Q\to \LQ\backslash Q,
\]
which then induces an isomorphism of spectra
\[
J_{(1,Q)}\theta \iso (F_{\bR^{k}}S^{0})/\LQ\iso
(F_{\bR^{k}}S^{0})/H
\]
where $H$ is the image of $\LQ$ in $\Sigma_{k}$.

We can now concentrate on an individual $i$, which amounts to studying
the $\sigma_{i}$-twisted $\Lambda(n)$ fixed points of the vector
bundle $\Sym^{q_{i}}\eta[n,m_{i},V_{i}]$.  Fixing $i$, we can tighten up
notation significantly.  
As a first step, since the action of $C_{\ell_{i}}(n)$ on
$V_{i}$ is trivial and $C_{\ell_{i}}(n)$ is in the kernel of
$\sigma_{i}$, we get an equivalent problem if we replace
$\bT(n)$ with $\bT(n)/C_{\ell_{i}}(n)$, $\Lambda(n)$ with
$\Lambda(n)/C_{\ell_{i}}(n)$ and $C_{m_{i}}(n)$ with
$C_{m_{i}}(n)/C_{m_{i}}(\ell_{i})$: let 
\[
\Gamma=\bT(n)/C_{\ell_{i}}(n), \quad 
\Lambda=\Lambda(n)/C_{\ell_{i}}(n),\quad 
C=C_{m_{i}}(n)/C_{\ell_{i}}(n).
\]
We also let $q=q_{i}$, $V=V_{i}$, $W_{j}=V_{i,j}$, and $\sigma'=\sigma_{i}$. 
In this notation, the hypothesis that $\aQ\in D(\aACyc)$ implies that
$V$ and each $W_{j}$ is a non-zero permutation representation of $C$,
and we are studying the pullback $(1,Q)$ vector bundle
\[
\xymatrix@R-1pc{%
Q^{\op}\times_{\LQ^{\op}}(V^{q})^{\Lambda,\sigma'}
\ar@{.>}[r]\ar[d]_{\theta'}
&((\Gamma \times_{C} V)^{q})^{\Lambda,\sigma'}\ar[d]\\
Q^{\op}/\LQ^{\op}\ar@{.>}[r]
&((\Gamma/C)^{q})^{\Lambda,\sigma'}.
}
\]
Let $x'=(x'_{1},\dotsc,x'_{q})$ for $x'_{j}=x_{i,j}$ be the point in
$\theta'$ obtained by restricting coordinates. 
We now recycle the symbols $i$, $\ell$, $m$, $n$, and $r$ for reuse. 

Section~\ref{sec:dfpl} studies twisted fixed points on cartesian
powers of bundles.  (To apply this section, understand the group $Q$
in the work there to be the trivial group.)  The homomorphism
$\sigma'\colon \Lambda \to \Sigma_{q}$ endows the set $\{1,\dotsc,q\}$
with a $\Lambda$ action.  Let $O_{1},\dotsc,O_{k}$ be the orbits for
this action, and choose an element $m_{j}\in O_{j}$ for each
$j=1,\dotsc,k$. Let $H_{j}<\Lambda$ be the isotropy subgroup of
$m_{j}$. Since $\Lambda$ is abelian, $H_{j}$ is the isotropy subgroup
for every element in $O_{j}$.  Moreover, we note that
$H_{j}<\Lambda<\Gamma$ is also a subgroup of $C$ since every element
in it fixes $x'_{m_{j}}\in \Gamma/C$. For each $i\in \{1,\dotsc,q\}$,
define $j_{i}$ by $i\in O_{j_{i}}$ and choose $\ell_{i}\in \Lambda$
such that $\ell_{i}\cdot m_{j_{i}}=i$, with $\ell_{m_{j}}=\id$.
Then for the inherent action of $\Gamma$ on the vector bundle $\Gamma
\times_{C}V\to \Gamma/C$, $\ell_{i}$ sends $W_{m_{j_{i}}}$ (the fiber
over $x_{m_{j_{i}}}$) to $W_{i}$ (the fiber over $x_{i}$).
Define
\[
\phi \colon W_{m_{i}}^{H_{m_{1}}}\times \dotsb \times W_{m_{k}}^{H_{m_{k}}}\to W_{1}\times \dotsb \times W_{q}
\]
by
\[
\phi(w_{1},\dotsc,w_{k}):=
\left(\tfrac{1}{\sqrt{\#O_{j_{1}}}}\ell_{1}w_{j_{1}},\dotsc,
\tfrac{1}{\sqrt{\#O_{j_{q}}}}\ell_{q}w_{j_{q}}\right).
\]
The scaling ensures that this is an isometry, and
Lemma~\ref{lem:dfplem} shows that $\phi$ is an isomorphism onto the
$\sigma'$-twisted $\Lambda$ fixed points $(W_{1}\times \dotsb \times
W_{q})^{\Lambda,\sigma'}$. 

For each $i=1,\dotsc,q$, choose an element $\bar x_{i}$ in $\Gamma$
lifting the element $x'_{i}\in \Gamma/C$.  The choice of $\bar x_{i}$
induces a $C$-equivariant isomorphism $\psi_{i}\colon V\to W_{i}$.  If
$x'_{i}=x'_{i'}$ for some $i'$, then $\bar x_{i'}\bar x_{i}^{-1}$
is an element of $C$ and the composite endomorphism 
\[
V\overto{\psi_{i}} W_{i}=W_{i'}\overto{\psi_{i'}^{-1}} V
\]
is induced by $\bar x_{i}^{-1}\bar x_{i'}\in C$ for the
inherent $C$ action on $V$.  More generally if under the
$\sigma'$-twisted $\Lambda$ action, the element $\lambda \in \Lambda$ takes
$x'_{i}$ to $x'_{i'}$, then $\bar x_{i'}\lambda^{-1}\bar x_{i}^{-1}$
with the composite endomorphism of $V$ given by multiplication by $\bar
x_{i}^{-1}\lambda^{-1} \bar x_{i'}$.  (These formulas and the
$C$-equivariance of the $\psi_{i}$ strongly use the fact that $\Gamma$
is abelian.)

The isomorphisms $\psi_{i}$ together with the
isomorphism $\phi$ above identify the fiber of $x'$ in $\theta'$ as 
\[
V^{H_{m_{1}}}\times \dotsb \times V^{H_{m_{k}}}.
\]
Because $V$ is a permutation representation of $C$, each
$V^{H_{j}}$ is a permutation representation of $C/H_{j}$.  We can
therefore choose an orthonormal basis of $V^{H_{j}}$ on which the $C$
action restricts.  If $H_{j}=H_{j'}$, we choose the same basis for
$V^{H_{j}}$ and $V^{H_{j'}}$.  We assemble the bases of $V^{H_{j}}$ to
an orthonormal basis of the fiber of $x'$ in $\theta'$.

The action of $\LQ$ on $\theta'$ is induced by the permutation action.
Since $\LQ$ is contained in the isotropy subgroup of $x'$ for the
$Q$ action on the base of $\theta'$, every element of $\LQ<Q$ fixes
the homomorphism $\sigma'$, and that is equivalent to saying that the
right action on $\LQ$ on $\{1,\dotsc,q\}$ commutes with the
$\Lambda$ action induced by $\sigma'$.  In particular, if some element
of $\LQ$ sends $i$ to $i'$, then it induces an isomorphism of
$\Lambda$-sets $O_{j_{i}}\to O_{j_{i'}}$ and we must
have $H_{j_{i}}=H_{j_{i'}}$.  Because for each $j$, the elements
$\ell_{i}$ for $i\in O_{j}$ form a complete list of coset
representatives for $\Lambda/H_{j}$, under the identification above,
the action of each element $\rho\in \LQ$ on $V^{H_{m_{1}}}\times \dotsb
\times V^{H_{m_{k}}}$ does the permutation on the orbits followed by
the action of an element of $C$ on each factor (namely, on the
$j$th factor, it acts by 
\[ 
\ell_{i}\bar x_{m_{j_{i}}}\bar x_{i}^{-1}\bar x_{i}\bar x_{m_{j}}^{-1}
=\ell_{i}\bar x_{m_{j_{i}}}\bar x_{m_{j}}^{-1},
\]
where $\rho\in \LQ<Q<\Sigma_{q}$ 
sends $m_{j}$ to $i$). It follows that the action of $\LQ$ on the 
fiber of $x'$ restricts to an action the orthonormal basis constructed 
above.

\chapter{The Multiplicative tom Dieck Splitting (Theorem~\sref$main:tds$)}
\oldlabel{chap:tds}

A fundamental structural property of the equivariant stable category
is the tom Dieck splitting, which in its original form gave a formula
for the categorical fixed points of a suspension spectrum.  It arises
from the existence of transfers on categorical fixed points.  In the
multiplicative context, equivariant commutative ring orthogonal
spectra come with certain transfers on their geometric fixed points,
which we review below.  When the equivariant commutative ring spectrum
comes from a non-equivariant commutative ring spectrum, these
transfers are independent of each other and assemble to a
multiplicative splitting that is a multiplicative analogue
of the tom Dieck splitting.  We gave a concrete statement in
Theorem~\ref{main:tds}, which we restate as Theorem~\ref{thm:tomdieck}
in Section~\ref{sec:multtD}.  This chapter just 
scratches the surface of the properties of the multiplicative
transfers, which we believe to be foundational in the equivariant
multiplicative theory and intend to return to in future work.
Section~\ref{sec:2tds} provides some additional information on
the pieces in the splitting in special cases; little beyond this
seems to currently be known.

\section{Construction of the multiplicative splitting}\label{sec:multtD}

We begin with the construction of the multiplicative transfer in the
context of commutative ring $G$-spectra (with $G$ a
finite group). For $H<G$, the norm $N_{H}^{G}$ restricts to a functor
from commutative ring $H$-spectra to
commutative ring $G$-spectra and it is left adjoint to the
reduction of structure functor $i_{H}^{*}$ from 
commutative ring $G$-spectra to
commutative ring $H$-spectra
\[
\xymatrix{%
\Com_{H}\ar@<.5ex>[r]^{N_{H}^{G}}
&\Com_{G}.\ar@<.5ex>[l]^{i_{H}^{*}}
}
\]
In particular, for $A$ a
commutative ring $G$-spectrum, the counit of
the adjunction gives a natural map of commutative ring
$G$-spectra
\[
N_{H}^{G}i_{H}^{*}A\to A.
\]
Applying $\Phi^{G}$ to this map and composing with 
the Hill-Hopkins-Ravenel diagonal map
\[
\Phi^{H}i_{H}^{*}A\to \Phi^{G}N_{H}^{G}i_{H}^{*}A
\]
(which is an isomorphism when $A$ is nice enough) 
we get a natural map
\[
\tau_{H}\colon \Phi^{H}i_{H}^{*}A\to \Phi^{G}N_{H}^{G}i_{H}^{*}A\to \Phi^{G}A.
\]
of non-equivariant commutative ring orthogonal spectra on the
point-set level, or working with derived functors, on the homotopy
category level.  

\begin{defn}
The \term{multiplicative transfer} for $H<G$ is the natural map
$\tau_{H}$ displayed above. 
\end{defn}

For the multiplicative tom Dieck splitting, the first observation we
need is that $\tau_{H}$ is actually equivariant, where we regard
$\Phi^{H}A$ with its natural $WH$ action and $\Phi^{G}A$ with its
trivial action. The following gives a precise statement.

\begin{thm}
The multiplicative transfer $t_{H}\colon \Phi^{H}A\to \Phi^{G}A$ is a
map of left $WH$-objects in the category of non-equivariant commutative ring orthogonal spectra, where $\Phi^{H}A$ is a
$WH$-object by neglect of structure (from its commutative ring $WH$-spectrum structure) and $\Phi^{G}A$ is a
$WH$-object with trivial action.
\end{thm}

\begin{proof}
The key observations are:
\begin{enumerate}
\item $N_{H}^{G}i_{H}^{*}A$ has a natural action of $WH$
in the category of commutative ring $G$-spectra,
\item The counit map $N_{H}^{G}i_{H}^{*}A\to A$ is a map of
$WH$-objects in the category of commutative ring
$G$-spectra where we give the target the trivial action, and
\item The diagonal map $\Phi^{H}A\to \Phi^{G}N_{H}^{G}i_{H}^{*}A$ is a
map of $WH$-objects in non-equivariant commutative ring orthogonal spectra.
\end{enumerate}
For point~(i), we use the fact~\cite[2.28]{HHR} that
$N_{H}^{G}i_{H}^{*}A$ is naturally isomorphic to the tensor
$G/H\otimes A$ of $A$ with the left $G$-space $G/H$, and under this
isomorphism, the map $N_{H}^{G}i_{H}^{*}A\to A$ is induced by the
$G$-map $G/H\to *$.  The commutative ring
$G$-spectrum $N_{H}^{G}i_{H}^{*}A$ then obtains a $WH$ action (in that
category) from the right action of $WH$ on $G/H$.  For~(ii),
formulated this way, the map is obviously a map of $WH$-objects where
the target has the trivial action.  For~(iii), 
\cite[B.190]{HHR} implies that the wedge of maps from
$\Phi^{H}F_{V}X(V)$ to $\Phi^{H}X$ is a categorical epimorphism, so it
suffices to check~(iii) in the case when $A$ is free on
$F_{V}G/H_{+}$, where it is easy from the description of the diagonal
in~\cite[B.209]{HHR}.
\end{proof}

The previous theorem was a point set statement. While it is true that
the point-set category of left $WH$-objects in non-equivariant
orthogonal spectra is equivalent to the point-set category of
$WH$-equivariant orthogonal spectra, for our purposes, the
corresponding homotopy theory statement we need is in terms of the
$WH$-Borel equivariant category of commutative ring spectra.  The
Borel category is the category of $WH$-objects in non-equivariant
commutative ring orthogonal spectra where we understand the weak
equivalences to be the maps that are weak equivalences of the
underlying non-equivariant spectra.  Homotopically, then we should
consider the map
\[
\Phi^{H}A\otimes EWH\to \Phi^{G}A
\]
obtained by tensoring with the universal free $WH$-space.  Since the
target is trivial, this factors through the quotient (in the category
of commutative ring equivariant spectra)
\begin{equation}\label{eq:mapabove}
\Phi^{H}A\otimes_{WH} EWH\to \Phi^{G}A.
\end{equation}

We can now begin to study the case when the commutative ring
$G$-spectrum comes from a non-equivariant commutative ring spectrum
$R$.  We have a point-set functor $\epsilon^{*}$ that takes
non-equivariant commutative ring spectra to commutative ring
$G$-spectra by giving it the trivial action: precisely, it is the
restriction of action along the final homomorphism $G\to 1$.  A check
of basic definitions shows that $\epsilon^{*}$ is left adjoint to the
categorical $G$ fixed point functor, and since categorical fixed point
functors preserve fibrations and acyclic fibrations, $\epsilon^{*}$ is
a Quillen left adjoint.  In particular, $\epsilon^{*}$ admits a left
derived functor that is calculated by applying it to a cofibrant
object.  In what follows, we use $\epsilon^{*}$ to denote this left
derived functor, and for convenience, we assume without loss of
generality that $R$ is a cofibrant commutative ring spectrum, so that
$\epsilon^{*}R$ is represented by the point-set functor applied to $R$.

We want to study the left derived geometric fixed points of
$\epsilon^{*}R$, and for this we take the commutative ring
$G$-spectrum $A$ in the above to be a
cofibrant approximation in a model structure satisfying the hypotheses
of Theorem~\ref{thm:Finexample}.  We get a map in the homotopy
category of non-equivariant commutative ring spectra
\[
R\to L\Phi^{H}\epsilon^{*}R\simeq \Phi^{H}A
\]
as the composite with the canonical map from $R$ to the derived
categorical $H$ fixed points of $\epsilon^{*}R$ with the canonical map
from the categorical fixed points to the geometric fixed points.
Formally, these are maps in the $WH$-Borel
category of commutative ring orthogonal spectra where $R$ has the
trivial action.  Composing with the map~\eqref{eq:mapabove} above, we
get a well defined 
map in the homotopy category of non-equivariant commutative ring
orthogonal spectra
\[
R\otimes^{L} BWH \iso R\otimes^{L}_{WH}EWH\to \Phi^{H}A\otimes^{L}_{WH}EWH\to 
L\Phi^{G}\epsilon^{*}R.
\]
(This is not a point-set map.)
Taking the coproduct over conjugacy classes of subgroups $H$, we get a
map of non-equivariant commutative ring orthogonal spectra
\[
\bigwedge_{[H]}R\otimes^{L} BWH\to L\Phi^{G} \epsilon^{*}R.
\]
We prove the following theorem.

\begin{thm}[Multiplicative tom Dieck splitting for $\epsilon^{*}$]
\label{thm:tomdieck}
Let $R$ be a non-equi\-variant commutative ring orthogonal spectrum, and
let $\epsilon^{*}$ denote the derived functor from non-equivariant to
commutative ring $G$-spectra. The map 
\[
\bigwedge_{[H]}R\otimes^L BWH\to L\Phi^{G} \epsilon^{*}R.
\]
induced by the multiplicative transfers is a weak equivalence.
\end{thm}

\begin{proof}
We take $R$ to be a cofibrant non-equivariant commutative ring
orthogonal spectrum.  Both the domain and codomain of the natural map
in the statement preserve homotopy colimits in commutative ring
spectra, and so it suffices to check
the case when $R$ is the domain and codomain of a cell, that is, when
$R=\bP(F_{m}X_{+})$, $q\geq 1$, the free commutative ring 
spectrum on $F_{m}X_{+}$ for $X=D^{n}$ or $\partial D^{n}$.  We can
then take $A$ to be the free commutative ring
$G$-spectrum $\bP Y$ on $Y=F_{\bR^{m}\oplus
\bR\langle G\rangle}\Sigma^{\bR\langle G\rangle}X_{+}$, where we have written 
$\bR\langle G\rangle$ for the regular representation of $G$.  While the symmetric
powers $Y^{(q)}/\Sigma_{q}$ of $Y$ are not themselves 
generalized orbit desuspension spectra, they are each the quotient of
a Hurewicz cofibration of generalized orbit desuspension spectra, for
the (globally trivial) $(G,\Sigma_{q})$ vector bundles
\[
\xymatrix{%
(X^{q}\times \partial D(\bR\langle G\rangle^{q}))\times (\bR^{m}\oplus \bR\langle G\rangle)^{q}\ar[d]
\ar@{.>}[r]
&(X^{q}\times D(\bR\langle G\rangle^{q}))\times (\bR^{m}\oplus \bR\langle G\rangle)^{q}\ar[d]\\
X^{q}\times \partial D(\bR\langle G\rangle^{q})\ar@{.>}[r]
&X^{q}\times D(\bR\langle G\rangle^{q})
}
\]
for $q\geq 1$.  Since the geometric fixed points functor $\Phi^{G}$
commutes with quotients by spacewise closed inclusions,
Theorem~\ref{thm:bundlegeofix} then gives the isomorphism
\[
\Phi^{G}(Y^{(q)}/\Sigma_{q})\iso
\biggl(\bigvee_{\sigma \colon G\to \Sigma_{q}}
F_{((\bR^{m}\oplus \bR\langle G\rangle)^{q})^{(G,\sigma)}}
(\Sigma^{\bR\langle G\rangle^{q}}X^{q}_{+})^{(G,\sigma)}
\biggr)/\Sigma_{q}
\]
for $q\geq 1$. 

The remainder of the argument is a version of the proof of
Theorem~\ref{thm:calcgeoSym} where (in the notation 
there) $Q$ is the trivial group.  The argument becomes significantly
more concrete in this case and we include it to make
this section more self-contained.

If we think of a homomorphism $G\to \Sigma_{q}$ as a $G$ action on
$\{1,\dotsc,q\}$, much of the quotient by $\Sigma_{q}$ above is relating
isomorphic $G$-sets, and we can rewrite the above in terms of
isomorphism classes of $G$-sets of size $q$:
\[
\Phi^{G}(Y^{(q)}/\Sigma_{q})\iso
\bigvee_{[Z], \#Z=q}
(F_{((\bR^{m}\oplus \bR\langle G\rangle)\otimes \bR\langle Z\rangle)^G}
(\Sigma^{\bR\langle G\times Z\rangle}X^{Z}_{+})^{G}
/\Aut_{G}(Z))
\]
where the action is the combined action of $G$ on $Z$ and the other
terms.  The previous formula was only for $q\geq 1$, but when $Z$ is the
empty set the formula reduces to a copy of $\bS$, which is the
zero summand of $A$.  Adding up over $q\geq 0$, we get a wedge over
isomorphism classes of finite $G$-sets $Z$
\[
\Phi^{G}A\iso 
\bigvee_{[Z]}
(F_{((\bR^{m}\oplus \bR\langle G\rangle)\otimes \bR\langle Z\rangle)^G}
(\Sigma^{\bR\langle G\times Z\rangle}X^{Z}_{+})^{G}
/\Aut_{G}(Z)).
\]
The construction 
\[
F_{((\bR^{m}\oplus \bR\langle G\rangle)\otimes \bR\langle Z\rangle)^G}
(\Sigma^{\bR\langle G\times Z\rangle}X^{Z}_{+})^{G}
\]
takes coproducts in $Z$ to smash products of spectra.  Reorganizing
the wedge above in terms of isomorphism classes of orbits then gives
\[
\Phi^{G}A\iso \bigwedge_{[H]}\biggl(\bigvee_{q\geq 0}
(F_{(\bR^{m}\oplus \bR\langle G/H\rangle)^{q}}
(\Sigma^{\bR\langle G/H\rangle}X_{+})^{(q)})/\Sigma_{q}\wr WH
\biggr)
\]
with the outer smash product indexed over conjugacy classes of
subgroups and where we have identified
\[
\Aut_{G}(G/H_{1}^{q_{1}}\amalg \dotsb \amalg G/H_{r}^{q_{r}})\iso
(\Sigma_{q_{1}}\wr WH_{1})\times \dotsb \times (\Sigma_{q_{r}}\wr WH_{r})
\]
for pairwise non-conjugate subgroups $H_{1},\dotsc,H_{r}$.  Finally,
moving the quotients by $WH$ inside the quotients by $\Sigma_{q}$, we
get
\[
\Phi^{G}A\iso \bigwedge_{[H]}\biggl(\bigvee_{q\geq 0}
((F_{\bR^{m}\oplus \bR\langle G/H\rangle}
(\Sigma^{\bR\langle G/H\rangle}X_{+}))/WH)^{(q)}/\Sigma_{q}
\biggr)
\]
The map in the homotopy category of non-equivariant commutative ring
orthogonal spectra $R\to \Phi^{G}A$ induced by
$\tau_{H}$ is determined by a map in the derived category of
non-equivariant spectra  
\[
F_{\bR^{m}}X_{+}
\to \Phi^{G}A.
\]
Using the description of the diagonal map from~\cite[B.209]{HHR}, we
can identify this as the map
\[
F_{\bR^{m}}X_{+}
\overfrom{\simeq}
F_{\bR^{m}\oplus \bR\langle G/H\rangle}\Sigma^{\bR\langle G/H\rangle}X_{+}
\to (F_{\bR^{m}\oplus \bR\langle G/H\rangle}
(\Sigma^{\bR\langle G/H\rangle}X_{+}))/WH
\]
followed by the inclusion of the latter as the $q=1$ summand in the
smash factor for $G/H$.  Because the action of $WH$ on $G/H$ is
faithful, the constituent spaces of $F_{\bR^{m}\oplus \bR\langle G/H\rangle} 
(\Sigma^{\bR\langle G/H\rangle}X)$ are free based $WH$-CW complexes and the maps 
\begin{align*}
F_{\bR^{m}}X_{+}\sma BWH&=(F_{\bR^{m}}X_{+})\sma_{WH}EWH\\
&\overfrom{\simeq}
(F_{\bR^{m}\oplus \bR\langle G/H\rangle}
(\Sigma^{\bR\langle G/H\rangle}X_{+}))\sma_{WH}EWH\\
&\to (F_{\bR^{m}\oplus \bR\langle G/H\rangle}
(\Sigma^{\bR\langle G/H\rangle}X_{+}))/WH
\end{align*}
are weak equivalences.  It follows that the map $R\otimes BWH\to
\Phi^{G}A$ induces a weak equivalence into the $G/H$ smash factor, and
the smash of these maps is a weak equivalence.
\end{proof}

\section{The factors in the multiplicative splitting}
\label{sec:2tds}

For $R$ a (non-equivariant) commutative ring orthogonal spectrum, the
commutative ring orthogonal spectra $R\otimes^{L} BWH$ in the tom
Dieck splitting are examples of \term{Loday constructions} on $R$.
Most work to date on Loday constructions $R\otimes^{L}X$ study
examples when $X$ is a finite complex, most often a sphere or product
of spheres.  The multiplicative tom Dieck splitting of the previous
section provides new motivation to study Loday constructions when $X$
is the classifying space of a finite group.  While little seems to be
known calculationally about this case, we do offer the following
result.

\begin{thm}\label{thm:loday}
Let $R$ be a connective commutative ring orthogonal spectrum and let $G$ be a
finite group.  If the cardinality of $G$ is invertible in $\pi_{0}R$,
then the map $R\to R\otimes^{L}BG$ induced by the inclusion of the
base point of $BG$ is a weak equivalence.
\end{thm}

The theorem is an immediate consequence of the following theorem,
which is well-known to experts.

\begin{thm}
Let $R$ be a commutative ring orthogonal spectrum and $X\to
Y$ a map of CW spaces that induces an isomorphism on $R$-homology.
Then $R\otimes^{L}X\to R\otimes Y$ is a weak equivalence.
\end{thm}

For the proof, we follow the outline of \cite[2.7]{DundasTenti}
(and~\cite[2.2.1.3]{DundasGoodwillieMcCarthy}).  

\begin{proof}
We can assume without loss of generality that $R$ is a cofibrant
commutative ring spectrum, and then
the derived tensor $R\otimes^{L}X$ is calculated by the point-set
tensor $R\otimes X$ for any CW complex $X$. Because $R\otimes X$ is
continuous in $X$, it preserves homotopies in $X$, and so it suffices
to consider the case when $X$, $Y$ are geometric realizations of
simplicial sets $X\subdot, Y\subdot$, and the map $X\to Y$ is the
geometric realization of a simplicial map.  We also assume that
$X\subdot$ and $Y\subdot$ are non-empty.

Let $\mathbf{F}$ denote the skeleton of the category of non-empty finite sets,
with objects $\mathbf{n}=\{1,\dotsc,n\}$ for $n> 0$.  Since $R$ is a
commutative monoid object for $\sma$, this structure gives a covariant functor
$\mathbf{n}\mapsto R^{(n)}$, while the structure on a non-empty space
$Z$ of being a commutative comonoid object for $\times$ (with diagonal map $Z\to
Z\times Z$ and unique map $Z\to *$) gives a contravariant functor
$\mathbf{n}\mapsto Z^{n}$.  Regarding $\mathbf{n}$ as a discrete
space, $Z^{n}$ is the space of maps from $\mathbf{n}$
to $Z$.  The coend
\[
CE(Z):=\int^{\mathbf{n}\in \mathbf{F}}R^{(n)}\sma Z^{n}_{+}
\]
is then easily identified as the left Kan extension of the functor
$R^{(-)}$ along the inclusion of $\mathbf{F}$ in the category of
non-empty spaces.  Since the tensor functor $Z\mapsto R\otimes Z$ is a functor
from non-empty spaces to orthogonal spectra that restricts to $\mathbf{n}\mapsto
R^{(n)}$ on $\mathbf{F}$, we get a natural transformation $CE(Z)\to
R\otimes Z$, which is an isomorphism when $Z=\mathbf{n}$.  Commuting
colimits then show that it is an isomorphism when the non-empty space
$Z$ is discrete, 
and the properties of geometric realization (commuting with colimits
and finite products) show that it is an isomorphism when $Z$ is the
geometric realization of a non-empty simplicial set.  (In fact, it is an
isomorphism for all non-empty spaces, but we do not need that here.)

Next consider the homotopy coend
\[
HCE(Z):=B(R^{(-)},\mathbf{F},Z^{-})
\]
constructed using the categorical bar construction: it is the
geometric realization of the simplicial orthogonal spectrum
$HCE\subdot(Z)$ that in
degree $q$ is 
\[
HCE_{q}(Z):=\bigvee_{n_{0},\dotsc,n_{q}}
R^{(n_{q})}\sma \mathbf{F}(\mathbf{n}_{q-1},\mathbf{n}_{q})_{+}
\sma\dotsb\sma \mathbf{F}(\mathbf{n}_{0},\mathbf{n}_{1})_{+}
\sma Z^{n_{0}}_{+}
\]
(with degeneracies induced by inserting identity maps in $\mathbf{F}$
and face maps induced by the composition in $\mathbf{F}$ and the
action of $\mathbf{F}$ on $R^{(-)}$ and $Z^{-}$).  We have a canonical
map from the homotopy coend to the coend
\[
HCE(Z)\to CE(Z)
\]
and the usual extra degeneracy argument shows that it is a homotopy
equivalence when $Z=\mathbf{n}$.  Because $R$ is cofibrant as a
commutative ring orthogonal spectrum, the map is a weak equivalence
when $Z$ is discrete: the homotopy groups of both sides can be
calculated as the colimit over the non-empty finite subsets.  When $Z$ is
discrete, $CE(Z)=R\otimes Z$ is the coproduct in commutative ring
orthogonal spectra of $Z$ copies of $R$. Since $R$ is cofibrant, it
follows that for any inclusion of discrete spaces $Z\to Z'$, the map
$CE(Z)\to CE(Z')$ is a cofibration of cofibrant commutative ring
orthogonal spectra.  In particular, for any simplicial set $Z\subdot$, the
degeneracies in the simplicial object $CE(Z\subdot)$ are
cofibrations of cofibrant commutative ring orthogonal spectra
and in particular spacewise Hurewicz cofibrations of non-degenerately
based spaces.  The degeneracies of $HCE(Z\subdot)$ are likewise
spacewise Hurewicz cofibrations of non-degenerately based spaces.
These technical checks suffice to imply that geometric realization
preserves level weak equivalences.
Because geometric realization commutes with both $HCE(-)$ and $CE(-)$,
we see that the map is a weak equivalence whenever $Z$ is the
geometric realization of a simplicial set.

Returning to the proof of the theorem, by the work above it now
suffices to show that the map $HCE(X)\to HCE(Y)$ is a weak
equivalence, and for this, it suffices to show that the map
$HCE_{q}(X)\to HCE_{q}(Y)$ is a weak equivalence for all $q$.  To show
this, it is enough to see that
\[
R^{(m)}\sma X^{n}\to R^{(m)}\sma Y^{n}
\]
is a weak equivalence for all $m,n > 0$. This is an easy consequence
of the hypothesis that $X\to Y$ is an $R$-homology isomorphism. 
\end{proof}

\chapter{Generalized Orbit Desuspension Spectra: Main Properties}
\oldlabel{chap:bundleone}

The purpose of this chapter is to introduce a class of equivariant
orthogonal spectra generalizing the orbit desuspension spectra with
similarly nice point-set and homotopical properties, particularly with
respect to the multiplicative geometric fixed point functor and smash
product.  In a departure from our standard notation, in this chapter
and the next we write $\Gamma$ rather than $G$ for a general compact
Lie group with the idea that $G\iso \Gamma /\KG$ for some $\KG \lhd
\Gamma$, as this is the case needed for arguments in
Section~\ref{sec:ACyc}.  The first section provides the basic
definitions. The next two sections explore how these spectra behave
under geometric fixed point functors, smash products, symmetric power
functors, and the interactions of these functors.

\section{Generalized orbit desuspension spectra}
\label{sec:startbundle}
\label{sec:Sym}
\label{sec:repnote}

Classically, an orbit desuspension spectra is a $\Gamma$-spectrum of
the form $\Gamma/H_{+}\sma S^{-V}$ for some orthogonal $\Gamma$-representation $V$ and closed subgroup $H<\Gamma$. The object $F_{V}S^{0}$
models $S^{-V}$ in the $\Gamma$-equivariant stable category, where
$F_{V}$ is the functor of~\cite[II.4.6]{MM}, left adjoint to
evaluation at $V$.  (Indeed,
\cite{HHR} denotes $F_{V}S^{0}$ as $S^{-V}$; see, for example,
p.~146.)  The description 
of $\Gamma$-equivariant orthogonal spaces as equivariant diagram
spaces in~\cite[II\S4]{MM} allows a concise description of
$F_{V}S^{0}$ as the representable functor
\[
F_{V}S^{0}(W)=\aJ_{\Gamma}(V,W)
\]
where $\aJ_{\Gamma}$ is the (based $\Gamma$-equivariantly enriched)
category defined 
in \cite[II.4.1]{MM}. Precisely, $\aJ_{\Gamma}(V,W)$ 
is the Thom space of the $\Gamma$-equivariant vector bundle
$\Im^{\perp}\aI(V,W)$ whose base space is the $\Gamma$-space
$\aI(V,W)$ of isometric embeddings of $V$ in $W$ and whose fiber over
$f\colon V\to W$ is the orthogonal complement of $f(V)$.  The
based $\Gamma$-spaces $\aJ_{\Gamma}(V,W)$ give the morphisms in the
category $\aJ_{\Gamma}$ (with objects the orthogonal $\Gamma$-representations), where
composition is induced by composition of embeddings and (internal)
direct sum of orthogonal complement subspaces.  The gist
of~\cite[II\S4]{MM} is that the category of $\Gamma$-spectra is isomorphic to the category of based
$\Gamma$-space enriched functors from $\aJ_{\Gamma}$ to based
$\Gamma$-spaces. 

To generalize orbit desuspension spectra, we replace the vector space
$V$ with a vector bundle $\eta$ and we also take the quotient by a
finite group action; to be concise about the input data, we use the
following terminology.

\begin{ter}\label{term:bundle}
Let $\Gamma$ be a compact Lie group and $Q$ a finite group.  For
the purposes of this paper, we use the terminology \term{$(\Gamma,Q)$
vector bundle} to mean a $(\Gamma \times Q^{\op})$-equivariant real vector bundle $E\to B$ with $(\Gamma \times
Q^{\op})$-invariant inner product, where $B$ is a Hausdorff (and
not just weak Hausdorff) space.  We require $E\to B$ to be
locally finite dimensional but do not require the dimension to be
constant over the different components of the base or the dimension
function to be bounded.  When needed, we refer to $\Gamma$ as the \term{group
of equivariance} and $Q$ as the \term{group for the quotient}.
\end{ter}

We then have the following construction.

\begin{cons}\label{cons:orbitds}
Let $\eta$ be a $(\Gamma,Q)$ vector bundle $E\to B$.  For a
vector space $W$ with inner product, let $\aI(\eta,W)\to B$ be the
locally trivial fiber bundle with fiber over $b$ the space
$\aI(E_{b},W)$ of linear isometries $E_{b}\to W$. (The
homeomorphism class of the fiber may vary with
the component of $B$.)  We then have a vector bundle
$\Im^{\perp}\aI(\eta,W)\to \aI(\eta,W)$ where the fiber over an
isometry $f\colon E_{b}\to W$ is the orthogonal complement of the
image.  Let $\aJ(\eta,W)$ be the Thom space of this vector bundle. 
We have a left $Q$ action on $\aJ(\eta,W)$ via the right
$Q$ action on $\eta$ and we convert it back to a right
$Q$ action using the inverse:
\[
\alpha \cdot \sigma := \alpha \circ \sigma_{*}^{-1}
\]
where $\sigma_{*}\colon E_{b}\to E_{b\sigma}$ is the action of
$\sigma$ on $E$.  We can then form the 
quotient $\aJ(\eta,W)/Q$.  When $W$ has an orthogonal (left)
$\Gamma$ action, $\aJ(\eta,W)/Q$ obtains a left $\Gamma$ action
via  action on $W$ and (inverse) action on $\eta$.  Letting $W$
vary, $\aJ(\eta,-)/Q$ obtains the structure of a based
$\Gamma$-space enriched functor from $\aJ_{\Gamma}$ to based $\Gamma$-spaces using
composition of isometries and (internal) direct
sum of vector spaces as in the definition of $\aJ_{\Gamma}$.
Let $J\eta$ denote the corresponding $\Gamma$-spectrum.  We write $J_{(\Gamma,Q)}\eta$ when $\Gamma$ or $Q$ 
need to be specified.
\end{cons}

\begin{example}
Taking $Q$ to be the trivial group, an orthogonal $\Gamma$-representation $V$ gives a $(\Gamma,Q)$ vector bundle
$V\to *$, and $J(V\to *)$ is canonically isomorphic to $F_{V}S^{0}$.
More generally, for $H<\Gamma$ a closed subgroup, $J(\Gamma/H\times V\to
\Gamma/H)$ is canonically isomorphic to the orbit desuspension spectrum
$F_{V}(\Gamma/H_{+})\iso \Gamma/H_{+}\sma F_{V}S^{0}$.  When $V$ is an orthogonal $H$-representation
$\eta\colon \Gamma\times_{H}V\to \Gamma/H$ is a $(\Gamma,Q)$ vector bundle,
and $J\eta$ is canonically isomorphic to the induced orbit
desuspension spectrum $\Gamma_{+}\sma_{H}F_{V}S^{0}$ (denoted
in~\cite{HHR} as $\Gamma_{+}\sma_{H}S^{-V}$).
\end{example}

The following proposition gives natural examples with $Q$ non-trivial that
include the basic building blocks for the underlying equivariant orthogonal spectra of cells for commutative ring equivariant orthogonal spectra.

\begin{prop}\label{prop:Sym}
For $\eta \colon E\to B$ a $(\Gamma,Q)$ vector bundle and $q\geq
1$, let $\Sym^{q}\eta$ denote the $(\Gamma,\Sigma_{q}\wr Q)$ vector
bundle obtained by taking the $q$th cartesian power
\[
\Sym^{q}\eta \colon E^{q}\to B^{q}.
\]
Then 
\[
J_{(\Gamma,\Sigma_{q}\wr Q)}(\Sym^{q}\eta)\iso (J_{(\Gamma,Q)}\eta)^{(q)}/\Sigma_{q}.
\]
More generally, if $\Sigma$ is any subgroup of $\Sigma_{q}$, let
$\Sigma^{q}_{\Sigma}\eta$ denote the $(\Gamma,\Sigma \wr Q)$ vector
bundle obtained by restring the group for the quotient; then 
\[
J_{(\Gamma,\Sigma\wr Q)}(\Sym^{q}_{\Sigma}\eta)\iso (J_{(\Gamma,Q)}\eta)^{(q)}/\Sigma.
\]
\end{prop}

\begin{proof}
Consider the map
\[
\aJ(\eta,W_{1}) \sma \dotsb \sma \aJ(\eta,W_{1})
\to \aJ(\eta^{q},W_{1}\oplus \dotsb \oplus W_{q})
\]
induced by cartesian product of isometries and complementary
subspaces.  Applying the universal property of 
the left Kan extension defining the smash product, we get a
$\Sigma_{q}\wr Q$ equivariant map of 
$\Gamma$-spectra 
\begin{equation}\label{eq:gdsym}
(J_{(\Gamma,1)}\eta)^{(q)}\to J_{(\Gamma,1)}(\eta^{q})
\end{equation}
(where $1$ denotes the trivial group).  
Since smash product commutes with colimits in each variable, 
taking the quotient by $\Sigma\wr Q$ action,
we get a map of $\Gamma$-spectra 
\[
(J_{(\Gamma,Q)}\eta)^{(q)}/\Sigma\to J_{(\Gamma,\Sigma\wr Q)}(\Sym^{q}_{\Sigma}\eta).
\]
To show that it is an isomorphism, it suffices to work
non-equivariantly, and indeed to show that the map~\eqref{eq:gdsym} is
an isomorphism of non-equivariant spectra.  Choosing a
cover $\aU=\{U_{\alpha}\}$ of $B$ by trivial neighborhoods for $\eta$,
$\aU^{q}$ gives a cover of $B^{q}$ by trivial neighborhoods of $\eta^{q}$.
Choosing trivializations $\eta_{U_{\alpha}}\iso U_{\alpha}\times
\bR^{k_{\alpha}}$, for each such neighborhood $U_{\alpha_{1}}\times
\dotsb \times U_{\alpha_{q}}$, we get isomorphisms
\begin{gather*}
\aJ(\eta|_{\alpha_{i}})\iso F_{\bR^{k_{\alpha_{i}}}}(U_{\alpha_{i}})_{+}\\
\aJ(\eta|_{U_{\alpha_{1}}}\times \dotsb \times \eta|_{U_{\alpha_{q}}})
\iso F_{\bR^{k_{\alpha_{1}}}\times \dotsb \times \bR^{k_{\alpha_{q}}}}(U_{\alpha_{1}}\times \dotsb \times U_{\alpha_{q}})_{+}
\end{gather*}
and the composite of these isomorphisms with the map in~\eqref{eq:gdsym} 
\[
F_{\bR^{k_{\alpha_{1}}}}(U_{\alpha_{1}})_{+}
\sma\dotsb\sma
F_{\bR^{k_{\alpha_{q}}}}(U_{\alpha_{1}})_{+}
\to
F_{\bR^{k_{\alpha_{1}}}\times \dotsb \times \bR^{k_{\alpha_{q}}}}(U_{\alpha_{1}}\times \dotsb \times U_{\alpha_{q}})_{+}
\]
is the standard isomorphism.  Without loss of generality, we may
assume that $\aU$ is closed under intersection, and make it a
partially ordered set under inclusion.  The map
in~\eqref{eq:gdsym} is then the colimit over $\aU$ of the local maps, and
so it is also an isomorphism. 
\end{proof}

An argument similar to the proof of the previous proposition also proves
the following result.

\begin{prop}\label{prop:bundleprod}
Let $\eta$ and $\eta'$ be $(\Gamma,Q)$ and $(\Gamma,Q')$
vector bundles, respectively. Then
\[
J_{(\Gamma,Q\times Q')}(\eta_{1}\times \eta_{2})\iso
J\eta_{1}\sma J\eta_{2}.
\]
\end{prop}

We have one more basic piece of notation and terminology for generalized orbit
desuspension spectra, relevant to the case when we have a fixed
isomorphism $G\iso \Gamma/\KG$ for some normal subgroup $\KG
\lhd \Gamma$.  In the next
section, we will construct an important class of examples of $(\Gamma,Q)$
vector bundles $\eta$ with the property that when $W$ is a
$\KG$-trivial orthogonal $\Gamma$-representation, the based
$\Gamma$-spaces $\aJ(\eta,W)/Q$ have trivial $\KG$ action. We
use the following terminology.

\begin{ter}\label{term:compatible}
Let $\KG <\Gamma$ be a closed subgroup.  We say that a
$(\Gamma,Q)$ vector bundle $\eta$ is
\term{$\Gamma/\KG$-compatible} if the based $\Gamma$-space
$\aJ(\eta,W)/Q$ has trivial $\KG$ action whenever the
orthogonal $\Gamma$-representation $W$ has trivial $\KG$ action.
\end{ter}

Orthogonal $\Gamma/\KG$-representations are essentially the same thing as
$\KG$-trivial orthogonal $\Gamma$-representations.  By slight
abuse of notation, we therefore use 
$\aJ_{\Gamma/\KG}$ for the full subcategory of
$\aJ_{\Gamma}$ consisting of the $\KG$-trivial orthogonal $\Gamma$-representations.  When $V$ and $W$ are $\KG$-trivial
orthogonal $\Gamma$-representations, the $\Gamma$-space
$\aJ_{\Gamma}(V,W)$ is a $\Gamma/\KG$-spaces and coincides with
$\aJ_{\Gamma/\KG}(V,W)$.  In particular, when $\eta$ is
$\Gamma/\KG$-compatible, $\aJ(\eta,W)/Q$ is a based
$\Gamma/\KG$ space for all $\KG$ fixed orthogonal $\Gamma$-representations $W$, and the restriction of $\aJ(\eta,-)/Q$
to these representations has the structure of a 
based $\Gamma/\KG$-space enriched functor from $\aJ_{\Gamma/\KG}$
to based $\Gamma/\KG$-spaces, that is to say, it defines a $\Gamma/\KG$-spectrum.

\begin{notn}\label{notn:compatible}
For $\eta \colon E\to B$ a $\Gamma/\KG$-compatible
$(\Gamma,Q)$ vector bundle, we write $J^{\Gamma/\KG}\eta$ for the
corresponding $\Gamma/\KG$-spectrum.  We write
$J^{\Gamma/\KG}_{(\Gamma,Q)}\eta$ for
$J^{\Gamma/\KG}\eta$ when $\Gamma$ or $Q$
need to be indicated.
\end{notn}

\section{Geometric fixed points and smash products
}
\label{sec:geobundle}

This sections studies the point-set and homotopical properties of the
multiplicative geometric fixed point functor on the generalized orbit
desuspension spectra introduced in the previous section.  The
multiplicative geometric fixed point functor is well-behaved on orbit
desuspension spectra, and therefore on those equivariant orthogonal
spectra nicely built out of them, that is, the cofibrant equivariant
orthogonal spectra in the standard model structure.  But it is not 
known to have well-behaved point-set or homotopical properties on
arbitrary spectra. The purpose of this section is to expand the
results that hold for orbit desuspension spectra to a wide class of
generalized orbit desuspension spectra.  This section contains
definitions and statements only; proofs appear in
Chapter~\ref{chap:bundletwo}.

In this section and the next, we use the following notational
convention.  It is meant to capture the situation when the Lie group
of equivariance $G$ is isomorphic to the quotient $\Gamma/\KG$ of a
Lie group $\Gamma$ by a closed normal subgroup; a normal subgroup
$\Lambda\lhd \Gamma$ containing $\KG$ then corresponds to a normal
subgroup $L\lhd G$ with the isomorphism $G\iso \Gamma/\KG$ taking $L$
to $\Lambda/\KG$.  In many cases of interest, $\KG$ is the trivial
group and then the $\Gamma/\KG$-compatibility hypotheses (where it
appears) is vacuous.

\begin{notn}\label{notn:KGL}
In this section $\Gamma$ denotes a compact Lie group, and $\KG,\Lambda$
denote closed normal subgroups of $\Gamma$ with $\KG<\Lambda$.
\end{notn}

We are particularly interested in the generalized orbit desuspension
spectra that satisfy the following faithfulness property.

\begin{defn}\label{defn:Sigmafaithful}
We say that a $(\Gamma,Q)$ vector bundle $\eta\colon E\to B$ is
$Q$-faithful when for every $b\in B$, the isotropy subgroup
$Q_{b}< Q$ acts faithfully on the fiber $E_{b}$.
\end{defn}

Because the condition is defined fiberwise and equivariant maps do not
decrease stabilizer subgroups, the faithfulness property above is
preserved by all the usual kinds of restriction: 

\begin{prop}\label{prop:Sfrestr}
Let $\eta\colon E\to B$ be a $Q$-faithful $(\Gamma,Q)$ vector bundle.  
\begin{enumerate}
\item For any $Q'<Q$, and any homomorphism $\Gamma'\to \Gamma$, the
restriction of $\eta$ to a $(\Gamma',Q')$ vector bundle is $Q'$-faithful.
\item For any Hausdorff $(\Gamma \times Q^{\op})$-space $B'$ and any
$(\Gamma \times Q^{\op})$-equivariant map $f\colon B'\to B$, the pullback
$(\Gamma,Q)$ vector bundle $f^{*}\eta$ is $Q$-faithful.
\end{enumerate}
\end{prop}

We use the $Q$-faithfulness hypothesis for most results in this section.  In
particular, it appears in the following theorem underlying all work on the
multiplicative geometric fixed points. 

\begin{thm}\label{thm:bundlegeofix}
Let $\eta$ be a $\Gamma/\KG$-compatible $Q$-faithful $(\Gamma,Q)$ vector bundle.
Then there exists an isomorphism of
$\Gamma/\Lambda$-spectra
\[
\Phi^{\Lambda/\KG}J^{\Gamma/\KG}\eta\iso
J^{\Gamma/\Lambda}(\eta(\Lambda|\KG))
\]
where $\eta(\Lambda|\KG)$ is the $\Gamma/\Lambda$-compatible $(\Gamma,Q)$
vector bundle of Definition~\ref{defn:philambdatwo} below.
\end{thm}

In the case of a non-normal closed subgroup $H<\Gamma$ containing
$\KG$, $\KG$ is a normal subgroup of the normalizer $N_{\Gamma}H$ of $H$ in
$\Gamma$.  Then restricting the group of equivariance, $\eta$ becomes
a $Q$-faithful $(N_{\Gamma}H,Q)$ vector bundle and the previous theorem then
asserts an isomorphism of $N_{\Gamma}H/H$-spectra.

The previous theorem now gives some control over the multiplicative
geometric fixed points on generalized orbit desuspension spectra with
enough faithfulness hypotheses.  For example, the following theorem
asserts that iterating multiplicative geometric fixed point functors
behaves properly when its hypotheses hold.  (It is not known to hold
for arbitrary $\Gamma$-spectra,
cf.~\cite[2.6]{BM-cycl}.) 

\begin{thm}\label{thm:bundleiterphi}
Let $M$ be a closed normal
subgroup of $\Gamma$ with $\Lambda < M$, and let $\eta$ be a
$\Gamma/\KG$-compatible $(\Gamma,Q)$ vector bundle.  If $\eta$ and $\eta(\Lambda|\KG)$ are
both $Q$-faithful, then the canonical map of
$\Gamma/M$-spectra
\[
\Phi^{M/\KG}(J^{\Gamma/\KG}\eta)\to
\Phi^{M/\Lambda}(\Phi^{\Lambda/\KG}(J^{\Gamma/\KG}\eta))
\]
is an isomorphism.
\end{thm}

A technical issue throughout~\cite{BM-cycl} is that although the
multiplicative geometric fixed point functor is lax symmetric
monoidal, it is not known to be strong symmetric monoidal, and a
difficult technical result in~\textit{ibid.} is that it is strong
symmetric monoidal when one variable is an orbit desuspension
spectrum.  The following theorem generalizes this to $Q$-faithful
generalized orbit desuspension spectra.  

\begin{thm}\label{thm:sm}
Suppose $\eta$ is a $Q$-faithful $\Gamma/\KG$-compatible
$(\Gamma,Q)$ vector bundle where the base $B$ admits the structure of a
$(\Gamma \times Q^{\op})$-equivariant cell complex.  Then
for any $\Gamma/\KG$-spectrum $X$, the canonical map of
orthogonal spectra  
\[
\Phi^{\Lambda/\KG}(J^{\Gamma/\KG}\eta)\sma \Phi^{\Lambda/\KG}X\to 
\Phi^{\Lambda/\KG}(J^{\Gamma/\KG}\eta\sma X)
\]
is an isomorphism.
\end{thm}

The previous theorems study point-set properties.  The next three
theorems study the homotopical properties.  The first of these gives
hypotheses that ensure that the multiplicative geometric fixed point
functor represents the derived geometric fixed point functor.  We
state it for the normal subgroup $\Lambda$, but as in the remarks
following Theorem~\ref{thm:bundlegeofix}, restricting from
$\Gamma$ to the normalizer gives the desired theorem for a general
closed subgroup.

\begin{thm}\label{thm:derphi}
Let $\eta$ be a $\Gamma/\KG$-compatible
$(\Gamma,Q)$ vector bundle.  If $\eta$ and $\eta(\Lambda|\KG)$ are
both $Q$-faithful, then the canonical map in the genuine
$\Gamma/\Lambda$-equivariant stable category 
$L\Phi^{\Lambda/\KG}(J^{\Gamma/\KG}\eta) \to
\Phi^{\Lambda/\KG}(J^{\Gamma/\KG}\eta)$ is an isomorphism.
\end{thm}

The final two theorems do not involve the multiplicative geometric fixed
point functor but share much of the same work as the previous
results.  The first is about derived functors of symmetric powers.

\begin{thm}\label{thm:dersym}
Suppose $\eta$ is a $Q$-faithful $\Gamma/\KG$-compatible $(\Gamma,Q)$
vector bundle with each fiber positive dimensional and $X\to
J^{\Gamma/\KG}\eta$ is a cofibrant approximation in the positive
complete model structure on $G/\KG$-spectra.
Then for any $k\geq 1$, and any subgroup $\Sigma <\Sigma_{k}$, the map
$X^{(k)}/\Sigma\to (J^{\Gamma/\KG}\eta)^{(k)}/\Sigma$ is a weak
equivalence. 
\end{thm}

We also have the following flatness result.

\begin{thm}\label{thm:bundleflat}
Suppose $\eta$ is a $Q$-faithful $\Gamma/\KG$-compatible
$(\Gamma,Q)$ vector bundle and assume the base $B$ admits the structure of a
$(\Gamma \times Q^{\op})$-equivariant cell complex.  Then
$J^{\Gamma/\KG}\eta$ is flat for the smash product in
$\Gamma/\KG$-spectra.
\end{thm}

The rest of the section is devoted to the construction and basic
properties of
$\eta(\Lambda|\KG)$.  We start with the easier construction of
$\eta(\Lambda)$ which corresponds to the case when $\KG$  is the
trivial subgroup.

\begin{defn}\label{defn:philambdaone}
Let $\Lambda$ be a closed normal subgroup of $\Gamma$ and let $\eta$
be a $Q$-faithful $(\Gamma,Q)$ vector bundle $E\to B$.  For
a continuous homomorphism $\sigma\colon \Lambda \to Q$, we get a
\term{$\sigma$-twisted $\Lambda$ action} on $E\to B$ from the
isomorphism of $\Lambda$ to the graph subgroup of $\Gamma \times
Q^{\op}$ associated to the homomorphism $\sigma((-)^{-1})$.
Concretely,
\[
\lambda \cdot_{\sigma} e := \lambda e \sigma(\lambda^{-1}).
\]
We write $\eta^{\Lambda,\sigma}\colon E^{\Lambda,\sigma}\to
B^{\Lambda,\sigma}$ for the fixed points under this action. 
Define the map
\[
\eta(\Lambda)\colon E(\Lambda)\to B(\Lambda)
\]
as 
\[
\coprod_{\sigma\colon \Lambda\to Q} (\eta^{\Lambda,\sigma} \colon E^{\Lambda,\sigma}\to B^{\Lambda,\sigma}).
\]
\end{defn}

We assert that the map $\eta(\Lambda)$ in the previous definition has
a canonical $\Gamma/\Lambda$-compatible $(\Gamma,Q)$ vector
bundle structure, with the $\Gamma\times Q^{\op}$ action
inherited from $\eta$ and the conjugation action on
$\Hom(\Lambda,Q)$.

For the vector bundle structure, we use the (trivial) fact that for
any $\sigma \colon \Lambda \to 
Q$, the $\sigma$-twisted $\Lambda$ action on
$\eta|_{B^{\Lambda,\sigma}}$ is fiberwise.  Then for any $b\in
B^{\Lambda,\sigma}$, there is an open set $b\in U\subset
B^{\Lambda,\sigma}$ on which the restriction $E|_{U}\to U$ is
$\Lambda$-equivariantly isomorphic (with the $\sigma$-twisted
$\Lambda$ action) to the product vector bundle $U\times V\to U$ for
the orthogonal $\Lambda$-representation $V=E_{b}$. (This
follows for example by character theory.) Taking $\Lambda$ fixed
points, we get the local trivialization for $\eta(\Lambda)$.

We construct the $\Gamma \times Q^{\op}$ action as follows.
An element $x$ of $E(\Lambda)$ consists of a homomorphism $\sigma\colon
\Lambda \to Q$ together with an element $e\in
E^{\Lambda,\sigma}$; we write $x=(\sigma,e)$.  For $\gamma \in\Gamma$, 
\begin{equation}\label{eq:Gammaaction}
\gamma (\sigma,e)=(\sigma^{\gamma}, \gamma e)
\end{equation}
where $\sigma^{\gamma} \colon \Lambda \to Q$ is the homomorphism
\[
\sigma^{\gamma} \colon \lambda \mapsto \sigma (\gamma^{-1}\lambda\gamma).
\]
To see that this specifies a left $\Gamma$ action, we need to check
that $\gamma e\in E^{\Lambda,\sigma^{\gamma}}$.  For $\lambda \in
\Lambda$,
\[
\lambda \cdot_{\sigma^{\gamma}}\gamma e:=
\lambda \gamma e\sigma^{\gamma}(\lambda^{-1})
:=\lambda \gamma e\sigma(\gamma^{-1} \lambda^{-1}\gamma)
=\gamma \gamma^{-1}\lambda \gamma e\sigma(\gamma^{-1} \lambda^{-1}\gamma)
=\gamma e
\]
with the last equality holding because $e\in E^{\Lambda,\sigma}$. For
the right $Q$ action, given $s \in Q$, 
\begin{equation}\label{eq:Qaction}
(\sigma,e)s=(\sigma_{s},es)
\end{equation}
where
\[
\sigma_{s}(\lambda)=s^{-1}\sigma(\lambda)s.
\]
To see that this
specifies a right $Q$ action, we need to check that  $es\in
E^{\Lambda,\sigma_{s}}$.  For $\lambda \in \Lambda$,
\[
\lambda \cdot_{\sigma_{s}}es:=\lambda es\sigma_{s}(\lambda^{-1})
:=\lambda ess^{-1}\sigma(\lambda^{-1})s=\lambda e\sigma(\lambda^{-1})s=es
\]
with the last equality holding because $e\in E^{\Lambda,\sigma}$. It
is clear from the formulas that the left $\Gamma$ action and right
$Q$ action commute.

Finally, we need to check that $\eta(\Lambda)$ is $\Gamma/\Lambda$-compatible.
It suffices to show that if $W$ is a $\Lambda$-trivial orthogonal $\Gamma$-representation, the $\Gamma$ space $\Im^{\perp}\aJ(\eta(\Lambda),W)/Q$
has trivial $\Lambda$ action.  An element of
$\Im^{\perp}\aJ(\eta(\Lambda),W)$ is specified by a homomorphism $\sigma
\colon \Lambda \to Q$, a point $b\in B^{\Lambda,\sigma}$, an
isometry $f\colon E^{\Lambda,\sigma}|_{b}\to W$, and a point $w$ in the complement of
its image. We write this as a tuple $(\sigma,b,f,w)$, and we write
$[\sigma,b,f,w]$ for the image in the quotient by $Q$.  For $\lambda
\in \Lambda$,
\begin{align*}
\lambda(\sigma,b,f,w)=(\sigma^{\lambda},\lambda b,f^{\lambda},\lambda w)
&=(\sigma^{\lambda},\lambda b, f\circ \lambda_{*}^{-1},w)\\
&=(\sigma^{\lambda},b\sigma(\lambda), f\circ \sigma(\lambda)^{-1}_{*},w)
\end{align*}
where the first equality follows from the fact that $W$ is
$\Lambda$-trivial and the second equality follows from the fact that
$b$ and the domain of $f$ lie in the $\sigma$-twisted $\Lambda$ fixed
points.  Since 
\[
\sigma^{\lambda}(-)=\sigma(\lambda^{-1}(-)\lambda)
=\sigma(\lambda)^{-1}\sigma(-)\sigma(\lambda)=\sigma_{\sigma(\lambda)}(-),
\]
we see that
\[
\lambda(\sigma,b,f,w)=(\sigma,b,f,w)\sigma(\lambda)
\]
and so $[\sigma,b,f,w]$ is fixed by $\lambda$.

We now construct $\eta(\Lambda|\KG)$ by restricting $\eta(\Lambda)$ to
a subset of its base space.

\begin{defn}\label{defn:philambdatwo}
Let $\KG,\Lambda \lhd \Gamma$, with $\KG<\Lambda$ and let $\eta$ be a
$\Gamma/\KG$-compatible $(\Gamma,Q)$ 
vector bundle.  Given $\sigma \colon \Lambda\to Q$, the $\sigma$-twisted
$\Lambda$ action on the restricted bundle $\eta|_{B^{\Lambda,\sigma}}$ is
fiberwise, and for each component of $B^{\Lambda,\sigma}$, the isomorphism 
class of the representation of the fiber is constant.  Let
$B^{\Lambda,\sigma}\{\KG\}$ denote the subspace of $B^{\Lambda,\sigma}$
consisting of those components where the fiberwise
$\sigma|_{\KG}$-twisted $\KG$ action is trivial.  Write
\[
\eta^{\Lambda,\sigma}\{\KG\}:=(\eta|_{B^{\Lambda,\sigma}\{\KG\}})^{\Lambda,\sigma}
\]
for the $\sigma$-twisted $\Lambda$ fixed points of the restriction of $\eta$
to $B^{\Lambda,\sigma}\{\KG\}$, and note that
$\eta^{\Lambda,\sigma}\{\KG\}$ is also the restriction of
$\eta^{\Lambda,\sigma}$ to the components $B^{\Lambda,\sigma}\{\KG\}\subset
B^{\Lambda,\sigma}$ of its base.  Write 
\[
\eta(\Lambda|\KG):=
\coprod_{\sigma\colon \Lambda\to Q} \eta^{\Lambda,\sigma}\{\KG\}.
\]
By construction, $\eta(\Lambda|\KG)$ is the restriction of
$\eta(\Lambda)$ to a certain $(\Gamma\times Q^{\op})$-stable subset of
components in its base.  We give $\eta(\Lambda)\{\KG\}$ the
$\Gamma/\Lambda$-compatible $(\Gamma,Q)$ vector bundle structure
inherited from $\eta(\Lambda)$. 
\end{defn}

As a special case of Proposition~\ref{prop:Sfrestr}, we point out the following observation.

\begin{prop}\label{prop:flk}
Let $\eta$ be a $\Gamma/\KG$-compatible $(\Gamma,Q)$ vector bundle.
If $\eta(\Lambda)$ is $Q$-faithful, then so is $\eta(\Lambda|\KG)$
\end{prop}

\section{Geometric fixed points of symmetric powers
}
\label{sec:endbundle}
\label{sec:geoSym}

This section again follows the notational conventions of
Notation~\ref{notn:KGL}.  Its purpose is to 
provide results on and tools for working with the
point-set geometric fixed points of symmetric powers of generalized
orbit suspension spectra $J^{\Gamma/\KG}\eta$ under certain faithfulness hypotheses
we state below.  The main theorem, Theorem~\ref{thm:calcgeoSym}, and
its proof give a complete description of
$\Phi^{\Lambda/\KG}((J^{\Gamma/\KG}\eta)^{(q)}/\Sigma_{q})$ when $\Sym^{q}\eta$ is
$(\Sigma_{q}\wr Q)$-faithful.  As discussed in
Proposition~\ref{prop:Sym} and Theorem~\ref{thm:bundlegeofix}, under
this faithfulness hypotheses, $\Phi^{\Lambda/\KG}((J^{\Gamma/\KG}\eta)^{(q)}/\Sigma_{q})$
is itself the generalized orbit desuspension spectrum
$J^{\Gamma/\Lambda}((\Sym^{q}\eta)(\Lambda|\KG))$ and we describe a
criterion (Proposition~\ref{prop:ifcritex}, or more generally,
Theorem~\ref{thm:ifcrit}) for $(\Sym^{q}\eta)(\Lambda)$ to be
$(\Sigma_{q}\wr Q)$-faithful. We need these results for the bundles considered
in the arguments in Section~\ref{sec:ACyc}.  Both
Theorem~\ref{thm:calcgeoSym} and~\ref{thm:ifcrit} are proved in later
sections. 

Let $\bP$ denote the free functor from equivariant
orthogonal spectra to equivariant commutative ring orthogonal spectra;
it is the wedge sum of symmetric powers
\[
\bP X=\bigvee_{n\geq 0}X^{(n)}/\Sigma_{n}
\]
where we understand $X^{(0)}/\Sigma_{0}$ as the sphere spectrum $\bS$.
We prove the following theorem in Section~\ref{sec:calgeoSym}.  Along
the way to the proof, we give more details on the individual symmetric
powers including an intrinsic description of the spectra $X_{q}$ in
terms of $\eta$ and $\Lambda$.

\begin{thm}\label{thm:calcgeoSym} 
Let $\eta$ be a $\Gamma/\KG$-compatible $(\Gamma,Q)$ vector bundle and
assume that
$\Sym^{q}\eta$ is $(\Sigma_{q}\wr 
Q)$-faithful for every $q\geq 1$.  Then there exist $\Gamma/\Lambda$-spectra $X_{q}(\eta)$,
$1\leq q\leq \#(\pi_{0}\Lambda)$, such that $X_{q}(\eta)$ is a wedge summand of
$\Phi^{\Lambda/\KG}((J^{\Gamma/\KG}\eta)^{(q)}/\Sigma_{q})$ and the inclusions
induce an isomorphism  
\[
\bigwedge_{q=1}^{\#\pi_{0}\Lambda}\bP(X_{q}(\eta))\iso 
\bP\biggl(\bigvee_{q}X_{q}(\eta)\biggr)\to
\Phi^{\Lambda/\KG}(\bP(J^{\Gamma/\KG}\eta)) 
\]
of commutative ring $\Gamma/\Lambda$-spectra.
Moreover, if $D$ is a (non-equivariant) Hausdorff space, then
$X_{q}(\eta\times D)\iso X_{q}(\eta)\sma D^{q}_{+}$, naturally in $D$.
\end{thm} 

The hypothesis that each $\Sym^{q}\eta$ is 
$(\Sigma_{q}\wr Q)$-faithful is not usually much more stringent than the
hypothesis that $\eta$ is $Q$-faithful.  An easy argument gives
the following observation.

\begin{prop}\label{prop:trivsf}
If $\eta\colon E\to B$ is $Q$-faithful and either $E_{b}$ is
positive dimensional for all $b\in B$ or $Q$ is not
the trivial group, then $\Sym^{q}\eta$ is $(\Sigma_{q}\wr
Q)$-faithful for all $q\geq 1$. 
\end{prop}

Theorem~\ref{thm:derphi} allows us to deduce homotopical conclusions
from Theorem~\ref{thm:calcgeoSym} when the $(\Gamma,\Sigma_{q}\wr Q)$ vector
bundle $(\Sym^{q}\eta)(\Lambda)$ (or its base restriction
$(\Sym^{q}\eta)(\Lambda|\KG)$) is $(\Sigma_{q}\wr Q)$-faithful
for all $q\geq 1$. We use the following terminology.

\begin{defn}\label{defn:isf}
A $(\Gamma,Q)$ vector bundle is inheritably faithful for
$\Phi^{\Lambda}$ means that for every $q\geq 1$,
$(\Sym^{q}\eta)(\Lambda)$ is $(\Sigma_{q}\wr Q)$-faithful.
(See Proposition~\ref{prop:Sym} and
Definitions~\ref{defn:Sigmafaithful}, \ref{defn:philambdaone}.)
\end{defn}

The previous definition asks for a minimal hypothesis to check but
leads to the following wider conclusion as an application of Proposition~\ref{prop:Sfrestr}.

\begin{prop}\label{prop:isfrestr}
Let $\eta$ be a $(\Gamma,Q)$ vector bundle and assume that
$(\Sym^{q}\eta)(\Lambda)$ is $(\Sigma_{q}\wr Q)$-faithful.  Then for any $\Sigma <\Sigma_{q}$, $(\Sym^{q}_{\Sigma}\eta)(\Lambda)$ is
$(\Sigma\wr Q)$-faithful.  Moreover if $\eta$ is $\Gamma/\KG$-compatible then 
for any $\Sigma <\Sigma_{q}$, $(\Sym^{q}_{\Sigma}\eta)(\Lambda|\KG)$ is
$(\Sigma\wr Q)$-faithful.
\end{prop}

To employ this theory, we offer the following result for identifying inheritably faithful vector
bundles.  In the statement $\Lambda_{0}$ denotes the identity
component of $\Lambda$.  We prove the following theorem in
Section~\ref{sec:ifcrit}. 

\begin{thm}\label{thm:ifcrit}
Let $\eta \colon E\to B$ be a $(\Gamma,Q)$ vector bundle such that for every
closed subgroup $\Lambda_{0}<H<\Lambda$, every $\sigma\colon H\to
Q$, and every $b\in B^{H,\sigma}$:
\begin{enumerate}
\item $(E_{b})^{H,\sigma}$ is positive dimensional,
\item Given $\lambda \in \Lambda - H$, $s\in Q$ such that
$\lambda b=bs$, there exists $v\in (E_{b})^{H,\sigma}$ for which
$\lambda v\neq vs$, and
\item Given $s\in Q$, not the identity element, such that $bs=b$
and $s$ commutes with $\sigma(h)$ for all $h\in H$, there exists $v\in
(E_{b})^{H,\sigma}$ for which $v\neq vs$.
\end{enumerate}
Then $\eta$ is inheritably faithful for $\Phi^{\Lambda}$.
\end{thm}

We note that condition~(iii) is
equivalent to $\eta(H)$ being $Q$-faithful for all
$\Lambda_{0}<H<\Lambda$.  The conditions given in the previous theorem
are technical; it is helpful to have concrete special cases
that provide examples for when the conditions hold.  We provide these
in the following proposition.  We
note that in the first set of examples, the hypothesis does not
involve $\Lambda$, and so it gives examples of $(\Gamma,Q)$
vector bundles that are inheritably faithful for $\Phi^{H}$ for all
subgroups $H$ (even non-normal subgroups after restricting $\Gamma$ to
the normalizer of $H$).

\begin{prop}\label{prop:ifcritex}
If $\eta\colon E\to B$ is a $(\Gamma,Q)$ vector bundle that satisfies
either of the following hypotheses, then $\eta$ satisfies the
conditions of Theorem~\ref{thm:ifcrit} and is therefore inheritably
faithful for $\Phi^{\Lambda}$.
\begin{enumerate}
\item[(a)] For every $b\in B$, the isotropy subgroup $(\Gamma\times Q^{\op})_{b}$ is
finite and the orthogonal $(\Gamma\times
Q^{\op})_{b}$-representation $E_{b}$ contains 
a copy of the regular representation.
\item[(b)] $Q$ is the trivial group and for every $b\in B$ whose isotropy subgroup $\Gamma_{b}<\Gamma$
contains $\Lambda_{0}$, the (restricted) orthogonal $(\Gamma_{b}\cap
\Lambda)$-representation $E_{b}$ contains a copy of the regular representation
of $\pi_{0}(\Gamma_{b}\cap \Lambda)$.
\end{enumerate}
\end{prop}

\begin{proof}
Let $\Lambda_{0}<H<\Lambda$, $\sigma \colon H\to Q$, and $b\in
B^{H,\sigma}$;  we check conditions~(i), (ii), and~(iii) of
Theorem~\ref{thm:ifcrit}.  

For hypothesis~(a), since $b\in B^{H,\sigma}$, the subgroup
$H_{\sigma}=\{(h,\sigma(h^{-1}))\}$ of $\Gamma \times Q^{\op}$ is
contained in the isotropy subgroup $L:=(\Gamma \times
Q^{\op})_{b}$.  The $H_{\sigma}$ fixed points of the regular
$L$-representation are positive dimensional, so (i) holds.  Now
suppose that $\lambda b=b\tau$ for $\tau\in Q$ and either 
$\lambda\in \Lambda -H$ or $\lambda =1$ with $\tau \neq 1$.  In either case,
$(\lambda,\tau^{-1})\in L$ but not in $H_{\sigma}$ so
$(\lambda,\tau^{-1})$ acts non-trivially on the $H_{\sigma}$ fixed
points of the regular $L$-representation.

For hypothesis~(b), since $H$ fixes $b$, we have $H<\Gamma_{b}$ and in
particular $\Lambda_{0}<\Gamma_{b}$.  Condition~(iii) holds vacuously
and condition~(i) holds because the fixed points of the regular
representation form a positive dimensional subspace. For 
condition~(ii), let
$\lambda \in \Lambda -H$ with $\lambda b=b$; then $\lambda \in
\Gamma_{b}\cap (\Lambda -H)$ and the image of $\lambda$ in
$\pi_{0}(\Gamma_{b}\cap \Lambda)$ is disjoint from $\pi_{0}H$.  It
follows that $\lambda$ acts non-trivially on the $H$ fixed points of
the regular $\pi_{0}(\Gamma_{b}\cap \Lambda)$-representation,
establishing condition~(ii).
\end{proof}

\chapter{Generalized Orbit Desuspension Spectra: Proofs}
\oldlabel{chap:bundletwo}

This chapter provides the proofs for the main theorems of
Chapter~\ref{chap:bundleone}.

\section{The quotient fixed point lemma and
Theorems~\sref$thm:bundlegeofix$, \sref$thm:bundleiterphi$, and~\sref$thm:derphi$}
\label{sec:qfpl}

This section proves Theorems~\ref{thm:bundlegeofix},
\ref{thm:bundleiterphi}, and~\ref{thm:derphi}. We follow the notational convention of
Notation~\ref{notn:KGL}: $\Gamma$ is a compact Lie group with
closed normal subgroups
$\KG,\Lambda$ satisfying $\KG<\Lambda$.  The proofs have in 
common a lemma about fixed points of the quotients of the isometry
bundles intrinsic to the construction of generalized orbit
desuspension spectra for $Q$-faithful $(\Gamma,Q)$ vector
bundles.  This lemma in turn is based on the following (easy)
observation.

\begin{prop}\label{prop:faithfulfree}
If $\eta$ is a $Q$-faithful $(\Gamma,Q)$ vector bundle, then for any
orthogonal $\Gamma$-representation $W$, the 
action of $Q$ on $\aI(\eta,W)$ and $\Im^{\perp}\aI(\eta,W)$ is free.
\end{prop}

The following is the main technical result we need in the proofs of the
theorems.  As in the definition of the bundle $\eta(\Lambda)$
(Definition~\ref{defn:philambdaone}), we write $(-)^{\Lambda,\sigma}$ for
the fixed points of the $\sigma$-twisted $\Lambda$ action for a
continuous homomorphism $\sigma \colon \Lambda \to Q$.

\begin{lem}\label{lem:bundlegeofix}
Let $\eta\colon E\to B$ be a $Q$-faithful $(\Gamma,Q)$ vector bundle, and
let $W$ be an orthogonal $\Gamma$-representation.  Then we have a
canonical isomorphism 
\[
(\aJ(\eta,W)/Q)^{\Lambda}\iso
\biggl(\bigvee_{\sigma \colon \Lambda \to Q} 
   \aJ(\eta,W)^{\Lambda,\sigma}
\biggr)/Q
\]
where $Q$ acts by conjugation on the set of homomorphisms, acts
as usual on $\eta$, and acts trivially on $W$.  Letting $\Gamma$ act on
the right hand side by conjugation on
the set of homomorphisms and as usual on $\eta$ and $W$, the 
isomorphism is $\Gamma$- and therefore $\Gamma/\Lambda$-equivariant.
Furthermore, for each summand on the right hand side (that is, for each
continuous homomorphism $\sigma \colon \Lambda \to Q$):
\begin{enumerate}
\item $(\Im^{\perp}\aI(\eta,W))^{\Lambda,\sigma}\to \aI(\eta,W)^{\Lambda,\sigma}$
is a vector bundle with Thom space\break $\aJ(\eta,W)^{\Lambda,\sigma}$.
\item $\aI(\eta,W)^{\Lambda,\sigma}\to B^{\Lambda,\sigma}$
is a locally trivial fiber bundle with fiber
$\aI(E_{b},W)^{\Lambda,\sigma}$. (The fiber may vary with the component
of $B^{\Lambda,\sigma}$.)
\end{enumerate}
\end{lem}

\begin{proof}
The numbered statements follow from the fact that the $\sigma$-twisted
$\Lambda$ action on 
\[
\aI(\eta|_{B^{\Lambda,\sigma}},W)\to B^{\Lambda,\sigma}
\]
is fiberwise, and locally in $b\in B^{\Lambda,\sigma}$,
$\eta|_{B^{\Lambda,\sigma}}$ is a product bundle $U\times V\to U$ for
an orthogonal $\Lambda$-representation $V$ (that may vary in the
component of $B^{\Lambda,\sigma}$). 

To study $(\aJ(\eta,W)/Q)^{\Lambda}$, we use its identification
as the Thom space of the vector bundle 
\[
(\Im^{\perp}\aI(\eta,W)/Q)^{\Lambda}\to (\aI(\eta,W)/Q)^{\Lambda}.
\]
Let $q\colon \aI(\eta,W)\to \aI(\eta,W)/Q$ and consider 
\[
q^{-1}((\aI(\eta,W)/Q)^{\Lambda})\subset \aI(\eta,W).
\]
An element on the left consists of a point $b\in B$ and an isometry
$f\colon E_{b}\to W$ such that for all $\lambda\in \Lambda$, there
exists $\sigma(\lambda)\in Q$ such that $\lambda
b=b\sigma(\lambda)$ and $\lambda \cdot f=f\circ \sigma(\lambda)_{*}^{-1}$
as isometries $E_{\lambda b}=E_{b\sigma(\lambda)}\to W$. Since
$Q$ acts freely on $\aI(\eta,W)$
(Proposition~\ref{prop:faithfulfree}), $\sigma(\lambda)$ is unique and
$\lambda \mapsto \sigma(\lambda)$ defines a continuous homomorphism
$\sigma \colon \Lambda \to Q$.  In other words, $f$ is a
$\sigma$-twisted $\Lambda$ fixed point for a unique $\sigma \colon
\Lambda \to Q$ (where we understand the $Q$ action on $W$ to
be trivial). Likewise, for any $\sigma \colon \Lambda \to
Q$, the $\sigma$-twisted $\Lambda$ fixed points of
$\aI(\eta,W)$ lie in $q^{-1}((\aI(\eta,W)/Q)^{\Lambda})$.
This shows
\[
(\aI(\eta,W)/Q)^{\Lambda}=
\biggl(\coprod_{\sigma \colon \Lambda \to Q}
\aI(\eta,W)^{\Lambda,\sigma}\biggr)/Q
\]
and we can similarly identify $(\Im^{\perp}\aI(\eta,W)/Q)^{\Lambda}$ as
\[
\biggl(\coprod_{\sigma \colon \Lambda \to Q}
(\Im^{\perp}\aI(\eta,W))^{\Lambda,\sigma}\biggr)/Q.
\]
The identification of $\aJ(\eta,W)/Q$ in the statement is now
clear.  The proof that the isomorphism is $\Gamma$-equivariant is a
straightforward check of the action on the righthand side.
\end{proof}

Each of the theorems we are proving in this section concerns a
multiplicative geometric fixed point functor of the form
$\Phi^{\PLambda}$ from the category
of $\PGamma$-spectra to the category of $\PGamma/\PLambda$-spectra,
where $\PLambda$ is a normal subgroup of a 
compact Lie group $\PGamma$; in the theorems, we have in various
cases (in the respective notation there), $G=\Gamma/\KG$ and $H=\Lambda/\KG$, $G=\Gamma/\KG$ and
$H=M/\KG$, or $G=\Gamma/\Lambda$ and $H=M/\Lambda$. The
arguments require some background in the construction of these functors.
We recall that $\Phi^{\PLambda}$ is a composite of a functor
$\Fix^{\PLambda}$ and a left Kan extension as follows.  

Let $\PGamma$ be a compact Lie group with $\PLambda\lhd G$ normal subgroup.
As in previous sections, we take the perspective that
$\PGamma$-spectra are \term{$\PGamma$-equivariant $\aJ_{\PGamma}$-spaces}: based $\PGamma$-space enriched functors 
from the category $\aJ_{\PGamma}$ (reviewed in Section~\ref{sec:startbundle}) to based
$\PGamma$-spaces.  Following~\cite[V\S4]{MM} (with slightly altered
notation), let $\aJ_{\PGamma}^{\PLambda}$ be the category enriched in
based $\PGamma/\PLambda$-spaces obtained as the $\PLambda$ fixed points
of $\aJ_{\PGamma}$.  It has the same objects (namely, the orthogonal $\PGamma$-representations) and for objects $V,W$, the based
$\PGamma/\PLambda$-space of morphisms
$\aJ_{\PGamma}^{\PLambda}(V,W)=(\aJ_{\PGamma}(V,W))^{\PLambda}$.  We
define the category of $\PGamma/\PLambda$-equivariant $\aJ^{\PLambda}_{\PGamma}$-spaces to be the 
category of based $\PGamma/\PLambda$-space enriched functors
from $\aJ_{\PGamma}^{\PLambda}$ to based $\PGamma/\PLambda$-spaces.  The
functor $\Fix^{\PLambda}$ from $\PGamma$-spectra to $\PGamma/\PLambda$-equivariant $\aJ_{\PGamma}^{\PLambda}$-spaces is defined by
\[
(\Fix^{\PLambda}X)(W):=(X(W))^{\PLambda}
\]
for $W$ an orthogonal $\PGamma$-representation.

For the left Kan extension, we use the based $\PGamma/\PLambda$-space
enriched functor $\phi \colon
\aJ^{\PLambda}_{\PGamma}\to \aJ_{\PGamma/\PLambda}$ that on objects takes
$V$ to its $\PLambda$ fixed points $V^{\PLambda}$ and on morphisms
restricts to fixed points.  Precisely, the morphism space
$\aJ_{\PGamma}^{\PLambda}(V,W)$ is the Thom space of the vector bundle
\[
(\Im^{\perp}\aI(V,W))^{\PLambda}\to (\aI(V,W))^{\PLambda}.
\]
where an element in the total space consists of a $\PLambda$-equivariant isometry
$V\to W$ together with an element of $W^{\PLambda}$ in the orthogonal complement of the image.  On morphisms $\phi$ is the map on Thom
spaces induced by the map of vector bundles 
\[
\bigl((\Im^{\perp}\aI(V,W))^{\PLambda}\rightarrow (\aI(V,W))^{\PLambda}\bigr)\to
\bigl(\Im^{\perp}\aI(V^{\PLambda},W^{\PLambda})\to \aI(V^{\PLambda},W^{\PLambda})\bigr)
\]
that takes the $\PLambda$-equivariant isometry $f\colon V\to W$ to the
restriction of the map on $\PLambda$ fixed points $f^{\PLambda}\colon
V^{\PLambda}\to W^{\PLambda}$ and maps
$f^{\perp}\cap W^{\PLambda}$ to $(f|_{V^{\PLambda}})^{\perp}\subset
W^{\PLambda}$ by the identity.
The based $\PGamma/\PLambda$-space enriched left Kan extension along
$\phi$ is then 
a based $\PGamma/\PLambda$-space enriched functor from
$\PGamma/\PLambda$-equivariant $\aJ_{\PGamma}^{\PLambda}$-spaces
to $\PGamma/\PLambda$-spectra.

The point-set multiplicative geometric fixed point functor
$\Phi^{\PLambda}$ from $\PGamma$-spectra to
$\PGamma/\PLambda$-spectra is the composite of
$\Fix^{\PLambda}$ from $\PGamma$-spectra to
$\PGamma/\PLambda$-equivariant $\aJ_{\PGamma}^{\PLambda}$-spaces and 
based $\PGamma/\PLambda$-space enriched left Kan extension along
$\phi$.

Returning to the context of Theorems~\ref{thm:bundlegeofix},
\ref{thm:bundleiterphi}, and \ref{thm:derphi},
the functor $\Phi^{\Lambda/\KG}$ from $\Gamma$-spectra to
$\Gamma/\Lambda$-spectra is the instance of
this functor for $\Lambda/\KG\lhd \Gamma/\Lambda$ combined with the canonical
isomorphism of compact Lie groups $\Gamma/\Lambda\iso
(\Gamma/\KG)/(\Lambda/\KG)$.  As in Section~\ref{sec:repnote}, we
switch between orthogonal $\Gamma/\KG$-representations and
$\KG$-trivial orthogonal $\Gamma$-representations without notation and
comment. We are now ready for the proof of Theorem~\ref{thm:bundlegeofix}.

\begin{proof}[Proof of Theorem~\ref{thm:bundlegeofix}]
The left Kan extension along $\phi$ is left adjoint to the functor
$u_{\phi}$ given by restriction
along the functor $\phi$, 
\[
u_{\phi}X=X\circ \phi.
\]
In this context, $u_{\phi}$ is 
a based $\Gamma/\Lambda$-space enriched functor from
$\Gamma/\Lambda$-spectra to 
$\Gamma/\Lambda$-equivariant $\aJ_{\Gamma/\KG}^{\Lambda/\KG}$-spaces.
Thus, to construct a map of $\Gamma/\Lambda$-spectra
\[
\Phi^{\Lambda/\KG}J^{\Gamma/\KG}\eta\to
J^{\Gamma/\Lambda}(\eta(\Lambda|\KG))
\]
it suffices to construct a map of $\Gamma/\Lambda$-equivariant $\aJ_{\Gamma/\KG}^{\Lambda/\KG}$-spaces
\[
\Fix^{\Lambda/\KG}J^{\Gamma/\KG}\eta\to
(J^{\Gamma/\Lambda}\eta(\Lambda|\KG))\circ \phi.
\]
For $W$ a $\KG$-trivial orthogonal $\Gamma$-representation, we need to construct a
$\Gamma/\Lambda$-equivariant map
\begin{equation}\label{eq:Fixmap}
(\aJ(\eta,W)/Q)^{\Lambda}\to \aJ(\eta(\Lambda|\KG),W^{\Lambda})/Q
\end{equation}
natural in $W$ in $\aJ_{\Gamma/\KG}^{\Lambda/\KG}$.  

Before constructing this map, we study $(\aJ(\eta,W)/Q)^{\Lambda}$ further.
From the lemma, $(\aJ(\eta,W)/Q)^{\Lambda}$ is the Thom space of
the vector bundle 
\[
\biggl(\coprod_{\sigma \colon \Lambda \to Q}
(\Im^{\perp}\aI(\eta,W))^{\Lambda,\sigma}\to \aI(\eta,W)^{\Lambda,\sigma}\biggr)/Q.
\]
We note that on base spaces
\[
\aI(\eta,W)^{\Lambda,\sigma}
=
\aI(\eta|_{B^{\Lambda,\sigma}},W)^{\Lambda,\sigma}
=
\aI(\eta|_{B^{\Lambda,\sigma}\{\KG\}},W)^{\Lambda,\sigma}
\]
where $\eta \colon E\to B$ and $B^{\Lambda,\sigma}\{\KG\}$ is as in
Definition~\ref{defn:philambdatwo}.  This is because for $b\in B$, a
$\sigma$-twisted $\Lambda$ fixed point, the 
$\sigma$-twisted $\Lambda$ action on $E_{b}$ does not admit any
equivariant isometries into a $K$-trivial representation unless $b\in
B^{\Lambda,\sigma}\{K\}$.  Thus, $(\aJ(\eta,W)/Q)^{\Lambda}$ is the Thom space of
the vector bundle
\[
\biggl(\coprod_{\sigma \colon \Lambda \to Q}
(\Im^{\perp}\aI(\eta|_{B^{\Lambda,\sigma}\{\KG\}},W))^{\Lambda,\sigma}
\to \aI(\eta|_{B^{\Lambda,\sigma}\{\KG\}},W))^{\Lambda,\sigma}\biggr)/Q.
\]

We construct~\eqref{eq:Fixmap} as a variant of
$\phi$ as follows. 
We note that in the vector bundle displayed above, on each fiber, the vector
space is a subspace of $W^{\Lambda}$; 
specifically, over a point $b$, $f\colon E_{b}\to W$, it is the orthogonal complement of $f(E_{b}^{\Lambda,\sigma})$ in $W^{\Lambda}$.  We define the
map~\eqref{eq:Fixmap} to be the induced map on Thom spaces that takes
$f\colon E_{b}\to W$ to $f|_{E_{b}^{\Lambda,\sigma}}$ and is the
identity on the fiber vector spaces (viewed as subspaces of
$W^{\Lambda}$). The resulting map is clearly 
natural in $W$ in $\aJ_{\Gamma/\KG}^{\Lambda/\KG}$. To see that it is
$\Gamma/\Lambda$-equivariant, it suffices to check that it is
$\Gamma$-equivariant, and this follows from the formula
in~\eqref{eq:Gammaaction} for the $\Gamma$ action on
$\eta(\Lambda)$: acting by $\gamma \in \Gamma$ on a
$\sigma$-twisted fixed point $f\colon E_{b}\to W$ of $\aI(\eta,W)$
takes it to the $\sigma^{\gamma}$-twisted fixed point $\gamma \cdot
f=\gamma \circ f\circ \gamma^{-1}$. 

Having constructed the comparison map of $\Gamma/\Lambda$-spectra 
\[
\Phi^{\Lambda/\KG}J^{\Gamma/\KG}\eta\to
J^{\Gamma/\Lambda}(\eta(\Lambda|\KG)),
\]
it suffices to show that it
is an isomorphism non-equivariantly.  Consider the categories of
non-equivariant $\aJ_{\Gamma/\KG}^{\Lambda/\KG}$-spaces and
$\aJ_{\Gamma/\Lambda}$-spaces defined to be categories of based space
enriched functors from $\aJ_{\Gamma/\KG}^{\Lambda/\KG}$ or
$\aJ_{\Gamma/\Lambda}$ to based spaces, respectively.  (The category
of non-equivariant $\aJ_{\Gamma/\Lambda}$-spaces is equivalent to the
category of non-equivariant spectra by the functor that
restricts to the full subcategory of trivial representations.)  The
functor $\phi$ induces a restriction functor from non-equivariant $\aJ_{\Gamma/\Lambda}$-spaces to non-equivariant $\aJ_{\Gamma/\KG}^{\Lambda/\KG}$-spaces which has a left adjoint given by left
Kan extension.  For formal reasons, forgetting equivariance commutes
with left Kan extension.  Thus, forgetting
$\Gamma/\Lambda$-equivariance, we can study
$\Phi^{\Lambda/\KG}J^{\Gamma/\KG}\eta$ by applying the non-equivariant left Kan
extension along $\phi$ to the underlying non-equivariant $\aJ_{\Gamma/\KG}^{\Lambda/\KG}$-space of
$\Fix^{\Lambda/\KG}J^{\Gamma/\KG}\eta$.
We denote the left Kan extension along $\phi$ as $P_{\phi}$.

Non-equivariantly, $\Fix^{\Lambda/\KG}J^{\Gamma/\KG}\eta$ and
$J^{\Gamma/\Lambda}(\eta(\Lambda|\KG))$ decompose as the coproduct over $\sigma
\colon \Lambda \to Q$ of sub-$\aJ_{\Gamma/\KG}^{\Lambda/\KG}$-spaces and
sub-$\aJ_{\Gamma/\Lambda}$-spaces (respectively), and the comparison
map preserves this decomposition.  As a left adjoint, $P_{\phi}$
preserves coproducts, so it suffices to study each summand separately,
and we now fix $\sigma \colon \Lambda \to Q$.  Let
$B'=B^{\Lambda,\sigma}\{\KG\}$, and let $\eta'\colon E'\to B'$ be the vector
bundle obtained from $\eta$ by restriction.  When we regard $\eta'$ as
a $\Lambda$-equivariant vector bundle with the $\sigma$-twisted
$\Lambda$ action, the $\Lambda$ action is $\KG$-trivial, making
$\eta'$ a $\Lambda/\KG$-equivariant vector bundle; moreover, 
the vector bundle $\eta^{\Lambda,\sigma}$ is precisely its
$\Lambda/\KG$ fixed points.  For
every $b\in B'$, there exists an open set $U$ around $b$ where the
restriction $\eta'|_{U}$ is $\Lambda/\KG$-equivariantly trivializable,
i.e., $\eta'|_{U}$ is isomorphic to $U\times E_{b}\to U$ as a
$\Lambda/\KG$-equivariant vector bundle.  Let $\aU$ be the partially
ordered set of open subsets $U$ of $B'$ on which $\eta'$ admits a
$\Lambda/\KG$-equivariant trivialization.  Both
$\Fix^{\Lambda/\KG}J^{\Gamma/\KG}\eta$ and
$J^{\Gamma/\Lambda}(\eta(\Lambda|\KG))$ decompose 
as a colimit over $\aU$, compatibly with the comparison map, and so it
suffices to show that for $U\in \aU$, the restriction of the
comparison map
\begin{multline*}
P_{\phi}(\aJ(U\times V\sto U,-)^{\Lambda/\KG})\iso
P_{\phi}(\aJ(\eta'|_{U},-)^{\Lambda/\KG})\\
\to \aJ(\eta^{\Lambda,\sigma}\{\KG\}|_{U},-)\iso \aJ(U\times
V^{\Lambda/\KG}\sto U,-)
\end{multline*}
is an isomorphism of non-equivariant $\aJ_{\Gamma/\Lambda}$-spaces,
where $V=E_{b}$ denotes the fiber $\Lambda/\KG$-representation of
$\eta'$ at an arbitrarily chosen point $b\in U$.  Since $P_{\phi}$
commutes with smash product with spaces, the map in the previous
display is easily identified as the canonical map
\[
U_{+}\sma P_{\phi}(\aJ(V,-)^{\Lambda/\KG})
\to U_{+}\sma \aJ(V^{\Lambda/\KG},-),
\]
which is an isomorphism by~\cite[V.4.5]{MM} (as extended
by~\cite[3.8.12]{Kro-SOmega}). 
\end{proof}

Theorem~\ref{thm:bundleiterphi} is now an easy consequence.

\begin{proof}[Proof of Theorem~\ref{thm:bundleiterphi}]
By Theorem~\ref{thm:bundlegeofix}
\begin{align*}
\Phi^{M/\Lambda}\Phi^{\Lambda/\KG}J^{\Gamma/\KG}\eta&\iso
\Phi^{M/\Lambda}(J^{\Gamma/\Lambda}(\eta(\Lambda|\KG)))
\iso J^{\Gamma/M}((\eta(\Lambda|\KG))(M|\Lambda))\\
\Phi^{M/\KG}J^{\Gamma/\KG}\eta&\iso
J^{\Gamma/M}(\eta(M|\KG)).
\end{align*}
In the notation of Definition~\ref{defn:philambdatwo}, we have
\begin{align*}
(\eta(\Lambda|\KG))(M|\Lambda)
&=\coprod_{\sigma \colon M\to Q} (\eta(\Lambda|\KG))^{M,\sigma}\{\Lambda\}\\
&=\coprod_{\sigma \colon M\to Q} \biggl(
\biggl(
  \coprod_{\tau\colon \Lambda \to Q} \eta^{\Lambda,\tau}\{\KG\}
\biggr)\mathstrut^{M,\sigma}\{\Lambda\}\biggr).
\end{align*}
Consider $\sigma \colon M\to Q$, $\tau \colon \Lambda \to Q$, and an
element $b\in B^{\Lambda,\tau}\{\KG\}$, which when viewed as an element of
the base of $\eta(\Lambda|\KG)$ (in the $\tau$ summand) is fixed by the
$\sigma$-twisted $M$ action.  In particular, $b$ is then also in
$B^{M,\sigma}$ and the
$\sigma|_{\Lambda}$-twisted $\Lambda$ action on the total space of
$\eta(\Lambda|\KG)$ restricts to an action on the fiber
$E^{\Lambda,\tau}_{b}$.  Since $\eta(\Lambda|\KG)$ is assumed to be
$Q$-faithful, the $\sigma|_{\Lambda}$-twisted $\Lambda$ action on
$E^{\Lambda,\tau}_{b}$ is trivial if and only if
$\sigma|_{\Lambda}=\tau$. Moreover, $\sigma|_{\Lambda}$ is a fixed
point for the $\sigma$-twisted $M$ action on the set of homomorphisms
$\Lambda \to Q$, and 
\[
(\eta^{\Lambda,\sigma|_{\Lambda}}\{K\})^{M,\sigma}=
\eta^{M,\sigma}\{K\}.
\]
It follows that the inclusion of $\sigma|_{\Lambda}$ in
the set of all homomorphisms $\tau \colon \Lambda \to Q$ induces an
isomorphism of $(\Gamma,Q)$ vector bundles
\begin{multline*}
\eta(M)=\coprod_{\sigma \colon M\to Q}\eta^{M,\sigma}\{\KG\}
= \coprod_{\sigma \colon M\to Q}(\eta^{\Lambda,\sigma|_{\Lambda}}\{\KG\})^{M,\sigma}\\
\to
\coprod_{\sigma \colon M\to Q} 
\biggl(
  \coprod_{\tau\colon \Lambda \to Q} \eta^{\Lambda,\tau}\{\KG\}
\biggr)\mathstrut^{M,\sigma}\{\Lambda\}=
(\eta(\Lambda|\KG))(M|\Lambda).
\end{multline*}
This induces an isomorphism $\Phi^{M/\KG}J^{\Gamma/\KG}\eta \to
\Phi^{M/\Lambda}\Phi^{\Lambda/\KG}J^{\Gamma/\KG}\eta$
and it is straightforward to
identify this map as the canonical one.
\end{proof}

For the proof of Theorem~\ref{thm:derphi}, recall the geometric
homotopy groups: for a compact Lie group $G$, closed subgroup $H<G$,
and $G$-spectrum $X$, let 
\[
\pi_{q}^{\Phi H}(X)=\begin{cases}
\colim_{V}\pi_{q}(\Omega^{V^{H}}(X(V)^{H}))&q\geq 0\\
\colim_{V}\pi_{0}(\Omega^{V^{H}}(X(\bR^{-q}\oplus V)^{H}))&q<0\\
\end{cases}
\]
where the colimit ranges over the finite dimensional subspaces of an
$H$-universe, ordered by inclusion; this depends on the universe but
any other choice of universe with the same isomorphism classes of
representations gives
canonically isomorphic groups. By convention, except when we
explicitly mention a universe, we always understand this colimit to be
over a complete universe.

The functor $\pi_{*}^{\Phi H}$ is denoted $\rho^{K}_{*}$
in \cite[V.4.8.(iii)]{MM} for $K$ the trivial group in this case; then
\cite[V.4.12--16]{MM} gives a natural isomorphism from
$\pi_{q}^{\Phi H}X$ to the homotopy groups of the underlying
non-equivariant spectrum of the derived geometric fixed point
spectrum,  
\begin{equation}\label{eq:derphi}
\pi^{\Phi H}_{q}X\iso \pi_{q}(L\Phi^{H}X).
\end{equation}
According to~\cite[XVI.6.4]{May-Alaska}, these homotopy groups detect
weak equivalences and family equivalences of $G$-spectra; the
following proposition holds for any $G$-universe provided we use the
same universe to define weak equivalences and (its restriction to
subgroups) to construct $\pi_{*}^{\Phi H}$: 

\begin{prop}\label{prop:pivsphi}
Let $G$ be a compact Lie group and let $\aF$ be a family of
subgroups.  A map of $G$-spectra $X\to Y$ is an $\aF$-equivalence if
and only if it induces an 
isomorphism $\pi^{\Phi H}_{*}X\to \pi^{\Phi H}_{*}Y$ for all
$H\in\aF$.
\end{prop}

For $N$ a normal subgroup of $G$
contained in $H$, we have a natural map 
\[
\pi_{q}^{\Phi H}(X)\to\pi_{q}^{\Phi (H/N)}(\Phi^{N}X).
\]
induced by the map $\Fix^{N}X\to \Phi^{N}X\circ \phi$ as follows.  For
$q\geq 0$, and $V$ an orthogonal $G$-representation, we use the map
\begin{multline*}
\pi_{q}(\Omega^{V^{H}}(X(V)^{H}))\iso
\pi_{q}(\Omega^{(V^{N})^{H/N}}((X(V)^{N})^{H/N}))\\\to 
\pi_{q}(\Omega^{(V^{N})^{H/N}}((\Phi^{N}X(V^{N}))^{H/N}))
\to \pi_{q}^{\Phi (H/N)}(\Phi^{N}X)
\end{multline*}
(where we take the colimit implicit in $\pi_{q}^{\Phi(H/N)}$ over the
inclusions in the complete $H/N$-universe obtained as the $N$ fixed
points of the chosen complete $H$-universe).  Similar formulas apply
when $q<0$.  Considering the case when $X$ is cofibrant in the
standard model structure, the point-set geometric fixed point
functors model the derived geometric fixed point functors, and
(using~\cite[V.4.13]{MM}) we see that the diagram of natural transformations
\[
\xymatrix{%
\pi_{q}^{\Phi H}(X)\ar[rr]\ar[d]_{\iso}
&&\pi_{q}^{\Phi(H/N)}(\Phi^{N}X)\ar[d]^{\iso}\\
\pi_{q}(L\Phi^{H}(X))\ar[r]_-{\iso}
&\pi_{q}(L\Phi^{H/N}(L\Phi^{N}X))\ar[r]
&\pi_{q}(L\Phi^{H/N}(\Phi^{N}X))
}
\]
commutes.  In the context of Notation~\ref{notn:KGL}, taking
$G=\Gamma/\KG$, the previous
proposition now gives the following conclusion.

\begin{prop}\label{prop:derphi}
Let $X$ be a $\Gamma/\KG$-spectrum.  The universal map
$L\Phi^{\Lambda/\KG}X\to \Phi^{\Lambda/\KG}X$ in the
$\Gamma/\Lambda$-equivariant stable category is an isomorphism if and
only if the natural maps $\pi_{*}^{\Phi(M/\KG)}X\to
\pi_{*}^{\Phi(M/\Lambda)}(\Phi^{\Lambda/K}X)$ are isomorphisms for all
closed subgroups $M<\Gamma$ containing $\Lambda$.
\end{prop}

Our strategy for the proof of Theorem~\ref{thm:derphi} is to apply the
previous proposition and calculate the maps.

\begin{proof}[Proof of Theorem~\ref{thm:derphi}]
By Theorem~\ref{thm:bundlegeofix}, it suffices to show that the map 
\[
\pi_{*}^{\Phi(M/\KG)}J^{\Gamma/\KG}\eta\to
\pi_{*}^{\Phi(M/\Lambda)}(J^{\Gamma/\Lambda}\eta (\Lambda|\KG))
\]
is an isomorphism for all $M<\Gamma$ containing $\Lambda$.  We now
fix one such $M$.
For $V$ an orthogonal $\Gamma/\KG$-representation, by
Lemma~\ref{lem:bundlegeofix}, we have 
\[
(J^{\Gamma /\KG}\eta (V))^{M}
\iso\biggl(\bigvee_{\sigma \colon M\to Q}\aJ(\eta,V)^{M,\sigma}\biggr)/Q
=\biggl(\bigvee_{\sigma \colon M\to Q}\aJ(\eta|_{B^{M,\sigma}\{\KG\}},V)^{M,\sigma}\biggr)/Q
\]
since $V$ is $\KG$-trivial and $\eta$ is $Q$-faithful.  Likewise, we have
\[
((J^{\Gamma /\Lambda}\eta(\Lambda|\KG))(V^{\Lambda}))^{M}
\iso
\biggl(\bigvee_{\sigma \colon M\to Q}
\aJ(\eta^{\Lambda,\sigma|_{\Lambda}}|_{B^{M,\sigma}\{\KG\}},
  V^{\Lambda})^{M,\sigma}\biggr)/Q;
\]
as in the proof of Theorem~\ref{thm:bundleiterphi}, the summands
$\tau\colon \Lambda \to Q$ for $\tau\neq \sigma|_{\Lambda}$ disappear
since $V^{\Lambda}$ is $\Lambda$-trivial and $\eta(\Lambda|\KG)$ is
$Q$-faithful.  The map we are calculating is induced by the map
\[
\aJ(\eta|_{B^{M,\sigma}\{\KG\}},V)^{M,\sigma}
\to 
\aJ(\eta^{\Lambda,\sigma|_{\Lambda}}|_{B^{M,\sigma}\{\KG\}},
  V^{\Lambda})^{M,\sigma}
\]
which is the induced map on Thom spaces from
\[
\Im^{\perp}\aI(\eta|_{B^{M,\sigma}\{\KG\}},V)^{M,\sigma}
\to 
\Im^{\perp}\aI(\eta^{\Lambda,\sigma|_{\Lambda}}|_{B^{M,\sigma}\{\KG\}},
  V^{\Lambda})^{M,\sigma}
\]
which takes a $\sigma$-twisted $M$-equivariant isometry $E_{b}\to V$
to its $\sigma|_{\Lambda}$-twisted $\Lambda$ fixed points
$E_{b}^{\Lambda,\sigma|_{\Lambda}}\to V^{\Lambda}$ and is the identity
on the complementary ($\sigma$-twisted $M$ fixed) vector space.

To calculate the map on homotopy groups of colimits, consider a
complete $M/\KG$-universe $U=U_{0}\oplus U_{1}$, where
$U_{0}=\bR^{\infty}$ and $U_{1}^{M}=0$.  Then for $X$ a $\Gamma/\KG$-spectrum we can calculate $\pi_{q}^{\Phi M/\KG}X$ as
\[
\colim_{n\gg -q}\colim_{W<U_{1}}\pi_{n+q}(X(\bR^{n}\oplus W)^{M/\KG})
\]
for any $q\in \bZ$.  If we extend the definition of $\aJ(\eta,-)$ to
infinite dimensional inner product spaces by taking the colimit over
finite dimensional subspaces, then we get
\begin{align*}
\pi^{\Phi M/\KG}(J^{\Gamma/\KG}\eta)&\iso
\colim_{n} \pi_{n+q}\biggl(
\biggl(\bigvee_{\sigma \colon M\to Q}\aJ(\eta|_{B^{M,\sigma}\{\KG\}},\bR^{n}\oplus U_{1})^{M,\sigma}\biggr)/Q\biggr)\\
\pi^{\Phi M/\Lambda}(J^{\Gamma/\KG}\eta(\Lambda|\KG))&\iso
\colim_{n} \pi_{n+q}\biggl(\biggl(\bigvee_{\sigma \colon M\to Q}
\aJ(\eta^{\Lambda,\sigma|_{\Lambda}}|_{B^{M,\sigma}\{\KG\}},
  \bR^{n}\oplus U_{1}^{\Lambda})^{M,\sigma}\biggr)/Q\biggr).
\end{align*}
Thus, it suffices to show that the map
\[
\bigvee_{\sigma\colon M\to Q}\aJ(\eta|_{B^{M,\sigma}\{\KG\}},\bR^{n}\oplus U_{1})^{M,\sigma}
\to
\bigvee_{\sigma \colon M\to Q}\aJ(\eta^{\Lambda,\sigma|_{\Lambda}}|_{B^{M,\sigma}\{\KG\}},
  \bR^{n}\oplus U_{1}^{\Lambda})^{M,\sigma}
\]
is a $Q$-equivariant weak equivalence.  Both sides are
$Q$-free because the bundles are $Q$-faithful, so it suffices to show
for each $\sigma \colon M\to Q$, the map 
\[
\aJ(\eta|_{B^{M,\sigma}\{\KG\}},\bR^{n}\oplus U_{1})^{M,\sigma}
\to
\aJ(\eta^{\Lambda,\sigma|_{\Lambda}}|_{B^{M,\sigma}\{\KG\}},
  \bR^{n}\oplus U_{1}^{\Lambda})^{M,\sigma}
\]
is a non-equivariant homotopy equivalence.  

We understand the $M$ action on the vector bundle
$\eta|_{B^{M,\sigma}\{K\}}$ to be the $\sigma$-twisted $M$ action.
This is a fiberwise $\KG$ fixed $M$ action and the usual integration
argument decomposes $\eta|_{B^{M,\sigma}\{K\}}$ into an internal
Whitney sum 
\[
\eta|_{B^{M,\sigma}\{K\}}=\theta_{0}\oplus \theta_{1}
\]
where $\theta_{0}=(\eta|_{B^{M,\sigma}\{K\}})^{M/\KG}$ and
$\theta_{1}=\theta_{0}^{\perp}$ satisfies $\theta_{1}^{M/\KG}$ is zero
dimensional.  Because $\Lambda$ is normal in $M$, $\theta_{1}$
decomposes further as  
\[
\theta = \theta_{1}^{\Lambda/\KG}\oplus \theta'_{1}.
\]
We then have 
\begin{align*}
\aJ(\eta|_{B^{M,\sigma}\{\KG\}},\bR^{n}\oplus U_{1})^{M,\sigma}
&\iso \aJ(\theta_{0},\bR^{n})\times \aI(\theta_{1},U_{1})^{M/\KG}\\
\aJ(\eta^{\Lambda,\sigma|_{\Lambda}}|_{B^{M,\sigma}\{\KG\}},
  \bR^{n}\oplus U_{1}^{\Lambda})^{M,\sigma}
&\iso \aJ(\theta_{0},\bR^{n})\times
\aI(\theta_{1}^{\Lambda/\KG},U_{1}^{\Lambda/\KG})^{M/\KG};
\end{align*}
the map preserves this product decomposition, is the identity on the
first factor, and is a homotopy equivalence on the second factors
because both are contractible.
\end{proof}

\section{Reduction of Theorems~\sref$thm:sm$ and~\sref$thm:bundleflat$}
\label{sec:redsm}

This section reduces Theorems~\ref{thm:sm}
and~\ref{thm:bundleflat} to a simpler but more technical result stated
as Theorem~\ref{thm:fixsmflat} below and proved in the next section.  As in the
statements, we follow the notational convention of
Notation~\ref{notn:KGL}: $\Gamma$ is a compact Lie group with a closed
normal subgroups $\KG,\Lambda$ satisfying $\KG<\Lambda$.

Theorem~\ref{thm:fixsmflat} is stated in terms of a variant of the
functor $\Fix^{\Gamma}$ that we write as $\Fix^{\Gamma}_{\KG}$.
The domain of $\Fix^{\Gamma}_{\KG}$ is the category of
$\Gamma$-spectra indexed on $\KG$-trivial
orthogonal $\Gamma$-representations, or equivalently, orthogonal $\Gamma/\KG$-representations.  It is equivalent to the category of
$\Gamma$-equivariant $\aJ_{\Gamma/\KG}$-spaces, and we mostly refer to
it this way, except in some headline statements, where we write 
$\Gamma$-spectra indexed on $\KG$-trivial
orthogonal $\Gamma$-representations for historical and philosophical
reasons. Starting with a
$\Gamma$-equivariant $\aJ_{\Gamma/\KG}$-space and taking $\Gamma$ fixed points
levelwise gives a (non-equivariant)
$\aJ^{\Gamma/\KG}_{\Gamma/\KG}$-space. (We emphasize that
$\aJ^{\Gamma/\KG}_{\Gamma/\KG}$-spaces here are always understood to
be non-equivariant.)

\begin{defn}\label{defn:fixgk}
Let $\Fix^{\Gamma}_{\KG}$ be the functor from $\Gamma$-equivariant $\aJ_{\Gamma/\KG}$-spaces to 
$\aJ^{\Gamma/\KG}_{\Gamma/\KG}$-spaces defined by
\[
(\Fix^{\Gamma}_{\KG}X)(W)=(X(W))^{\Gamma}.
\]
\end{defn}

We often want to apply $\Fix^{\Gamma}_{\KG}$ to a
$\Gamma$-spectrum $Y$; to avoid confusion, we write
$\imath Y$ for the underlying $\Gamma$-equivariant $\aJ_{\Gamma/\KG}$-space obtained by restricting $Y$ to the
$\KG$-trivial orthogonal $\Gamma$-representations. In parallel, for
a $\Gamma/\KG$-spectrum $X$, we get a
$\Gamma$-equivariant $\aJ_{\Gamma/\KG}$-space $\jmath X$ obtained
by regarding $\Gamma/\KG$-spaces as $\Gamma$-spaces with trivial
$\KG$ action. 

The category of $\aJ_{\Gamma/\KG}^{\Gamma/\KG}$-spaces has a smash product defined by
Day convolution for the smash product in $\aJ^{\Gamma/\KG}_{\Gamma/\KG}$ induced by
direct sum; see~\cite[\S 21]{MMSS}.  The functor
$\Fix^{\Gamma}_{\KG}$ is lax symmetric monoidal for this smash
product, with structure map
\[
\Fix^{\Gamma}_{\KG}X\sma \Fix^{\Gamma}_{\KG}Y\to
\Fix^{\Gamma}_{\KG}(X\sma Y)
\]
induced by the natural transformation 
\begin{multline*}
(\Fix^{\Gamma}_{\KG}X)(V)\sma (\Fix^{\Gamma}_{\KG}Y)(W)
=(X(V))^{\Gamma}\sma (Y(W))^{\Gamma}\\
\to ((X\sma Y)(V\oplus W))^{\Gamma}
=(\Fix^{\Gamma}_{\KG}(X\sma Y))(V\oplus W)
\end{multline*}
where the middle map is the induced map on $\Gamma$ fixed points of the universal
map $X(V)\sma Y(W)\to (X\sma Y)(V\oplus W)$ for the smash product of
$\Gamma$-spectra indexed on $\KG$-trivial orthogonal $\Gamma$-representations.

For some of the work, we will need the homotopy theory of
$\aJ^{\Gamma/\KG}_{\Gamma/\KG}$-spaces. We understand a
\term{weak equivalence} as a 
$\pi_{*}$-isomorphism, that is an isomorphism on homotopy groups: for
a $\aJ^{\Gamma/\KG}_{\Gamma/\KG}$-space $X$, we define homotopy groups by
\[
\pi_{q}X:=\begin{cases}
\colim_{V}\pi_{q}\Omega^{V^{\Gamma/\KG}}X(V)&q\geq 0\\
\colim_{V}\pi_{0}\Omega^{V^{\Gamma/\KG}}X(\bR^{-q}\oplus V)&q< 0\\
\end{cases}
\]
where the colimit ranges over the finite dimensional subspaces of a
complete $\Gamma/\KG$-universe (cf.~\cite[V.4.8.i]{MM}).  These are the weak equivalences in a model
structure whose cofibrant objects are the retracts of cell complexes built by
cells of the form $\aJ^{\Gamma/\KG}_{\Gamma/\KG}(V,-)\sma (D^{n},\partial D^{n})_{+}$. 

In the following theorem, $F=\aF(1\times Q^{\op})$ denotes the
family of graph subgroups, and $\EF$ denotes the universal space for this
family (sometimes elsewhere written $E_{\Gamma}Q^{op}$). The graph
subgroups are precisely those subgroups whose 
intersection with $1\times Q^{\op}$ is the trivial subgroup.  The
universal space $\EF$ is a $(\Gamma \times Q^{\op})$-CW complex
characterized up to equivariant homotopy equivalence by the condition
that $\EF^{H}$ is 
contractible when $H$ is a graph subgroup of $\Gamma \times Q^{\op}$
and is empty when $H$ is not.

\begin{thm}\label{thm:fixsmflat}
Let $\eta$ be the $(\Gamma,Q)$-vector bundle 
\[
(\Gamma \times Q^{\op})\times_{M}V\times Y\to (\Gamma\times Q^{\op})/M\times Y
\]
for $M$ a closed subgroup of $\Gamma \times Q^{\op}$, $V$ an
orthogonal $M$-representation, and $Y$ a $(\Gamma \times Q^{\op})$-equivariant cell
complex.  Assume that $\eta$ is $Q$-faithful.
Then for any $\Gamma$-spectrum $X$ indexed on
$\KG$-trivial orthogonal $\Gamma$-rep\-resen\-tations:
\begin{enumerate}
\item The
canonical map
\[
\Fix^{\Gamma}_{\KG}(\imath J\eta)\sma
\Fix^{\Gamma}_{\KG}(X)
\to
\Fix^{\Gamma}_{\KG}(\imath J\eta\sma X)
\]
is an isomorphism.
\item The map $\EF\to *$ induces a weak equivalence 
\[
\Fix^{\Gamma}_{\KG}(\imath J(\EF\times \eta)\sma X)\to
\Fix^{\Gamma}_{\KG}(\imath J\eta\sma X).
\]
\end{enumerate}
\end{thm}

We now deduce Theorems~\ref{thm:sm} and~\ref{thm:bundleflat}
from Theorem~\ref{thm:fixsmflat}.

\begin{proof}[Proof of Theorem~\ref{thm:sm} from Theorem~\ref{thm:fixsmflat}]
Let $\eta$ and $X$ be as in the
statement of Theorem~\ref{thm:sm}.  Since $\Phi^{\Lambda/\KG}$ is the composite
of $\Fix^{\Lambda/\KG}$ with left Kan extension, it suffices to show that the
map
\[
\Fix^{\Lambda/\KG}(J^{\Gamma/\KG}\eta)\sma\Fix^{\Lambda/\KG}X
\to \Fix^{\Lambda/\KG}(J^{\Gamma/\KG}\eta\sma X)
\]
is an isomorphism of non-equivariant $\aJ_{\Gamma/\KG}^{\Lambda/\KG}$-spaces.
Non-equivariantly, the restriction of equivariance functor induces an
equivalence of (non-equivariant) topological categories from
$\aJ_{\Gamma/\KG}^{\Lambda/\KG}$ to a full subcategory of
$\aJ_{\Lambda/\KG}^{\Lambda/\KG}$, and so it suffices 
to show that the map above is an isomorphism of non-equivariant
$\aJ_{\Lambda/\KG}^{\Lambda/\KG}$-spaces for any $\Lambda$-equivariant
$\aJ_{\Lambda/\KG}$-space $X$.  Thus, by a change of
notation, we can assume without loss of generality that $\Lambda =\Gamma$.
Then 
\begin{gather*}
\Fix^{\Lambda/\KG}(J^{\Gamma/\KG}\eta)\iso
\Fix^{\Gamma}_{\KG}(\imath J\eta),\quad
\Fix^{\Lambda/\KG}(X)\iso\Fix^{\Gamma}_{\KG}(\jmath X), \quad \text{and}\\
\Fix^{\Lambda/\KG}(J^{\Gamma/\KG}\eta\sma X)
\iso\Fix^{\Gamma}_{\KG}(\imath J\eta\sma \jmath X),
\end{gather*}
compatibly with the universal map, so it suffices to show that the map
\[
\Fix^{\Gamma}_{\KG}(\imath J\eta)\sma 
\Fix^{\Gamma}_{\KG}(\jmath X)\to
\Fix^{\Gamma}_{\KG}(\imath J\eta\sma\jmath X).
\]
is an isomorphism. Since both sides commute with
cofiber sequences and sequential colimits over Hurewicz cofibrations,
working a cell at a time, it suffices to consider the case when the
base of $\eta$ consists of a single cell or its boundary, i.e., when
$\eta$ is as in the statement of Theorem~\ref{thm:fixsmflat} for
$Y=D^{n}$ or $\partial D^{n}$.  Part~(i) of
Theorem~\ref{thm:fixsmflat} establishes this case.
\end{proof}

For Theorem~\ref{thm:bundleflat}, we need the following additional observation.

\begin{prop}\label{prop:bundlefixcw}
Let $\eta\colon E\to B$ be a $(\Gamma,Q)$ vector bundle such that
$B$ is a $(\Gamma \times Q^{\op})$-equivariant CW complex with
isotropy groups in $\aF(1\times Q^{\op})$.  Then
$\Fix^{\Gamma}_{\KG}(\imath J\eta)$ is a cellular
$\aJ^{\Gamma/\KG}_{\Gamma/\KG}$-space, i.e., is built from
cells of the form $\aJ(V,-)\sma (D^{n},\partial D^{n})_{+}$ for $V$ an
orthogonal $\Gamma/\KG$-representation.
\end{prop}

\begin{proof}
Working over the cells of $B$, it suffices to check the case when
$B=(\Gamma \times Q^{\op})/\tilde H$ for $\tilde H\in \aF(1\times
Q^{\op})$. Let $H<\Gamma$ be the image of $\tilde H$ under the
projection; then the projection $\tilde H\to H$ is an isomorphism and
there exists a unique homomorphism $\sigma \colon H\to Q$ such
that
\[
\tilde H=H_{\sigma}=\{(h,\sigma(h^{-1}))\mid h\in H\}.
\]
The description of $B$ as an orbit implicitly chooses a point with
isotropy subgroup $\tilde H$.
Let $V$ denote the fiber over that point; then $V$ has
the structure of an orthogonal $\tilde H$-representation. Because
$\tilde H\cap (1\times Q^{\op})$ is trivial, for any
orthogonal $\Gamma/\KG$-representation $W$, we have 
\begin{multline*}
(\Fix^{\Gamma}_{\KG}(\imath J\eta))(W)=(\aJ(\eta,W)/Q)^{\Gamma}
\iso (((\Gamma \times Q^{\op})_{+}
   \sma_{\tilde H}\aJ(V,W))/Q)^{\Gamma}\\
\iso (\Gamma_{+} \sma_{H}\aJ(V,W))^{\Gamma}
\end{multline*}
where in the last isomorphism $H$ acts on $V$ via the isomorphism
$\tilde H\to H$.  We see that this based space is trivial unless $H=\Gamma$, in
which case we get $\aJ(V,W)^{\Gamma}$.  Since the $\KG$ action on
$W$ is trivial, there are no $\Gamma$-equivariant isometries $V\to W$
unless the $\KG$ action on $V$ is also trivial, in which case 
\[
\aJ(V,W)^{\Gamma}=\aJ^{\Gamma/\KG}_{\Gamma/\KG}(V,W).
\]
Thus, $\Fix^{\Gamma}_{\KG}(\imath J\eta)$ is either $*$ or
$\aJ^{\Gamma/\KG}_{\Gamma/\KG}(V,-)$ for an orthogonal
$\Gamma/\KG$-representation $V$. 
\end{proof}

The argument of~\cite[12.3]{MMSS} shows that cellular
$\aJ^{\Gamma/\KG}_{\Gamma/\KG}$-spaces are flat for the smash
product.  To make this argument work, it suffices to see that for
any $\KG$-trivial $\Gamma$-representation $V$, the map
\[
S^{(V\oplus V)^{\Gamma}}\to 
\aJ^{\Gamma/\KG}_{\Gamma/\KG}(V,V\oplus V)\sma
S^{V^{\Gamma}}
\]
(induced by the inclusion of $V$ as the first summand) is a $2(\dim
V^{\Gamma})$-equivalence.  This is easy to see by decomposing $V$ into
irreducible $\Gamma$-representations: each non-trivial irreducible
representation contributes a smash factor of $O(2n)/O(n)_{+}$,
$U(2n)/U(n)_{+}$, or $Sp(2n)/Sp(n)_{+}$ (where $n$ is the number of
summands of that type), and the trivial representation together with
$S^{V^{\Gamma}}$ contribute a smash factor of 
\[
(O(2n)_{+}\sma_{O(n)}S^{n})\sma S^{n}\iso O(2n)/O(n)_{+}\sma S^{2n}
\]
(where $n=\dim(V^{\Gamma})$).
We can now deduce Theorem~\ref{thm:bundleflat}.  In fact, we prove the
following more general statement.

\begin{thm}\label{thm:varbundleflat}
Suppose $\eta$ is a $Q$-faithful 
$(\Gamma,Q)$ vector bundle and assume the base $B$ admits the structure of a
$(\Gamma \times Q^{\op})$-equivariant cell complex.  Then
$\imath J\eta$ is flat for the smash product in
$\Gamma$-spectra indexed on $\KG$-trivial
orthogonal $\Gamma$-representations.
\end{thm}

Theorem~\ref{thm:bundleflat} is a consequence of the previous theorem
because of the following easy facts:
\begin{enumerate}
\item When $\eta$ is $\Gamma/\KG$-compatible, $\imath J\eta=\jmath
J^{\Gamma/\KG}\eta$.
\item When $X$ and $Y$ are $\Gamma/\KG$-spectra
$\jmath X\sma \jmath Y\iso \jmath(X\sma Y)$
\item When $X$ and $Y$ are $\Gamma/\KG$-spectra
a map $X\to Y$ is a weak equivalence of $\Gamma/\KG$-spectra if and only if applying $\jmath$ takes it to a weak
equivalence of $\Gamma$-spectra indexed on $\KG$-trivial
orthogonal $\Gamma$-representations.
\end{enumerate}

\begin{proof}[Proof of Theorem~\ref{thm:varbundleflat} from Theorem~\ref{thm:fixsmflat}]
Let $\eta$ be as in the
statement of Theorem~\ref{thm:varbundleflat}. To show that
$\imath J\eta$ is flat, by definition, we need to show that for any
weak equivalence $X\to X'$ of $\Gamma$-equivariant $\aJ_{\Gamma/\KG}$-spaces, the induced map 
\[
\imath J\eta\sma X\to 
\imath J\eta\sma X'
\]
is a weak equivalence. For this, by~\cite[XVI.6.4]{May-Alaska}
(see also~\cite[V.4.12--16]{MM} or Proposition~\ref{prop:pivsphi}), it
suffices to show that for every closed subgroup $H$ of $\Gamma$, the map
\[
\Fix^{H}_{H\cap \KG}(j^{*}_{H}\imath J\eta \sma j^{*}_{H} X)
\to \Fix^{H}_{H\cap \KG}(j^{*}_{H}\imath J\eta\sma j^{*}_{H} X')
\]
is a weak equivalence, where $j_{H}$ denotes the inclusion of $H$ in
$\Gamma$, and $j_{H}^{*}$ denotes the reduction of structure functor from
$\Gamma$-equivariant $\aJ_{\Gamma/\KG}$-spaces to $H$-equivariant $\aJ_{H/(H\cap \KG)}$-spaces.  Working one $H$ at a time, we can change
notation, replacing $\Gamma$ with $H$, and it suffices to show that the
induced map
\[
\Fix^{\Gamma}_{\KG}(\imath J\eta\sma X)
\to \Fix^{\Gamma}_{\KG}(\imath J\eta\sma X')
\]
is a weak equivalence (for any $X\to X'$ and $\eta$ as in the
statement of Theorem~\ref{thm:varbundleflat}).
Working a cell at a time over the base of $\eta$, it suffices to prove
the case when $\eta$ is as in the statement of
Theorem~\ref{thm:fixsmflat} for $Y=D^{n}$ or $\partial D^{n}$.  If $\EF$
is a universal $F$-space for $F=\aF(1\times Q^{\op})$ (as above), then
$\Fix^{\Gamma}_{\KG}(\imath J(\EF\times \eta))$ is a cellular
$\aJ^{\Gamma/\KG}_{\Gamma/\KG}$-space by the
Proposition~\ref{prop:bundlefixcw}, and hence the top map in the
commuting diagram 
\[
\xymatrix{%
\Fix^{\Gamma}_{\KG}(\imath J(\EF\times \eta))\sma
\Fix^{\Gamma}_{\KG}X
\ar[r]\ar[d]
&
\Fix^{\Gamma}_{\KG}(\imath J(\EF\times \eta))\sma
\Fix^{\Gamma}_{\KG}X'\ar[d]\\
\Fix^{\Gamma}_{\KG}(\imath J\eta)\sma
\Fix^{\Gamma}_{\KG}X
\ar[r]
&
\Fix^{\Gamma}_{\KG}(\imath J\eta)\sma
\Fix^{\Gamma}_{\KG}X'
}
\]
is a weak equivalence.  Applying part~(ii) of
Theorem~\ref{thm:fixsmflat}, we see that the vertical maps are weak
equivalences, and we conclude that the bottom horizontal map is a weak
equivalence. Part~(i) of Theorem~\ref{thm:fixsmflat} then finishes the
proof. 
\end{proof}

\section{Proof of Theorem~\sref$thm:fixsmflat$}
\label{sec:fixsmflat}
\def\RQ{H}

Let $\eta$, $M$, $V$, $X$, and $Y$ be as
in the statement of Theorem~\ref{thm:fixsmflat}.  For the proof of the
theorem, it suffices to treat the case when the projection
map $M\to \Gamma$ is onto, as otherwise all
$\aJ^{\Gamma/\KG}_{\Gamma/\KG}$-spaces mentioned in (i)
and~(ii) are trivial.  

Let $\RQ\lhd M$ be the kernel of the projection
$M\to \Gamma$.  Let $L\lhd M$ be the inverse image of
$\KG$. 
Then 
\[
M/L\overto{\iso}\Gamma/\KG
\]
and we use this isomorphism to pass back and forth between
$L$-trivial $M$ actions, $\KG$-trivial $\Gamma$ actions,
and $\Gamma/\KG$ actions without comment or notation. 

As a subgroup of $\Gamma \times Q^{\op}$, $\RQ=M\cap (1\times
Q^{\op})$ and so in particular its projection to $Q^{\op}$ is an
injection; write $\RQ'$ for the subgroup of
$Q$ opposite to $\RQ$.  Whenever we have a left action of $\RQ$, we
have a corresponding right action by $\RQ'$ that we will use without
further comment.
For any orthogonal $\Gamma/\KG$-representation $W$, we then have an
$\Gamma$-equivariant isomorphism 
\begin{equation*}\label{eq:etatoV}
\aJ(\eta,W)/Q\iso (\aJ(V,W)\sma Y_{+})/\RQ',
\end{equation*}
which is a natural isomorphism of based $\Gamma$-space enriched functors
from $\aJ_{\Gamma/\KG}$ to $\Gamma$-equivariant based spaces,
i.e., a map of $\Gamma$-equivariant $\aJ_{\Gamma/\KG}$-spaces.   Using this isomorphism and the
fact that smash product commutes with colimits,
as in~\cite[4.4]{MMSS} (or see the proofs of
\cite[A.1]{BM-cycl} or \cite[III.8.4]{MM}) we get an
isomorphism 
\begin{equation}\label{eq:machina}
(\imath J\eta\sma X) (W)\iso 
(\aI(V\oplus \bR^{n},W)\times  Y)_{+}\sma_{\RQ'\times
O(n)}(S^{0}\sma X(\bR^{n})),
\end{equation}
where $\dim W=n+\dim V$. The right action on $\alpha\in
\aI(V\oplus \bR^{n},W)$ is defined by
\[
\alpha \cdot (\sigma,r):=\alpha \circ (\sigma_{*}^{-1}\oplus r)
\]
for $\sigma \in
\RQ'$ and $r\in O(n)$. The right action of $O(n)$ on $Y$ is
trivial.  We note that since the action of $\RQ'<Q$ on $\eta$
is faithful by hypothesis, the action of $\RQ'\times O(n)$ on
\[
\aI(\eta,W)=\aI(V\oplus \bR^{n},W)\times Y
\]
is free. 

We use the isomorphism in~\eqref{eq:machina} to analyze
$\Fix^{\Gamma}_{\KG}(\imath J\eta\sma X)$ as follows.
To compress notation, we write
\[
\bar\Xi(W):=(\aI(V\oplus \bR^{n},W)\times  Y)_{+}\sma (S^{0}\sma X(\bR^{n})),
\]
so that
\[
\bar\Xi(W)/(\RQ'\times O(n))\iso 
(\aI(V\oplus \bR^{n},W)\times Y)_{+}\sma_{\RQ'\times O(n)} (S^{0}\sma X(\bR^{n}))
\]
and per~\eqref{eq:machina}
\begin{equation}\label{eq:XiFix}
(\bar\Xi(W)/(\RQ'\times O(n)))^{\Gamma}\iso 
\Fix^{\Gamma}_{\KG}(\imath J\eta\sma X)(W).
\end{equation}
We also work with the space
\[
\Xi(W):=\aI(V\oplus \bR^{n},W)\times  Y\times (S^{0}\sma X(\bR^{n})).
\]
Let
\[
\Psi(W)\subset \Xi(W), \qquad \bar \Psi(W)\subset \bar \Xi(W)
\]
denote the inverse image of the $\Gamma$ fixed point subspaces of
$\Xi(W)/(\RQ'\times O(n))$ and $\bar\Xi(W)/(\RQ'\times O(n))$, respectively.

We next observe that $\Psi(W)$ and $\bar \Psi(W)$ have the obvious
quotient relationship, but this takes a bit of work because of the
point-set topology of quotients of (compactly generated) weak
Hausdorff spaces.  
Since being $\Gamma$ fixed is a closed condition, these are closed
subspaces.  Let $\Psi(W)_{*}\subset \Psi(W)$ denote the inverse image of the
basepoint of $S^{0}\sma X(\bR^{n})$ under the projection.  Then
$\Psi(W)_{*}\subset \Psi(W)$ is closed and the map from the pushout to the
union of subspaces
\[
\Psi(W)\cup_{\Psi(W)_{*}}(\aI(V\oplus \bR^{n},W)\times  Y\times *)\to 
\Psi(W) \cup (\aI(V\oplus \bR^{n},W)\times  Y\times *)
\]
is an isomorphism.  In particular, the map $\Psi(W)/\Psi(W)_{*}\to
\bar\Psi(W)$ is an isomorphism.  Because the group $\RQ'\times
O(n)$ is compact, the equivalence relation associated to an 
$(\RQ'\times O(n))$ action is closed. The
underlying set of a quotient (in the category of compactly generated
weak Hausdorff spaces) is then the quotient of the underlying set;
moreover, 
the quotient of a closed $(\RQ'\times O(n))$-stable subspace
includes as a closed subspace of the quotient.  Specifically, we have
that the maps
\begin{multline}\label{eq:PsibarPsi}
\qquad(\Psi(W)/\Psi(W)_{*})/(\RQ'\times O(n))\overto{\iso} 
\bar \Psi(W)/(\RQ'\times O(n))\\\overto{\iso} 
(\bar\Xi(W)/(\RQ'\times O(n)))^{\Gamma}\qquad
\end{multline}
are isomorphisms.  This lets us work with $\Psi(W)$ to study
$\Fix^{\Gamma}_{\KG}(\imath J\eta\sma X)(W)$.  Our strategy
is to decompose $\Psi(W)$ into a disjoint union of pieces corresponding
to graph subgroups as in the study of fixed points in Section~\ref{sec:qfpl} 
and then to analyze the individual pieces to identify $\Psi(W)$ and
$\bar \Psi(W)$.  (The precise statement is given as
Theorem~\ref{thm:recharpsi} below.) 

We write $(\alpha,y,x)$ for an element of $\aI(V\oplus \bR^{n},W)\times Y\times (S^{0}\sma X(\bR^{n}))$ with $\alpha \in
\aI(V\oplus \bR^{n},W)$, $y\in Y$, and $x\in X(\bR^{n})$.  By
definition, $(\alpha,y,x)$ is in $\Psi$ if and only if its image $[\alpha,y,x]$ in
the quotient $\Xi/(\RQ'\times O(n))$ is fixed by
$\Gamma$ if and only if that image is fixed by the restriction action of $M$ if and only
if for each $g\in M$, there exists a $\tilde\sigma(g)\in \RQ'$ and
$r(g)\in O(n)$ such that
\begin{equation}\label{eq:Psicond}
g\cdot\alpha=\alpha\cdot (\tilde\sigma(g),r(g)),\qquad
gy=y\tilde\sigma(g),\qquad 
gx=r(g)^{-1}x.
\end{equation}
Because the action of $\RQ'\times O(n)$ on $\aI(V\oplus \bR^{n},W)\times Y$
is free, $\tilde\sigma(g)$ and $r(g)$ are uniquely determined, and the
existence of local sections shows that
\[
\tilde\sigma\times r\colon M\to \RQ'\times O(n)
\]
is continuous.  

Looking at what happens for the product of elements in $M$, we see
that $r$ is a homomorphism.  Indeed, because $M$ acts trivially on
$\bR^{n}$ and $L$ acts trivially on $W$, $r$ descends to a
homomorphism $\Gamma/\KG\to O(n)$. On the other hand,
$\tilde\sigma$ is not a homomorphism but rather satisfies the
following formulas:
\begin{enumerate}
\item Restricted to $\RQ\lhd M$, $\tilde\sigma$ is the projection
anti-isomorphism $\RQ^{\op}\iso \RQ'$, and
\item For $g,h\in M$,
$\tilde\sigma(gh)=\tilde\sigma(g)({}^{g}\tilde\sigma(h))$, where
${}^{g}(-)$ denotes the left action of $g$ on $\RQ'$ corresponding
to the conjugation action $h\mapsto ghg^{-1}$ on $\RQ$.
\end{enumerate}
Equivalently, the graph $h\mapsto (h,\tilde\sigma(h))$ defines a
continuous graph homomorphism from $M$ into the semidirect product
$\RQ' \rtimes M$ extending the diagonal on $\RQ$; for the
purposes of this argument we call a continuous map $M\to \RQ'$
satisfying (i) and~(ii) above an \term{sdde-graph homomorphism} and
write $\tilde\sigma \colon M\sdde \RQ'$.  
We were led to this formulation by the fixed point formulas, but we
can rewrite the structure in terms of graph subgroups of $M<\Gamma \times
Q^{\op}$ and $\sigma$-twisted $\Gamma$ actions as follows.

\begin{prop}
There is a one-to-one correspondence between the set of sdde-graph homomorphisms
$\tilde\sigma\colon M\sdde \RQ'$ and the set of homomorphisms $\sigma \colon \Gamma
\to Q$ for which the corresponding graph subgroup
\[
\Gamma_{\sigma}=
\{(\gamma,\sigma(\gamma^{-1}))\mid\gamma \in \Gamma\}<\Gamma \times Q^{\op}
\]
lies in $M$.  Given $\tilde\sigma$ and an $M$-space $Z$, the
corresponding $\sigma$-twisted $\Gamma$ action on $z\in Z$ is given by
\[
\gamma \cdot_{\sigma}z=(gz)\tilde\sigma(g)^{-1}
\]
where the projection $M\to \Gamma$ sends $g$ to $\gamma$.
\end{prop}

\begin{proof}
The formulas are easier when we use the anti-isomorphism to write $\tilde\sigma$
in terms of $\RQ$: let $\tilde\sigma'\colon M\to \RQ$ denote the composite
of $\tilde\sigma$ with the anti-isomorphism $\RQ'\iso \RQ$.  The
defining rules for a sdde-graph homomorphism then become
\begin{enumerate}
\item Restricted to $\RQ\lhd M$, $\tilde\sigma'$ is the identity, and
\item For $g_{1},g_{2}\in M$,
$\tilde\sigma'(g_{1}g_{2})=g_{1}\tilde\sigma'(g_{2})g_{1}^{-1}\tilde\sigma'(g_{1})$.
\end{enumerate}
Given $\tilde\sigma'\colon M\to \RQ$, define $\sigma \colon \Gamma \to
Q$ by
\[
(\gamma,\sigma(\gamma)^{-1})=\tilde\sigma'(g)^{-1}g\in \Gamma \times Q^{\op}
\]
where $g\mapsto \gamma$; this is well-defined because if $h\in \RQ$,
then 
\[
\tilde\sigma'(hg)^{-1}hg=(h\tilde\sigma'(g)h^{-1}\tilde\sigma'(h))^{-1}hg
=(h\tilde\sigma'(g))^{-1}hg=\tilde\sigma(g)^{-1}g
\]
and $\sigma$ is a homomorphism because
\begin{align*}
(\gamma_{1},\sigma(\gamma_{1})^{-1})(\gamma_{2},\sigma(\gamma_{2})^{-1})
&=\tilde\sigma'(g_{1})^{-1}g_{1}\tilde\sigma'(g_{2})^{-1}g_{2}\\
&=\tilde\sigma'(g_{1})^{-1}g_{1}\tilde\sigma'(g_{2})^{-1}g_{1}^{-1}g_{1}g_{2}\\
&=(g_{1}\tilde\sigma'(g_{2})g_{1}^{-1}\tilde\sigma'(g_{1}))^{-1}g_{1}g_{2}\\
&=\tilde\sigma'(g_{1}g_{2})^{-1}g_{1}g_{2}
=(\gamma_{1}\gamma_{2},\sigma(\gamma_{1}\gamma_{2})^{-1}),
\end{align*}
noting that the multiplication rule in $\Gamma \times Q^{\op}$
satisfies
\[
(\gamma_{1},a_{1}^{-1})(\gamma_{2},a_{2}^{-1})=(\gamma_{1}\gamma_{2},(a_{1}a_{2})^{-1}).
\]
On the other hand, given $\sigma \colon \Gamma \to Q$ such that
$\Gamma_{\sigma}\subset M$, define 
\[
\tilde\sigma'(g)=g(\gamma^{-1},\sigma(\gamma))
\]
where $g\mapsto \gamma$. Since the right side is in $M$ and the
projection $M\to \Gamma$ sends $\tilde\sigma'(g)$ to the identity,
$\tilde\sigma'(g)\in \RQ$.  If $g\in \RQ$, then $\gamma$ is the identity
and $\tilde\sigma'(g)=g$.  Given $g_{1},g_{2}\in M$, 
\begin{align*}
\tilde\sigma'(g_{1}g_{2})
&=g_{1}g_{2}((\gamma_{1}\gamma_{2})^{-1},\sigma(\gamma_{1}\gamma_{2}))\\
&=g_{1}g_{2}(\gamma_{2}^{-1},\sigma(\gamma_{2}))(\gamma_{1}^{-1},\sigma(\gamma_{1}))\\
&=g_{1}g_{2}(\gamma_{2}^{-1},\sigma(\gamma_{2}))
    g_{1}^{-1}g_{1}(\gamma_{1}^{-1},\sigma(\gamma_{1}))\\
&=g_{1}\tilde\sigma'(g_{2})g_{1}^{-1}\tilde\sigma'(g_{1}).
\end{align*}
These rules are obviously inverse, establishing the one-to-one
correspondence.  For $z\in Z$, by definition, the $\sigma$-twisted
$\Gamma$ action on $z$ is given by
\[
\gamma \cdot_{\sigma}z:= (\gamma,\sigma(\gamma)^{-1})z 
\]
which by the above is $\tilde\sigma'(g)^{-1}gz$, or written in terms
of the right $\RQ'$ action $\gamma \cdot_{\sigma}z=(gz)\tilde\sigma(g)^{-1}$.
\end{proof}

We call a continuous homomorphism with the property that the
corresponding graph subgroup $\Gamma_{\sigma}<\Gamma \times
Q^{\op}$ lies $M$ an \term{$M$ graph subgroup homomorphism} and
write $\sigma \colon \Gamma \Mto Q$.  Given such a homomorphism
$\sigma$, we write $\tilde\sigma \colon M\sdde \RQ'$ for the
corresponding sdde-graph homomorphism.  We note that care must be
taken with the action formula $\gamma
\cdot_{\sigma}z=(gz)\tilde\sigma(g)^{-1}$ as the left action of $M$
and the right action of $\RQ'$ do not commute, rather
$g(z\sigma)=(gz)({}^{g}\sigma)$.

Returning to the study of $\Psi$, given $\sigma \colon \Gamma\Mto
Q$ and $r\colon \Gamma/\KG \to
O(n)$, define the
$(\sigma,r)$-twisted $\Gamma$ actions on $\aI(V\oplus \bR^{n},W)$, $Y$ and
$X(W)$ by
\begin{align*}
\gamma \cdot_{(\sigma,r)}\alpha 
&:= (\gamma \cdot_{\sigma} \alpha)\cdot (\id,r(\gamma)) 
= (g\cdot \alpha)\cdot (\tilde\sigma(g),r(\gamma))^{-1}
\\
\gamma \cdot_{(\sigma,r)}y &:= \gamma\cdot_{\sigma}y=\gamma y\sigma(\gamma^{-1})=(gy)\tilde\sigma(g)^{-1}\\
\gamma \cdot_{(\sigma,r)}x &:= r(\gamma)\cdot \gamma x
\end{align*}
where $g\mapsto \gamma$; these formulas give well-defined actions
because the $\sigma$-twisted $\Gamma$ action on $\aI(V\oplus \bR^{n},W)$ and the
natural $\Gamma$ action on $X(W)$ commute with the $O(n)$ action.
We then put these together to define the $(\sigma,r)$-twisted $\Gamma$ action on~$\Xi$,
\[
\gamma \cdot_{(\sigma,r)}(\alpha,y,x)
:=(\gamma \cdot_{(\sigma,r)}\alpha,\gamma \cdot_{(\sigma,r)}y,\gamma \cdot_{(\sigma,r)} x).
\]
%
%
The following proposition now summarizes the formulas
in~\eqref{eq:Psicond}. 

\begin{prop}\label{prop:Psicond}
An element $\xi$ of $\Xi$ is in $\Psi$ if and only if there exists a 
continuous homomorphisms $r\colon \Gamma/\KG \to O(n)$ and an
$M$ graph subgroup homomorphism $\sigma \colon \Gamma \Mto Q$
such that $\xi$ is fixed by the $(\sigma,r)$-twisted $\Gamma$ action
on $\Xi$.  When $r$ and $\sigma$ exist, they are unique.
\end{prop}

We can rewrite the $(\sigma,r)$-twisted $\Gamma$ action of the previous
paragraph more intrinsically to develop an intrinsic identification of
$\Fix^{\Gamma}_{\KG}(\imath J\eta\sma X)$.  Given a homomorphism
$r\colon \Gamma/\KG\to O(n)$, we get an orthogonal $\Gamma/\KG$-representation $R_{r}$
with underlying inner product space $\bR^{n}$ and $\Gamma/\KG$ action given by
$r$.  Given $\sigma \colon \Gamma \Mto
Q$, we get $\sigma$-twisted $\Gamma$ actions on $V$ and $Y$; we
write $V_{\sigma}$ for the orthogonal $\Gamma$-representation with
underlying inner 
product space $V$ and $\Gamma$ action given by the $\sigma$-twisted
$\Gamma$ action on $V$, and we write $Y_{\sigma}$ for the
space $Y$ with $\Gamma$ action given by the $\sigma$-twisted
$\Gamma$ action on $Y$.
Then the $(\sigma,r)$-twisted $\Gamma$ action on $\aI(V\oplus \bR^{n},W)$
is precisely the canonical $\Gamma$ action on $\aI(V_{\sigma}\oplus
R_{r},W)$ and the $(\sigma,r)$-twisted $\Gamma$ action on
$(\aI(V\oplus \bR^{n},W)\times Y)_{+}\sma (S^{0}\sma X(\bR^{n}))$ is precisely the canonical
$\Gamma$ action on $(\aI(V_{\sigma}\oplus R_{r},W)\times
Y_{\sigma})_{+}\sma (S^{0}\sma X(R_{r}))$.
We see from this reformulation that if there exists an
$(\sigma,r)$-twisted fixed point of $\aI(V\oplus \bR^{n},W)$, then
$V_{\sigma}\oplus R_{r}$ is $\Gamma$-isomorphic to $W$, and this implies
that $\KG$ acts trivially $V_{\sigma}$ as well as $R_{r}$.  In
particular, $R_{r}$ and $V_{\sigma}$
are both orthogonal $\Gamma/\KG$-representations.  

We say that an $M$ graph subgroup homomorphism is \term{$W$-compatible} when
$V_{\sigma}$ is $\Gamma$-isomorphic to a subrepresentation of $W$, and
write $\sigma \colon \Gamma \WMto Q$.  By the previous
proposition, only the $W$-compatible $M$ graph subgroup homomorphisms
are relevant to $\Psi(W)\subset \Xi(W)$ and $\xi \in \Xi(W)$ cannot be
a $(\sigma,r)$-twisted $\Gamma$ fixed point unless $\sigma$ is
$W$-compatible.  

The (conjugation) action of $\RQ'$ on the set of $M$ graph subgroup
homomorphisms preserves the subset of $W$-compatible ones.  Indeed, if
$\alpha \colon V_{\sigma}\oplus R_{r}\to W$ is an isomorphism of
orthogonal $\Gamma/\KG$-representations, then for all $\tau \in
\RQ'$, $\alpha \cdot (\tau,\id)$ is an isomorphism of orthogonal $\Gamma/\KG$-representations $V_{\sigma_{\tau}}\oplus R_{r}\to W$.
For each $\RQ'$-orbit of $W$-compatible $M$ graph subgroup
homomorphisms $[\sigma]$, choose and fix a homomorphism $r[\sigma]$ so
that $V_{\sigma}\oplus R_{r[\sigma]}\iso W$.  We then get a map
\[
(\aI(V_{\sigma}\oplus R_{r[\sigma]},W)
\times Y_{\sigma}\times X(R_{r[\sigma]}))^{\Gamma}
\to \Psi(W)^{\Gamma,\sigma}\subset \Xi(W)
\]
which is $O(R_{r[\sigma]})^{\Gamma/\KG}$-equivariant, where
\[
O(R_{r[\sigma]})^{\Gamma/\KG}\subset O(R_{r[\sigma]})=O(n)
\]
denotes the Lie subgroup of $\Gamma/\KG$-equivariant isometries (the
final equality holds because the underlying inner product space of
$R_{r[\sigma]}$ is $\bR^{n}$).  Inducing up to $O(n)$, we get an
$O(n)$-equivariant map;  let $T(W)$ be the disjoint
union of these spaces over the $W$-compatible $M$ graph subgroup homomorphisms
\[
T(W):=\coprod_{\sigma\colon \Gamma \WMto Q}
(\aI(V_{\sigma}\oplus R_{r[\sigma]},W) \times Y_{\sigma}
\times X(R_{r[\sigma]}))^{\Gamma}\times_{O(R_{r[\sigma]})^{\Gamma/\KG}}O(n).
\]
We give this the right $\RQ'$ action where $\tau \in \RQ'$ sends $v\in
V_{\sigma}$ to $v\tau \in V_{\sigma_{\tau}}$ and $y\in Y_{\sigma}$ to
$y\tau \in Y_{\sigma_{\tau}}$ (acting trivially on
$R_{r[\sigma]}$).    Likewise, let
\[
\bar T(W)=\bigvee_{\sigma\colon \Gamma \WMto Q}
((\aI(V_{\sigma}\oplus R_{r[\sigma]},W) \times Y_{\sigma})_{+}
\sma
X(R_{r[\sigma]}))^{\Gamma}\sma_{O(R_{r[\sigma]})^{\Gamma/\KG}}O(n)_{+}
\]
with the analogous $\RQ'$ action.  
The induced maps
\[
T(W)\to \Psi(W), \qquad \bar T(W)\to \bar \Psi(W)
\]
are then $(\RQ'\times O(n))$-equivariant.
We prove the following theorem. 

\begin{thm}\label{thm:recharpsi}
The $(\RQ'\times O(n))$-equivariant maps
\[
T(W)\to \Psi(W), \qquad \bar T(W)\to \bar \Psi(W)
\]
are isomorphisms.
\end{thm}

\begin{proof}
We construct an inverse isomorphism $\Psi(W)\to T(W)$ as follows.
For fixed $\sigma\colon \Gamma \WMto Q$, let 
\[
T(\sigma)=(\aI(V_{\sigma}\oplus R_{r[\sigma]},W)\times Y_{\sigma}
\times X(R_{r[\sigma]}))^{\Gamma}
\times_{O(R_{r[\sigma]})^{\Gamma/\KG}}O(n)\subset T(W)
\]
(a summand in the disjoint union), let
\[
\aI(\sigma)\subset \aI(V\oplus \bR^{n},W)\times Y
\]
denote the union of the $(\sigma,r)$-twisted $M$ fixed points as $r$
varies, and let 
\[
\Psi(\sigma)\subset \Psi(W)
\]
denote the subspace that projects to $\aI(\sigma)$. The subsets
$\aI(\sigma)$ are then disjoint and they are 
closed since the inclusion of $\aI(\sigma)$ in $\aI(V\oplus
\bR^{n},W)\times Y$ is the inverse image of the inclusion of the $\sigma$-twisted
$\Gamma$ fixed points in $\aI(V\oplus \bR^{n},W)/O(n)\times Y$.  The action of
$O(n)$ on $\aI(V\oplus \bR^{n},W)$ restricts to each $\aI(\sigma)$.
These observations for $\aI(\sigma)$ imply the corresponding
observations for $\Psi(\sigma)$.  In particular, $\Psi(W)$ is the
disjoint union of $\Psi(\sigma)$, and since by construction the map 
\[
T(\sigma)\subset T(W)\to \Psi(W)
\]
factors through $\Psi(\sigma)$, it suffices to construct an inverse
isomorphism $\Psi(\sigma)\to T(\sigma)$.

For each $(\alpha,y) \in \aI(\sigma)$, there is a unique $r_{\alpha}\colon
\Gamma/\KG\to O(n)$ for which $\alpha$ is a
$(\sigma,r_{\alpha})$-twisted $\Gamma$ fixed point (independently of $y\in Y$), and the function
$\alpha \mapsto r_{\alpha}$ is a continuous map $r$ from $\aI(\alpha)$
to the space $\Hom(\Gamma/\KG,O(n))$ of continuous homomorphisms
$\Gamma/\KG \to O(n)$ (with the compact-open topology).  It is
convenient to write $\Hom(\sigma)$ for the subspace of
$\Hom(\Gamma/\KG,O(n))$ of elements $r\colon \Gamma/\KG \to
O(n)$ such that $R_{r}\iso R_{r[\sigma]}$.

Because $V_{\sigma}\oplus R_{r_{\alpha}}$ is
$\Gamma/\KG$-isomorphic to $W$, $R_{r_{\alpha}}$ is
$\Gamma/\KG$-isomorphic to $R_{r[\sigma]}$.  Choose an arbitrary
$\Gamma/\KG$-isomorphism $f_{\alpha}\colon R_{r_{\alpha}}\to
R_{r[\sigma]}$.  Since $R_{r_{\alpha}}$ and $R_{r[\sigma]}$ both have
underlying inner product space $\bR^{n}$, $f_{\alpha}\in O(n)$.  In
general there is no way to choose the $f_{\alpha}$ so that the
resulting function $f\colon \aI(\sigma)\to O(n)$ is continuous;
however, we can choose $f$ so that it factors as
\[
\aI(\sigma)\overto{r} \Hom(\sigma)\overto{\tilde f}O(n) 
\]
and the usual compact Lie group integration argument shows that for
any fixed $r_{0}$ and fixed choice of $\tilde f_{r_{0}}\colon
R_{r_{0}}\iso R_{r[\sigma]}$, there exists a neighborhood $U$ of
$r_{0}$ in $\Hom(\sigma)$ where we can choose $\tilde f_{r}$ to
assemble to a continuous function $\tilde f\colon U\to O(n)$ extending
the chosen value at $r_{0}$.

Now consider the function $\theta \colon \Psi(\sigma)\to T(\sigma)$ defined by 
\begin{multline*}
\theta \colon (\alpha,y,x)\mapsto 
((\alpha \circ (\id_{V}\oplus f^{-1}_{\alpha}),
y,f_{\alpha}\cdot x),f_{\alpha})\\
\in
(\aI(V_{\sigma}\oplus R_{r[\sigma]},W)\times Y_{\sigma}\times X(R_{r[\sigma]}))^{\Gamma}
\times_{O(R_{r[\sigma]})^{\Gamma/\KG}}O(n).
\end{multline*}
It is not obvious that $\theta$ is continuous, but it is well defined: 
\begin{itemize}
\item By hypothesis $\alpha$ is an $(\sigma,r_{\alpha})$-twisted
$\Gamma$ fixed
point of $\aI(V\oplus \bR^{n},W)$ and hence a
$\Gamma$ fixed element of $\aI(V_{\sigma}\oplus R_{r_{\alpha}},W)$; then $a\circ
(\id_{V}\oplus f_{\alpha}^{-1})$ is a $\Gamma$ fixed element of 
$\aI(V_{\sigma}\oplus R_{r[\sigma]},W)$ since $f_{\alpha}$ is
$\Gamma$-equivariant.
\item By hypothesis $y$ is a $(\sigma,r_{\alpha})$-twisted $\Gamma$ fixed
point of $Y$ and hence is a $\Gamma$ fixed element of $Y_{\sigma}$.
\item By hypothesis $x$ is a $(\sigma,r_{\alpha})$-twisted $\Gamma$ fixed
point of $X(\bR^{n})$ and hence is a $\Gamma$ fixed element of
$X(R_{r_{\alpha}})$; then $f_{\alpha}\cdot x$ is a $\Gamma$ fixed element of
$X(R_{r[\sigma]})$ since $f_{\alpha}$ is
$\Gamma$-equivariant.
\end{itemize}
Next we check that $\theta$ is continuous.
If $f'_{\alpha}$ is another choice of $\Gamma/\KG$-equivariant isomorphism $R_{r_{\alpha}}\to
R_{r[\sigma]}$, then $\psi:=f'_{\alpha}\circ f_{\alpha}^{-1}$ is an element of
$O(R_{r[\sigma]})^{\Gamma/\KG}$ and 
\begin{align*}
((\alpha \circ (\id_{V}\oplus f^{-1}_{\alpha}),
y,f_{\alpha}\cdot x),f_{\alpha})
&=((\alpha \circ (\id_{V}\oplus f^{-1}_{\alpha}),
y,f_{\alpha}\cdot x)\cdot \psi^{-1},\psi\circ f_{\alpha})\\
&=((\alpha \circ (\id_{V}\oplus f^{-1}_{\alpha}\circ \psi^{-1}),
y,((\psi \circ f_{\alpha})\cdot x)),\psi \circ f_{\alpha})\\
&=((\alpha \circ (\id_{V}\oplus (f')^{-1}_{\alpha}),
y,(f'_{\alpha})\cdot x),f'_{\alpha}).
\end{align*}
Thus, the function $\theta$ is independent of the
choice of $f$.  Since we can choose $f$ to be continuous in a
neighborhood of $\alpha$ for any fixed $\alpha$, $\theta$ is
continuous in a neighborhood of every point $(\alpha,y,x)$, and so
$\theta$ is continuous on all of $\Psi(\sigma)$.  

By inspection, $\theta$ is inverse to the original map $T(\sigma)\to
\Psi(\sigma)$, and this proves the statement for $\Psi(W)$.  The
statement for $\bar\Psi(W)$ follows from the statement for $\Psi(W)$.
\end{proof}

We can now prove Theorem~\ref{thm:fixsmflat}.

\begin{proof}[Proof of Theorem~\ref{thm:fixsmflat}]
For part~(i), 
the previous theorem applied in the case when $X=\bS$ gives
\begin{multline*}
\Fix^{\Gamma}_{\KG}(\imath J\eta)(W)=
\biggl(\bigvee_{\sigma \colon \Gamma \WMto Q}
(\aI(V_{\sigma}\oplus R_{r[\sigma]},W)^{\Gamma}\times Y_{\sigma}^{\Gamma})_{+}\sma_{O(R_{r[\sigma]})^{\Gamma/\KG}}
S^{R_{r[\sigma]}^{\Gamma/\KG}}\biggr)/\RQ'\\[-.25em]
=\biggl(\bigvee_{\sigma \colon \Gamma \WMto Q}
\aJ^{\Gamma/\KG}_{\Gamma/\KG}(V_{\sigma},W)\sma (Y_{\sigma}^{\Gamma})_{+}
\biggr)/\RQ'
\end{multline*}
and it is straightforward to check that the latter formula specifies an
isomorphism functorial for $W\in
\aJ^{\Gamma/\KG}_{\Gamma/\KG}$, i.e., describes
$\Fix^{\Gamma}_{\KG}(\imath J\eta)$ as a
$\aJ^{\Gamma/\KG}_{\Gamma/\KG}$-space.  (This also follows
from the proof of Theorem~\ref{thm:bundlegeofix}.) Since
smash product commutes with colimits in each variable, this
isomorphism induces an isomorphism
\[
\Fix^{\Gamma}_{\KG}(\imath J\eta)\sma 
\Fix^{\Gamma}_{\KG}(X)
\iso 
\biggl(\bigvee_{\sigma \colon \Gamma \WMto Q}
(\aJ^{\Gamma/\KG}_{\Gamma/\KG}(V_{\sigma},-)\sma
(Y_{\sigma}^{\Gamma})_{+})\sma
\Fix^{\Gamma}_{\KG}(X)
\biggr)/\RQ'.
\]
As in \cite[4.4]{MMSS} (or as in the proof of~\cite[A.1]{BM-cycl}), we
get for the $W$ space
\begin{multline*}
(\Fix^{\Gamma}_{\KG}(\imath J\eta)\sma
\Fix^{\Gamma}_{\KG}(X))(W)
\\
\iso 
\biggl(\bigvee_{\sigma \colon \Gamma \WMto Q}
(\aJ^{\Gamma/\KG}_{\Gamma/\KG}(V_{\sigma},W)\sma (Y_{\sigma}^{\Gamma})_{+})
\sma_{O(R_{r[\sigma]})^{\Gamma/\KG}}X(R_{r[\sigma]})^{\Gamma}
\biggr)/\RQ'\\[-.5em]
\iso \bar T(W)/(\RQ'\times O(n))
\end{multline*}
with the universal map 
\[
(\Fix^{\Gamma}_{\KG}(\imath J\eta))(W')\sma
(\Fix^{\Gamma}_{\KG}(X))(W'')
\to \bar T(W)/(\RQ'\times O(n))
\]
for $W=W'\oplus W''$
induced by the inclusion of $W'$ in $W$ 
\begin{multline*}
\aJ^{\Gamma/\KG}_{\Gamma/\KG}(V_{\sigma},W')\sma (Y_{\sigma}^{\Gamma})_{+}
\iso
(\aI(V_{\sigma}\oplus R,W')\times Y_{\sigma})^{\Gamma}_{+}
\sma_{O(R)^{\Gamma/\KG}}S^{R^{\Gamma/\KG}}\\
\to
(\aI(V_{\sigma}\oplus R\oplus W'',W)\times Y_{\sigma})^{\Gamma}_{+}
\sma_{O(R)^{\Gamma/\KG}}S^{R^{\Gamma/\KG}}
\end{multline*}
in the $J\eta$ factor
and the left suspension map 
\[
S^{R^{\Gamma/\KG}}\sma X(W'')^{\Gamma}\to X(R\oplus W'')^{\Gamma}
\]
in the $X$ factor.
(To compare notation, note that $R_{r[\sigma]}\iso R\oplus W''$ since
they both equivariantly complement $V_{\sigma}$ in $W$.)
The previous theorem (combined with~\eqref{eq:XiFix} and~\eqref{eq:PsibarPsi}) gives an isomorphism 
\[
\bar T(W)/(O(n)\times \RQ')\to 
(\Fix^{\Gamma}_{\KG}(\imath J\eta\sma X))(W)
\]
and by inspection, the composite
\begin{multline*}
(\Fix^{\Gamma}_{\KG}(\imath J\eta))(W')\sma
(\Fix^{\Gamma}_{\KG}X)(W'')
\to \bar T(W)/(\RQ'\times O(n))\\
\to (\Fix^{\Gamma}_{\KG}(\imath J\eta\sma X))(W)
\end{multline*}
is the universal map for the natural transformation 
\[
\Fix^{\Gamma}_{\KG}(\imath J\eta)\sma \Fix^{\Gamma}_{\KG}(X)
\to \Fix^{\Gamma}_{\KG}(\imath J\eta\sma X).
\]
This proves part~(i).

For part~(ii), the previous theorem identifies the map 
\[
(\Fix^{\Gamma}_{\KG}(\imath J(\EF\times \eta)\sma X))(W)\to
(\Fix^{\Gamma}_{\KG}(\imath J\eta\sma X))(W)
\]
as the map
\begin{multline*}
\biggl(
\bigvee_{\sigma\colon \Gamma \WMto Q}
((\aI(V_{\sigma}\oplus R_{r[\sigma]},W)^{\Gamma} \times (\EF \times Y)_{\sigma}^{\Gamma})_{+}
\sma_{O(R_{r[\sigma]})^{\Gamma/\KG}}
X(R_{r[\sigma]})^{\Gamma})
\biggr)/\RQ'
\\[-1em]
\to
\biggl(
\bigvee_{\sigma\colon \Gamma \WMto Q}
((\aI(V_{\sigma}\oplus R_{r[\sigma]},W)^{\Gamma} \times Y_{\sigma}^{\Gamma})_{+}
\sma_{O(R_{r[\sigma]})^{\Gamma/\KG}}
X(R_{r[\sigma]})^{\Gamma})
\biggr)/\RQ'
\end{multline*}
For any homomorphism $\sigma \colon \Gamma \to
Q$, the graph subgroup $\Gamma_{\sigma}<\Gamma \times
Q^{\op}$ is in the family $\aF=\aF(1\times Q^{\op})$ and so
the map $(\EF\times Y)_{\sigma}^{\Gamma}\to Y_{\sigma}^{\Gamma}$ is a
homotopy equivalence.  Let $O_{1},\dotsc,O_{q}$ be the set of
orbits for the action of $\RQ'$ on the set of
$W$-compatible $M$ graph subgroup homomorphisms $\Gamma \WMto Q$, and
for each $i$, choose a representative $\sigma_{i}\in O_{i}$.
Because $\RQ'\times O(n)$ acts freely on $\aI(V\oplus \bR^{n},W)\times Y$,
for each $i$, $\RQ'\times (O(R_{r[\sigma_{i}]}))^{\Gamma/\KG}$
acts freely on 
\[
\coprod_{\sigma \in O_{i}}\aI(V_{\sigma}\oplus
R_{r[\sigma]},W)^{\Gamma}\times Y_{\sigma}^{\Gamma}\subset 
\aI(V\oplus \bR^{n},W)\times Y
\]
so for each $i$, the map
\[
\coprod_{\sigma \in O_{i}}
\aI(V_{\sigma}\oplus R_{r[\sigma]},W)^{\Gamma}
\times (\EF\times Y)_{\sigma}^{\Gamma}
\to
\coprod_{\sigma \in O_{i}}\aI(V_{\sigma}\oplus
R_{r[\sigma]},W)^{\Gamma}
\times Y_{\sigma}^{\Gamma}
\]
is an $\RQ'\times (O(R_{r[\sigma]})^{\Gamma/\KG})$-equivariant homotopy equivalence.  The maps
\begin{multline*}
\biggl(
\bigvee_{\sigma \in O_{i}}
(\aI(V_{\sigma}\oplus R_{r[\sigma]},W)^{\Gamma}
\times (\EF\times Y)_{\sigma}^{\Gamma})_{+}
\sma_{O(R_{r[\sigma]})^{\Gamma/\KG}}
X(R_{r[\sigma]})^{\Gamma}\biggr)/\RQ'\\[-1em]
\to
\biggl(
\bigvee_{\sigma \in O_{i}}
(\aI(V_{\sigma}\oplus R_{r[\sigma]},W)^{\Gamma}
\times Y_{\sigma}^{\Gamma})_{+}
\sma_{O(R_{r[\sigma]})^{\Gamma/\KG}}
X(R_{r[\sigma]})^{\Gamma}\biggr)/\RQ'
\end{multline*}
are therefore based homotopy equivalences.  Taking a wedge over $i$,
we see that the map
\[
\Fix^{\Gamma}_{\KG}(\imath J(\EF\times \eta)\sma X) 
\to
\Fix^{\Gamma}_{\KG}(\imath J\eta\sma X) 
\]
induces a based homotopy equivalence on $W$ spaces for all $W$.  This
proves part~(ii). 
\end{proof}

\section{Proof of Theorem~\sref$thm:dersym$}
\label{sec:dersym}

The strategy for the proof of Theorem~\ref{thm:dersym} is an indirect
approach using Proposition~\ref{prop:charsym} below, which gives 
criteria to identify when symmetric powers preserve a weak
equivalence out of a cofibrant object.  

The criterion involves indexed smash powers and HHR diagonal maps,
which are reviewed and slightly generalized in Section~\ref{sec:B209}.
In the current context, for $G$ a compact Lie group, $Y$ 
a $G$-spectrum, 
and $S$ a finite $G$-set, we can form the indexed smash power
$Y^{(S)}$ which gets its $G$-spectrum structure
both from $Y$ and the $G$ action on $S$. 
When we decompose $S$ as a disjoint union of orbits 
\[
S\iso G/H_{1}\amalg \dotsb \amalg G/H_{r},
\]
we then have an \term{HHR diagonal map}~\cite[B.209]{HHR} (or Proposition~\ref{prop:diagexist}) of
non-equivariant spectra
\[
\Phi^{H_{1}} Y\sma \dotsb \sma \Phi^{H_{r}} Y
\to \Phi^{G}(Y^{(S)}).
\]
We remark that the map depends on the choice of isomorphism, that is,
the choice of the elements of $S$ that correspond to the identity
cosets $eH_{1}$,\dots, $eH_{r}$: the map
is analogous to and induced by the diagonal map for based $G$-spaces
\[
A^{H_{1}}\sma\dotsb\sma A^{H_{r}}\to (A^{(S)})^{G}
\]
that sends $a_{1}\sma \dotsb \sma a_{r}$ to the element of the smash
product $A^{(S)}$ that is given by $g_{\zeta} a_{i}$ in the $\zeta \in
S$ coordinate, where $\zeta$ corresponds to the element $g_{\zeta}
H_{i}$ in the identification of $S$ as a disjoint union of orbits
above.

In the statement of the following proposition,
for a finite group $\Sigma$, we write $E_{G}\Sigma$ for the
universal space for the 
family of subgroups of $G\times \Sigma$ whose intersection with $1\times
\Sigma$ is trivial. It is a $(G \times \Sigma)$-CW
complex whose fixed points are contractible for subgroups in the
family and empty for subgroups not in the family.

\begin{prop}\label{prop:charsym}
Let $G$ be a compact Lie group, let $k>1$, and let $\Sigma <
\Sigma_{k}$. Let $Y$ be a $G$-spectrum that
satisfies the following conditions:
\begin{enumerate}
\item The quotient map
$E_{G}\Sigma_{+}\sma_{\Sigma}Y^{(k)}\to Y^{(k)}/\Sigma$
is a weak equivalence.
\item For every $M<G$ and every $M$-set
$S=M/H_{1}\amalg \dotsb \amalg M/H_{r}$ of cardinality $k$, the map
\[ \qquad 
L\Phi^{H_{1}}Y\sma^{L}\dotsb\sma^{L}L\Phi^{H_{r}}Y
\to 
\Phi^{H_{1}}Y\sma\dotsb\sma \Phi^{H_{r}}Y
\to \Phi^{M}(Y^{(S)}) 
\]
in the non-equivariant stable category is an isomorphism.
\end{enumerate}
Then for any cofibrant approximation $X\to Y$ in the positive complete
model structure on $G$-spectra, the
induced map $X^{(k)}/\Sigma\to Y^{(k)}/\Sigma$ is a weak
equivalence. 
\end{prop}

\begin{proof}
Let $X\to Y$ be a cofibrant approximation in the positive complete
model structure on $G$-spectra.  In the diagram
\[
\xymatrix{%
E_{G}\Sigma_{+}\sma_{\Sigma}X^{(k)}\ar[d]\ar[r]
&X^{(k)}/\Sigma\ar[d]\\
E_{G}\Sigma_{+}\sma_{\Sigma}Y^{(k)}\ar[r]
&Y^{(k)}/\Sigma
}
\]
the top horizontal map is a weak equivalence by~\cite[B.116]{HHR} and
the bottom horizontal map is a weak equivalence by hypothesis~(i).  We
want to show that the right horizontal map is a weak equivalence, so
it suffices to show that the left horizontal map is a weak
equivalence.  The $(G\times \Sigma)$-CW complex $E_{G}\Sigma$
is built by cells of the form $(G\times \Sigma)/\tilde H\times
(D^{n},\partial D^{n})$ where $\tilde H\cap (1\times \Sigma)$ is trivial,
and working a cell at a time, it suffices to show that the map
\[
(((G\times \Sigma)/\tilde H) \sma X^{(k)})/\Sigma
\to
(((G\times \Sigma)/\tilde H) \sma Y^{(k)})/\Sigma
\]
is a weak equivalence for every such subgroup $\tilde H$.  Let $H$ be
the image of $\tilde H$ in $G$ under the projection.  Because $\tilde
H \cap (1\times \Sigma)$ is trivial, the projection is an isomorphism
$\tilde H\to H$ and there exists a homomorphism $\sigma \colon H\to
\Sigma$ with $\tilde H$ the graph of this homomorphism:
\[
\tilde H = H_{\sigma}=\{(h,\sigma(h))\mid h\in H\} < G\times \Sigma.
\]
Let $S=\{1,\dotsc,k\}$ be the $H$-set obtained using $\sigma$ to
define the action. We then have a commutative diagram 
\[
\xymatrix@R-1pc{%
(((G\times \Sigma)/\tilde H) \sma X^{(k)})/\Sigma\ar[d]\ar[r]^-{\iso}
&G_{+}\sma_{H}X^{(S)}\ar[d]\\
(((G\times \Sigma)/\tilde H) \sma Y^{(k)})/\Sigma\ar[r]_-{\iso}
&G_{+}\sma_{H}Y^{(S)}
}
\]
with the horizontal maps isomorphisms (see, for example,
Proposition~\ref{prop:indexedsma} with 
$H$ playing the role of $G$ there).  We are now
reduced to showing that the map of $H$-spectra 
\[
X^{(S)}\to Y^{(S)}
\]
is a weak equivalence.  By~\cite[XVI.6.4]{May-Alaska} (or
Proposition~\ref{prop:pivsphi}), it is now enough to show that for
every $M<H$, the map $L\Phi^{M}(X^{(S)})\to L\Phi^{M}(Y^{(S)})$ is a
non-equivariant weak equivalence, and this follows from hypothesis~(ii).
\end{proof}

We now start the argument for Theorem~\ref{thm:dersym}, using the
notational convention of Notation~\ref{notn:KGL}: $\Gamma$ is a
compact Lie group with a closed normal subgroups $\KG,\Lambda$
satisfying $\KG<\Lambda$. We treat the case $k>1$, 
noting that in the statement of Theorem~\ref{thm:dersym}, there is
nothing to show in the case $k=1$. The following
proposition gives the first step.

\begin{prop}\label{prop:exttosym}
Let $\eta$ be a $\Gamma/\KG$-compatible $Q$-faithful $(\Gamma,Q)$
vector bundle where each fiber is positive dimensional.  Then for
every $\Sigma<\Sigma_{k}$, the map
\[
E_{\Gamma/\KG}\Sigma_{+}\sma_{\Sigma}(J^{\Gamma/\KG}\eta)^{(k)}
\to (J^{\Gamma/\KG}\eta)^{(k)}/\Sigma
\]
is a level equivalence.
\end{prop}

\begin{proof}
By Proposition~\ref{prop:trivsf}, $\Sym^{k}\eta$ is $(\Sigma_{k}\wr
Q)$-faithful.  In particular, the $\Sigma_{k}$ action on
\[
(J^{\Gamma/\KG}\eta)^{(k)}(V)\iso 
(J^{\Gamma/\KG}\eta^{k})(V)
=\aJ(\eta^{k},V)/Q^{k}
\]
is free for every $\KG$-trivial orthogonal $\Gamma$-representation
$V$. We get that 
\[
E_{\Gamma/\KG}\Sigma_{+}\sma \aJ(\eta^{k},V)/Q^{k}\to 
\aJ(\eta^{k},V)/Q^{k}
\]
is a $(\Gamma/\KG\times \Sigma)$-equivalence; taking quotients by
$\Sigma$, the map 
\[
E_{\Gamma/\KG}\Sigma_{+}\sma_{\Sigma} \aJ(\eta^{k},V)/Q^{k}\to 
(\aJ(\eta^{k},V)/Q^{k})/\Sigma
\]
is a $\Gamma/\KG$-equivalence.
\end{proof}

Next we study the geometric fixed points $\Phi^{\Gamma/\KG}$ of
$(J^{\Gamma/\KG}\eta)^{(S)}$ for finite $\Gamma/\KG$-sets $S$. As in
Section~\ref{sec:redsm}, we need the extra flexibility of considering
$\Gamma$ actions which are not $\KG$ fixed and the appropriate context
for this is the category of $\Gamma$-spectra
indexed on $\KG$-trivial spectra, or in other words,
$\Gamma$-equivariant $\aJ_{\Gamma/\KG}$-spaces.  As in that section,
we use $\imath$ for the functor from $\Gamma$-spectra to $\Gamma$-equivariant $\aJ_{\Gamma/\KG}$-spaces that
restricts to the $\KG$-trivial orthogonal $\Gamma$-representations,
and $\jmath$ for the functor from $\Gamma/\KG$-spectra to $\Gamma$-equivariant $\aJ_{\Gamma/\KG}$-spaces that regards
a $\Gamma/\KG$ action as a $\Gamma$ action.  For $\eta$ a
$\Gamma/\KG$-compatible $(\Gamma,Q)$ vector bundle, we then have
\[
\imath J\eta = \jmath J^{\Gamma/\KG}\eta.
\]

Section~\ref{sec:B209} reviews the construction of the indexed smash
product and HHR diagonal in this context.  To interface with the
terminology and notation there, we let $U$ be a complete
$\Gamma/\KG$-universe, regarded as a $\Gamma$-universe (a
$\KG$-trivial $\Gamma$-universe containing a copy of every
$\KG$-trivial orthogonal $G$-representation).   While we can define
the indexed smash product for any finite $\Gamma$-set, we will
restrict to considering finite
$\Gamma/\KG$-sets.  Then when $X$ is a $\Gamma/\KG$-spectrum, we have a natural isomorphism
\[
(\jmath X)^{(S)}\iso \jmath(X^{(S)}).
\]
For generalized orbit desuspension spectra $J\eta$, using the
construction of indexed smash powers in
Proposition~\ref{prop:indexedsma}, we can identify $(\imath
J\eta)^{(S)}$ as $\imath J(\eta^{(S)})$ with $\eta^{(S)}$ defined as
follows. 

\begin{defn}\label{defn:etaS}
Let $S$ be a finite $\Gamma/\KG$-set and let $\eta$ be a $(\Gamma,Q)$
vector bundle.  Write $|S|$ for the underlying set of $S$. Let
$\Sigma_{|S|}$ denote the group of automorphisms of the set $|S|$, 
let $\sigma_{S}\colon \Gamma\to 
\Sigma_{|S|}$ be the tautological homomorphism for the $\Gamma$ action on
$S$, and let $\Sigma_{|S|}(\sigma_{S})$ be the $(\Gamma \times
\Sigma_{|S|}^{op})$-set of Proposition~\ref{prop:indexedsma}, given by
$\Sigma_{|S|}$ with its natural right action and left
$\Gamma$ action induced by $\sigma_{S}$.  Define $\eta^{(S)}$ to be the
$(\Gamma,\Sigma_{|S|}\wr Q)$ vector bundle
\[
\eta^{(S)}= \Sigma_{|S|}(\sigma_{S})\times \eta^{|S|}
\]
where the left $\Gamma$ action is the diagonal of the left $\Gamma$
action on $\Sigma_{|S|}(\sigma_{S})$ and
the action of $\Gamma$ on each copy of $\eta$ coordinatewise; the
right action of 
$\Sigma_{|S|}\wr Q$ uses the right action of
$\Sigma_{|S|}$ on $\Sigma_{|S|}(\sigma_{S})$ (with trivial
$Q^{|S|}$ action), the action of $\Sigma_{|S|}$ on $\eta^{|S|}$ that
permutes the coordinates, and the product action of $Q^{|S|}$ on
$\eta^{|S|}$. 
\end{defn}

As a consequence of Proposition~\ref{prop:bundleprod} and
Proposition~\ref{prop:indexedsma}, we get the following observation. 

\begin{prop}\label{prop:bundleindexsma}
Let $S$ be a finite $\Gamma/\KG$-set and let $\eta$ be a $(\Gamma,Q)$
vector bundle.  Then $(\imath J\eta)^{(S)}\iso \imath J(\eta^{(S)})$.
\end{prop}

We use this in the following proposition, which allows us to study the
indexed smash power of a generalized orbit desuspension spectrum
$J^{\Gamma/\KG}\eta$ in terms of the indexed smash power of $\imath
J\eta'$ where we have taken an equivariant CW approximation of the
base space.  This kind of CW approximation does not preserve
$\Gamma/\KG$-compatibility and is the reason we consider
generalizations to $\Gamma$-equivariant $\aJ_{\Gamma/\KG}$-spaces.
In the statement, $\aF(1\times Q^{\op})$ denotes the
family of subgroups of $\Gamma \times Q^{\op}$ whose intersection with
$1\times Q^{\op}$ is trivial.

\begin{prop}\label{prop:bundlemultiwe}
Let $\eta$ be a $Q$-faithful $(\Gamma,Q)$ vector bundle, let $f\colon
B'\to B$ be an $\aF(1\times Q^{\op})$-equivalence of $(\Gamma\times
Q^{\op})$-spaces.  Assume that $B'$ is Hausdorff and let $f^{*}\eta$ denote the
pullback $(\Gamma,Q)$-vector bundle with base space $B'$.  Then: 
\begin{enumerate}
\item $f^{*}\eta$ is $Q$-faithful.
\item The induced map $J(f^{*}\eta)\to J\eta$ is a level equivalence.
\item For any finite $\Gamma/\KG$-set $S$, the induced map $(\imath
J(f^{*}\eta))^{(S)}\to (\imath J\eta)^{(S)}$ is a level equivalence.
\end{enumerate}
\end{prop}

\begin{proof}
Conclusion~(i) is clear from the fiberwise definition of $Q$-faithful.
For conclusion~(ii), we need to show that for an arbitrary orthogonal $\Gamma$-representation $W$ and an arbitrary 
closed subgroup $H <\Gamma$, the map
\[
(\aJ(f^{*}\eta,W)/Q)^{H}\to (\aJ(\eta,W)/Q)^{H}
\]
is a (non-equivariant) weak equivalence.  By 
Lemma~\ref{lem:bundlegeofix}, this is the $Q$-quotient of the induced
map on Thom spaces of the map of vector bundles (of the same local
dimension) which on total spaces is 
\[
\coprod_{\sigma \colon H \to Q}
\Im^{\perp}\aI(f^{*}\eta,W)^{H,\sigma}\to
\coprod_{\sigma \colon H \to Q}
\Im^{\perp}\aI(\eta,W)^{H,\sigma}.
\]
Since each side is a $Q$-stable subspace of a free $Q$-space
(Proposition~\ref{prop:faithfulfree}), the $Q$ actions are free.
Since these are Hausdorff spaces, to see that the map induces a weak
equivalence on $Q$-quotients, it is enough to see that the map
itself is a weak equivalence, and for this it is enough to see that
for each $\sigma \colon H \to Q$, the map
\[
\aI(f^{*}\eta,W)^{H,\sigma}\to
\aI(\eta,W)^{H,\sigma}
\]
is a weak equivalence.  Applying Lemma~\ref{lem:bundlegeofix} again, the previous map is
the map on total spaces of a map of locally trivial fiber
bundles which is the restriction $f^{H,\sigma}$ of $f$ on the base spaces
\[
f^{H,\sigma}\colon (B')^{H,\sigma}\to B^{H,\sigma}
\]
(with the identity map on the fibers).
The $\sigma$-twisted $H$ fixed points are the fixed points of
the subgroup 
\[
(H,\sigma)=\{(h,\sigma(h^{-1}))\mid h \in H \}<\Gamma \times Q^{\op}
\]
and hence by hypothesis, $f^{H,\sigma}$ is a weak equivalence.
This establishes~(ii).  To prove~(iii), we apply
Proposition~\ref{prop:bundleindexsma} and observe that $(f^{*}\eta)^{(S)}$
is obtained from $\eta^{(S)}$ as the pullback bundle of the
$(\Sigma_{|S|}\wr Q)$-equivariant map of base spaces
\[
\Sigma_{|S|}(\sigma_{S})\times (B')^{|S|}\to 
\Sigma_{|S|}(\sigma_{S})\times B^{|S|}.
\]
Because $B'\to B$ is assumed to be an $\aF(1\times
Q^{\op})$-equivalence, the displayed map is a
$\aF(1\times(\Sigma_{|S|}\wr Q)^{\op})$-equivalence: for any space
$X$, the fixed points of $X^{k}$ for graph subgroups of $\Sigma_{k}\wr
Q$ can be written in terms of products of fixed points of $X$ for
subgroups of $Q$ (see, for example,  Lemma~\ref{lem:dfplem}). Next we
observe that $\eta^{(S)}$ is $(\Sigma_{|S|}\wr Q)$-faithful: this is
clear when $Q$ is the trivial group since $\Sigma_{|S|}$ acts freely
on $\Sigma_{|S|}(\sigma_{S})$ and is covered by
Proposition~\ref{prop:trivsf} when $Q$ is non-trivial.
Conclusion~(ii) now implies conclusion~(iii).
\end{proof}

We defined the functor $\Fix^{\Gamma}_{\KG}$ from
$\Gamma$-equivariant $\aJ_{\Gamma/\KG}$-spaces to non-equivariant
$\aJ^{\Gamma}_{\Gamma/\KG}$-spaces in
Definition~\ref{defn:fixgk}. Although not needed in that section, 
left Kan extension along the functor $\phi \colon
\aJ^{\Gamma}_{\Gamma/\KG}\to \aJ$ produces a functor
$\Phi^{\Gamma}_{\KG}$ from $\Gamma$-equivariant
$\aJ_{\Gamma/\KG}$-spaces to non-equivariant spectra; the functor
$\Phi^{\Gamma}_{\KG}$ ; it
is canonically naturally isomorphic to the functor denoted $\Phi^{\Gamma}_{U}$ in
Section~\ref{sec:B209} for the universe $U$ described above.
More generally, for any
closed subgroup $H$ of $\Gamma$ containing $\KG$,  we likewise have a
functor $\Phi^{H}_{\KG}=\Phi^{H}_{U}$ defined using the 
functor $\Fix^{H}_{\KG}$ from $N_{\Gamma}H$-equivariant
$\aJ_{N_{\Gamma}H/\KG}$-spaces to non-equivariant
$\aJ^{H}_{N_{\Gamma}H/\KG}$-spaces by the formula  
\[
(\Fix^{H}_{\KG}Y)(W)=Y(W)^{H}
\]
and left Kan extension along the functor $\phi
\colon \aJ^{H}_{N_{\Gamma}H/\KG}\to \aJ$ defined as usual (on objects
it takes $V$ to $V^{H}$ regarded as a non-equivariant inner product
space, and on maps it takes the restriction of the equivariant
isometry to its fixed points and sends the embedding complement 
element by the identity).  The functors $\Phi^{H}_{\KG}$ are instances
of the geometric fixed point functors of~\cite[V\S4]{MM} when we
identify $\Gamma$-equivariant $\aJ_{\Gamma/\KG}$-spaces as
$\Gamma$-equivariant orthogonal spectra indexed on the $\KG$-trivial
orthogonal $\Gamma$-representations.  The work in \textit{ibid.} does
not restrict to the complete universe and the work there establishes
various point-set and homotopical properties of these functors,
including the existence of left derived functors, computed by applying
the point-set functors to cofibrant approximations (in the standard
stable model structure).

For $\Gamma/\KG$-spectra, we have natural isomorphisms relating
$\Fix^{H}_{\KG}$ and $\Phi^{H}_{\KG}$ to the underlying
non-equivariant spectra of $\Fix^{H/\KG}$ and $\Phi^{H/\KG}$ which are
essentially the identity map. As reviewed in Section~\ref{sec:B209},
the construction of the HHR 
diagonal extends to the case of $\Gamma$-equivariant $\aJ_{\Gamma/\KG}$-spaces to produce a natural diagonal map
\[
\Phi^{H}_{\KG}(Y)\to \Phi^{\Gamma}_{\KG}(Y^{(\Gamma/H)}),
\]
and for $X$ a $\Gamma/\KG$-spectrum, the
diagram of comparison isomorphisms
\[
\xymatrix{%
\Phi^{H}_{\KG}(\jmath X)\ar[r]\ar[d]_{\iso}
&\Phi^{\Gamma}_{\KG}((\jmath X)^{(\Gamma/H)})\ar[d]^{\iso}\\
\Phi^{H/\KG}(X)\ar[r]
&\Phi^{\Gamma/\KG}(X^{(\Gamma/H)})
}
\]
commutes.  The following proposition generalizes~\cite[B.209]{HHR} in
the current context.

\begin{thm}\label{thm:bundlediag}
Let $\eta\colon E\to B$ be a $(\Gamma,Q)$ vector bundle such that
$B$ is a $(\Gamma \times Q^{\op})$-equivariant CW complex with
isotropy groups in $F=\aF(1\times Q^{\op})$.  Then for any closed
subgroup $H$ of $\Gamma$ containing $\KG$, 
the diagonal map 
\[
\Phi^{H}_{\KG}\imath J\eta
\to
\Phi^{\Gamma}_{\KG}((\imath J\eta)^{(\Gamma/H)})
\]
is an isomorphism.
\end{thm}

\begin{proof}
By hypothesis, the base $B$ is built out of cells of the form
\[
(\Gamma \times Q^{\op})/M_{\sigma}\times (D^{n},\partial D^{n})
\]
for a $M_{\sigma}$ a subgroup of $\Gamma \times Q^{\op}$ in the family
$F$.  Restricted to such a cell, the fiber $V=E_{b}$ at a point has an
action of $M_{\sigma}$ and a trivialization giving an isomorphism 
\[
\eta|_{D^{n}}\iso ((\Gamma\times Q^{\op}) \times_{M_{\sigma}} V\to 
(\Gamma\times Q^{\op}) /M_{\sigma})\times D^{n}.
\]
We then see that $\eta$ is built out of cells of the form
\[
((\Gamma\times Q^{\op})_{+}\sma_{M_{\sigma}} \aJ(V,-))/Q^{\op}\sma (D^{n},\partial D^{n})_{+}
\]
where $M_{\sigma}$ is in $F$, and $V$ is an orthogonal
$M_{\sigma}$-representation.  Since $M_{\sigma}$ is in the family $F$,
letting $M$ denote the image of $M_{\sigma}$ under the projection to
$\Gamma$, we have
\[
M_{\sigma}=\{(h,\sigma(h)^{-1}))\mid h\in M\}
\]
for a unique homomorphism $\sigma \colon M\to Q$. Using the projection
isomorphism $M_{\sigma}\to M$ to regard $V$ as an orthogonal $M$-representation, we get an isomorphism of $\Gamma$-spectra
\[
((\Gamma\times Q^{\op})_{+}\sma_{M_{\sigma}} \aJ(V,-))/Q^{\op}
\iso \Gamma_{+}\sma_{M}\aJ(V,-). 
\]
This shows that $J\eta$ is cofibrant
in the complete model structure on $\Gamma$-spectra; the result here is now a special case of
Corollary~\ref{cor:B209}. 
\end{proof}

Finally, we prove Theorem~\ref{thm:dersym}.

\begin{proof}[Proof of Theorem~\ref{thm:dersym}]
Let $\eta\colon E\to B$ be a $Q$-faithful $\Gamma/\KG$-compatible
$(\Gamma,Q)$ vector bundle where each fiber is positive dimensional.
To prove the statement of Theorem~\ref{thm:dersym}, we just need to
check that $J^{\Gamma/\KG}\eta$ satisfies hypotheses~(i) and~(ii) of
Proposition~\ref{prop:charsym} for every $k>1$ and every $\Sigma <
\Sigma_{k}$.  Proposition~\ref{prop:exttosym} establishes
hypothesis~(i).  

For hypothesis~(ii), by restricting the group of equivariance and
changing notation if necessary, it suffices to consider finite
$\Gamma/\KG$-sets $S$.  Let $f\colon B'\to B$ be an $\aF(1\times
Q^{\op})$-colocal CW-approximation, i.e., an $\aF(1\times
Q^{\op})$-equivalence from a CW complex with cells whose isotropy
groups are in $\aF(1\times Q^{\op})$.  By
Proposition~\ref{prop:bundlemultiwe}, $J(f^{*}\eta)\to J\eta$ induces
a weak equivalence of $\Gamma$-equivariant $\aJ_{\Gamma/\KG}$-spaces
\[
(\imath J(f^{*}\eta))^{(S)}\to 
(\imath J\eta)^{(S)}
\]
for all finite $\Gamma/\KG$-sets
$S$. Proposition~\ref{prop:bundlemultiwe} also asserts that
$J(f^{*}\eta)\to J\eta$ is a level equivalence, and so induces
non-equivariant weak equivalences 
\[
\Fix^{H}_{\KG}\imath J(f^{*}\eta)\to
\Fix^{H}_{\KG}\imath J\eta
\]
for all closed subgroups $H<\Gamma$ containing $\KG$.  This gives a
commuting diagram in the stable category of $\Gamma$-spectra indexed on $\KG$-trivial representations
\[
\xymatrix{%
L\Phi^{H_{1}}_{\KG}(\imath J(f^{*}\eta))\sma^{L}\dotsb\sma^{L}
L\Phi^{H_{r}}_{\KG}(\imath J(f^{*}\eta))
\ar[d]_{\iso}\ar[r]
&(\imath J(f^{*}\eta))^{(S)}\ar[d]^{\iso}\\
L\Phi^{H_{1}}_{\KG}(\imath J\eta)\sma^{L}\dotsb\sma^{L}
L\Phi^{H_{r}}_{\KG}(\imath J\eta)
\ar[r]\ar[d]_{\iso}
&(\imath J\eta)^{(S)}\ar[d]^{\iso}\\
L\Phi^{H_{1}}_{\KG}(\jmath J^{\Gamma/\KG}\eta)\sma^{L}\dotsb\sma^{L}
L\Phi^{H_{r}}_{\KG}(\jmath J^{\Gamma/\KG}\eta)
\ar[r]
&(\jmath J^{\Gamma/\KG}\eta)^{(S)}
}
\]
for $S=\Gamma/H_{1}\amalg \dotsb \amalg \Gamma /H_{r}$, with the
indicated arrows isomorphisms.  Noting that the functor $\jmath$ is a
left Quillen adjoint (for the standard stable model structures
of~\cite[III\S4]{MM}) and that it preserves and reflects weak equivalences,
the isomorphism $\Phi^{\Gamma/\KG}\iso 
\Phi^{\Gamma}_{\KG}\jmath$ induces an isomorphism of derived functors
\[
L\Phi^{\Gamma/\KG}\iso 
(L\Phi^{\Gamma}_{\KG})\circ \jmath.
\]
In particular, the bottom map in the square above is an isomorphism in
the stable category of $\Gamma$-spectra indexed
on $\KG$-trivial representations if and only if the map
\[
L\Phi^{H_{1}/\KG}(J^{\Gamma/\KG}\eta)\sma^{L}\dotsb\sma^{L}
L\Phi^{H_{r}/\KG}(J^{\Gamma/\KG}\eta)
\to
(J^{\Gamma/\KG}\eta)^{(S)}
\]
is an isomorphism is the $\Gamma/\KG$-equivariant stable
category. Thus, we are reduced to showing that the map
\[
L\Phi^{H_{1}}_{\KG}(\imath J(f^{*}\eta))\sma^{L}\dotsb\sma^{L}
L\Phi^{H_{r}}_{\KG} (\imath J(f^{*}\eta))
\to
(\imath J(f^{*}\eta))^{(S)}
\]
is an isomorphism in 
the stable category of $\Gamma$-spectra indexed on $\KG$-trivial representations.

By Theorem~\ref{thm:varbundleflat}, each $\imath J(f^{*}\eta)$ is flat
for the smash product of $\Gamma$-equivariant $\aJ_{\Gamma/\KG}$-spaces, which implies that its smash product with
any $\Gamma$-equivariant $\aJ_{\Gamma/\KG}$-space represents the
derived smash product.  Since
\[
(\imath J(f^{*}\eta))^{(S\amalg S')}\iso (\imath
J(f^{*}\eta))^{(S)}\sma (\imath J(f^{*}\eta))^{(S')},
\]
it suffices to consider just the case when $S=\Gamma/H$.  
Although $\imath J(f^{*}\eta)$ is generally not a cofibrant
$\Gamma$-equivariant $\aJ_{\Gamma/\KG}$-space, by
Proposition~\ref{prop:bundlefixcw}, $\Fix^{H}_{\KG}(\imath
J(f^{*}\eta))$ is a cofibrant non-equivariant $\aJ^{H}_{\Gamma/\KG}$-space and this is enough to ensure that the
point-set functor $\Phi^{H}_{\KG}$ represents the derived functor 
$L\Phi^{H}_{\KG}$.  The argument is now completed
by Theorem~\ref{thm:bundlediag}.
\end{proof}

\section{Proof of Theorem~\sref$thm:calcgeoSym$}
\label{sec:calgeoSym}

This section proves Theorem~\ref{thm:calcgeoSym} and along the way
provides more details about the spectra in the decomposition given
there.  It is independent of the quotient fixed point lemma,
and could have been placed just after Section~\ref{sec:endbundle}: it
uses only the statement of Theorem~\ref{thm:bundlegeofix} and not the
techniques of its argument.

To compress notation, we write
$\Sigma'_{q}:=\Sigma_{q}\wr Q$.  The key idea in the construction is
that homomorphisms $\Lambda \to \Sigma'_{q}$ generalize the theory of
finite $\Lambda$-sets. When $Q$ is the trivial group, a homomorphism
$\sigma \colon \Lambda \to \Sigma_{q}$ specifies an action of
$\Lambda$ on $\{1,\dotsc,q\}$; if we denote the corresponding
$\Lambda$-set as $S(\sigma)$, then the (left) action by an element
$s^{-1}\in \Sigma_{q}$ specifies an isomorphism of $\Lambda$-sets
$S(\sigma)\to S(\sigma_{s})$, where (as in~\eqref{eq:Qaction}) $\sigma_{s}$ denotes
the homomorphism obtain by right conjugation by $s$,
\[
\sigma_{s}\colon \lambda\mapsto s^{-1}\sigma(\lambda)s.
\]
When conjugation by $s$ fixes $\sigma$, we get an
automorphism of $S(\sigma)$.  The corresponding theory when $Q$ is
non-trivial goes as follows.

Given a homomorphism $\sigma \colon \Lambda \to \Sigma'_{q}$, we call
$q$ the \term{cardinality} of $\sigma$.  We say that homomorphisms
$\sigma$ and $\sigma'$ of the 
same cardinality $q$ are \term{conjugate} when they are conjugate
under the $\Sigma'_{q}$ action on the set of all homomorphisms
$\Lambda \to \Sigma'_{q}$; homomorphisms of different cardinality are
never conjugate.  

Composing with the projection homomorphism $\Sigma'_{q}\to
\Sigma_{q}$, a homomorphism $\sigma \colon \Lambda \to \Sigma'_{q}$
gives a homomorphism $\Lambda \to \Sigma_{q}$, which defines an action
on $\{1,\dotsc,q\}$. We say that $\sigma$ is \term{irreducible} when
the action on $\{1,\dotsc,q\}$ is transitive.  Given any homomorphisms
$\sigma_{i} \colon \Lambda \to \Sigma'_{q_{i}}$, $i=1,\dotsc,r$, we
get a homomorphism $\sigma\colon \Lambda \to \Sigma'_{q}$,
$q=q_{1}+\dotsb+q_{r}$ by taking the block sum; we then write $\sigma
=\sigma_{1}\oplus \dotsb \oplus \sigma_{r}$ and call
$\sigma_{1},\dotsc,\sigma_{r}$ a block sum decomposition of $\sigma$.

An arbitrary homomorphism $\sigma \colon \Lambda \to \Sigma'_{q}$ is
conjugate to a homomorphism where the orbits
$O_{1},\dotsc,O_{k}$ consists of intervals in $\{1,\dotsc,q\}$.  When
the orbits are intervals, then $\sigma$ has a block sum
decomposition by irreducible homomorphisms $\sigma_{i}\to
\Sigma'_{q_{i}}$, $i=1,\dotsc,k$, with $q_{i}=\#O_{i}$.  

For $\sigma \colon \Lambda \to \Sigma'_{q}$, we write $C(\sigma)$ for
the subgroup of $\Sigma'_{q}$ whose conjugation action on $\sigma$ is
trivial.  The number of distinct conjugates of $\sigma$
is then $\#(\Sigma'_{q}/C(\sigma))$.  If $\sigma\colon \Lambda \to
\Sigma'_{q}$ is irreducible then any element in $C(\sigma^{\oplus
n})<\Sigma'_{nq}$ must have image in $\Sigma_{nq}$ that permutes the
blocks 
\[
\{1,\dotsc,q\},\{q+1,\dotsc,2q\},\dotsc,\{(n-1)q+1,\dotsc,nq\}.
\]
Numbering these blocks in the obvious way and using the induced isomorphism
between $\Sigma_{n}$ and the subgroup of $\Sigma_{nq}$ that permutes
these blocks, we get an isomorphism
\[
\Sigma_{n}\wr C(\sigma)\to C(\sigma^{\oplus n})<\Sigma'_{nq}.
\]
If $\sigma_{i}\colon \Lambda \to \Sigma'_{q_{i}}$, $i=1,\dotsc,r$ are
indecomposable and pairwise non-conjugate, then the isomorphism above
and block sum induces an isomorphism
\begin{equation}\label{eq:C}
(\Sigma_{n_{1}}\wr C(\sigma_{1}))\times \dotsb \times 
(\Sigma_{n_{r}}\wr C(\sigma_{r}))\to 
C(\sigma_{1}^{\oplus n_{1}}\oplus \dotsb \oplus \sigma_{r}^{\oplus n_{r}}).
\end{equation}

When $q>\#(\pi_{0}\Lambda)$, no homomorphism of cardinality $q$ can be
irreducible and so there are only finitely many total irreducible
homomorphisms of any cardinality.  Let $T$ be the number of conjugacy
classes of irreducible homomorphisms; we choose one in each conjugacy
class and enumerate them $\tau_{1},\dotsc,\tau_{T}$.  Let $t_{i}$
denote the cardinality of $\tau_{i}$, and let
$\vec t=(t_{1},\dotsc,t_{T})\in \bN^{T}$.  For $\vec
n=(n_{1},\dotsc,n_{T})$, define
\[
\tau(\vec n):=\tau_{1}^{\oplus n_{1}}\oplus \dotsb \oplus \tau_{T}^{\oplus n_{T}}
\colon \Lambda \to \Sigma'_{\vec n\cdot \vec t}
\]
when at least one of the $n_{i}$ is non-zero.  The cardinality of
$\tau(\vec n)$ is then given by the vector dot product $\vec
n\cdot\vec t$.  Let 
\[
\eta[\vec n]\colon 
(E^{\vec n\cdot \vec t})^{\Lambda,\tau(\vec n)}\{\KG\}\to 
(B^{\vec n\cdot \vec t})^{\Lambda,\tau(\vec n)}\{\KG\}
\]
be the $(\Gamma,C(\tau(\vec n)))$ vector bundle obtained by raising $\eta$
to the power of the cardinality of $\tau(\vec n)$, taking the
$\tau(\vec n)$-twisted $\Lambda$ fixed points, and restricting the
base to those points for which the $\tau(\vec n)|_{\KG}$-twisted
$\KG$ action on the fiber is trivial. (The $C(\tau(\vec n))$
action is the restriction to the fixed points of the $C(\tau(\vec
n))<\Sigma'_{\vec n\cdot \vec t}$ action on $\eta^{\vec n\cdot \vec
t}$.)  When $q=\vec n\cdot \vec t$ and we write
$(\Sym^{q}\eta)(\Lambda|\KG)|_{C(\tau(\vec n))}$ for the restriction of
$(\Sym^{q}\eta)(\Lambda|\KG)$ to a $(\Gamma,C(\tau(\vec n)))$ vector
bundle, $\eta[\vec n]$ includes as a $(\Gamma,C(\tau(\vec n)))$ vector
bundle in $(\Sym^{q}\eta)(\Lambda|\KG)|_{C(\tau(\vec n))}$ as the summand
indexed by $\tau(\vec n)$.  
Let $I\eta[\vec n]$ be the $(\Gamma,\Sigma'_{q})$ vector bundle 
\[
I\eta[\vec n]:= 
(\eta [\vec n])\times_{C(\tau(\vec n))}\Sigma'_{q}.
\]
Then formally, we get a map of $(\Gamma,\Sigma'_{q})$ vector bundles
$I\eta[\vec n]\to (\Sym^{q} \eta)(\Lambda|\KG)$. The summand of
$I\eta[\vec n]$ indexed by the coset $C(\tau(\vec n))s$ maps to the
summand of $(\Sym^{q} \eta)(\Lambda|\KG)$ indexed by $\tau(\vec
n)_{s}$. Since the cosets are in bijective correspondence with the
conjugates of $\tau(\vec n)$, the map $I\eta[\vec n]\to (\Sym^{q}
\eta)(\Lambda|\KG)$ is an inclusion.  Because every $\sigma \colon \Lambda
\to \Sigma'_{q}$ is conjugate to $\tau(\vec n)$ for some (unique)
$\vec n$ with $\vec n\cdot \vec t=q$, we get the following result.

\begin{prop}\label{prop:calcSym1}
For $q\geq 1$, the map of $(\Gamma,\Sigma'_{q})$ vector bundles
\[
\coprod_{\vec n\cdot \vec t=q} I\eta[\vec n]\to 
(\Sym^{q}\eta)(\Lambda|\KG)
\]
is an isomorphism.
\end{prop}

Because $J_{(\Gamma,C(\tau(\vec n)))}(\eta[\vec n])\iso
J_{(\Gamma,\Sigma'_{\vec n\cdot \vec t})}I\eta[\vec n]$, we then get
an isomorphism of $\Gamma/\Lambda$-spectra
\[
\bigvee_{\vec n\cdot \vec t=q}J^{\Gamma/\Lambda}(\eta[\vec n])\overto{\iso}
J^{\Gamma/\Lambda}((\Sym^{q}\eta)(\Lambda|\KG)).
\]

For $i=1,\dotsc,T$, let $\vec e_{i}$ be the vector with $1$ in the $i$
position and zero elsewhere so that $\tau(\vec e_{i})=\tau_{i}$.  The
decomposition of $\tau(\vec n)$ into indecomposable pieces decomposes
the $\tau(\vec n)$-twisted $\Lambda$ action on $\eta^{\vec n\cdot \vec
t}$ into a product, and we get the following further identification of
$\eta[\vec n]$. In the formula, when $n_{i}=0$, we should understand
$\Sym^{n_{i}}\eta[\vec e_{i}]$ as the trivial bundle $*\to *$.

\begin{prop}\label{prop:calcSym2}
Let $f_{\vec n}$ denote the group isomorphism~\eqref{eq:C} and
$f^{*}_{\vec n}$ the induced equivalence of $(\Sigma_{n_{1}}\wr
C(\tau_{1}))\times \dotsb \times \Sigma_{n_{r}}\wr
C(\tau_{r}))$-spaces with $C(\tau(\vec n))$-spaces. We then have an
isomorphism of $(\Gamma,C(\tau(\vec n)))$ vector bundles
\[
\eta[\vec n]\iso f^{*}_{\vec n}((\Sym^{n_{1}}\eta[\vec e_{1}])\times \dotsb \times
(\Sym^{n_{T}}\eta[\vec e_{T}]))
\]
for any non-zero $\vec n$.
\end{prop}

We can now prove Theorem~\ref{thm:calcgeoSym} 

\begin{proof}[Proof of Theorem~\ref{thm:calcgeoSym}]
We use the notation established above.  We assume that $\Sym^{q}\eta$
is $\Sigma'_{q}$-faithful for all $q$ so that by
Theorem~\ref{thm:bundlegeofix} (and Proposition~\ref{prop:Sym}), we
have an isomorphism of $\Gamma/\Lambda$-spectra 
\[
\Phi^{\Lambda}((J^{\Gamma/\KG}\eta)^{(q)}/\Sigma_{q})\iso 
J^{\Gamma/\Lambda}((\Sym^{q}\eta)(\Lambda|\KG))
\]
for all $q\geq 1$.  For each $q\geq 1$, if there is no irreducible
homomorphism $\sigma \colon \Lambda \to \Sigma'_{q}$ then set
$X_{q}=*$; otherwise let
\[
X_{q}=\bigvee_{t_{i}=q}J^{\Gamma/\Lambda}(\eta[\vec e_{i}]).
\]
In particular, $X_{q}=*$ for $q>\#(\pi_{0}\Lambda)$.
By Proposition~\ref{prop:calcSym1} and the observation below it, each
$X_{q}$ is then a wedge summand of
$\Phi^{\Lambda}((J^{\Gamma/\KG}\eta)^{(q)}/\Sigma_{q})$.
Propositions~\ref{prop:calcSym1} and~\ref{prop:calcSym2} (together
with Proposition~\ref{prop:Sym}) then imply that the inclusion of the
$X_{q}$ in $\Phi^{\Lambda}((J^{\Gamma/\KG}\eta)^{(q)}/\Sigma_{q})$
induces an isomorphism 
\[
\bP\biggl(\bigvee_{q}X_{q}\biggr)\to \Phi^{\Lambda}(\bP(J^{\Gamma/\KG}\eta)) 
\]
of commutative ring $\Gamma/\Lambda$-spectra.
\end{proof}

\section{The diagonal fixed point lemma and Theorem~\sref$thm:ifcrit$}
\label{sec:ifcrit}
\label{sec:dfpl}

In  this section, we prove a general lemma about fixed points of
cartesian powers of equivariant spaces with respect to subgroups of
wreath product groups.  The main application in this section is the
proof of Theorem~\ref{thm:ifcrit}, which like the arguments in the previous
section uses the statement of Theorem~\ref{thm:bundlegeofix} but not
the ideas in its proof.  We also used a version of the lemma in
Section~\ref{sec:neACyc}. 

The setup for the diagonal fixed point lemma is as follows.  Let
$\Lambda$ and $Q$ be groups, and $q\geq 1$.  Let $X$ be a left
$(\Lambda \times Q^{\op})$-space, and let $\sigma \colon \Lambda \to
\Sigma_{q}\wr Q$ be a homomorphism.  By definition, the
$\sigma$-twisted $\Lambda$ action on $X^{q}$ is the
diagonal of the given left $\Lambda$ action and the right action via $\sigma$.
\[
\lambda\cdot_{\sigma}(x_{1},\dotsc,x_{q}):= 
(\lambda x_{1},\dotsc,\lambda x_{q})\sigma(\lambda^{-1})
\]
We denote the fixed points of this action as
$(X^{q})^{\Lambda,\sigma}$ and the lemma below gives a diagonal
description of this action.  Since
$\Lambda$ and $Q$ are involved only in taking fixed points, any
topology on these groups is irrelevant; we only need to assume
their actions on $X$ are continuous.

To fix notation with simplest formulas, we
write the wreath product $\Sigma_{q}\wr Q$ as the semidirect
product $Q^{q}\rtimes \Sigma_{q}$ and write an element as
$(a_{1},\dotsc,a_{q};s)$ with $a_{i}\in Q$ and $s\in \Sigma_{q}$.
Then
\[
(a_{1},\dotsc,a_{q};s)
(a'_{1},\dotsc,a'_{q};s')
=(a_{1}a'_{s^{-1}(1)},\dotsc,a_{q}a_{s^{-1}(q)};ss').
\]
We also use the
$a_{i}$ and $s$ notation for writing the coordinates of $\sigma$: 
\[
(a_{1}(\lambda),\dotsc,a_{q}(\lambda);s(h)):=\sigma(\lambda),
\]
noting that $s\colon \Lambda \to \Sigma_{q}$ is a homomorphism, but
the multiplication intertwines the other coordinates and $s$.
To avoid confusion in notation with multiple arguments, we write the
action of $s\in \Sigma_{q}\wr Q$ on $\{1,\dotsc,q\}$ using brackets:
$s(\lambda)$ sends $i$ to $s(\lambda)[i]$. 
The formula for multiplication in the $a_{i}$ coordinates is then
\[
a_{i}(\lambda \mu)=a_{i}(\lambda)a_{s(\lambda^{-1})[i]}(\mu)
\]
The action of an element $(a_{1},\dotsc,a_{q};s)$ on
an element $(x_{1},\dotsc,x_{q})$ in $X^{q}$ in this notation is 
\[
(x_{1},\dotsc,x_{q})(a_{1},\dotsc,a_{q};s)
=(x_{s[1]}a_{s[1]},\dotsc,x_{s[q]}a_{s[q]}).
\]
For $\lambda \in \Lambda$, we then have
\[
\lambda \cdot_{\sigma}(x_{1},\dotsc,x_{q})
=(\lambda x_{s(\lambda^{-1})[1]}a_{s(\lambda^{-1})[1]}(\lambda^{-1}),\dotsc,
\lambda x_{s(\lambda^{-1})[q]}a_{s[q]}(\lambda^{-1})).
\]

In general, given an element of, or formula in $X^{q}$, we write a
subscript $i$ for its $i$th component; e.g., if $x\in X^{q}$,
$x_{i}$ denotes its $i$th component, and for $\lambda \in \Lambda$,
$(\lambda \cdot_{\sigma}x)_{i}$ denotes the $i$th component of
$\lambda \cdot_{\sigma}x$. 

For $i=1,\dotsc,q$, let $H_{i}$ denote the subgroup of $\Lambda$ that
is the inverse image under $s$ of the subgroup of $\Sigma_{q}$ that
fixes $i$; restricted to $H_{i}$, $a_{i}$ becomes a homomorphism,
which we denote as $\alpha_{i}\colon H_{i}\to Q$.  (We note that in
the typical case, when $\Lambda$ is a compact Lie group, $Q$ is a
discrete group, and the homomorphism $\sigma$ is continuous, we have
that the identity component $\Lambda_{0}$ of $\Lambda$ is contained in
$H_{i}$ for every $i$.)  We then have an $\alpha_{i}$-twisted $H_{i}$
action on $X$ defined by 
\[
h\cdot_{\alpha_{i}}x:=hx\alpha_{i}(h^{-1})
\]
(for $x\in X$, $h\in H_{i}$) using the original left $\Lambda$ action
and right $Q$ action on $X$.

The homomorphism $s$ gives an action of $\Lambda$ on the set
$\{1,\dotsc,q\}$; let $O_{1},\dotsc,O_{k}$ denote its orbits.  For
$i=1,\dotsc,q$, define $j_{i}\in \{0,\dotsc,k\}$ by $i\in O_{j_{i}}$.
For each $j=1,\dotsc,k$, choose an element $m_{j}\in O_{j}$. For each
$i=1,\dotsc,q$, choose an element $\ell_{i}\in \Lambda$ such that
$s(\ell_{i})[m_{j_{i}}]=i$.  We can now state the lemma.

\begin{lem}\label{lem:dfplem}
In the notation above, the map 
\[
f\colon X^{H_{m_{1}},\alpha_{m_{1}}}\times \dotsb \times X^{H_{m_{k}},\alpha_{m_{k}}}
\to X^{q}
\]
defined by
\[
f(x_{1},\dotsc,x_{k})_{i}=\ell_{i}x_{j_{i}}a_{m_{j_{i}}}(\ell_{i}^{-1})
\]
induces a homeomorphism onto the $\sigma$-twisted fixed points
$(X^{q})^{\Lambda,\sigma}$. Moreover, this map is independent of the
choice of the elements $\ell_{i}$ above.
\end{lem}

The map does depend on the choice of elements $m_{j}\in O_{j}$.

\begin{proof}
To show that every $\sigma$-twisted $\Lambda$ fixed point is in the
image, assume $x\in (X^{q})^{\Lambda,\sigma}$ and observe that
\begin{align*}
(\ell_{i}\cdot_{\sigma}x)_{i}
&=\ell_{i}x_{s(\ell_{i}^{-1})[i]}a_{s(\ell_{i}^{-1})[i]}(\ell_{i}^{-1})\\
&=\ell_{i}x_{m_{j_{i}}}a_{m_{j_{i}}}(\ell_{i}^{-1}).
\end{align*}
In the case when $i=m_{j}$ for some $j$ and $h\in H_{m_{j}}$, we have
\begin{align*}
(h\cdot_{\sigma}x)_{m_{j}}
&=hx_{s(h^{-1})[m_{j}]}a_{s(h^{-1})[m_{j}]}(h^{-1})\\
&=hx_{m_{j}}a_{m_{j}}(h^{-1}).
\end{align*}
Since $h\in H_{m_{j}}$, we have
$\alpha_{m_{j}}(h^{-1}):=a_{m_{j}}(h^{-1})$, and since $x\in
(X^{q})^{\Lambda,\sigma}$, we conclude that 
\[
x_{m_{j}}=(h\cdot_{\sigma}x)_{m_{j}}
=hx_{m_{j}}\alpha_{m_{j}}(h^{-1})=h\cdot_{\alpha_{m_{j}}}x_{m_{j}},
\]
i.e., $x_{m_{j}}\in X^{H_{m_{j}},\alpha_{m_{j}}}$.  So $x\in
(X^{q})^{\Lambda,\sigma}$ implies 
$x=f(x_{m_{1}},\dotsc,x_{m_{k}})$ 
with $x_{m_{j}}\in X^{H_{m_{j}},\alpha_{m_{j}}}$.
Next to show that $f$ lands in the $\sigma$-twisted $\Lambda$ fixed
points, for $\lambda \in \Lambda$, writing $s=s(\lambda^{-1})$, we
have 
\begin{align*}
(\lambda\cdot_{\sigma}f(x_{1},\dotsc,x_{k}))_{i}
&=(\lambda\ell_{j_{s[i]}})x_{j_{s[i]}}
    (a_{m_{j_{s[i]}}}(\ell_{j_{s[i]}}^{-1})a_{s[i]}(\lambda^{-1}))\\
&=(\lambda\ell_{j_{s[i]}})x_{j_{s[i]}}
    (a_{m_{j_{s[i]}}}(\ell_{j_{s[i]}}^{-1}\lambda^{-1}))
\end{align*}
since $s(\ell_{s[i]})[m_{j_{s[i]}}]=s[i]$.  Since $s(\lambda
\ell_{s[i]})[m_{j_{i}}]=s^{-1}s[i]=i$, $\lambda \ell_{j_{s[i]}}$ is an
alternative choice for the element $\ell_{i}\in \Lambda$.  Thus, to
see that $f(x_{1},\dotsc,x_{k})$ is fixed by $\lambda$ (under the
$\sigma$-twisted action), we just need to verify the last statement:
$f$ is independent of the choice of $\ell_{i}$.  Let $\ell'_{i}$
be another choice and $\mu=\ell_{i}^{-1}\ell'_{i}$ so that
$\ell_{i}\mu=\ell_{i}'$.  To reduce subscripts, let $j=j_{i}$.  Then
$s(\mu)[m_{j}]=m_{j}$ and $\mu \in H_{m_{j}}$, so
\[
\mu x_{j}a_{m_{j}}(\mu^{-1})=x_{j}
\]
since $x_{j}\in X^{H_{m_{j}},\alpha_{m_{j}}}$.  Then 
\begin{align*}
\ell_{i}x_{j}a_{m_{j}}(\ell_{i}^{-1})
&=\ell_{i}\mu x_{j}a_{m_{j}}(\mu^{-1})a_{m_{j}}(\ell_{i}^{-1})\\
&=(\ell_{i}\mu)x_{j}a_{m_{j}}(\mu^{-1}\ell_{i}^{-1})\\
&=\ell'_{i}x_{j}a_{m_{j}}((\ell'_{i})^{-1})
\end{align*}
with the second equality holding because $\mu$ fixes $m_{j}$.
\end{proof}

We now move on to Theorem~\ref{thm:ifcrit}, and we use the notations
and conventions from there, so $Q$ is a finite group, and $\Lambda$ is
a closed normal subgroup of a compact Lie group $\Gamma$.  As in the
previous section, we write $\Sigma'_{q}$ for $\Sigma_{q}\wr Q$.  We
look at a homomorphism $\sigma \colon \Lambda \to \Sigma_{q'}$ and the
$\sigma$-twisted $\Lambda$ action on $\Sym^{q}\eta$ for a $(\Gamma,Q)$
vector bundle $\eta$.  The previous lemma gives us a diagonal
description of the $\sigma$-twisted $\Lambda$ fixed points of both the
total space $E^{q}$ and the base space $B^{q}$.  If we restrict to
an element $b=(b_{1},\dotsc,b_{q})\in (B^{q})^{\Lambda,\sigma}$, the
$\sigma$-twisted $\Lambda$ action is fiberwise on $E_{b}$, and the
lemma identifies the space $E_{b}^{\Lambda,\sigma}$ as 
\[
E_{b_{m_{1}}}^{H_{m_{1}},\alpha_{m_{1}}}\times \dotsb \times
E_{b_{m_{k}}}^{H_{m_{k}},\alpha_{m_{k}}}.
\]
By inspection of the formula, this homeomorphism is an isomorphism
of vector spaces.  It may not be an isometry, but this is easy to
correct by scaling.  In the following proposition, for an element $v$
of $E_{b_{i}}$, let $\delta_{i}v$ denote the element of $E_{b}$ that
is $v$ in the $i$th coordinate and zero elsewhere; notation aside, the
homomorphism in it is just the composite of the one in the lemma with
the rescaling $v_{j}\mapsto v_{j}/\sqrt{\#O_{j}}$.  

\begin{prop}\label{prop:lsfp}
Define the isometry
\[
f\colon E_{b_{m_{1}}}^{H_{m_{1}},\alpha_{m_{1}}}
\times \dotsb \times 
E_{b_{m_{k}}}^{H_{m_{k}},\alpha_{m_{k}}}
\to E^{q}_{b}
\]
by the formula
\[
f(v_{m_{1}},\dotsc,v_{m_{k}})
=\sum_{j=1}^{k}\sum_{i\in O_{j}}
  \frac{\ell_{i}\cdot_{\sigma}(\delta_{m_{j}}v_{m_{j}})}{\sqrt{\#O_{j}}}
=\sum_{j=1}^{k}\sum_{i\in O_{j}}
  \frac{\delta_{i}f_{i}(v_{m_{j}})}{\sqrt{\#O_{j}}}.
\]
Then $f$ is an isomorphism onto $(E^{q}_{b})^{\Lambda,\sigma}$.
\end{prop}

We can now prove theorem~\ref{thm:ifcrit}.

\begin{proof}[Proof of Theorem~\ref{thm:ifcrit}]
We fix $q\geq 1$ and show that $\eta':=(\Sym^{q}\eta)(\Lambda)$ is $\Sigma'_{q}$-faithful.
We use the notation of the discussion above: we study
the coproduct factor corresponding to a fixed $\sigma \colon \Lambda\to
\Sigma'_{q}$ at a point $b=(b_{1},\dotsc,b_{q})\in (B^{q})^{\Lambda,\sigma}$.
For this, it suffices to show that for every 
$\tau =(c_{1},\dotsc,c_{q};t)$, not the
identity, with $b\tau=b$, there exists
$w\in (E^{q}_{b})^{\Lambda,\sigma}$ such that $w\tau\neq w$.
We treat this in three cases, one corresponding to each of the conditions
in Theorem~\ref{thm:ifcrit}.

The first case is when $t$ sends an element $i$ in $O_{j}$ to an
element $i'$ in $O_{j'}$ for some $j\neq j'$. Then by
Proposition~\ref{prop:lsfp} and
condition~(i), we can choose an element $w$ where the components in
$O_{j}$ are non-zero and the components in $O_{j'}$ are
zero.  For such a $w$, $w\tau \neq w$.

The second case is when $t$ preserves each of the orbits $O_{1},\dotsc,O_{k}$
but is not the identity, say $t(i)=i'\neq i$ for $i\in O_{j}$ for some
$j$.  Since the numbering of
the orbits was arbitrary and the choice of
orbit representatives above was arbitrary, we can assume without loss
of generality that $j=1$ and $m_{1}=i'$; we note that $\ell_{i}\notin
H_{m_{1}}$.  Since $b\tau=b$ and $\ell_{i}\cdot_{\sigma}b=b$, looking
at $i$th coordinates, we see that 
\[
b_{i}=b_{m_{1}}c_{m_{1}}=\ell_{i}b_{m_{1}}a(\ell_{i}^{-1})
\]
and so, writing $a:=c_{m_{1}}a(\ell_{i}^{-1})^{-1}$, we have
$\ell_{i}b_{m_{1}}=b_{m_{1}}a$.  By condition~(ii), there exists $v\in
E_{b_{m_{1}}}^{H_{m_{1}},\alpha_{m_{1}}}$ such that
$\ell_{m_{1}}v\neq va$.  Let $w=f(\sqrt{\#O_{1}}v,0,\dotsc,0)\in
(E^{q}_{b})^{\Lambda,\sigma}$.  Then
\[
w_{i}=\ell_{m_{1}}va(\ell_{m_{1}}^{-1}), \qquad 
(w\tau)_{i}=v c_{m_{1}}
\]
and we have chosen $v$ so that these two elements are unequal.

The final case is when $t$ is the identity.  Then there exists some
$i$ for which $c_{i}$ is not the identity.  Without loss of generality
$m_{1}=i$.  By condition~(iii), we can choose $v\in
E_{b_{m_{1}}}^{H_{m_{1}},\alpha_{m_{1}}}$ so that
$vc_{m_{1}}\neq v$.  Then 
\[
f(v,0,\dotsc,0)\tau=f(vc_{m_{1}},0,\dotsc,0)\neq f(v,0,\dotsc,0).\qedhere
\]
\end{proof}

\section[The HHR diagonal]%
{Indexed smash powers, HHR norms, geometric fixed points, and the HHR diagonal}
\label{sec:B209}

The purpose of this section is to generalize the HHR diagonal map~\cite[B.209]{HHR} 
\[
\Phi^{H}X\to \Phi^{G}(X^{(G/H)})
\]
that computes the geometric fixed points of an indexed smash power.
We generalize in two directions: we consider universes other than the
complete universe and we consider compact Lie groups that are not
finite.  This map does not exist for arbitrary universes and arbitrary
orbit spaces $G/H$; we specify below when results hold.  The results
in this section are known to experts, but statements have not appeared in
the generality we consider here, though the proofs in the existing
literature generalize to cover them. Since we expect these results to be
of interest beyond the scope of the rest of the paper, we have made
this section as self-contained as possible at the cost of a bit of
redundancy.

In this section $G$ denotes a compact Lie group.  We let $U$ be a
$G$-universe.  We write $V<U$ to mean
that $V$ is a finite dimensional $G$-stable subspace of $U$; for
finite dimensional $G$-inner product spaces, we also write $V<W$ to
indicate that $V$ is a subspace of $W$ (with inherited $G$ action and
inner product) and we emphasize that in this case the notation does not imply
proper subspace: $W<W$.  For $H<G$ a closed subgroup, we write $U_{H}$
for $U$ viewed as an $H$-universe (and write $V<U_{H}$ for $V$ a
finute dimensional $H$-stable subspace).  We use
the notation of Section~\ref{sec:startbundle} (the unadorned form of
the notation in~\cite{MM}) for the categories $\aI$ and $\aJ$: for
vector spaces with inner products $V$, $W$, $\aI(V,W)$ denotes the
space of isometric embeddings $V\to W$, $\Im^{\perp}\aI(V,W)\to
\aI(V,W)$ denotes the vector bundle over $\aI(V,W)$ whose fiber is the
orthogonal complement of the image of an embedding, and $\aJ(V,W)$
denotes the Thom space of this vector bundle.  When $V$ and $W$ are
orthogonal $G$-representations, $\aI(V,W)$ gets the usual $G$ action
on a mapping space, $\Im^{\perp}\aI(V,W)\to \aI(V,W)$ becomes an
$G$-equivariant vector bundle with its natural action, and $\aJ(V,W)$
becomes a based $G$-space. By~\cite[II\S4]{MM}, the functor $\aJ(V,-)$
specifies a $G$-equivariant orthogonal spectrum that represents the
$V$th space functor $X\mapsto X(V)$.  In~\cite{MM}, $\aJ(V,-)$ is more
usually denoted $F_{V}S^{0}$, and in~\cite{HHR}, it is denoted
$S^{-V}$.

The results in this section are of a point-set rather than homotopical
nature, and therefore, when convenient, we can treat $G$-equivariant
orthogonal spectra as indexed on $\bR^{\infty}$ instead of $U$: the
forgetful functor that only remembers the indexing on $\bR^{\infty}$
is an equivalence of categories~\cite[V.1.5]{MM}. The convenience of
this is that the category of $G$-equivariant orthogonal spectra indexed on
$\bR^{\infty}$ is just the category of left $G$-objects in
(non-equivariant) orthogonal spectra and this makes constructions like
indexed smash powers and HHR norms significantly easier to describe.
In so doing, however, we need to be careful about which concepts or
constructions depend on $U$ and which do not.  The
concepts/constructions that are independent of the universe $U$
include:
\begin{itemize}
\item The smash product $\sma$ of $G$-equivariant orthogonal spectra
and the smash product of a $G$-equivariant orthogonal spectrum with a
based $G$-space.
\item Indexed smash powers / HHR norms (when defined).
\item Limits and colimits.
\item Restriction / (additive) induction along inclusions of subgroups.
\end{itemize}
The concepts/constructions that depend in the universe $U$ include:
\begin{itemize}
\item Weak equivalences / homotopy groups.
\item Cofibrancy / fibrancy in standard or complete model structures.
\item The tautological presentation (coequalizer).
\item The geometric fixed point functor (both point-set and derived).
\end{itemize}
It is in the last two points that the universe plays a role in this
section.  When indexing $G$-equivariant orthogonal spectra on
$\bR^{\infty}$ and working with the universe $U$, we understand the
$V$th space $X(V)$ of a $G$-equivariant orthogonal spectrum to be
\[
X(V):=\aI(\bR^{n},V)_{+}\sma_{O(n)}X(\bR^{n})=\aJ(\bR^{n},V)\sma_{O(n)}X(\bR^{n})
\]
when $V$ is $n$-dimensional; see~\cite[V\S1]{MM}.

Indexing on $\bR^{\infty}$, finite $G$-set indexed
smash powers become easy to describe.  For a finite $G$-space $S$ and a
$G$-equivariant orthogonal spectrum $X$, we get an orthogonal spectrum
$X^{(S)}$ with $n$th space
\[
X^{(S)}(\bR^{n}):=(X(\bR^{n}))^{(S)}
\]
the smash power indexed on $S$, with $G$ action both on $X(\bR^{n})$
and on $S$. In some contexts, it is useful to have a description of
the indexed smash power that is universe $U$ intrinsic in the sense
that it does not require reindexing.  We gives such a construction in
the following proposition; its proof is clear when indexing on the
universe $\bR^{\infty}$ and the construction is preserved under the
reindexing to $U$. 

\begin{prop}\label{prop:indexedsma}
Let $G$ be a compact Lie group.  Let $k\geq 1$, $\Sigma<\Sigma_{k}$,
and $\sigma \colon G\to \Sigma$ a continuous
homomorphism.  Let $S$ be the finite $G$-set obtained from
$\{1,\dotsc,k\}$ using $\sigma$ for the action, and let
$\Sigma(\sigma)$ be the $(G \times \Sigma^{\op})$-set whose
underlying right $\Sigma$-set is $\Sigma$ and which has left $\Gamma$
action induced by the homomorphism $\sigma$.
Then for $G$-equivariant symmetric spectra $X$, there is a natural
isomorphism 
\[
X^{(S)}\iso (\Sigma(\sigma)_{+} \sma X^{(k)})/\Sigma.
\]
\end{prop}

When $S=G/H$ for $H<G$ a codimension zero subgroup, the indexed smash
power $X^{(G/H)}$ is a functor of the underlying $H$-equivariant orthogonal spectrum, namely the HHR norm~\cite[A\S4]{HHR}.  In the
case of finite groups, the HHR norm can be made without choices; when
$G$ is a compact Lie group, we seem to have to choose a set of coset
representatives $\gamma_{\zeta}\in G$ for $\zeta \in G/H$ (with $\zeta
=\gamma_{\zeta} H$).  We then get a continuous homomorphism 
\[
\sigma_{G/H} \colon G\to \Sigma_{G/H}\wr H
\]
defined by $\sigma_{G/H}(g)=(s,(h_{\zeta}))$ where $s$ is the
permutation on $G/H$ induced by right multiplication by $g$, and
$h_{\zeta}$ is defined by $\gamma
\gamma_{\zeta}=\gamma_{s\zeta}h_{\zeta}$. We note that for an
$H$-equivariant orthogonal spectrum $X$, the spaces of the $k$th smash
power $X^{(k)}(\bR^{n})$ come with an action of the group
$\Sigma_{k}\wr H$, making it naturally a $(\Sigma_{k}\wr
H)$-equivariant orthogonal spectrum; then, regarding $G/H$ as just a set, the smash
power $X^{(G/H)}$ has the structure of a $(\Sigma_{G/H}\wr
H)$-equivariant orthogonal spectrum, and we use $\sigma_{G/H}$ to make
it a $G$-equivariant orthogonal spectrum.  This defines the HHR norm 
$N_{H}^{G}X$.

\begin{defn}\label{defn:norm}
Let $G$ be a compact Lie group and $H<G$ a codimension zero subgroup.
Define the norm $N_{H}^{G}$ from $H$-equivariant orthogonal spectra to
$G$-equivariant orthogonal spectra by 
\[
N_{H}^{G}X(\bR^{n})=X^{(G/H)}(\bR^{n})
\]
with the $G$ action defined by $\sigma_{G/H}$ above.
\end{defn}

We have a completely analogous construction of $N_{H}^{G}$ as
a functor from based $H$-spaces to based $G$-spaces.  When
$X=\aJ(V,-)\sma A$ where $V$ is a finite dimensional $H$-inner product
space and $A$ is a based $H$-spaces, then by inspection
\[
N_{H}^{G}(\aJ(V,-)\sma A)\iso \aJ(\Ind_{H}^{G}V,-)\sma N_{H}^{G}A
\]
where $\Ind_{H}^{G}V$ is the (finite dimensional) $G$-inner product space
defined by $H$ to $G$ induction; it is characterized by the universal
property $\aI(\Ind_{H}^{G}V,-)^{G}\iso \aI(V,-)^{H}$. (It is
constructed analogously to $N_{H}^{G}$ but using direct sum in place
of smash product.)

While the construction in Definition~\ref{defn:norm} depends on the
choice of coset representatives, a different choice gives a naturally
isomorphic functor via multiplication by the element of $H^{G/H}$ that
transforms the representatives.

In the case when $X$ is a $G$-equivariant orthogonal spectrum, the
norm $N_{H}^{G}$ is isomorphic to the indexed smash power $X^{(G/H)}$,
where the isomorphism multiplies by the element of $G^{G/H}$ given by the
coset representatives (for the indexing on the universe $\bR^{\infty}$). 

The geometric fixed point functor does depend on the universe $U$, and
for $H<G$ a closed subgroup, we write $\Phi^{H}_{U}$ to keep track of
the universe when it is not complete.  The functor is constructed
in~\cite[V\S4]{MM} as the composite of the spacewise $H$ fixed point
functor and a left Kan extension.  Rather than write out the Kan
extension, we follow~\cite[B\S10]{HHR} and describe it in terms of the
resulting coequalizer.  This coequalizer is well-motivated in terms of the
the \term{tautological presentation} of a $G$-equivariant orthogonal
spectrum $X$, which is the coequalizer 
\[
\xymatrix@-1pc{%
\bigvee\limits_{V,W<U}\aJ(W,-)\sma \aJ(V,W) \sma X(V)
\ar@<-.5ex>[r]\ar@<.5ex>[r]
&\bigvee\limits_{V<U}\aJ(V,-) \sma X(V)\ar[r]&X.
}
\]
In the case of $H$ fixed points, we restrict the
$G$ action to $H$ and consider the tautological
presentation for that group
\[
\xymatrix@-1pc{%
\bigvee\limits_{V,W<U_{H}}\aJ(W,-)\sma \aJ(V,W) \sma X(V)
\ar@<-.5ex>[r]\ar@<.5ex>[r]
&\bigvee\limits_{V<U_{H}}\aJ(V,-) \sma X(V)\ar[r]&X.
}
\]
The idea is that $\Phi^{H}_{U}\aJ(V,-)\iso \aJ(V^{H},-)$ for every
$V<U_{H}$ and although $\Phi^{H}_{U}$ does not preserve all reflexive
coequalizers, it does preserve reflexive coequalizers that are
spacewise split on the underlying spaces indexed on $U_{H}$ (see~\cite[2.11]{ABGHLM}).
The geometric fixed point spectrum $\Phi^{H}_{U}X$ is then the
coequalizer
\[
\xymatrix@-1pc{%
\bigvee\limits_{V,W<U_{H}}\aJ(W^{H},-)\sma \aJ(V,W)^{H} \sma X(V)^{H}
\ar@<-.5ex>[r]\ar@<.5ex>[r]
&\bigvee\limits_{V<U_{H}}\aJ(V^{H},-) \sma X(V)^{H}
}
\]
where one map is induced by the $H$ fixed points of the action map
$\aJ(V,W)\sma X(V)\to X(W)$ and the other is induced by the map
\[
\aJ(W^{H},-)\sma \aJ(V,W)^{H}\to 
\aJ(W^{H},-)\sma \aJ(V^{H},W^{H})\to 
\aJ(V^{H},-)
\]
that takes an $H$ fixed element of $\aJ(V,W)$, restricts to the fixed
points in $V$ and $W$ to get an element of $\aJ(V^{H},W^{H})$, and
composes. This defines $\Phi^{H}_{U}X$ as a functor to non-equivariant
orthogonal spectra, but in the case of interest to us when $H<G$ is
codimension zero, we can also recover the action of the Weyl
group $W_{G}H$ via the evident (continuous) action of the normalizer
$N_{G}H$ on the wedge sums (which involves permutation of the indexing
$V,W<U_{H}$ and $V<U_{H}$).

If we want the equivariance for more general $H<G$, we restrict the
$G$ action to the normalizer $N=N_{G}H$ and use the coequalizer
\[
\xymatrix@-1pc{%
\bigvee\limits_{V,W<U_{N}}\aJ(W^{H},-)\sma \aJ(V,W)^{H} \sma X(V)^{H}
\ar@<-.5ex>[r]\ar@<.5ex>[r]
&\bigvee\limits_{V<U_{N}}\aJ(V^{H},-) \sma X(V)^{H}
}
\]
deriving from the tautological presentation for the group $N$.  This
coequalizer is isomorphic to the one above using the group $H$
by~\cite[3.8.12]{Kro-SOmega} (the statement is phrased in terms of Kan
extensions, but the argument is exactly a comparison of these
coequalizers).  We restate this using the notation of~\cite{MM}.

\begin{prop}[{\cite[3.8.12]{Kro-SOmega}}]\label{prop:Kro}
When $N\lhd G$ is a normal subgroup, the functor $\Phi^{N}_{U}$ from
$G$-equivariant orthogonal spectra (indexed on $U$) to non-equivariant
orthogonal spectra is naturally isomorphic to the composite of the
geometric fixed point functor $\Phi^{N}$ of~\cite[V.4.3]{MM} from
$G$-equivariant orthogonal spectra (indexed on $U$) to
$G/N$-equivariant orthogonal spectra (indexed on $U^{N}$) and the
forgetful functor to non-equivariant orthogonal spectra.
\end{prop}

\begin{rem}\label{rem:Kro}
The statement of~\cite[3.8.12]{Kro-SOmega} includes a hypothesis on
$N$ fixed points of $N$-representations versus $N$ fixed points of
$G$-representations that amounts to the following: If $N$ is a
closed normal subgroup of a compact Lie group $G$, and $V$ is a
non-trivial irreducible (finite dimensional) $N$-representation, then
there exits an irreducible (finite dimensional) $G$-representation $W$
that contains $V$ as a sub $N$-representation and satisfies
$W^{N}=0$.  This hypothesis always holds.
The Peter-Weyl theorem implies that $V$ is contained as a sub
$N$-representation of some finite dimensional $G$-representation $W$,
which without loss of generality can be taken to be irreducible.
Choosing any non-zero $v$ in $V$, the underlying vector space of $W$ is
then generated by elements of the form $gv$ for $g\in G$.  The
sub $N$-representation generated by $gv$ is isomorphic to
$\Ad_{g}^{*}V$, where $\Ad_{g}\colon N\to N$ is automorphism of $N$
given by conjugation by $g$.  Since by hypothesis, $V$ is a
non-trivial irreducible $N$-representation, so is $\Ad_{g}^{*}V$, and
this writes the underlying $N$-representation of $W$ as a sum of
non-trivial irreducible $N$-representations, showing that $W^{N}=0$.
\end{rem}

\begin{rem}
If $H$ is not normal in $G$, the coequalizer using the
tautological presentation for the group $G$ does not in general give
an isomorphic orthogonal spectrum.  To give a specific example,
consider the case when $G=A_{5}$ and
$H=\langle (12)(34)\rangle$, with $U$ a complete $G$-universe, and
$\bR(\sigma)<U_{H}$ the sign representation of $H$; in this example,
the coequalizers 
\begin{gather*}
\xymatrix@-1pc{%
\bigvee\limits_{V,W<U}\aJ(W^{H},0)\sma \aJ(V,W)^{H} \sma \aJ(\bR(\sigma),V)^{H}
\ar@<-.5ex>[r]\ar@<.5ex>[r]
&\bigvee\limits_{V<U}\aJ(V^{H},0) \sma \aJ(\bR(\sigma),V)^{H}
}
\\
\xymatrix@-1pc{%
\bigvee\limits_{V,W<U_{H}}\!\!\!\!
\aJ(W^{H},0)\sma \aJ(V,W)^{H} \sma \aJ(\bR(\sigma),V)^{H}
\ar@<-.5ex>[r]\ar@<.5ex>[r]
&\bigvee\limits_{V<U_{H}}\aJ(V^{H},0) \sma \aJ(\bR(\sigma),V)^{H}
}
\end{gather*}
are non-isomorphic: the bottom is a split coequalizer, giving
$\aJ(\bR(\sigma)^{H},0)=S^{0}$ as the resulting space, but in the top,
every $A_{5}$ representation $V$ containing the sign representation
also contains copies of the trivial representation, so we get $*$ for the
coequalizer.  (The $G$-equivariant orthogonal spectrum $X=G_{+}\sma_{H}
\aJ(\bR(\sigma),-)$ likewise gives different answers for the two
coequalizers.) The point is that when $H<G$ is normal, $V$ is a
non-trivial irreducible $H$-representation, and $W$ is an irreducible
$G$-representation containing $V$ as a sub $H$-representation, we have
$W^{H}=0$, but this can fail when $H$ is not normal.
\end{rem}

To see that $\Phi^{H}_{U}$ depends on the $U$ (and that
it does not preserve arbitrary reflexive coequalizers), consider the
case $X=\aJ(V,-)$ for various finite dimensional $G$-inner product
spaces $V$.  As already noted, we have
\[
\Phi^{H}_{U}\aJ(V,-)\iso \aJ(V^{H},-) \qquad \text{for }V<U_{H}
\]
but if $V$ is not $H$-isomorphic to a subspace of $U$, then
$\aI(V,W)^{H}$ is empty for all $W<U_{H}$ and so $X(W)^{H}=\aJ(V,W)^{H}=*$
for all $W<U_{H}$.  This gives  
\[
\Phi^{H}_{U}\aJ(V,-)\iso * \qquad \text{for }\aI(V,U)^{H}=\{\}.
\]

Now we restrict to the case when $H<G$ is zero codimensional so that
$G/H$ is a finite $G$-set.  Under the hypothesis that the $G$-universe $U$
is isomorphic to the $G$-universe $U\otimes G/H$, we obtain a diagonal
natural transformation
\[
\Phi^{H}_{U}\to \Phi^{G}_{U}((-)^{(G/H)})
\]
for $G$-equivariant orthogonal spectra
as a special case of the natural transformation 
\[
\Phi^{H}_{U_{H}}\to \Phi^{G}_{U}(N_{H}^{G}(-))
\]
defined as follows, using the outline given
in~\cite[B.209]{HHR}. First, as already observed, for a based 
$H$-space $A$, we have a norm construction $N_{H}^{G}$ defined
using $\sigma_{G/H}$ in the same manner as the spectral construction
in Definition~\ref{defn:norm}.  The diagonal map 
\[
A^{H}\overto{\iso} (N_{H}^{G}A)^{G}
\]
that sends an $H$ fixed point $a$ to the element of the smash product
which is $a$ in each factor is an isomorphism. Likewise, on a
finite dimensional $H$-inner product space $V$, we get  an isomorphism
of vector spaces  
\[
V^{H}\overto{\iso} (\Ind_{H}^{G}V)^{G},
\]
which with the appropriate scaling becomes an isometry.
Then for a based $G$-space $A$ and $V<U_{H}$, we get isomorphisms
\begin{align*}
\Phi^{G}_{U}(N_{H}^{G}(\aJ(V,-)\sma A))
&\iso \Phi^{G}_{U}(\aJ(\Ind_{H}^{G}V,-)\sma N_{H}^{G}A)\\
&\iso \aJ((\Ind_{H}^{G}V)^{G},-)\sma (N_{H}^{G}A)^{G}
\iso \aJ(V^{H},-)\sma A^{H}.
\end{align*}
Applying the composite functor $\Phi^{G}_{U}N_{H}^{G}$ to the
tautological presentation (for the group $H$) and applying the isomorphism above, we
precisely get the coequalizer diagram for $\Phi^{H}_{U_{H}}$, and that
induces the diagonal map from $\Phi^{H}_{U_{H}}\to
\Phi^{G}_{U}N_{H}^{G}$.  We state the observation here in the
following proposition for easy reference.

\begin{prop}\label{prop:diagexist}
Let $G$ be a compact Lie group and let $H<G$ be a codimension
zero subgroup.  Assume the $G$-universe $U$ is isomorphic to the
$G$-universe $U\otimes G/H$.  Then for $H$-equivariant orthogonal spectra indexed on $U_{H}$, there is a natural diagonal map of
non-equivariant orthogonal spectra
\[
\Phi^{H}_{U_{H}}\to \Phi^{G}_{U}N_{H}^{G}
\]
and for $G$-equivariant orthogonal spectra indexed on $U$, a natural diagonal map of
non-equivariant orthogonal spectra
\[
\Phi^{H}_{U}\to \Phi^{G}_{U}((-)^{(G/H)}).
\]
\end{prop}

As in the case covered in~\cite[B.209]{HHR}, for nice $X$, the
diagonal map is an isomorphism.  It is asserted and proved there that
(in the setting of a finite group and the complete universe), whenever
$X$ is cofibrant for the positive complete model structure, the
diagonal is an isomorphism; moreover, \cite[2.23]{HHR} remarks that
positivity is not needed, and a careful check of the proof shows that
it applies to objects that are cofibrant in the complete model
structure, that is, retracts of cell complexes built out of cells of
the form
\[
G_{+}\sma_{M}\aJ(V,-) \sma (D^{n},\partial D^{n})_{+}
\]
where $M<G$ and $V$ is an $M$-inner product space.  (Positivity would
require $V$ to have a non-zero $M$ fixed vector.)  The same result
holds with essentially the same proof for the diagonal map in any
of the universes above.

\begin{thm}\label{thm:B209}
Let $G$ be a compact Lie group, $H<G$ a codimension zero
subgroup, and $U$ a $G$-universe, not necessarily complete, with
$U\otimes G/H$ isomorphic to $U$.  If $X$ is an $H$-equivariant orthogonal spectrum that is cofibrant in the complete model structure
on the complete universe, then the $U$ universe diagonal
\[
\Phi^{H}_{U_{H}}X\to \Phi^{G}_{U}N_{H}^{G}X
\]
is an isomorphism.
\end{thm}

\begin{proof}
The outline of the argument for the finite group version of the
theorem in~\cite[B.209]{HHR} goes through without difficulty to reduce
to the case when $X$ is a single cell of the form $H_{+}\sma_{M}
\aJ(V,-)$.  In this case, when $M$ is a proper subgroup of $H$, both
sides are clearly the trivial orthogonal spectrum $*$, and the result
holds.  In the case when $M=H$, 
and $V$ is not isomorphic to a subspace of the universe $U_{H}$,
$\Ind_{H}^{G}V$ is also not isomorphic to a subspace of the universe
$U_{H}\iso (U\otimes G/H)_{H}$, and again both sides are the trivial
orthogonal spectrum $*$.  Finally, in the last case $M=H$, $V<U_{H}$,
and we have that the diagonal map
\[
\aJ(V^{H},-)\iso \Phi^{H}_{U_{H}}\aJ(V,-)\to
\Phi^{G}_{U}N_{H}^{G}\aJ(V,-)\iso \aJ((\Ind_{H}^{G}V)^{G},-)
\]
is an isomorphism by construction.
\end{proof}

Working with $G$-equivariant orthogonal spectra, we get the following
immediate corollary.

\begin{cor}\label{cor:B209}
Let $G$ be a compact Lie group, $H<G$ a codimension zero
subgroup, and $U$ a $G$-universe, not necessarily complete, with
$U\otimes G/H$ isomorphic to $U$.  If $X$ is an $G$-equivariant orthogonal spectrum that is cofibrant in the complete model structure
on the complete universe, then the $U$ universe diagonal
\[
\Phi^{H}_{U}X\to \Phi^{G}_{U}(X^{(G/H)})
\]
is an isomorphism.
\end{cor}

\begin{proof}
Under the hypothesis, the underlying $H$-equivariant orthogonal spectrum is cofibrant in the complete model structure on the complete
universe. 
\end{proof}

We state one more result on the diagonal in the case when $H$ is the
trivial group, a version of
Theorem~\ref{thm:B209} for more general cofibrant objects.  A
non-equivariant orthogonal spectrum is \term{convenient $O$-model 
cofibrant} if it is a retract of an orthogonal
spectrum built as cell complex using cells of the form
\[
A\sma (D^{n},\partial D^{n})_{+}
\]
where 
\[
A=(F_{\bR^{n}}S^{0})/M
\]
for $M$ a closed subgroup of $O(n)$.  (These are the
cofibrant objects in a model structure, the further details of which
are not relevant.)  As a special case we note that the cofibrant
objects in the positive convenient 
$\Sigma$-model structure of Definition~\ref{defn:Sigmodel} are
convenient $O$-model cofibrant as are the underlying orthogonal
spectra of cofibrant objects in the standard model structure on
commutative ring orthogonal spectra. 

\begin{thm}\label{thm:convB209}
Let $G$ be a finite group and $U$ a complete $G$-universe.  If $X$ is
a non-equivariant orthogonal spectrum that is convenient $O$-model
cofibrant, then the diagonal 
\[
X\to \Phi^{G}_{U}N_{1}^{G}X
\]
is an isomorphism.
\end{thm}

\begin{proof}
As in the proof of Theorem~\ref{thm:B209}, it suffices to treat the
case when $X=(F_{\bR^{n}}S^{0})/M$ for $M$ a closed subgroup
of $O(n)$.  This case reduces
to showing that for every orthogonal $G$-representation $W$ in $U$,
the map
\[
\aJ(\Ind_{1}^{G}\bR^{n},W)^{G}/M\to 
(\aJ(\Ind_{1}^{G}\bR^{n},W)/N_{1}^{G}M)^{G}
\]
is an isomorphism, where $N_{1}^{G}M$ is the group object in
$G$-spaces given by the cartesian product of copies of $M$ indexed by
$G$.  Since both sides are compact Hausdorff spaces and the map is
obviously injective, it suffices to show that it is surjective.  
Let $(\phi,w)\in \aJ(\Ind_{1}^{G}\bR^{n},W)$ be a non-base point whose image in
$\aJ(\Ind_{1}^{G}\bR^{n},W)/N_{1}^{G}M$ is $G$ fixed.  Here $\phi$
denotes an isometry $\Ind_{1}^{G}\bR^{n}\to W$ and $w$ a point in its
complement.  Since the image of $(\phi,w)$ is $G$ fixed, $w$ must be
$G$ fixed and for each $g\in G$, there exists an element $\sigma(g)$
in $N_{1}^{G}M$ such that 
\begin{equation}\label{eq:act}
g\circ \phi\circ g^{-1}=\phi\circ
\sigma(g)^{-1}.
\end{equation}
In the identification of $\Ind_{1}^{G}\bR^{n}$ and
$N_{1}^{G}M$ as cartesian products indexed on $G$, let $\phi_{1}$
denote the restriction of $\phi$ to the coordinate indexed by the
identity and let $\sigma_{1}(g)$ denote the coordinate of $\sigma(g)$
indexed by the identity.  Define $\psi \colon \Ind_{1}^{G}\bR^{n}\to
W$ by sending the coordinate indexed by $g$ to $g\circ \phi_{1}$.  We
see from~\eqref{eq:act}, the image of $g\circ \phi_{1}$ is the image
of $\phi$ on $g$th coordinate of $\Ind_{1}^{G}\bR^{n}$, and so $\psi$
is a linear isometry and $(\psi,w)$ specifies an element of
$\aJ(\Ind_{1}^{G}\bR^{n},W)$, which is in the same $N_{1}^{G}M$ coset as
$(\phi,w)$.  But by inspection, $(\psi,w)\in
\aJ(\Ind_{1}^{G}\bR^{n},W)^{G}$. 
\end{proof}


\bibliographystyle{plain}
\bibliography{bluman}

\end{document}